\UseRawInputEncoding
\documentclass{article}
\usepackage{pgfplots}
\pgfplotsset{compat=1.18}
\usepackage[T1]{fontenc}
\usepackage[utf8]{inputenc}
\usepackage{amsfonts}
\usepackage{hyperref}
\usepackage[normalem]{ulem}
\usepackage{bbm}
\usepackage{amsmath}
\usepackage{xcolor}
\usepackage{amsthm}
\usepackage{authblk}
\usepackage{pdflscape}
\usepackage{pgfplots}
\usepackage{appendix}
\usepackage{mathrsfs}
\usepackage{quiver}
\usepackage{geometry}
\newcommand{\N}{\mathbb{N}}
\newcommand{\Einf}{\mathbb{E}_{\infty}}
\newcommand{\Z}{\mathbb{Z}}
\newcommand{\R}{\mathbb{R}}
\newcommand{\OO}{\mathcal{O}}
\newcommand{\X}{\mathcal{X}}
\newcommand{\MCS}{\mathcal{S}}
\newcommand{\MSL}{\mathscr{L}}
\newcommand{\Y}{\mathcal{Y}}

\newcommand{\LL}{\mathbb{L}}
\newcommand{\sset}{{\rm sSet}}

\newcommand{\CC}{\mathcal{C}}
\newcommand{\DD}{\mathcal{D}}
\newcommand{\llog}{\mathbf{Log}}
\newcommand{\rtr}{\sqrt[\infty]{\R\llog_R^{\wedge}}}
\newcommand{\logfr}{\R\llog_R^{\wedge}}
\newcommand{\llogd}{{\mathbf{Log}}^{\Delta}}

\newcommand{\aff}{\mathbf {Aff}}

\newcommand{\ani}{\mathbf {Ani}}
\newcommand{\rr}{\longrightarrow}
\newcommand{\spec}{{\rm Spec}}
\newcommand{\spet}{{\rm Sp\Acute{e}t}}
\newcommand{\spf}{{\rm Spf}}

\newcommand{\mon}{\mathbf{CMon}}
\newcommand{\mond}{\mathbf{CMon}^{\Delta}}
\newcommand{\mone}{\mathbf{CMon}}
\newcommand{\stk}{\mathbf{Stk}}
\newcommand{\cat}{\mathbf{Cat}_{\infty}}

\newcommand{\F}{\mathscr{F}}

\newcommand{\colim}{{\rm colim}}
\newcommand{\qcoh}{\mathbf{QCoh}}
\newcommand{\map}{{\rm Map}}

\newcommand{\plog}{\mathbf{PreLog}}
\newcommand{\plogd}{\mathbf{PreLog}^{\Delta}}

\newcommand{\Fun}{{\rm Fun}}
\newcommand{\Hom}{{\rm Hom}}
\newcommand{\tang}{{\rm T}}
\newcommand{\shv}{\mathbf{Shv}}
\newcommand{\algcn}{\mathbf{CAlg}^{\rm cn}}
\newcommand{\MT}{{\mathfrak M}^T}
\newcommand{\MTD}{{\mathfrak M}^{T^{\DD}}}
\newcommand{\Der}{\mathbf {Der}}
\newcommand{\Dlog}{\mathbf {Der}^{\mathbf{Log}}}

\newcommand{\Mod}{\mathbf {Mod}}
\newcommand{\CMod}{\mathbf {CatMod}}
\newcommand{\Spt}{\mathbf {Spt}}
\newcommand{\Sptdm}{\mathbf {SpDM}}
\newcommand{\ddm}{\mathbf {DM}^{\Delta}}
\newcommand{\alg}{\mathbf{CAlg}}
\newcommand{\tlog}{{\rm T}^{\rm rep}_{\mathbf{Log}} }
\newcommand{\MTL}{\mathfrak{M}^{\rm rep}_{\mathbf{Log}}}
\newcommand{\et}{{\rm \Acute{e}t}}
\newcommand{\set}{{\rm s\Acute{e}t}}

\theoremstyle{plain}
\newtheorem{thm}{Theorem}[section]
\newtheorem{cor}[thm]{Corollary}
\newtheorem{prop}[thm]{Proposition}
\newtheorem{lem}[thm]{Lemma}

\theoremstyle{definition}
\newtheorem{defn}[thm]{Definition}
\newtheorem{rmk}[thm]{Remark}
\newtheorem{cau}[thm]{Caution}
\newtheorem{rmkcaution}[thm]{Remark-Caution}

\newtheorem{exmp}[thm]{Example}

\newtheorem{cov}[thm]{Convention}

\newtheorem{no}[thm]{Notation}

\theoremstyle{remark}

\author{Ruichuan Zhang}
\date{\ }
\makeatletter
\let\c@equation\c@thm
\makeatother
\numberwithin{equation}{section}
\title{Derived logarithmic deformation theory and moduli stacks  of  derived logarithmic structures}
\begin{document}
\maketitle
\begin{abstract}
  This paper investigates the derived and spectral analogs of logarithmic geometry. We develop the deformation theory for animated log rings and $\mathbb{E}_\infty$-log rings and examine the corresponding theories of derived and spectral log stacks. Furthermore, we define moduli stacks for derived and spectral log structures and establish their representability. As an application, we will construct infinite root  stacks in the derived and spectral settings and  study the associated geometric properties.
\end{abstract}
\tableofcontents
\section{Introduction}
This article generalizes some concepts of logarithmic geometry from \cite{kato1996log} and \cite{ogus2018lectures}, to the settings of derived and spectral stacks, as in \cite{lurie2004derived} and \cite{lurie2018spectral}. While affine derived logarithmic geometry has been explored by Rognes \cite{rognes2009topological}, Sagave-Sch{\"u}rg-Vezzosi \cite{sagave2016derived}, Binda-Lundemo-Park-{\O}stvaer \cite{binda2023hochschild}, \cite{10.1093/imrn/rnad224}  and Lundemo \cite{lundemo2023deformation},  this work globalizes these constructions. We develop a theory of \textit{spectral} and \textit{derived log Deligne–Mumford stacks} and study their associated moduli stacks of logarithmic structures.\\ 
This article has two parts. The first part is the formulation of the deformation theory of $\Einf$-log rings. The second part is devoted to the study of log stacks: we will prove the representability of the moduli stack of log structures and globalize the deformation theory of $\Einf$-log rings for log stacks. As an application, we study the $\infty
$-root stacks in both derived and spectral settings.
\subsection{Moduli of classical log structures and log invariants}
Logarithmic geometry, originally introduced by Fontaine and Illusie, provides a framework for studying the geometry and arithmetic of open and singular varieties. Its foundations were later systematically established by Kato \cite{Kato1988}. This theory centers on \textit{logarithmic schemes} (or \textit{log schemes} for brevity) rather than classical schemes. A log scheme is a triple $(X, \mathcal{M}, \alpha)$, where $X$ is a scheme, $\mathcal{M}$ is an étale sheaf of monoids on $X$, and $\alpha: \mathcal{M} \to \mathcal{O}_X$ is a morphism of sheaves of monoids inducing an isomorphism $\mathcal{M} \times_{\mathcal{O}_X} \mathcal{O}_X^* \xrightarrow{\sim} \mathcal{O}_X^*$.\\
In practice, many log schemes arise from non-proper or non-smooth classical schemes. For example, let $k$ be a field and $X$ a singular variety over $k$. Suppose there exists a resolution of singularities $\pi: X' \to X$ (for instance, if ${\rm char}(k)=0$ or $\dim X\leq 3$) such that $X'$ is smooth and the exceptional divisor $D$ is a strict normal crossing divisor. Then the pair $(X', D)$ gives rise to a log smooth scheme $(X',\mathcal D)$, where $\mathcal D$ is the log structure on $X'$ induced by $D$ (see \cite[Examples 3.1.20]{ogus2018lectures} for details).\\
Another key example arises in the study of schemes over local fields with \textit{semi-stable reductions}. Let $R$ be a discrete valuation ring with maximal ideal $\mathfrak{m}$, and let $f: X \to \spec R$ be a proper, flat, surjective morphism. Assume the generic fiber $X_{\eta}:=X\times_R \spec({\rm Frac }R)$ is smooth over ${\rm Frac}R$, and the special fiber $X_s:=X\times_R \spec R/\mathfrak m$ is a strict normal crossing divisor on $X$. This pair $(X,X_s)$ also defines a log structure on $X$ (see \cite[Examples 1.2.7]{ogus2018lectures} for details).
\subsubsection{Moduli of log structures}
According to the fundamental insight of Olsson \cite{olsson2003logarithmic}, which builds on the ideas of Deligne and Illusie, the concepts of log geometry can be systematically reformulated using the theory of algebraic stacks. The definition of log schemes involves sheaves of monoids, which are poorly structured objects that are difficult to handle directly. To overcome this difficulty, one may study the moduli stack of log structures, thereby translating the study of mappings between log schemes into the study of relative stacks.\\
Let $(S,\mathcal{L})$ be a \textit{fine log scheme}\footnote{This means that there is an \'etale cover $\pi:U\rightarrow S$, and a prelog ring $(A,P)$, where $P$ is an integral and finitely generated monoid, such that $(U,\pi^*\mathcal 
 L)$ is isomorphic to the log scheme associated with $(A,P)$. see \cite[Definition 2.1.8]{ogus2018lectures} for  details.} we can define the moduli functor of points of log structures:
\begin{align*}
\llog_{(S,\mathcal L)}:&\mathbf{Sch}^{\rm op}_{/S}\rr{\mathbf{Grpd}},\\
& X/S\mapsto\{\text{maps of fine log schemes }(X,\mathcal M)\rightarrow(S,\mathcal L)\}
\end{align*}
Olsson proved the following result.
\begin{thm}\label{OlssonRep}\cite[Theorem 1.1]{olsson2003logarithmic}
    The functor $\llog_{(S,\mathcal L)}$ is an Artin stack that is locally finitely presented over $S$.
\end{thm} 
Using Theorem\autoref{OlssonRep}, one can rephrase  certain fundamental notions of log geometry in terms of the language of algebraic stacks. In \cite{olsson2003logarithmic}, Olsson shows that the notions of log-smoothness and log-\'etaleness for log schemes can be recovered from the classical smoothness and \'etaleness for maps of stacks. 

\subsubsection{Semi-classical log invariants from moduli of log structures}
A key feature of Olsson's stack $\llog_{(S,\mathcal{L})}$ is its utility in studying certain algebraic invariants of fine log schemes from a stack-theoretic perspective. Let $\mathbf H$ be an invariant defined for relative algebraic stacks,  i.e., it is defined for morphisms ${\mathbf H}_{X/Y}$ and is functorial in both the source and target. More precisely, $\mathbf H$ is a functor $\mathbf H:\Fun(\Delta^1,{\mathbf{ArtStk}})\rightarrow\mathcal C$, where $\mathcal C$ is a complete ($\infty$)-category. One can then define a new invariant for log schemes associated with $\mathbf H$ as follows:
$$\mathbf H^O:\Fun(\Delta
^1,\mathbf{LogSch})\rr \mathcal C, ((X,\mathcal M)\rightarrow(S,\mathcal L))\mapsto\mathbf{H}_{X/\llog_{(S,\mathcal L)}}.$$
In \cite{olsson2005logarithmic} and  \cite{olsson2024hochschildhomologylogschemes}, Olsson employed this approach to study the log cotangent complex ( i.e., Andr\'e-Quillen homology) and the log Hochschild homology functor, which we denote by $\LL^O$ and ${\rm HH}^O$, respectively.\\
In practice, a complication arises because such a logarithmically defined invariant is often "semi-classical" and may not yield (in general) the desired logarithmic invariants. For example, take $\mathbf{H}$ to be the cotangent complex functor $\LL$ in the sense of \cite{illusie1971complexe1}, thereby defining Olsson's cotangent complex $\LL^O_{(X,\mathcal M)/(S,\mathcal L)}:=\LL_{X/\llog_{(S,\mathcal L)}}\in {\mathbf D}_{\rm qcoh}(X)$ for a morphism of fine log schemes $f:(X,\mathcal M)\rightarrow(S,\mathcal L)$. The $0$-th homotopy group $\pi_0\LL^O{(X,\mathcal M)/(S,\mathcal L)}$ is naturally equivalent to the sheaf of logarithmic differential forms $\Omega^1_{(X,\mathcal M)/(S,\mathcal L)}$ \cite[Lemma 3.8]{olsson2005logarithmic}. However, the higher homotopy groups (${\rm deg}\geq 3$) of $\LL^O_{(X,\mathcal M)/(S,\mathcal L)}$ lack clear geometric interpretations, as noted in \cite[Theorem 8.32]{olsson2005logarithmic}.\\
On the other hand, in deformation theory, the cotangent complex functor is typically characterized by certain universal properties; namely, it should control the higher deformations of the geometric objects under consideration. In particular, one hopes that the "correct" log cotangent complex should control the higher deformations of log schemes (or log rings) in the sense of \cite[Section 3.3]{binda2023hochschild}.\\
To this end, there is another definition of the log cotangent complex due to Gabber, constructed in \cite[Definition 8.5]{olsson2005logarithmic} (see also \cite[Definition 3.5]{binda2023hochschild}), which we refer to as the \textit{(algebraic) Gabber's cotangent complex} $\LL^G$. This complex is first defined for maps of (animated) log rings $f:(R,L)\rightarrow(A,M)$ as the non-abelian derived functor of the module of log differential forms. It is then generalized to maps of (derived) log schemes via descent arguments, as detailed in \cite[Section 8.29]{olsson2003logarithmic} and \cite[Proposition 7.20]{binda2023hochschild}. The cotangent complex of the morphism $f$ is characterized by the following universal property:
$$\map_{A}(\LL^G_{(A,M)/(R,L)},J)\simeq\map_{(R,L)//(A,M)}((A,M),(A\oplus J,M\oplus J)),J\in\Mod_A^{\rm cn}$$
where $\Mod_A$ denotes the derived $\infty$-category of $A$-modules and $\Mod_A^{\rm cn}$ denotes its subcategory of connective objects. Furthermore, the mapping space $\map_{(R,L)//(A,M)}((A,M),(A\oplus J,M\oplus J))$ is homotopy equivalent to the anima of $(R,L)$-linear \textit{log derivations} \cite[Section 3.3]{binda2023hochschild}. This implies that the log cotangent complex functor corepresents log derivations. If the (animated) log rings $(A,M)=(A,GL_1(A))$ and $(R,L)=(R,GL_1(R))$ are (animated) rings endowed with trivial log structures, then Gabber's cotangent complex $\LL^G_{(A,GL_1(A))/(R,GL_1(R))}$ coincides with the ordinary cotangent complex $\LL_{A/R}$ defined in \cite{illusie1971complexe1} and \cite{lurie2004derived}. Unfortunately, it's not possible to reconstruct the space of log derivation using Olsson's stack of log structures. 
\subsubsection{Olsson's cotangent complex versus Gabber's cotangent complex}
Recall that a map of schemes $f:X\rightarrow S$ is smooth (resp. étale) if it is locally of finite presentation of relative dimension $d$ (resp. $0$), flat, and the sheaf of relative differentials $\Omega^1_{X/S}$ is a vector bundle of constant rank $d$ (resp. $0$). This definition is equivalent to the following derived characterization: $f$ is locally of finite presentation and the relative cotangent complex $\LL_{X/S}$ is a locally free sheaf of finite rank (resp. is contractible) \cite[Proposition 3.4.9]{lurie2004derived}. Consequently, the classical and derived notions of smoothness (resp. étaleness) coincide for schemes.\\
Furthermore, the theory of cotangent complexes satisfies  conormal fiber sequence and  flat base change formula. Specifically, for a sequence of maps of schemes (or algebraic stacks) $X\stackrel{f}\rightarrow Y\stackrel{g}\rightarrow S$, there exists a canonical fiber sequence of quasicoherent sheaves on $X$:
$$f^*\LL_{Y/S}\rr\LL_{X/S}\rr\LL_{X/Y}$$
Unfortunately, this will never happen for classical logarithmic geometry. As argued by Bauer and Illusie, following a proposal by Olsson in \cite[Section 7.1]{olsson2005logarithmic}, there exists no cotangent complex functor $\LL$ for maps of classical log schemes that satisfies the following properties:
\begin{enumerate}
    \item If $(X,\mathcal M)\rightarrow(S,\mathcal L)$ is strict, then $\LL_{(X,\mathcal M)/(S,\mathcal L)}\simeq\LL_{X/S}$;
    \item If  $(X,\mathcal M)\rightarrow(S,\mathcal L)$ is classically log-smooth, then $\LL_{(X,\mathcal M)/(S,\mathcal L)}\simeq\Omega_{(X,\mathcal{M})/(S,\mathcal{L})}^1$;
    \item Given a sequence of map $(X,\mathcal M)\stackrel{f}\rightarrow(Y,\mathcal N)\stackrel{g}\rightarrow(S,\mathcal L)$, then there is a functorial fiber sequence
    $$f^*\LL_{(Y,\mathcal N)/(S,\mathcal L)}\rr\LL_{(X,\mathcal M)/(S,\mathcal L)}\rr\LL_{(X,\mathcal{M})/(Y,\mathcal N)}.$$
\end{enumerate}
Note that Olsson’s cotangent complex $\LL^{\rm O}$ satisfies conditions (1) and (2) but not (3), whereas Gabber’s cotangent complex $\LL^{\rm G}$ satisfies condition (1) and (3) but not (2). The existence of a satisfactory cotangent complex theory for log schemes is precluded for three principal reasons:
\begin{enumerate}
    \item The underlying scheme map of a log smooth map may not have good behavior, in particular, it might not be flat\footnote{For example, let $V$ be a smooth complex algebraic variety with trivial log structure. Choose a closed point $x\in V$. Consider the blow-up ${\rm Bl}_xV$ of $V$ along $x$. The exceptional divisor $E\subset {\rm Bl}_xV$ defines a log structure $\mathcal{ E}$ on ${\rm Bl}_xV$. We have a log-scheme map $({\rm Bl}_xV,\mathcal{ E})\rightarrow (V,\mathcal{ O}_V^*)$.
    This  is a log-\'etale map, but the underlying map of schemes is not  flat.}; 
    \item The notions of classical log-smoothness (characterized by liftability against square-zero extensions) and derived log-smoothness (characterized by Gabber’s cotangent complex or liftability against higher square-zero extensions) do not coincide. Derived log-smoothness is strictly stronger than its classical counterpart;
\item The derived log-smoothness is not preserved by pullbacks in the ordinary category of quasi-coherent log-schemes\footnote{See \cite[Example 4.9]{binda2023hochschild} for an example that classically log-smooth maps are not necessarily derived log-smooth, and the derived log-smoothness is not preserved by classical base change.}.
\end{enumerate}
Let us analyze what the issues listed above mean. For issue $(1)$ and $(3)$, this is a flat pullback problem. The pullback of logarithmic cotangent complexes only involves the underlying scheme maps, and the base change behaves ugly if we work in the category of classical schemes. This issue disappears if we work in the world of derived or spectral log schemes (or stacks) instead of only dealing with classical objects, as 
we can drop the flatness assumption for pullback operations in derived and spectral geometry.
 For issue $(2)$, the point is that  Gabber's cotangent complex $\LL^G$, which characterizes higher deformations, is not equivalent to Olsson's cotangent complex. In particular, Gabber's cotangent complex  of a classical log-\'etale map might fail to be discrete. To fix this issue,  we work in the $\infty$-category of derived or spectral log schemes instead of the ordinary category of classical log schemes, and deal with the notions of derived log-smoothness and \'etaleness. In this setting,  smoothness means that the underlying scheme map  is finitely presented and Gabber's cotangent complex is locally free of finite rank.
\begin{rmk}
 In practice,  both of Olsson's and Gabber's cotangent complexes  have their own advantages and disadvantages. We will work with Gabber's cotangent complex because it is more adapted to the derived log geometry.
\end{rmk}
\subsection{Logarithmic deformation theory}
As explained in \cite{lurie2012derived}, Lurie posits that the geometric objects in spectral (or derived) algebraic geometry should be understood as (higher) deformations of  objects arising from classical algebraic geometry. For instance, any connective $\Einf$-ring $R$ admits a functorial tower—known as the Postnikov tower—of truncated $\Einf$-rings:
$$...\rr\tau_{\leq n}R\rr\tau_{\leq n-1}R\rr...\rr \tau_{\leq 0}R=\pi_0R$$
in which each successive map is a square-zero extension and the natural map $R \xrightarrow{\simeq} \varprojlim_{n} \tau_{\leq n}R$ is an equivalence of $\Einf$-rings. A key aspect of higher deformation theory is that such square-zero extensions are governed by the cotangent complex functor. Lurie established the following foundational results in this context.
\begin{thm}\cite[Theorem 7.4.2.7, simplified version]{lurie2017higher}
    Let $R$ be a connective  $\Einf$-ring, and let $A$ be a connective $R$-algebra. Denote by $\LL_{A/R}$ the corresponding relative cotangent complex. Then
    \begin{enumerate}
        \item $\LL_{A/R}$ is initial in the $\infty $-category of $R$-derivations of $A$, and for any $I\in\Mod_A^{\rm cn},$ there is a canonical equivalence 
        $$\map_{A}(\LL_{A/R},I)\simeq\map_{R//A}(A,A\oplus I).$$
        \item $\LL_{A/R}$ classifies square-zero extensions of $A$, that is , for any $I\in\Mod_A^{\rm cn}$, let $\mathcal C$ be the subcategory of $\algcn_R$ spanned by square-zero extensions of $A$, with kernel $I$. Then there is a canonical equivalence 
        $$\map_{A}(\LL_{A/R},I[1])\simeq\mathcal C^\simeq.$$
    \end{enumerate}
\end{thm}
\begin{thm}\cite[Theorem 7.5.1.11, simplified version]{lurie2017higher}
    Let $A$ be a connective $\Einf$-ring. The base change along the canonical map $A\rightarrow\pi_0A$ induces an equivalence of $\infty$-categories (in fact, they are ordinary categories) $A_{\et}\simeq (\pi_0A)_{\et}.$ 
    Here $A_{\et}$ and $(\pi_0 A)_{\et}$ are the subcategories of $\algcn_A$ and $\algcn_{\pi_0A}$ spanned by \'etale objects, respectively.
\end{thm}
The similar  results of the above two theorems also hold for animated rings, which are
formulated in Lurie's PhD thesis \cite{lurie2004derived}.
\subsubsection{Towards logarithmic geometry}
The primary technical framework for formulating $\Einf$-deformation theory is Lurie's \textit{tangent bundle formalism} \cite[Chapter 7]{lurie2017higher}. This formalism is particularly well-suited for constructing the deformation theory of $\Einf$-rings. It yields presentable fibrations
$$\begin{tikzcd}
\tang_{\alg} \arrow[r, "\Omega^{\infty}"] \arrow[rr, "\pi"', bend right] & {\Fun(\Delta^1,\alg)} \arrow[r, "{\rm ev}_1"] & \alg
\end{tikzcd}$$
where the functor $\Omega^{\infty}$ is the \textit{relative stabilization} of the functor ${\rm ev}_1$, and the fiber of $\pi$ over an $\mathbb{E}_{\infty}$-ring $A \in \alg$ is equivalent to the $\infty$-category of $A$-modules. The relative cotangent complex $\mathbb{L}$ ( i.e., topological André-Quillen homology) is, by definition, the relative left adjoint to $\Omega^{\infty}$. The $\infty$-category $\tang_{\alg}$ is called the tangent bundle of $\alg$.\\
We will consider the logarithmic analogue of Lurie's tangent bundle formalism in this paper.
\begin{defn}
    A prelog $\Einf$-ring is a triple $(A,M,\alpha)$, where $A$ is a connective $\Einf$-ring, $M$ is an $\Einf$-monoid, and $\alpha: M\rightarrow \Omega^{\infty} A$ is a map of $\Einf$-monoids. We say $(A,M,\alpha)$ is an $\Einf$-log ring if the induced map $M\times_{\Omega^{\infty}A}GL_1(A)\rightarrow GL_1(A)$ is an equivalence.
\end{defn}
However, it is not possible to produce the deformation theory of $\Einf$-log rings directly in this way. The reason Lurie's formalism fails is that it does not yield the correct notion of "modules over $\Einf$-log rings." The tangent bundle $\tang_{\llog}$, defined as the stabilization of the functor ${\rm ev}_1:\Fun(\Delta^1,\llog)\rightarrow\llog$, appears to be too complicated; moreover, the fiber $\tang_{\llog}\times_{\llog}\{(A,M)\}$ over an $\Einf$-log ring $(A,M)$ is not well understood. To resolve this issue, Lundemo \cite{lundemo2023deformation} introduces a variation of Lurie's tangent bundle, termed the replete tangent bundle formalism, which gives rise to presentable fibrations.
$$\begin{tikzcd}
\tang^{\rm rep}_{\llog} \arrow[r, "\Omega^{\infty}"] \arrow[rr, "\pi"', bend right] & {\Fun(\Delta^1,\llog)} \arrow[r, "{\rm ev}_1"] & \llog
\end{tikzcd}$$
such that the fiber of $\pi$ at an $\Einf$-log ring $(A,M)$ is equivalent to the $\infty$-category $\Mod_A$ of $A$-modules. The \textit{Gabber's cotangent complex} $\LL^G$ is the left adjoint of the functor $\Omega^{\infty}$. This gives rise to a topological variant of algebraic Gabber's cotangent complex functor given above defined in terms of the non-abelian derived functor of the log differential forms. Unwinding the definition,  the Gabber's cotangent complex is characterized by the following universal property:
$$\map_{A}(\LL^G_{(A,M)/(R,L)},J)\simeq\map_{(R,L)//(A,M)}((A,M),(A\oplus J,M\oplus J)),J\in\Mod_A^{\rm cn},$$
where we denote by $\Mod_A$ the  $\infty$-category of $A$-modules, and denote by $\Mod_A^{\rm cn}$ the  subcategory of $\Mod_A$ spanned by connective objects.\\
Besides the cotangent complex, another object central to deformation theory is that of \textit{derivations}.  A derivation of a ring $R$ is a connective module $J$, together with a map $R\rightarrow R\oplus J$, such that the composition $R\rightarrow R\oplus J\rightarrow R$ is the identity map. Or equivalently, it is an additive map $d:R\rightarrow J$ which satisfies the Leibniz rule.
The classical deformation theory shows that the square-zero extensions of $R$ are one to one corresponding to derivations over $R$.

In spectral algebraic geometry, Lurie employs the tangent correspondence formalism \cite[Section 7]{lurie2017higher} to define the notions of higher derivations and higher deformations of $\Einf$-rings. In \autoref{definiiton of einf-log rings}, we introduce an analogue of Lurie's tangent correspondence, which we term the replete tangent correspondence $\MTL\rightarrow \Delta^1\times\llog$, generalizing Lundemo's ideas. Informally, $\MTL$ can be described as follows: $\MTL\times_{\Delta^1}\{0\}\simeq\llog$ and $\MTL\times_{\Delta^1}\{1\}\simeq\tlog$, a morphism from $(A,M)\in\llog$ to $((B,N),J)\in\tlog$ is given as an $\Einf$-log ring map $(A,M)\rightarrow (B,N)$, together with a 
$B$-module map $\LL^G_{(A,M)}\otimes_AB\rightarrow J$.
Using this formalism, we define the $\infty$-category $\Der^{\llog}\rightarrow\Delta^1\times\llog$ of log derivations of $\Einf$-log rings, thereby generalizing the theory of derivations from the animated context \cite{sagave2016derived, binda2023hochschild} to the spectral context.
   \begin{defn}[Definition\autoref{LDER}]
 The $\infty$-category of log  derivations $\Der^\llog$  is the subcategory of $\Fun(\Delta^1, \MTL)$ spanned by these functors  
  $d:\Delta^1\rightarrow \MTL$, such that the composition 
  $$\Delta^1\rr\MTL\rr \Delta^1\times \llog$$
  is the functor $\Delta^1\rightarrow \Delta^1\times\{(A,M)\}\subset \Delta^1\times \llog$ for some $(A,M)\in\llog.$ For such functors we use the notation $d: (A,M)\rightarrow I$, where $I$ is the image $d(1)\in \tlog\times_{\llog}\{(A,M)\}\simeq \Mod_A.$
\end{defn}
\begin{defn}
\begin{enumerate}
    \item 
    Let $(A,M)$ be an $\Einf$-log ring, and let $I\in\Mod_A^{\rm cn}$ be a connective $A$-module, then a \textit{square-zero extension}  of $A$ with kernel $I$ is a strict  morphism  $i:(A',M')\rightarrow (A,M),$ such that $A'$ is a square-zero \cite[Definition 7.4.1.6]{lurie2017higher} extension of $A$ with kernel $I$.
    \item Let $(R,L)$ be an $\Einf$-log ring, and let $(R',L')$ be a square-zero extension of $(R,L)$ with kernel $I\in\Mod_R^{\rm cn}$. Let $(A,M)$ be an $(R,L)$-algebra,
a deformation of $(A,M)$ to $(R',L')$ is a 
square-zero extension  $(A',M')$ of $(A,M)$ with kernel $J\in\Mod_A^{\rm cn}$, such that there is an equivalence  $J\simeq I_A=I\otimes_RA$ of $A$-modules, and the following diagram 
\[\begin{tikzcd}
	{(R',L')} & {(R,L)} \\
	{(A',M')} & {(A,M)}
	\arrow[from=1-1, to=1-2]
	\arrow[from=1-1, to=2-1]
	\arrow[from=2-1, to=2-2]
	\arrow[from=1-2, to=2-2]
\end{tikzcd}\]
is a pushout square in $\llog$. 
\end{enumerate}
\end{defn}
We have the following analogue of Lurie's theorems of deformations.
\begin{thm}[Corollary\autoref{obstruction}]\label{1111}
     Let $(R',L')\rightarrow(R,L)$ be a square-zero extension of $\Einf$-log rings with kernel $I$, induced by the derivation $d:(R,L)\rightarrow I[1]$. 
\begin{enumerate}
    \item 
   Let $(R,L)\rightarrow(A,M)$ be a morphism of $\Einf$-log rings. The obstruction to the existence of deformations $(A',M')\rightarrow(A,M)$ to $(R',L')$ lies in $\mathrm{Ext}^2(\LL^G_{(A,M)/(R,L)},I_A).$
\item If the map $\LL_{(R,L)}\otimes_RA\rightarrow I_A[1]$ induced from $d$ vanishes (for example, $d$ is a trivial derivation), then there is an equivalence of animae
$$\map_{A}(\LL^G_{(A,M)/(R,L)},I_A[1])\simeq\mathbf{Def}^{\simeq}((A,M)/(R,L),d)$$
where $\mathbf{Def}^{\simeq}((A,M)/(R,L),d)$ is the anima of deformations $(A',M')$ to $(R',L').$
    \end{enumerate}
\end{thm}
\begin{thm}[Proposition\autoref{strictetalerig}]\cite[Theorem 1.3]{lundemo2023deformation}\label{2222}
      Let $(A,M)$ be an $\Einf$-log ring, and let $(A',M')$ be a square-zero extension of $(A,M)$. Then the base change gives rise to  equivalences
    $$-\otimes_{A'}A:(A',M')_{\set}\stackrel{\simeq}\rr(A,M)_{\set},(A',M')_{{\rm l}\et}\stackrel{\simeq}\rr(A,M)_{{\rm l}\et}.$$
    Here $(A,M)_{\set}$ and $(A,M)_{{\rm l}\et}$ are subcategories of $\llog_{(A,M)/}$ forms of strict \'etale and log \'etale $(A,M)$-algebras respectively.
\end{thm}
\subsubsection{Animated context}
As observed in \cite[Section VII.25.3.3]{lurie2018spectral}, Lurie noted that the tangent bundle formalism does not yield the desired deformation theory for animated rings. Proposition 25.1.2.2 of the same work implies that animated rings should be understood as connective $\Einf$-rings equipped with additional structure; that is, the $\infty$-category $\alg^\Delta$ is equivalent to the $\infty$-category of modules over $\alg$ for a specific monad. More precisely, there exists a functor $\Theta:\alg^{\Delta}\rightarrow\alg$, defined as the left Kan extension of the Eilenberg–Mac Lane functor $\alg^\heartsuit\rightarrow\alg$ along the fully faithful embedding $\alg^\heartsuit\subset\alg^\Delta$. The functor $\Theta$ admits both left and right adjoints and is both monadic and comonadic.\\
An analogous issue arises in logarithmic geometry. In \autoref{deformationofanimated}, we will present an animated version of the replete tangent bundle and replete tangent correspondence formalisms, along with animated analogues of Theorem\autoref{1111} and Theorem\autoref{2222}.
\subsection{Moduli of spectral and derived log structures}
We say an $\Einf$-log ring $(A,M)$ is $n$-admissible, if the logification of the induced $\Einf$-prelog ring $(\pi_0 A,M)$ is $n$-truncated for some $n\in\N$, see Definition\autoref{admlogrings}. A spectral log Deligne-Mumford stack $(\X,\mathscr M)$ is $n$-admissible if it is covered by $n$-admissible log-affine objects.
In this paper, we study the following moduli stacks of log structures in the contexts of spectral and derived algebraic geometry given in Definition\autoref{rlogdefn},\autoref{rlogdefn2} and\autoref{rlogdefn4}.
\begin{enumerate}
\item 
The functors $\R\llog^{\mu}_{(\MCS,\mathscr L)},\R\llog^{{\rm Adm},\mu}_
{(\MCS,\mathscr L)}:\Sptdm^{\rm op}_{/\MCS}\rightarrow\ani$, which send a quasi-compact, quasi-separated spectral Deligne-Mumford stack $\X$ to the anima of maps of (admissible) spectral $\mu$-log Deligne-Mumford stacks $(\X,\mathscr M)\rightarrow(\MCS,\MSL)$, where $\mu\in\{\rm qcoh,coh,fin\}$.
\item And these derived analog $\R\llog^{\mu,\Delta}_{(\MCS,\MSL)},\R\llog^{{\rm Adm},\mu,\Delta}_{(\MCS,\MSL)}:{\mathbf{ DM}^{\Delta,\rm op}}\rightarrow\ani$, which are defined over the $\infty$-category of derived Deligne-Mumford stacks over $\MCS$. 
\end{enumerate}
 Here, we refer to  spectral and derived Deligne-Mumford stacks as those defined in \cite[Definition 1.4.4.2]{lurie2012derived} and \cite[Definition 4.5.1]{lurie2004derived}, respectively.\\
We say a functor $X:\algcn\rightarrow\ani$ is a spectral Artin stack if it satisfies \'etale descent, the diagonal $X\rightarrow X\times X$ is representable by a relative algebraic space, and it admits a smooth surjective map from a disjoint union of spectral schemes $\coprod_j U_j\rightarrow X$. Similarly, one can also define $n$-Artin stacks for some $n\in\N$. It is an \'etale sheaf, and the diagonal is representable by a relative  $(n-1)$-Artin stack, which admits a smooth surjective map from a disjoint union of spectral schemes.
See Definition\autoref{artstkdef} for details.\\
We have the following result.
\begin{thm}[Theorem\autoref{algebraicity} and Theorem\autoref{algebraicityani}] {\ } 
      \begin{enumerate}
      \item Let $\mu\in\{\rm qcoh,coh,fin\}$. The stack $\R\llog^{\rm Adm,\mu}_{(\MCS,\MSL)}$ admits a filtration of open substacks 
$$\R\llog^{0-{\rm Adm,\mu}}_{(\MCS,\MSL)}\subset\R\llog^{1-{\rm Adm},\mu}_{(\MCS,\MSL)}\subset...\subset \R\llog^{n-{\rm Adm},\mu}_{(\MCS,\MSL)}\subset...$$
where $\R\llog^{n-{\rm Adm},\mu}_{(\MCS,\MSL)}$ is the substack of $\R\llog^{{\rm Adm},\mu}_{(\MCS,\MSL)}$ consisting of $n$-admissible objects. We have $\R\llog^{{\rm Adm},\mu}_{(\MCS,\MSL)}\simeq\bigcup_{n}\R\llog^{n-{\rm Adm},\mu}_{(\MCS,\MSL)}$.
        \item Let $\mu\in\{\rm qcoh,coh,fin\}$. The functors $\R\llog_{(\MCS,\MSL)}^{\mu,(\Delta)}$  $\R\llog_{(\MCS,\MSL)}^{n-{\rm Adm},\mu,(\Delta)}$ and $\R\llog_{(\MCS,\MSL)}^{{\rm Adm},\mu,(\Delta)}$ are sheaves with respect to the \'etale topology.
        \item Let $\mu\in\{\rm qcoh,coh,fin\}$. Let $(\MCS,\MSL)$ be a $\mu$-spectral (or derived) log Deligne-Mumford stack. Then $\R\llog_{(\MCS,\MSL)}^{\mu,(\Delta)}$ admits a $(-1)$-connective cotangent complex.
        \item Let $\mu\in\{\rm coh,fin\}$. Let $(\MCS,\MSL)$ be a $\mu$-spectral (or derived) log Deligne-Mumford stack. Assume that $\MCS$ admits an \'etale cover by union of  spectra of   $\Einf$-(or animated) $G$-rings\footnote{$R$ is called an $\Einf$- (or animated) $G$-ring, if $R$ is Noetherian, and the classical truncation $\pi_0 R$ is a G-ring.}. Then the functor  $\R\llog_{(\MCS,\MSL)}^{n-{\rm Adm},\mu,(\Delta)}$ is $(1+n)$-Artin, and $\R\llog_{(\MCS,\MSL)}^{{\rm Adm},\mu,(\Delta)}$ is locally Artin.
        \item Let $f:\spet(A,M)\rightarrow\spet(B,N)$ be a map induced by  derived log ring map $(B,N)\rightarrow(A,M)$, then there is a canonical equivalence $\LL_{\spet A/\R\llog^{\mu,(\Delta)}_{\spet(B,N)}}\simeq\LL^{G,(\Delta)}_{(A,M)/(B,N)}.$
    \end{enumerate}
\end{thm}
There is a natural map from Olsson's classical moduli stack of log structures to our derived moduli stack $\R\llog$. We let $(\MCS^\heartsuit,\MSL^\heartsuit)$ be a classical fine log Deligne-Mumford stack, and let $(\MCS,\MSL)$ be its corresponding discrete spectral log Deligne-Mumford stack. We denote by $\R\llog^\heartsuit_{(\MCS,\MSL)}$ the classical truncation of the stack $\R\llog^{\rm fin}_{(\MCS,\MSL)}$. There is a natural map
as follows:
$$\theta:\llog_{(\MCS^\heartsuit,\MSL^\heartsuit)}\rr\R\llog^\heartsuit_{(\MCS,\MSL)}.$$
The following result shows that Olsson's classical moduli stacks of log structures are the classical truncation of our derived moduli stacks:
\begin{thm}[Theorem\autoref{compder-clas}]
    The map $\theta$ above is an open immersion of stacks, whose image is isomorphic to the classical truncation of $\R\llog^{0-{\rm Adm,fin}}_{(\MCS,\MSL)}$. 
\end{thm}
\subsection{Log cotangent complexes for log stacks}
We also geometrize and globalize Gabber's cotangent complex in this paper. We can define the following logarithmic cotangent complex.
\begin{defn}[Definition\autoref{DEF OF COTANGENT CPX}]
       Let $f:(\X,\mathfrak M)\rightarrow(\MCS,\MSL)$ be a map of quasi-coherent spectral (resp. derived) Deligne-Mumford stack, the log cotangent complex of $f$ is the relative cotangent complex $\LL^{\rm Log,(\Delta)}_{(\X,\mathfrak M)/(\MCS,\MSL)}:=\LL_{\X/\R\llog^{\rm qcoh,(\Delta)}_{(\MCS,\MSL)}}$.
\end{defn}
We have a tautological equivalence
    $\LL^{G,(\Delta)}_{(A,M)/(R,L)}\simeq\LL^{\rm Log,(\Delta)}_{\spet(A,M)/\spet(R,L)}.$ Thus, the log cotangent complex functor $\LL^{\rm Log}$ should be thought of as a globalization of Gabber's cotangent complex. 
\begin{thm}[Theorem\autoref{global tri seq}]
        Let $(\X,\mathfrak M)\stackrel{f}\rightarrow(\Y,\mathscr N)\stackrel{g}\rightarrow(\MCS,\MSL)$ be a sequence of quasi-coherent spectral (resp. derived) Deligne-Mumford stacks, then there is a fiber sequence of quasi-coherent sheaves over $\X$:
    $$f^*\LL^{\rm Log,(\Delta)}_{(\Y,\mathscr N)/(\MCS,\MSL)}\rr\LL^{\rm Log,(\Delta)}_{(\X,\mathfrak M)/(\MCS,\MSL)}\rr\LL^{\rm Log,(\Delta)}_{(\X,\mathfrak M)/(\Y,\mathscr N)}.$$
\end{thm}
\subsection{($p$-typical) infinite root  stacks}
As an application of the  stack $\R\llog$, we can  develop a derived version of the theory of infinite root  stacks, as in \cite{Borne_2012} and \cite{Talpo_2018}. For simplicity, we fix a $p$-complete $\Einf$-ring $R$, and work on the $p$-completion $\logfr:=\R\llog^{\rm qcoh}_{\spet R}\times_{\spec R}\spf R$ of log stacks defined over $\spf R$, which is the moduli stack that classifies quasi-coherent  log stacks that are lying over $\spf R$.
\begin{defn}[Definition\autoref{qazwsxedc}]
 The \textit{universal infinite root  stack} $\rtr$ is defined as the inverse limit of the following sequence
$$...\rr\R\llog_R^{\wedge}\stackrel{p}\rr\R\llog_R^{\wedge}\stackrel{p}\rr...\stackrel{p}\rr \R\llog_R^{\wedge}.$$
Let $X\rightarrow\logfr$ be a map of stacks with respect to fpqc-topology, the infinite root  stack $\sqrt[\infty]{X}$ is the fiber product $X\times_{\logfr}\rtr$.
\end{defn}
We prove that the definition of infinite root  stacks is functorial.
\begin{thm}[Theorem\autoref{FUNCROOTSTK}]
   There is  a limit-preserving functor 
    $$\gamma:\llog\stk_{/\spf R}\rr\stk_{/\spf R},$$
    such that for any  quasi-coherent spectral  log stack $X$, the stack $\gamma(X)$ coincides with the infinite root  stack $\sqrt[\infty]{X}$ of $X.$
\end{thm}
The following theorem generalizes the result of Binda-Lundemo-Merici-Park in \cite[Theorem 4.12]{binda2024logarithmictcinfiniteroot} and \cite[Lemma 3.18]{binda2024motivicmonodromypadiccohomology}  to the quasi-coherent case.
\begin{thm}[Theorem\autoref{123256456}]
 Let $R$ be a connective $p$-complete  $\Einf$-ring.   Let $X$ be a topologically $p$-saturated (Definition\autoref{satdefn}) quasi-coherent spectral log Deligne-Mumford stack over $\spec R$, and let $\underline{X}_p^\wedge=\underline{X}\times_{\spec R}\spf R \rightarrow \logfr$ be the associated $p$-completion.
    Then the  map 
    $$\phi:\LL_{X/R}^{\rm Log}\rr\eta_*\LL_{\sqrt[\infty]{X}/\spf R}$$
     induced from the canonical map $\eta:\sqrt[\infty]{X}\rightarrow\underline{X}^\wedge_p$
    becomes an equivalence after $p$-completion. In particular, the following map
    $$\phi:{\rm R}\Gamma(\underline{X},\LL^{{\rm Log}}_{X/R})^\wedge_p\rr{\rm R}\Gamma(\sqrt[\infty]{X},\LL_{\sqrt[\infty]{X}/\spf R})$$
    is an equivalence.
\end{thm}
\subsection{Outline}
We now provide a brief outline of the contents of this paper.\\
In \autoref{Section2}, we present a concise introduction to animated and $\Einf$-log rings from a purely $\infty$-categorical perspective. We define the $\infty$-categories of animated and $\Einf$-log rings and compare them via the Barr-Beck-Lurie theorem.\\
In \autoref{definiiton of einf-log rings}, we develop the higher deformation theory in logarithmic geometry. We adapt Lundemo's replete tangent bundle formalism to define the cotangent complex of $\Einf$-log rings and introduce the notion of the replete tangent correspondence—a variant of Lurie's tangent correspondence—to define log derivations. The main result of this section is the log deformation theory, Theorem\autoref{deform}. We then employ Barr-Beck-Lurie arguments to derive the deformation theory for animated log rings in Theorem\autoref{deformani}.\\
In \autoref{logstacks}, we globalize the notions of animated and $\mathbb{E}_\infty$-log rings. We first examine the theory of charted log stacks, which generalizes the classical notion of charted log schemes. We study the deformation theory of the moduli prestack of charted log stacks, which possesses nearly all desired deformation properties except for the \'etale descent property. We then introduce the concept of log stacks and investigate their global deformation theory, culminating in the construction of the moduli stack of log structures in derived and spectral algebraic geometry. The main results concerning the descent and infinitesimal properties of this moduli stack are presented in \autoref{subsection43}.\\
In \autoref{ALREP}, we review Lurie's deformation theory in spectral algebraic geometry. We introduce the notion of spectral Artin stacks and generalize the Artin-Lurie representability theorem to spectral algebraic geometry in Theorem\autoref{alrep}; this result is necessary for proving the representability of our moduli of log stacks.\\
In \autoref{section6}, we prove the representability theorem for the moduli of log stacks in both spectral and derived settings. The main theorems are Theorem\autoref{algebraicity} and Theorem\autoref{algebraicityani}. We also compare the classical truncation of the moduli of derived log stacks with Olsson's classical moduli of log structures in Proposition\autoref{compder-clas}.\\
In \autoref{section7}, we use the moduli stack of log structures to construct the infinite root  stacks in spectral and derived log geometry. We will prove the functoriality of infinite root  stacks in Theorem\autoref{FUNCROOTSTK}, and generalize a theorem of  Binda-Lundemo-Merici-Park in Theorem\autoref{123256456}, which asserts that the $p$-completion of the log cotangent complex of a log stack could be calculated in terms of that of the associated infinite root  stacks. 
\subsection{Acknowledgements}
This article is part of my PhD thesis at \textit{Universit\`a  di Milano}. I am very grateful to my supervisor, Federico Binda. He introduced the topic to me and carefully checked all details of this paper. This work would not be finished  without the discussion with him in depth. I thank Tommy Lundemo and Doosung Park for many useful discussions on many details of proofs in this paper, and for showing me many inspiring examples. I also thank Bhavna Joshi,  Caixing Cao, Xiaomin Chu, Jens Hornbostel, Alberto Vezzani, Qixiang Wang, Xiangsheng Wei, Ziqian Yin, and Pengcheng Zhang for  their kind help in mathematics.\\
In addition, I think Italy's beautiful scenery and delicious cuisine are also important, which makes my time as a PhD student here truly enjoyable and undoubtedly helps my research work in a psychological sense.
\section{Affine logarithmic geometry}\label{Section2}
This section provides the basic definitions of log rings in both derived and spectral contexts, building upon work presented in \cite{rognes2009topological}, \cite{sagave2016derived}, \cite{binda2023hochschild}, \cite{10.1093/imrn/rnad224}, and \cite{lundemo2023deformation}. While these cited works employ model categorical language to varying degrees, we adopt a purely $\infty$-categorical framework herein for the sake of conceptual clarity.
\subsection{Animated log rings}
We review the fundamental definitions within the context of animated log rings, as established in \cite{binda2023hochschild} and \cite{sagave2016derived}. Instead of employing a model structure on the category of simplicial prelog rings, we work within the framework of \textit{animations} of ordinary categories (see e.g., \cite{vcesnavivcius2024purity}), to construct the $\infty$-category of animated log rings. This is the point of view adopted in \cite{binda2023hochschild}.
\subsubsection{Animations}
Denote by $\ani$ the $\infty$-category of \textit{animae} (i.e., \textit{spaces} or \textit{$\infty$-groupoids}).
For an  ordinary category $\CC$ that admits \textit{sifted colimits}, we let
$\CC^{\rm sfp}$ be the full subcategory generated by \textit{strongly finitely presented} objects, i.e., objects $X\in\CC$ such that the functor $\Hom(X,-):\CC\rightarrow\mathbf{Set}$ commutes with sifted colimits. Assume that $\CC$ is cocomplete and is generated by $\CC^{\rm sfp}$ under 1-sifted colimits, i.e., $\CC \simeq 1{\rm -}\mathrm{sInd}(\CC^{\rm sfp})$. 
\begin{defn}\cite[Section 5.1.4]{vcesnavivcius2024purity}
The animation of $\CC$ is the $\infty$-category $\ani(\CC)$  freely generated by $\CC^{\rm sfp}$ under sifted colimits, that is the $\infty$-category 
of functors $F:\CC^{\rm sfp,op}\rightarrow\ani$, which preserve finite products. The animations can be characterized by the following universal property 
$$\Fun^{\rm sift}(\ani(\CC),\mathcal A)\stackrel{\simeq}\rr \Fun(\mathcal C^{\rm sfp},\mathcal A),$$
for any $\infty$-category $\mathcal A.$ 
\end{defn}
\begin{rmk}
    The animation $\ani(\CC)$ of a small ordinary category $\CC$ is  presentable by \cite[Proposition 5.5.8.10]{lurie2009higher}.
\end{rmk}
\subsubsection{Animated log rings}
\begin{defn}
    Let $\mon^{\heartsuit}$ be the ordinary category of discrete commutative monoids. We let $\mon^{\Delta}$ denote the animation of $\mon^{\heartsuit}$, and refer to it as the $\infty$-category of \textit{animated monoids}.
\end{defn}
As a monoid $M$ is strongly finitely presented  if and only if it is  a retraction of finitely generated free monoids,  one can identify the $\infty$-category $\mond$ with the full subcategory of  $\Fun({\mathbf{FreeMon}}^{\rm op},\ani)$ spanned by those functors that preserve finite products, here $\mathbf{FreeMon}$ is the subcategory of $\mone^\heartsuit$ spanned by
finitely generated free monoids.
\begin{rmk}[Simplicial monoids]\cite[Remark 2.9]{binda2023hochschild}\label{smon}
    Denote by $s\mon^{\heartsuit}:=\Fun(\Delta^{\rm op},\mone^\heartsuit)$ the ordinary category of simplicial monoids. By \cite[Proposition 2.1]{sagave2016derived}, it admits a combinatorial simplicial model structure in which fibrations and weak equivalences are detected on the underlying simplicial sets (equipped with  Quillen's \textit{standard model structure}). We still denote by $s\mon^{\heartsuit}$ the associated $\infty$-category, in the sense of \cite[Definition 1.1.5.5]{lurie2009higher}. Using the universal property of animation, and together with the arguments in \cite[Section 5.58 and 5.59]{lurie2009higher}, we have a natural equivalence of $\infty$-categories 
$$\mon^{\Delta}\stackrel{\simeq}\rr s\mon^{\heartsuit},$$
    induced by the diagonal  embedding  $\mon^{\heartsuit}\subset s\mon^{\heartsuit}.$ We will therefore make no distinctions between $\mon^{\Delta}$ and $s\mon^{\heartsuit}$.
\end{rmk}
\begin{defn}
An animated monoid $M$ is called \textit{group-like} if $\pi_0(M)$ is an abelian group. The \textit{group completion} of $M$ is  the map  $M\rightarrow\Omega BM=:M^{\rm gp}$, which is initial for  maps from $M$ to group-like animated monoids.
\end{defn}
\begin{rmk}
The subcategory of $\mon^{\Delta}$ 
spanned by  group-like animated monoids is naturally equivalent to the $\infty$-category $\textbf {Ab}^{\Delta}$ of \textit{animated 
abelian groups}, see \cite[Section 2.1]{sagave2016derived} for details. If we identify these two $\infty$-categories, then the group-completion is nothing but the left derived functor  of the classical group completion of discrete monoids.

\end{rmk}

\begin{defn}\label{LOGPOLYNOMIAL}  An \textit{ordinary prelog ring} $(A,M,\alpha)$ consists of a discrete ring $A$, a discrete monoid $M$, and a monoid map $\alpha: M\rightarrow (A,*).$ According to the arguments in \cite[Section 5]{bhatt2012padicderivedrhamcohomology} and \cite[Remark 2.11]{binda2023hochschild}, the ordinary category of all ordinary  prelog rings is   generated by strongly finitely presented objects under sifted colimits, that is, retractions of \textit{free prelog algebras}, which have forms  $$(\Z[x_1,x_2,...,x_k,x_{k+1},...,x_n],\langle x_1,...,x_k \rangle,\alpha_k),$$
    in which the monoid $\langle x_1,...,x_k \rangle$ is the free monoid $\bigoplus^k_{i=1}\N x_i$ of rank $k$, and the monoid map $\alpha_k$ is 
    $$\alpha_k:\langle x_1,...,x_k \rangle\rr\Z[x_1,...,x_n],\sum n_ix_i\mapsto \prod x_i^{n_i}.$$
    We let $\plog^{\Delta}:=\ani(\plog^{\rm ord})$ denote the $\infty$-category of \textit{animated prelog rings}. 
    \end{defn}
    
\begin{defn}  
 An animated  prelog ring $(A,M,\alpha)$ is called an \textit{animated log ring} if the upper horizontal map of the following diagram is an equivalence
\[\begin{tikzcd}
	{\alpha^{-1}GL_1(A)} & {GL_1(A)} \\
	M & A
	\arrow[from=1-1, to=1-2]
	\arrow[from=1-1, to=2-1]
	\arrow[from=1-2, to=2-2]
	\arrow["\alpha", from=2-1, to=2-2]
\end{tikzcd}\]
Denote by $\llog^{\Delta}$ the full subcategory of $\plog^{\Delta}$ spanned by animated log rings. 
\end{defn}
Following \cite[Section 5]{bhatt2012padicderivedrhamcohomology}, and \cite[Section 2.10]{binda2023hochschild}. Consider the forgetful functor $F:\plog^{\rm ord}\rightarrow\alg^\heartsuit\times\mon^\heartsuit$ of ordinary categories. This functor is a right adjoint functor, whose left adjunction is given by $\alg^\heartsuit\times\mon^\heartsuit\ni(A,M)\mapsto (A[M],M,M\rightarrow A[M])$, where the map $M\rightarrow A[M]$ is defined in the obvious way. On the other hand, it is also a left adjoint functor, whose right adjunction is given by sending a pair $(A,M)\in\alg^\heartsuit\times\mon^\heartsuit$ to the triple $(\tilde A,\tilde M,\alpha)$ obtained by taking the limit of the following diagrams
$$\begin{tikzcd}
M \arrow[r] & M' \arrow[d, "\alpha'"] \\
A \arrow[r] & A'                    
\end{tikzcd}$$
where the index runs over all $A'$, $ M'$ and $\alpha'$. Passing to animation, we obtain that the forgetful functor $$F:\plog^\Delta\rr \alg^\Delta\times\mon^\Delta$$ is both limit-preserving and colimit-preserving. In particular, the functor $F$ detects limits and colimits of animated prelog rings, and the inclusion $\llog^\Delta\subset\plog^\Delta$ preserves limits, as the limits in $\plog^\Delta$ preserve animated log rings.
\begin{prop}\label{presentabilityoflog}
    The $\infty$-category  $\llog^\Delta$ is presentable.
\end{prop}
\begin{proof}
We have a functor $\phi:\plog^\Delta\rightarrow\Fun(\Delta^1,\mon), (A,M,\alpha)\mapsto \alpha^{-1}GL_1(A)\rightarrow GL_1(A)$. More precisely, it's  given by composing the following functors
$$\plog^\Delta\rr\Fun(\Delta^1{\coprod}_{\{1\}}\Delta^1,\mon^\Delta)\rr\Fun(\Delta^1\times\Delta^1,\mon^\Delta)\rr\Fun(\Delta^1,\mon^\Delta)$$
where the first functor is the forgetful functor $(A,M,\alpha)\mapsto(GL_1(A)\rightarrow A\leftarrow M)$, the second functor is the right Kan extension along the inclusion $\Delta^1\coprod_{\{1\}}\Delta^1\subset\Delta^1\times\Delta^1$, and the last functor is the restriction along the inclusion $\Delta^1\subset\Delta^1\times\Delta^1$ to the upper horizontal arrows. The diagonal embedding $\delta:\mon^\Delta\rightarrow\Fun(\Delta^1,\mon^\Delta)$ identifies $\mon^\Delta$ with the full subcategory of $\Fun(\Delta^1,\mon^\Delta)$ spanned by invertible arrows. Then one can identify $\llog^\Delta$ with the limit of the following diagram
$$\begin{tikzcd}   & \mon^\Delta \arrow[d, "\delta"] \\
\plog^\Delta \arrow[r, "\phi"] & {\Fun(\Delta^1,\mon^\Delta)}   
\end{tikzcd}$$
Indeed, we denote by $P$ the limit of the above diagram, and it identifies with the  subcategory of $\plog^\Delta$ consisting of those animated prelog rings $(A,M,\alpha)$ such that  the induced map $\alpha^{-1}GL_1(A)\rightarrow GL_1(A)$ is an equivalence.
\end{proof}

\begin{defn}
The \textit{logification} functor $a:\plog^\Delta\rightarrow\llog^\Delta$ is the left adjoint of the inclusion $\llog^\Delta\subset\plog^\Delta$. 
\end{defn}
The logification of an animated prelog ring $(A,M,\alpha)$ is equivalent to the animated log ring $(A,M^a,\alpha^a)$ given by the following coCartesian square 
\[\begin{tikzcd}
	{\alpha^{-1}GL_1(A)} & {GL_1(A)} \\
	M & {M^a}
	\arrow[from=1-1, to=1-2]
	\arrow[from=1-1, to=2-1]
	\arrow[from=1-2, to=2-2]
	\arrow[from=2-1, to=2-2]
\end{tikzcd}\]
\begin{rmk}[Model category theoretical approach of presentability]\cite[Remark 2.12]{binda2023hochschild}\label{modelcatrep}
     One can use model category technology to obtain the presentability of  $\llog^\Delta$. By \cite[Theorem 3.13]{sagave2016derived}, there is a combinatorial simplicial model structure, which is called \textit{log model structure}, on the ordinary category $s\mathcal P$ of simplicial prelog rings, see \cite{sagave2016derived} for details. The $\infty$-category associated with this model  category is naturally equivalent to $\llog^{\Delta}$. Then using \cite[Proposition A.3.7.6]{lurie2009higher}, one deduces that the $\infty$-category $\llog^{\Delta}$ is  presentable. {See also \cite[Remark 5.2]{lundemo2023deformation}}
\end{rmk}

\subsection{\texorpdfstring{$\Einf$-log rings}{E-infinity-log rings}}
  We recall the fundamental notions of $\Einf$-log rings from \cite[Part II]{rognes2009topological}. These objects serve as topological analogues of animated log rings, constructed from $\Einf$-rings and $\Einf$-monoids. In practice, the categorical behavior of $\Einf$-log rings is more tractable than that of their animated counterparts—a phenomenon with parallels in algebraic geometry, as detailed in \cite[Section 25.3.3]{lurie2018spectral}. This paper will primarily operate within the context of $\Einf$-log rings. The results established for this setting will largely extend to animated log rings, with proofs using the techniques of "$\Einf$-methods."
\begin{rmkcaution}
Throughout this paper, we will only consider a special class of $\Einf$-log rings, which only capture "log structures on $\pi_0$" in the following sense: log structures on $\Einf$-rings $R$ come from their infinite loop space $\Omega^\infty R$, and  the log structures from  higher homotopy classes of $R$ will not be considered. The latter one is also important in  algebraic topology and arithmetic, which we call  $\Einf^\mathcal{ J}$-log rings, made of (non-connective) $\Einf$-rings and $\mathcal{QS}^0$-graded $\Einf$-monoids, developed by Rognes, Sagave, and Schlichtkrull, c.f., \cite{Sagave_2012}, \cite{Sagave_2014}, \cite{Rognes_2015}, and \cite{lundemo2023deformation}.
\end{rmkcaution}
\begin{defn}
Denote by $\ani^\times$ the symmetric monoidal $\infty$-category of animae equipped with the Cartesian monoidal structure.    An \textit{$\Einf$-monoid} is an $\Einf$-algebra in $\ani^\times$. We denote by $\mone$ the $\infty$-category of $\Einf$-monoids. This is a presentable $\infty$-category by \cite[Corollary 3.2.3.5]{lurie2009higher}.
\end{defn}
\begin{defn}
An $\Einf$-monoid $M$ is called \textit{group-like} if $\pi_0(M)$ is an abelian group. The group completion of $M$ is  the map  $M\rightarrow\Omega BM=:M^{\rm gp}$, which is initial for maps from $M$ to group-like $\Einf$-monoids.
\end{defn}
\begin{rmk}\label{mayrecognition}
    By May's recognition principle \cite[Theorem 1.3]{MR420610} or \cite[Remark 5.2.6.26]{lurie2017higher},  the infinity loop space functor 
    $\Omega^\infty:\Spt\rr\ani$ induces an equivalence $\Spt_{\geq 0}\simeq\mone^{\rm gp}$ of the $\infty$-category of connective spectra and the $\infty$-category of group-like $\Einf$-monoids; therefore, the $\infty$-category $\mon^{\rm gp}$ is presentable. 
    \end{rmk}
    The inclusion $\mone^{\rm gp}\subset \mone$ preserves both limits and colimits, hence we have the following diagrams of adjoints
$$\begin{tikzcd}
\mone \arrow[r, "\rm gp", shift left=2] \arrow[r, "GL_1"', shift right=2] & \mone^{\rm gp} \arrow[l]
\end{tikzcd}$$
in which the left adjoint $\mathrm{gp}$ is the group completion functor and the right adjoint $GL_1$ is called the \textit{space of homotopy units} functor. 

\begin{defn}
  Let $\plog\Spt$ be the  $\infty$-category of triples $(E,A,\alpha)$, in which $E$ is a spectrum, $A$ is an anima, and $\alpha:\Sigma_+^{\infty}A\rightarrow E$ is a map of spectra. In other words, it's the  oriented fiber product $\Spt\overleftarrow{\times
  }_{\Spt}\ani,$ induced by the functors $id:\Spt\rightarrow\Spt$ and $\Sigma_+^{\infty}:\ani\rightarrow\Spt.$
\end{defn}
By the  definition of oriented fiber products introduced in \cite[Notation 2.1.4.19]{kerodon}, the $\infty$-category  $\plog\Spt$ is made of  the limit of the following diagram
    $$\begin{tikzcd}
\ani \arrow[rd, "\Sigma^{\infty}_{+}"] &      & {\Fun(\Delta^1,\Spt)} \arrow[ld, "{\rm ev}_0"'] \\  & \Spt &    
\end{tikzcd}$$
Note  that the  $\infty$-category $\Fun(\Delta^1,\Spt)$ can be  endowed with a natural symmetric  monoidal structure $\Fun(\Delta^1,\Spt)^{\otimes}$  with respect to the pointwise monoidal product induced by the smash product monoidal structure on $\Spt$. Then  the evaluation functor  $\mathrm{ev}_0$ is a monoidal coCartesian fibration. On the other hand, the infinite suspension $\Sigma^{\infty}_+$ is also symmetric monoidal if we endow the $\infty$-category $\ani$ with the Cartesian monoidal structure. Then we can endow $\plog\Spt$ with a symmetric  monoidal structure by promoting the above limit as a limit of symmetric  monoidal $\infty$-categories (as argued in \cite[Construction 2.22]{binda2024motivicmonodromypadiccohomology}, together with \cite[Propositions 2.2.4.4 and 2.2.4.9]{lurie2017higher}, this limit could also be formed in the $\infty$-category $\mathbf{Op}_{\infty}$ of operads)$$\plog\Spt^{\otimes}:=\lim(\ani^{\times}\stackrel{\Sigma^\infty_+}\rr\Spt^{\otimes}\stackrel{\mathrm{ev}_0}\longleftarrow\Fun(\Delta^1,\Spt)^{\otimes}).$$
Informally, we have $$(A,M,\alpha)\otimes(B,N,\beta)=(A\otimes B,M\times N,\alpha\otimes\beta),$$
where the morphism $\alpha\otimes\beta$ is given by 
$$\alpha\otimes\beta:\Sigma_+^{\infty}(M\times N)\stackrel{\simeq}\rr\Sigma^{\infty}_+M\otimes\Sigma^\infty_+N\rr A\otimes B.$$
\begin{defn}
An \textit{$\Einf$-prelog ring} is an $\Einf$-algebra in $\plog\Spt^{\otimes}$. We let $\plog$ denote  the $\infty$-category of $\Einf$-prelog rings.
\end{defn}
\begin{rmk}\label{1234543456}
    Informally, to give an $\Einf$-algebra object in the $\infty$-category $\plog\Spt^{\otimes}$, it needs to give a triple $(A,M,\alpha)$, a binary  operation $\mu:(A,M,\alpha)\otimes(A,M,\alpha)\rightarrow(A,M,\alpha),$ and a unit map $e:1=(\mathbb S,0,0)\rightarrow(A,M,\alpha),$ with some extra operation laws. If we unwind the definition of $\Einf$-prelog rings, it's equivalent  to say that $A$ is an $\Einf$-ring, $M$ is an $\Einf$-monoid, and the map $\alpha:\Sigma_+^\infty M=:\mathbb S[M]\rightarrow A$ is an $\Einf$-ring map.\\
    Formally, according to the definition of $\Einf$-algebra objects, the functor $$\alg(-):\cat^\otimes\rr\cat,\mathcal{C}\mapsto\alg(\mathcal{ C})$$ preserves limits. In particular, one has $$\plog\simeq\lim(\mon\stackrel{\Sigma^\infty_{+}}\rr\alg \stackrel{{\rm ev}_0}\longleftarrow\Fun(\Delta^1,\alg)).$$
\end{rmk}
\begin{rmk}
Passing to the adjunction, one can also show that
$$\plog\simeq\lim(\alg\stackrel{\Omega^\infty}\rr\mon \stackrel{{\rm ev}_1}\longleftarrow\Fun(\Delta^1,\mon)).$$
\end{rmk}
\begin{defn}
 An  $\Einf$-prelog ring $(A,M,\alpha)$ is  an \textit{$\Einf$-log ring}, if the upper horizontal map of the following diagram is an equivalence
\[\begin{tikzcd}
	{\alpha^{-1}GL_1(A)} & {GL_1(A)} \\
	M & \Omega^{\infty}A
	\arrow[from=1-1, to=1-2]
	\arrow[from=1-1, to=2-1]
	\arrow[from=1-2, to=2-2]
	\arrow["\alpha", from=2-1, to=2-2]
\end{tikzcd}\]
Denote by $\llog$ the full subcategory of $\plog$ spanned by $\Einf$-log rings.
\end{defn}
\begin{rmk}\label{loopgl}
    The space of homotopy units $GL_1(A)$ of an $\Einf$-ring $A$ is equivalent to the fiber product $\Omega^{\infty}(A)\times_{\pi_0(A)}\pi_0(A)^*$.
\end{rmk}
By \cite[Proposition 5.5.3.12]{lurie2009higher} and  \cite[Corollary 3.2.3.5]{lurie2017higher}, the $\infty$-categories $\plog\Spt$ and $\plog$ are presentable.
Consider the following  monoidal functor $$F_0:\plog\Spt^\otimes\rr\ani^\times\times\Fun(\Delta^1,\Spt)^\otimes\stackrel{1\times{\rm ev}_1}\rr \ani^\times\times\Spt^\otimes.$$
It's easy to see that the functor $F_0$ preserves both limits and colimits, and thus it admits both left and right adjoints, which are automatically symmetric monoidal. Passing to $\Einf$-algebra objects, it turns out that the induced functor
$$F:\plog\rr\alg\times\mon$$
admits both left and right adjoints by \cite[Proposition 5.2.2.8]{lurie2009higher}. In particular, the forget functor $F$ detects limits and colimits of $\Einf$-prelog rings. As a consequence, the limits in the inclusion $\llog\subset\plog$ preserve limits.
\begin{prop}\label{presentability}
The  $\infty$-category $\llog$ is presentable. The inclusion $\llog\subset\plog$ admits a left adjoint $a:\plog\rightarrow\llog$, which provides a presentation of $\llog$. 
\end{prop}
\begin{proof}
The same arguments as in the proof of Proposition \autoref{presentabilityoflog} prove the presentability of $\llog$. Therefore, there is an adjoint pair  
\[\begin{tikzcd}
	\plog & \llog
	\arrow["a", shift left, from=1-1, to=1-2]
	\arrow["i", shift left, from=1-2, to=1-1]
\end{tikzcd}\] 
where $i$ is the inclusion $\llog\subset\plog$. Let us prove that this provides a presentation of $\llog$. We must prove the following facts: $a$ is an accessible localization, i.e., the functor  $i$ preserves $\kappa$-filtered colimits for some regular cardinal $\kappa$.  We have to prove that $GL_1$ preserves $\kappa$-filtered colimits. Let $I\rightarrow \alg$ be a $\kappa$-filtered diagram in $\alg$, we have a natural map of anime $$\varinjlim_{i\in I}GL_1(A_i)\rr GL_1(\varinjlim_{i\in I} A_i).$$
Using Remark \autoref{loopgl}, one can deduce that the homotopy group of $GL_1(R)$ for some $\Einf$-ring $R$ should be 
$$\pi_n(GL_1(R)) = \begin{cases}
\pi_0(R)^* & n=0,\\
\pi_n(R)  & n\geq 1.
\end{cases}$$
Therefore, we have $$\pi_n(\varinjlim_{i\in I}GL_1(A_i))\simeq \pi_n(GL_1(\varinjlim_{i\in I} A_i)).$$ 
\end{proof}
\begin{defn}
    The functor $a:\plog\rightarrow\llog$ is called the \textit{logification} functor.
\end{defn}
\begin{rmk}[Model category theoretical approach of presentability]\cite[Section 4.5.1]{lundemo2022formally}
One can use model category technology to obtain the presentability of the $\infty$-categories $\llog$, as same as Remark\autoref{modelcatrep}. By \cite[Proposition 4.5.1.1]{lundemo2022formally}, there is a combinatorial simplicial model structure, which is called \textit{log model structure}, on the ordinary category $\mathcal P$ of prelog ring spectra. The $\infty$-category associated with this model  category is naturally equivalent to $\llog$. Then using \cite[Proposition A.3.7.6]{lurie2009higher}, one deduces that the $\infty$-category $\llog$ is  presentable. 
\end{rmk}
\begin{rmk}[$\Einf^\mathcal{J}$-log rings]
    In the definition of $\Einf$-log rings, the log structure over an $\Einf$-ring $A$ is recorded by a map  of $\Einf$-monoids $\alpha:M\rightarrow\Omega^\infty A$. {This of course imposes some significant restrictions on the kind of log structure that we are allowing.} The $\Einf$-log rings we define here only capture log structures on $\pi_0$, i.e., 
    only non-invertible elements that we  can handle are in $\pi_0(A)$. On the other hand, in algebraic topology and $K$-theory, one might as well consider log structures over $\Einf$-rings which are determined  by elements in  higher homotopies. For example, consider the ($p$-completed) connective topological $K$-theory spectrum $ku_p$ and its \textit{Adams summand} $l_p$, the natural map $$i:l_p\rr ku_p$$ is possible to view as a  \textit{formally log-\'etale} map (in the sense of Definition\autoref{aaabbbvvvddd}) if we give certain "log structures" induced by the Bott classes $\beta^{p-1}\in\pi_{2p-2}(l_p)=\pi_{2p-2}(ku_p)=\beta^{p-1}\Z_p$ and  $\beta\in\pi_2ku_p=\beta\Z_p$ on both sides respectively, and we denote by 
    $$i:(l_p,\langle\beta^{p-1}\rangle_*)\rr(ku_p,\langle \beta\rangle_*)$$
    the corresponding map of "$\Einf$-log rings", 
    c.f. \cite[Section 6]{Sagave_2014} and \cite[Section 6]{Rognes_2015}. This will make sense if we work in the $\infty$-category of $\Einf^\mathcal{J}$-log rings studied in \cite{Sagave_2012},  \cite{Sagave_2014} and \cite{lundemo2022formally}. 
\end{rmk}
\begin{prop}\label{tensorlog}
  Let $\plog^\otimes\rightarrow\mathbb E^\otimes_{\infty}$\footnote{We use the notation $\Einf^\otimes$ for the $\Einf$-algebra operad, as an  $\infty$-category, it's equivalent to the nerve of the  category of pointed finite sets $\mathrm{N}(\mathbf{Fin}_*)$, see \cite[Example 2.1.1.18]{lurie2017higher}. } be the coCartesian symmetric monoidal structure.   Up to homotopy, there  is a unique symmetric monoidal structure $\otimes$ on $\llog$, such that the logification promotes to a colimit-preserving symmetric monoidal functor 
$$a^\otimes:\plog^\otimes\rr\llog^\otimes.$$
    \begin{proof}
       We let $\llog^\otimes$ be the subcategory of $\plog^\otimes$ spanned by objects having the form
       $$\widetilde {\coprod}_{i\in<n>^\circ}(A_j,M_j):= ((A_1,M_1),(A_2,M_2),...,(A_n,M_n))$$
       such that $(A_i,M_i)\in\llog$. We first prove that $\llog^\otimes\subset\plog^\otimes$ is a Bousfield localization of $\plog^\otimes$. Indeed, let $\widetilde{\coprod}_{j\in<m>^\circ}(B_j,N_j)$ be an object in $\plog^\otimes$, and we have $\widetilde {\coprod}_{j\in<m>^\circ}(B_j,N_j)^a\in\llog^\otimes$. Let $\widetilde {\coprod}_{i\in<n>^\circ}(R_i,L_i)$ be an object in $\llog^\otimes$. We have a natural map 
       $$\theta:\map_{\plog^\otimes}(\widetilde {\coprod}_{j\in<m>^\circ}(B_j,N_j)^a,\widetilde {\coprod}_{i\in<n>^\circ}(R_i,L_i))\rr \map_{\llog^\otimes}(\widetilde {\coprod}_{j\in<m>^\circ}(B_j,N_j),\widetilde {\coprod}_{i\in<n>^\circ}(R_i,L_i)),$$
       induced from the natural map 
       $\widetilde {\coprod}_{j\in<m>^\circ}(B_j,N_j)\rightarrow \widetilde {\coprod}_{j\in<m>^\circ}(B_j,N_j)^a.$
       We will show that $\theta$ is an equivalence. Note that both  sides of the map $\theta$  can be written as the disjoint union over maps $f:<m>\rightarrow <n>$ of $\prod_{i\in<n>}\map_{\plog}(\otimes_{f(j)=i}(A_j,M_j), (R_i,L_i))$ and of $\prod_{i\in<n>}\map_{\plog}(\otimes_{f(j)=i}(A_j,M_j)^a, (R_i,L_i))$ respectively, where the tensor product $\otimes$ is formed in $\plog^\otimes$. We observe that the associated components of the map $\theta$ are equivalences, as the canonical map $\otimes_{f(j)=i}(A_j,M_j)\rightarrow \otimes_{f(j)=i}(A_j,M_j)^a$ becomes a homotopy equivalence after taking logification. This shows that  $\theta$ is an equivalence, and that $\llog^\otimes$ is a Bousfield localization of $\plog^\otimes$, and we denote by $a^\otimes$ a left adjoint of the inclusion. One can see that the underlying functor $a^\otimes_{<1>}:\plog^\otimes_{<1>}\rightarrow\llog^\otimes_{<1>}$ is equivalent to $a$. Using \cite[Lemma 2.2.1.11]{lurie2017higher}, we deduce that  $\llog^\otimes\rightarrow\Einf^\otimes$ is a symmetric monoidal $\infty$-category, and the functor $a^\otimes$ is symmetric monoidal.\\
       The uniqueness of the functor $a^\otimes$ can be deduced from the universal property of Bousfield localizations: for any $\infty$-category $\DD$, and any colimit-preserving functor $F:\plog^\otimes\rightarrow\DD$ such that sending maps having the form $\widetilde{\coprod}_{j\in<m>^\circ}(B_j,N_j)\rightarrow \widetilde{\coprod}_{j\in<m>^\circ}(B_j,N_j)^a$ to equivalences, up to homotopy,  there is a unique factorization
       $$\begin{tikzcd}
\plog^\otimes \arrow[d, "a^\otimes"'] \arrow[r, "F"] & \DD \\
\llog^\otimes \arrow[ru, "F'"', dashed]              &    
\end{tikzcd}$$
through $a^\otimes$. Then the result follows.
    \end{proof}
\end{prop}
\subsection{Truncatedness of \texorpdfstring{$\Einf$}{E-infinity}- and animated log rings}
Let $\CC$ be an $\infty$-category. An object $X \in \CC$ is called $n$-truncated if the associated Yoneda functor $h_X: \CC^{\rm op} \rightarrow \ani$ factors through the subcategory $\ani_{\leq n}$ of $\ani$ spanned by $n$-truncated animae. In this subsection, we study the truncatedness of  $\Einf$- and animated log rings, a necessary prerequisite for analyzing the representability properties of the moduli stack of log structures defined in \autoref{logstacks}.
\begin{defn}
\begin{enumerate}
\item An $\Einf$-(or animated) log ring $(A,M)$ is \textit{connective} if the underlying $\Einf$-ring $A$ is connective.
\item An $\Einf$-(or animated) log ring $(A,M)$ is \textit{discrete} (resp. \textit{n-truncated}) if it is discrete (resp. $n$-truncated) in $\llog$.
\end{enumerate}
\end{defn}
\begin{rmk}
    An $\Einf$-log ring $(A,M)$ is $n$-truncated if and only if both $A$ and $M$ are $n$-truncated. Indeed, the sufficiency is obvious. On the other hand, we let $(A,M)$ be $n$-truncated. Then one has the following equivalences 
    $$\map_{\llog}((\mathbb S\{x\},\coprod_{n\geq 0}B\Sigma_n)^a,(A,M))\simeq\map_{\alg}(\coprod_{n\geq 0}B\Sigma_n,M)\times_{\map_{\mon}(\coprod_{n\geq 0}B\Sigma_n,\Omega^\infty A)}\map_{\alg}(\mathbb S\{x\},A)\simeq M$$ and $$\map_{\llog}((\mathbb S\{x\},0)^a,(A,M))\simeq\map_{\mon}(0,M)\times_{\map_{\mon}(0,A)}\map_{\alg}(\mathbb S\{x\}, A)\simeq\Omega^\infty A.$$ A similar result also holds for animated log rings. 
\end{rmk}
\begin{rmk}
 A discrete $\Einf$-(or animated) ring $R$ may admit non-discrete log structures. For example, let $R\in\alg^{\heartsuit}$ be  an ordinary commutative ring, and let $M$ be an arbitrary non-discrete $\Einf$-monoid. We can consider the unique  map $\mathfrak{o}:M=\coprod_{n\geq 0}B\Sigma_n\rightarrow  R$ determined by $0\in R$.  This monoid map gives rise to an  $\Einf$-log ring $(R,M^a,\mathfrak{o}^a)$, where $M^a:=M\oplus R^*$ is not discrete in general. In the $\Einf^{\mathcal{J}}$-setting, c.f. \cite{Sagave_2014},  discrete $\Einf^{\mathcal{J}}$-log rings with  non-discrete monoids  also naturally appear. For example, one has a formally log-\'etale map $(\Z_p,\langle\beta^{p-1}\rangle_*)\rightarrow (\Z_p\otimes_{l_p}ku_p,\langle\beta\rangle_*)$ given by base change of the map $i:(l_p,\langle\beta^{p-1}\rangle_*)\rr(ku_p,\langle \beta\rangle_*)$ along $l_p\rightarrow\pi_0(l_p)=\Z_p$. Here the $\mathcal{J}$-monoid $\langle\beta^{p-1}\rangle_*$ is equivalent to the natural embedding $\mathcal{QS}^0\times_\Z\N\rightarrow\mathcal{QS}^0$, which is non-discrete. 
\end{rmk}
\begin{defn}\label{admlogrings}
    An $\Einf$-(or animated) log ring $(A,M)$ is called $n$-\textit{admissible} if  $(\pi_0A,M)^a$ is $n$-truncated for some $n<+\infty$. We say $(A,M)$ is admissible if it is $n$-admissible for some $n$.
\end{defn}
\begin{rmk}\label{truncadmissible}
If $(A,M)$ is $n$-admissible and $A$ is $m$-truncated, then $(A,M)$ is $(1+\max\{n,m\})$-truncated in $\llog$. Indeed, we have a pushout square 
$$\begin{tikzcd}
GL_1(A) \arrow[r] \arrow[d] & GL_1(\pi_0A) \arrow[d] \\
M \arrow[r]                 & M^a                   
\end{tikzcd}$$
where $GL_1(A)$ is $m$-truncated and $M^a$ is $n$-truncated by assumption.
\end{rmk}
\begin{rmk}[Strict morphisms\footnote{We say that $f:(A,M)\rightarrow (B,N)$ is a strict morphism, if $(B,N)\simeq(B,M)^a$, or equivalently, $(B,N)\simeq (A,M)\otimes_AB$. Here the tensor product is formed in $\llog^\otimes$ given by Proposition\autoref{tensorlog}}]\label{admissiblelocal}
    Let $f:(A,M)\rightarrow(B,N)$ be a strict map. Then if  $(A,M)$ is $n$-admissible, then so is $(B,N)$. Conversely, if $\spec\pi_0 B\rightarrow\spec\pi_0A$ is a flat cover, and $(B,N)$ is $n$-admissible, then so is $(A,M)$, see Lemma\autoref{1884310}. 
\end{rmk}
\subsection{Comparison of animated and \texorpdfstring{$\Einf$}{E-infinity}-contexts}\label{einfvsani}
Similar to derived algebraic geometry, animated and $\Einf$-log rings provide two foundational frameworks for derived logarithmic geometry. The behavior of animated log rings is well-studied in \cite{sagave2016derived} and \cite{binda2023hochschild}, while the $\Einf$-contextual approach is explored in \cite{rognes2009topological}, \cite{Sagave_2014}, and \cite{lundemo2023deformation}.\\
Both types of log rings share many common properties and permit analogous constructions of associated invariants. However, they are not equivalent—even when restricted to log rings over a field of characteristic $0$—which stands in contrast to the relationship between commutative algebras in the animated and $\Einf$-contexts. This discrepancy arises because the definition of an $\Einf$-log ring $(A, M, \alpha)$ involves an $\Einf$-monoid $M$ and a morphism of $\Einf$-monoids $\alpha: M \rightarrow \Omega^\infty A$, which is not always equivalent to a morphism of animated monoids. 
\subsubsection{Animated monoids versus $\Einf$-monoids}
By the universal property of animation, the canonical fully faithful embedding $\mone^\heartsuit\rightarrow\mone$ can be uniquely extended to a functor $\theta:\mone^\Delta\rightarrow\mone
$, which is given by taking left Kan extension. We refer to it as the forgetful functor and refer to the object $\theta(M)$  as the underlying $\Einf$-monoid of $M$.
\begin{lem}\label{ani-sptmon}
    The functor $\theta$ preserves both arbitrary small colimits and arbitrary small limits, and is conservative. The functor $\theta$ is both monadic and comonadic.
\end{lem}
\begin{proof}
    The colimit-preserving property is clear.   To prove that $\theta$ preserves limits and is conservative, we need to show that the composition $\mone^\Delta\stackrel{\theta}\rightarrow\mone\stackrel{F}\rightarrow\ani$ preserves arbitrary limits, sifted colimits  and is conservative; here the functor $F$ is the forgetful functor. This is because $F$ detects limits, sifted colimits and equivalences of $\Einf$-monoids.  Since any animated monoid $M$ is equivalent to some sifted colimit of a family of finitely generated free monoids, the functor $F\circ\theta$ is equivalent to the left Kan extension of  its restriction $F\circ\theta|_{\mone^\heartsuit}$ along the embedding $\mone^\heartsuit\subset\mone^\Delta$,
    which is equivalent to the functor $M\mapsto M(\mathbb N).$ The monadicity and comonadicity follow from the fact that the functor $\theta$ is conservative and preserves geometric realizations (i.e. $\Delta^{\rm op}$-indexed colimits) and totalizations (i.e. $\Delta$-indexed limits) \cite[Theorem 4.7.0.3]{lurie2017higher}.
\end{proof}
Applying the adjoint functor theorem for presentable $\infty$-categories, the functor $\theta$ admits a left adjoint $\theta^L$ and a right adjoint $\theta^R.$ We have an adjoint triple 
$$\begin{tikzcd}
\mone^\Delta \arrow[r, "\theta" description] & \mone \arrow[l, "\theta^L" description, shift right=3] \arrow[l, "\theta^R" description, shift left=3]
\end{tikzcd}$$
The monadicity of the functor $\theta$ implies that there is a  monad $T$, whose underlying functor is equivalent to $\theta\circ\theta^L$ on $\mone$, exhibits the $\infty$-category $\mone^\Delta$ as the $\infty$-category of $T$-modules ${\mathbf{L}}\Mod_T(\mone)$ in $\mone.$ We denote by $\theta':{\mathbf{L}}\Mod_T(\mone)\rightarrow \mone^\Delta$ the unique equivalence. In other words, to give an animated monoid $M$, it is equivalent to give  a $T$-module $M_{-1}$, whose underlying $\Einf$-monoid is equivalent to $\theta(M)$. Let $M_\bullet:[n]\mapsto M_n=T^{\circ (n+1)}   M_{-1}$ be the corresponding Bar resolution of $M_{-1}$ with respect to $T$. The geometric realization of $\theta'(M_\bullet)$ is canonically equivalent to $M$.\\

\subsubsection{Animated log rings versus $\Einf$-log rings}
The natural embeddings 
$\alg^{\heartsuit}\hookrightarrow\algcn$ and $\mon^{\heartsuit}\hookrightarrow\mone$ define a functor $\plog^{\rm ord,sfp}\rightarrow\plog,$
which is given by sending  a free prelog  algebra $(\Z[x_1,...,x_k,...,x_n],\langle x_1,...,x_k\rangle)$ to the underlying  $\Einf$-prelog ring $(\Z[x_1,...,x_k,...,x_n],\langle x_1,...,x_k\rangle)$, where the prelog structure is given by the following composition
$$\Sigma^{\infty}_+\langle x_1,...,x_k\rangle=\mathbb S[ x_1,...,x_k]\rr\Z[x_1,...,x_k]\rr\Z[x_1,...,x_k,...,x_n].$$
By the universal property of animation, up to homotopy, there exists a unique functor 
$$\Theta: \plogd\rr\plog_\Z,$$
 $$(\Z[x_1,...,x_k,...,x_n],\langle x_1,...,x_k\rangle)\mapsto(\Z[x_1,...,x_k,...,x_n],\langle x_1,...,x_k\rangle).$$
 Here we let $\plog_\Z$ be the slice category of $\plog$ below $(\Z,\{\pm 1\})$.
 The functor preserves both limits and colimits. Indeed, limits and colimits on both sides are calculated pointwise, and the forgetful functors $\theta:\mond\rightarrow\mone$ and $\theta_{\rm alg}:\alg^{\Delta}\rightarrow\algcn$ preserve both arbitrarily small limits and arbitrary small colimits as discussed before in the proof of Lemma\autoref{ani-sptmon}. Using the adjoint functor theorem for presentable $\infty$-categories, we get a left adjoint $\Theta^L$ and a right adjoint $\Theta^R$ of $\Theta$:
\[\begin{tikzcd}
	\plog_\Z && \plogd
	\arrow["{\Theta^L}"{description}, shift left=4, from=1-1, to=1-3]
	\arrow["{\Theta^R}"{description}, shift right=4, from=1-1, to=1-3]
	\arrow["\Theta"{description}, from=1-3, to=1-1]
\end{tikzcd}\] 
Let $(A,M,\alpha)\in\llogd$ be an animated log ring, then the image of $(A,M,\alpha)$ under $\Theta$ is the $\Einf$-pre log ring $(\Theta(A),M,\alpha')$, where $$\alpha':\Sigma^{\infty}_+M=\mathbb S[M]\rr\Z[M]\rr \Theta(A).$$ 
This gives rise to a  morphism of $\Einf$-monoids
$M\rightarrow\Omega^{\infty}\Theta(A)\simeq \theta(A)$.
The  induced map  $M\times_{\theta (A)}GL_1(\Theta(A))\rightarrow GL_1(\Theta(A))\simeq \theta(GL_1(A))$ is an equivalence because we already proved the fact that  the forgetful functor $\theta$  preserves the space of homotopy units.
Therefore, the functor $\Theta$ preserves log objects, and when we restrict it to $\Theta_{\rm Log}:\llogd\rightarrow \llog$, it still preserves both limits and colimits. We therefore get an adjoint triple
\[\begin{tikzcd}
	\llog_\Z && \llogd
	\arrow["{\Theta_{\rm Log}^L}"{description}, shift left=4, from=1-1, to=1-3]
	\arrow["{\Theta_{\rm Log}^R}"{description}, shift right=4, from=1-1, to=1-3]
	\arrow["\Theta_{\rm Log}"{description}, from=1-3, to=1-1]
\end{tikzcd}\]

\begin{prop}\label{ani-spt}
The functors $\Theta$ and $\Theta_{\rm Log}$ are both monadic and comonadic.
\end{prop}
\begin{proof}
  This is because  the functors $\Theta$ and $\Theta_{\rm Log}$ are conservative, preserving arbitrary small  limits and arbitrary small colimits.
\end{proof}

\section{Deformation theory of \texorpdfstring{$\Einf$}{E-infinity}- and animated log rings}\label{definiiton of einf-log rings}
In this section, we study the deformation theory of $\Einf$-log rings. We begin by reviewing the formalism of \textit{replete tangent bundles} and the definition of \textit{Gabber's cotangent complexes} introduced by Lundemo \cite{lundemo2023deformation}, which generalizes the construction of cotangent complexes for $\Einf$-rings from \cite{lurie2017higher}. We then introduce the notion of \textit{replete tangent correspondence} and use it to define the $\infty$-category of \textit{log derivations} for $\Einf$-log rings.\\
Subsequently, we define \textit{deformations} of $\Einf$-log rings and formulate a derived version of logarithmic deformation theory. This framework generalizes classical deformation theory to the context of spectral log geometry. Finally, we consider the animated analogues of the replete tangent bundle and replete tangent correspondence formalisms, thereby formulating a deformation theory for animated log rings.
\begin{no}
From now on, we only  consider connective $\Einf$-log rings, and denote by $\llog$ the $\infty$-category of connective $\Einf$-log rings. We will call a "connective $\Einf$-log ring" an $\Einf$-log ring, unless we specifically say "a non-connective $\Einf$-ring". 
\end{no}

\subsection{Replete tangent bundles and replete tangent correspondences}
In this subsection, we develop the theoretical foundation for studying the deformation theory of $\Einf$-log rings. We first introduce the notion of replete tangent bundles and Gabber's cotangent complex for presentable $\infty$-categories, which generalizes Lundemo's replete bundle formalism and recovers Gabber's cotangent complex for $\llog$ as presented in \cite{lundemo2023deformation}. We then define the concept of replete tangent correspondences and use it to construct the $\infty$-category of replete derivations.
\subsubsection{Tangent bundles}\label{321321}
Denote by $\CC$ a presentable  $\infty$-category. The \textit{tangent bundle} \cite[Definition 7.3.1.9]{lurie2017higher}   of  $\CC$ is the \textit{stable envelope}\footnote{This is a relative version of stabilization of an $\infty$-category, see \cite[Definition 7.3.1.1]{lurie2017higher}.}  of the evaluation map $\mathrm{ev}_1: \Fun(\Delta^1,\CC)\rightarrow\CC$. This is a presentable fibration  $\Omega^\infty:\tang_\mathcal{C}\rightarrow\Fun(\Delta^1,\CC)$. We get  a sequence of presentable fibrations
\[\begin{tikzcd}
	{{\rm T}_{\CC}} & {\Fun(\Delta^1,\CC)} & \CC
	\arrow["{\Omega^{\infty}}", from=1-1, to=1-2]
	\arrow["{ev_1}", from=1-2, to=1-3]
\end{tikzcd}\]
The functor $\Omega^{\infty}$, which we will call the \textit{infinity loop functor}, admits a relative left adjoint $\LL$, see \cite[Section 7.3.2]{lurie2017higher} for the notion of relative adjunctions,  which we will call a \textit{cotangent complex functor} for $\CC$. For an object $B\in\CC$, denote by $\CMod_B$ the fiber of the functor ${\rm T}_{\CC}\rr\CC$ at $B$, i.e., we let
$\CMod_B:={\rm T}_{\CC}\times_{\CC}\{B\}\simeq\Spt(\CC_{/B})$\footnote{The stabilization  $\Spt(\CC)$ of $\CC$ can be realized as the $\infty$-category ${\rm Exc}_*(\ani_*^{\rm fin},\CC)$ of \textit{spectrum objects} in $\CC$ defined in \cite[Definition 1.4.2.8]{lurie2017higher}. The functor $\CC\mapsto\Spt(\CC)$ could be characterized by certain universal property, see \cite[Proposition 1.4.2.22 and Corollary 1.4.2.23]{lurie2017higher}. Roughly speaking, it is universal for left exact (and hence right exact) functors from $\CC$ to  stable $\infty$-categories. Informally, one can view $\Spt(\CC)$ as the $\infty$-category of generalized homology theories of topological spaces taking values in $\CC$.}, which is a presentable stable $\infty$-category. We have the following fact, given by the adjunction between $\LL$ and $\Omega^\infty$,
\begin{prop}\label{cotsus}
    Let $(f:A\rightarrow B)\in \Fun(\Delta^1,\CC)$ be a morphism in $\CC,$ and let $M$ be an object in $\CMod_B$, then there is a functorial equivalence of animae,
    $$\map_{\CC_{/B}}(f,\Omega^{\infty}M)\simeq\map_{\CMod_B}(\LL_f,M).$$
\end{prop}
\begin{rmk}[Tangent bundle of $\Einf$-algebras]\label{tangentmodule}
In \cite{lurie2017higher}, the tangent bundle of $\infty$-category of   $\Einf$-rings $\alg$ has been well studied. Let $A$  be an $\Einf$-ring, the stabilization $\CMod_A=\Spt(\alg_{/A})$ 
is naturally equivalent to $\Mod_A$, the $\infty$-category of $A$-module spectra.  The infinity loop functor $\Omega^{\infty}:\Mod_A\rightarrow\alg_{/A}$
is given by sending an $A$-module $M$ to $A\oplus M\rightarrow A$ and the cotangent complex $\LL_{A/(-)}:\alg_{/A}\rightarrow\Mod_A$ coincides with the \textit{topological Andr\'e-Quillen homology} defined in \cite{article}.
\end{rmk}
\begin{rmk}[Passing to connective algebras]\label{tangentconnective}
Let $(\mathcal{C}^\otimes,\tau)$ be a presentable symmetric monoidal stable $\infty$-category equipped with a complete $t$-structure, such that the monoidal product $\otimes$  is right $t$-exact in $\CC$. In this case, the connective cover functor $\tau_{\geq 0}:\alg(\CC)\rightarrow\alg(\CC_{\geq 0})$ is a right adjoint of the inclusion $\alg(\CC_{\geq 0})\subset\alg(\CC).$ We claim that this induces an equivalence ${\tau}':\Spt(\alg(\CC)_{/A})\stackrel{\simeq}\rightarrow\Spt(\alg(\CC_{\geq 0})_{/A})$ for any $A\in\alg(\CC_{\geq 0})$. Indeed, the functor ${\tau}'$ is a right adjoint functor by the definition of stabilizations, whose left adjoint  $\widetilde{(-)}:\Spt(\alg(\CC_{\geq 0})_{/A})\rightarrow\Spt(\alg(\CC)_{/A})$ is given by sending a spectrum object $F:\ani_*^{\rm fin}\rightarrow \CC_{\geq 0}$ to $\widetilde{F}:X\mapsto\varinjlim_n\Omega^n_{\CC}F(\Sigma^nX)$. We show that $\widetilde{(-)}$ is a quasi-inverse of ${\tau}'$. First, we note that  ${\tau}'(\widetilde{F})\stackrel{\simeq}\leftarrow F$. On the other hand, assuming that $G\in\Spt(\alg(\CC)_{/A})$, we have to show that the natural map $\widetilde{{\tau}'(G)}\rightarrow G$ is an equivalence. Since equivalences of algebra objects are detected by the forgetful functor $\alg(\CC)\rightarrow\CC$. In addition, since $\CC$ is $t$-complete, the equivalence of a map in $\CC$ can be detected by the induced maps of homotopies. Since $\ani^{\rm fin}_*$ is  generated by the $0$-sphere $S^0$ under finite colimits, we just need to show that $$\widetilde{{\tau'}(G)}(S^0)=\varinjlim_n\Omega^n_{\CC}\tau_{\geq 0}G(S^n)\rr G(S^0)$$
is an equivalence. This is clear because this map becomes an isomorphism passing to $n$-th homotopy $\pi_n$ for any integer $n$. As a consequence, we have $$\tang_{\alg(\CC_{\geq 0})}\simeq\tang_{\alg(\CC)}\times_{\alg(\CC)}\alg(\CC_{\geq 0}).$$
\end{rmk}
\subsubsection{Tangent correspondences}
Let $\mathcal C$ and $\mathcal D$ be $\infty$-categories. A \textit{correspondence} $\mathfrak M:\mathcal C\rightsquigarrow\mathcal D$ from $\CC$ to $\DD$ is a relative $\infty$-category 
$\mathfrak M\rightarrow\Delta^1$, such that $\mathfrak M_0\simeq\CC$ and $\mathfrak M_1\simeq \DD.$ 
We have a functor $\pi_\mathfrak M:\Fun_{\Delta^1}(\Delta^1,{\mathfrak M})\rightarrow \CC\times \DD$ which is given by the product  of the two evaluation functors at the source ${\rm ev}_0$ and at the target ${\rm ev}_1$. The functor 
$\pi_{\mathfrak M}$ is a \textit{bifibration} in the sense of \cite[Definition 2.4.7.2]{lurie2009higher}, see \cite[Proposition 2.4.7.10]{lurie2009higher}. On the other hand, \cite[Theorem A and Theorem B]{stevenson2018modelstructurescorrespondencesbifibrations} shows that the functor
$\mathfrak M\mapsto \pi_\mathfrak M$ gives rise to an equivalence of the $\infty$-category ${\rm Corr}(\CC,\DD)$ of correspondences from $\CC$ to $\DD$,
and the $\infty$-category ${\rm BFib}_{/\CC\times \DD}$ of bifibrations over $\CC\times \DD.$ A variant of \textit{Grothendieck construction} gives
an equivalence of $\infty$-categories $${\rm BFib}_{/\CC\times\DD}\stackrel{\simeq}\rr\Fun(\CC^{\rm op}\times\DD,\ani), \mathscr E_{/\CC\times\DD}\mapsto ((c,d)\mapsto\mathscr E_{(c,d)}),$$
c.f. \cite[Theorem 1.1.7]{ayala2020fibrationsinftycategories}. We can still define an $(\infty,2)$-category ${\mathbf{ Corr}}$  whose objects are $\infty$-categories, and mapping $\infty$-categories ${\rm Corr}(\CC,\DD)\simeq{\rm BFib}_{/\CC\times\DD}\simeq\Fun(\CC^{\rm op}\times\DD,\ani),$ see \cite[Theorem 1.26 and Definition 1.27]{ayala2020fibrationsinftycategories}.
The composition of  correspondences
is given by the coends:
\begin{align*}
-\circ -:\Fun(\CC^{\rm op}\times\DD,\ani)&\times\Fun(\DD^{\rm op}\times\mathcal{E},\ani)\rr\Fun(\CC^{\rm op}\times\mathcal E,\ani),\\&(F,G)\mapsto\int^{d\in\DD}F(-,d)\times G(d,-).
\end{align*}
\begin{exmp}[Fundamental correspondence]
Let $\mathfrak M_\CC^0:\CC\rightsquigarrow\Fun(\Delta^1,\CC)$ be a correspondence, whose {correspondent} functor is defined by
$$ \mathfrak M_\CC^0:\mathcal{C}^{\rm op}\times\Fun(\Delta^1,\mathcal{C})\rr
\ani,
(A,f:B\rightarrow  C)\mapsto\{A\}\times_{\Fun(\Delta^{\{0\}},\CC)}\Fun(\Delta^2,\CC)\times_{\Fun(\Delta^{\{1,2\}},\CC)}\{f\}.$$
In \cite[Notation 7.3.6.4]{lurie2017higher}, the fundamental correspondence is  characterized  by the following universal property:
Let $A\in \sset_{/\Delta
^1}$ be a simplicial set over $\Delta^1$, let $K\subset\Delta^1\times\Delta^1$ be a sub simplicial set spanned by verticals
$(i,j), i\leq j.$ We let $\overline{A}$
 be the pullback
 $$\begin{tikzcd}
\overline{A} \arrow[d] \arrow[r] & K \arrow[d]            \\
\Delta^1\times A \arrow[r]       & \Delta^1\times\Delta^1
\end{tikzcd}$$
Then there is a functorial isomorphism of sets
$$\Hom_{\sset_{/\Delta^1}}(A,\mathfrak M^0_\CC)\stackrel{\simeq}\rr\Hom_{\sset}(\overline{A},\CC).$$
\end{exmp}
\begin{exmp}[Categorical cylinder]
    Let $F:\CC\rightarrow\DD$ be a functor. The \textit{categorical cylinder} $\mathbf{C}(F):\DD\rightsquigarrow\CC$ is given by:
    $$\mathbf{C}(F):\DD^{\rm op}\times\CC\rr\ani, (d,c)\mapsto\map_{\DD}(d,F(c)).$$
\end{exmp}
\begin{defn}
    Let $\CC$ be a presentable $\infty$-category, the \textit{tangent correspondence} of $\CC$ is the composition of correspondences
    $$\MT_\CC:=\mathbf{C}(\Omega^{\infty})\circ\mathfrak M^0_\CC:\CC\rightsquigarrow\Fun(\Delta^1,\CC)\rightsquigarrow \tang_{\CC}$$
\end{defn}
By definition, we have $\MT_\CC\times_{\Delta^1}\{0\}\simeq\CC$ and $\MT_\CC\times_{\Delta^1}\{1\}\simeq\tang_\CC.$ Let $A\in\CC$ and $(M,B)\in\tang_{\CC}.$  By the definition of tangent correspondence and the adjoint relation asserted in Proposition\autoref{cotsus}, we get that the anima 
$\map_{\MT_{\CC}}(A,(M,B))$  is spanned by morphisms in $\CMod_B$ which have the  form $f_!\LL_A\rightarrow  M,$ in which $f:A\rightarrow B$ is an arrow in $\CC$, and $f_!$ is the corresponding coCartesian arrow in $\mathrm{T}_{\CC}$. 

Hence we have the following definition.
\begin{defn}\cite[Definition 7.4.1.1]{lurie2017higher}
 The $\infty$-category of derivations $\Der(\CC)$ in  $\CC$ is the sub-category of $\Fun(\Delta^1, \MT_\CC)$ spanned  by  functors  
  $d:\Delta^1\rightarrow \MT_\CC$ such that the composition 
  $$\Delta^1\rr \MT_{\CC}\rr \Delta^1\times\Fun(\Delta^1, \CC)\stackrel{{\rm ev}_1}\rr \Delta^1\times \CC$$
  is the inclusion $\Delta^1\rightarrow \Delta^1\times\{A\}\subset \Delta^1\times \mathcal C$ for some $A\in\CC$.  We will use the notation $d:A\rightarrow M$ for a derivation, where $M=d(1)\in\tang_{\CC}\times_{\mathcal{C}}\{A\}.$
\end{defn}
\begin{rmk}
Assume that $\CC$ has an initial object $\emptyset$. The cotangent complex $\LL_{A/\emptyset}$ is initial in the $\infty$-category $\Der(\CC)_A=\Der(\CC)\times_\CC\{A\}$. The relative cotangent complex $\LL_{B/A}$ of a map $f:A\rightarrow B$ is equal to the  cofiber of the natural map
$\LL_{A/\emptyset}\rightarrow\LL_{B/\emptyset}$ in $\MT_\CC$. Therefore we get  a cofiber sequence 
$$f_{!}\LL_{A/\emptyset}\rr\LL_{B/\emptyset}\rr\LL_{B/A}$$
in $\CMod_B.$
\end{rmk}

\subsubsection{Replete tangent correspondences}
Using the tangent correspondence, one can define   cotangent complexes ( i.e., \textit{topological Andr\'e-Quillen homology}) and derivations  for $\Einf$-rings. One can also develop the   deformation theory  in the spectral  context \cite[Chapter IV,V]{lurie2018spectral}, which is almost  parallel to the deformation theory of classical commutative rings (and schemes) developed in \cite{grothendieck1957fondements},\cite{illusie1971complexe1}and \cite{illusie1971complexe2}. Unfortunately, the tangent bundle formalism is not suitable for logarithmic geometry. One reason  is that there is not a good definition of "\textit{log modules}" over log rings. Of course, there is a stable $\infty$-category $\CMod_{(A,M)}$ associated with the $\Einf$-log ring $(A,M)$ constructed in the last section by letting $\CC=\llog$ be the $\infty$-category of $\Einf$-log rings. Sadly, this category is very difficult to handle, and what's worse is that  we cannot do any computation  in it. So we will provide a variant of tangent correspondences, which is called \textit{replete tangent correspondences}. We generalize the constructions of \textit{replete tangent bundles} for $\Einf$-log rings in \cite{lundemo2023deformation}. 
\begin{defn}\label{gabtan}
Let $U:\CC\rightarrow\DD$ be a presentable fibration. A replete tangent bundle with respect to $U$ is an $\infty$-category $\tang^\DD_\CC$, such that there exists an  $\infty$-category $\Fun^{\DD}(\Delta^1,\CC)\subset\Fun(\Delta^1,\CC)$, and a relative  Bousfield localization  $$\begin{tikzcd}
{\Fun(\Delta^1,\CC)} \arrow[rd, "{\rm ev}_1"'] \arrow[rr, shift left] &     & {\Fun^\DD(\Delta^1,\CC)} \arrow[ld] \arrow[ll, hook, shift left] \\
& \CC &             
\end{tikzcd}$$
, and there exist presentable fibrations  $U_*:\Fun^{\DD}(\Delta^1,\CC)\rightarrow\CC$ and 
$U_*':\tang_{\CC}^{\DD}\rightarrow\tang_{\DD}$, 
fitting  into the following commutative diagram 
\[\begin{tikzcd}
	{\tang_{\CC}} && {\tang_{\CC}^{\DD}} & {\tang_{\DD}} \\
	\\
	{\Fun(\Delta^1,\CC)} && {\Fun^{\DD}(\Delta^1,\CC)} & {\Fun(\Delta^1,\DD)} \\
	& \CC && \DD
	\arrow[from=1-1, to=1-3]
	\arrow["{\Omega_{\CC}^{\infty}}"', from=1-1, to=3-1]
	\arrow["{U_*'}", from=1-3, to=1-4]
	\arrow["{\Omega^{\infty}}"', from=1-3, to=3-3]
	\arrow["{\Omega_{\DD}^{\infty}}", from=1-4, to=3-4]
	\arrow[from=3-1, to=3-3]
	\arrow[from=3-1, to=4-2]
	\arrow["{U_*}", from=3-3, to=3-4]
	\arrow[from=3-3, to=4-2]
	\arrow[from=3-4, to=4-4]
	\arrow["U", from=4-2, to=4-4]
\end{tikzcd}\]
where the functor $\Omega^{\infty}:\tang_{\CC}^{\DD}\rightarrow\Fun^{\DD}(\Delta^1,\CC)$ is the  stable envelope  of the functor $U_*:\Fun^{\DD}(\Delta^1,\CC)\rightarrow\CC$. In addition, the following conditions are satisfied:
\begin{enumerate}
\item The functor $U$ is a right adjoint, whose left adjoint $G$ is fully faithful 
\item The functor $U'_*$ is a right adjoint.
 
    \item The functor $\Fun^{\DD}(\Delta^1,\CC)\rightarrow\CC$ is a presentable fibration.
     \item The commutative diagram 
\[\begin{tikzcd}
	{\tang_{\CC}^{\DD}} & {\tang_{\DD}} \\
	\CC & \DD
	\arrow[from=1-1, to=1-2]
	\arrow[from=1-1, to=2-1]
	\arrow[from=2-1, to=2-2]
	\arrow[from=1-2, to=2-2]
\end{tikzcd}\]
is a pullback square.
\end{enumerate}
\end{defn}
 By the definition of replete tangent bundle given above,  the presentable fibration $$\Omega^{\infty}:\tang_{\CC}^{\DD}\rr\Fun^\DD(\Delta^1,\CC)$$
 admits a left  adjoint, which we denote  by $\mathcal L^{\DD}$. We have the following definition:
 \begin{defn}
The \textit{Gabber's cotangent complex} of the presentable fibration  $U:\CC\rightarrow\DD$ is the composition 
$$\LL^{\DD}:\Fun(\Delta^1,\CC)\rr\Fun^{\DD}(\Delta^1,\CC)\stackrel{\mathcal L^{\DD}}\rr\mathrm{T}_{\CC}^{\DD}.$$
\end{defn}
By definition, the Gabber's cotangent complex $\LL^{\DD}:\Fun(\Delta^1,\CC)\rightarrow\mathrm{T}_{\CC}^{\DD}$ admits a right adjoint $\Omega^{\infty,\DD}:\Fun(\Delta^1,\CC)\rightarrow\mathrm T_{\CC}^{\DD}.$

\begin{defn}
    The replete tangent correspondence $\MTD_\CC$ of $U:\CC\rightarrow\DD$ is defined as the compositions of correspondences  in the following diagram 
\[\begin{tikzcd}
	\CC & {\Fun(\Delta^1,\CC)} \\
	& {\Fun^{\DD}(\Delta^1,\CC)} & {\tang^\DD_{\CC}}
	\arrow["{{\mathfrak M^0}}", squiggly, from=1-1, to=1-2]
	\arrow["{\mathbf{C}(i)}", squiggly, from=1-2, to=2-2]
	\arrow["{\mathbf{C}(\Omega^{\infty})}", squiggly, from=2-2, to=2-3]
\end{tikzcd}\]
where the functor  $i:\Fun^\DD(\Delta^1,\DD)\rightarrow\Fun(\Delta^1,\CC)$ is the right adjoint of the Bousfield localization $\Fun(\Delta^1,\CC)\rightarrow\Fun^\DD(\Delta^1,\DD).$
\end{defn}
    \begin{defn}\label{LDER}
 The $\infty$-category of replete  derivations $\Der^{\DD}(\CC)$ in  $\CC$ is the subcategory of $\Fun(\Delta^1, \MTD_{\CC})$ spanned by these functors  
  $d:\Delta^1\rightarrow \MTD_{\CC}$ such that the composition 
  $$\Delta^1\rr \MTD_{\CC}\rr \Delta^1\times \CC$$
  is the functor $\Delta^1\rightarrow \Delta^1\times\{A\}\subset \Delta^1\times C$ for some $A\in\CC.$ For such functors we use the notation $d: A\rightarrow M$, where $M$ is the image of $d$ at $1\in\Delta^1$
  $$d(1)\in \tang^{\DD}_{\CC}\times_\CC\{A\}\simeq \CMod_{U(A)}\simeq\Spt(\DD_{/U(A)}).$$ 
\end{defn}
The $\infty$-category of replete  derivations $\Der^{\DD}(\CC)$ is  equivalent to the fiber product $$\Fun(\Delta^1,\MTD_{\CC})\times_{\Fun(\Delta^1,\Delta^1\times\CC)}\CC.$$
\begin{rmk}
Assume that $\CC$ has an initial object $\emptyset$. The Gabber's cotangent complex $\LL^{\DD}_{A/\emptyset}$ is initial in the $\infty$-category $\Der^\DD(\CC)_A=\Der^\DD(\CC)\times_\CC\{A\}$. The relative Gabber's cotangent complex $\LL^\DD_{B/A}$ of a map $f:A\rightarrow B$ is equal to the  cofiber of the natural map
$\LL^\DD_{A/\emptyset}\rightarrow\LL^\DD_{B/\emptyset}$ in $\MTD$. Therefore we get  a cofiber sequence  
$$f_{!}\LL^\DD_{A/\emptyset}\rr\LL^\DD_{B/\emptyset}\rr\LL^\DD_{B/A}$$
in $\CMod_B.$ Here we let $f_!$ be a functor $\CMod_A\rightarrow \CMod_B$ given by  the coCartesian fibration  $\tang^\DD_\CC\rightarrow\CC$. 
\end{rmk}
\subsection{Log derivations}
We apply the constructions of  replete tangent bundles and replete tangent correspondences to $\Einf$-log rings. For this, we introduce the notion of exactification of $\Einf$-log rings, which generalizes the definition of \textit{repletion of ($\Einf$-)log rings} introduced by Kato \cite[ Section 2]{Kato1988} and Rognes \cite[Section 3]{rognes2009topological} for virtually surjective maps. We will also define the $\infty$-category of log derivations in this subsection.\\
\subsubsection{Exactifications}
We give a categorical description of the conception \textit{exactifications} of log rings  introduced in  \cite{Kato1988}, \cite{rognes2009topological} and \cite{sagave2016derived} in the spectral context.   This construction is essential for building  the deformation theory in logarithmic geometry.
\begin{lem}\label{Cf}
    Let $f:M\rightarrow N$ be a morphism of  $\Einf$-monoids, and let $\CC(f)$ be the subcategory of $\mone_{M//N}$ spanned by objects $M\stackrel{a}\rightarrow M'\stackrel{b}\rightarrow N$, such that the composition $b\circ a$ is equivalent to $f$, and 
    $a^{\rm gp}: M^{\rm gp}\stackrel{\simeq}\rightarrow M'^{\rm gp}$. Then $\CC$ admits a terminal object $M\rightarrow M^{\rm ex}\rightarrow N$. 
\end{lem}
\begin{proof}
First, the $\infty$-category $\CC(f)$ is filtered. Indeed, any functor $F:K\rightarrow\CC(f),k\mapsto(M\rightarrow M'_k\rightarrow N)$ can be extended to $F^\triangleright:K^\triangleright\rightarrow\CC(f)$, as we can let the value $F^\triangleright(\infty)$ at the cone point $\infty$ be $(M\rightarrow M\coprod_{\colim_KM}\colim_K M'_k\rightarrow N)$.\\
We let $M^{\rm ex}:=\colim_{M'\in\CC(f)}M'.$ Then $(M\rightarrow M'\rightarrow N)\in\CC(f)$ since there is an equivalence $M^{\rm ex,gp}\simeq(\colim_{M'\in\CC(f)}M')^{\rm gp}\simeq\colim_{M'\in\CC(f)}M'^{\rm gp}=M^{\rm gp}. $ 
\end{proof}
Therefore, we have the following definition:
\begin{defn}
    Let $f:M\rightarrow N$ be a morphism of $\Einf$-monoids. We say that $f$ is exact if $f\simeq f^{\rm ex}$. Here $f^{\rm ex}$ is a terminal object in $\CC(f)$.
    \end{defn}
\begin{cor}\cite[Proposition 2.21]{binda2023hochschild}\label{exactificationofmonoids}
 If the morphism  $f:M\rightarrow N$ admits a section $s:N\rightarrow M$, then the sequence $N\stackrel{s}\rightarrow M^{\rm ex}\stackrel{f}\rightarrow N$ splits,
    and $M^{\rm ex}\simeq N\oplus M^{\rm gp}/N^{\rm gp}$. The functor $$Q:\mone^{\rm ex}_{N//N}\rightarrow \mone^{\rm gp}, (N\rightarrow M\rightarrow N)\mapsto M/N$$ is an equivalence.  Here $\mone^{\rm ex}_{N//N}$ is the subcategory of $\mon_{N//N}$ spanned by exact maps. 
\end{cor}
\begin{proof}
    We see that any $M'\in\CC(f)$, up to homotopy, there is a unique morphism $M'\rightarrow N\oplus M^{\rm gp}/N^{\rm gp}$, such that the following diagram commutes,
   $$ \begin{tikzcd}
M' \arrow[d] \arrow[r]         & N\oplus M^{\rm gp}/N^{\rm gp} \arrow[r] \arrow[d] & N \arrow[d] \\
M^{\rm gp} \arrow[r, "\simeq"] & N^{\rm gp}\oplus M^{\rm gp}/N^{\rm gp} \arrow[r]  & N^{\rm gp} 
\end{tikzcd}$$
This implies that $N\oplus M^{\rm gp}/N^{\rm gp}$ is a terminal object in $\CC(f)$. In other words, this gives rise to an equivalence $M^{\rm ex}\simeq N\oplus M^{\rm gp}/N^{\rm gp}$.\\
The functor $Q$ is clearly essentially surjective. We want to show that $Q$ is fully faithful. Indeed, let $M,M'\in\mon^{\rm ex}_{N//N}$ we have equivalences
\begin{align*}
\map_{\mone^{\rm ex}_{N//N}}(M,M')&\simeq \map_{\mone^{\rm ex}_{N//N}}(N\oplus M^{\rm gp}/N^{\rm gp},N\oplus M'^{\rm gp}/N^{\rm gp})\\
&\simeq\map_{\mone^{\rm gp}}(M^{\rm gp}/N^{\rm gp},M'^{\rm gp}/N^{\rm gp})\\
&\simeq\map_{\mone^{\rm gp}}(Q(M),Q(M')).
\end{align*}
\end{proof}

    \begin{lem}\label{exactificationpresentable}
Let $\mone^{\rm ex}_{/N}$ be the  subcategory of $\mone_{/N}$ spanned by exact morphisms. Then the inclusion $\mone^{\rm ex}_{/N}\subset\mone_{/N}$ admits a left adjoint $M\mapsto M^{\rm ex},$ which is called the exactification functor. The exactification functor provides a presentation of $\mone^{\rm ex}_{/N}$, and therefore it is presentable.
   \end{lem}

\begin{proof}
  Let $P\rightarrow N$ be an exact morphism, and let $M\rightarrow P$ be a morphism in $\mone_{/N}$. We have a commutative diagram
  $$\begin{tikzcd}
M \arrow[d] \arrow[r] & P \arrow[r] \arrow[d] & N \arrow[d] \\
M^{\rm gp} \arrow[r]  & P^{\rm gp} \arrow[r]  & N^{\rm gp} 
\end{tikzcd}$$  
Now let $M'$ be an object in $\CC(M\rightarrow N)$, i.e., 
there is a morphism $M\rightarrow M'$ over $N$, such that $M^{\rm gp}\simeq M'^{\rm gp}.$ Let 
$P'$ be the pushout $P\coprod_M M'$. We have $P\simeq P'$, since $P'^{\rm gp}=P^{\rm gp}\coprod_{M^{\rm gp}}M^{\rm gp}\simeq P^{\rm gp}.$ Then after inverting  the map $P\rightarrow P'$, we get a morphism $M'\rightarrow P$, which is  unique, up to homotopy, such that the following  diagram commutes
$$\begin{tikzcd}
M \arrow[r] \arrow[rd] & M' \arrow[d] \\
                       & P           
\end{tikzcd}$$
Passing to colimit, we get a functorial equivalence of animae
$$\map_{\mone_{/N}}(M,P)\simeq\map_{\mone_{/N}}(M^{\rm ex},P)\simeq\map_{\mone^{\rm ex}_{/N}}(M^{\rm ex},P).$$
\end{proof}
\begin{rmk}
 If the morphism $f:M\rightarrow N$ is \textit{virtually surjective}, i.e., the induced map $\pi_0(M^{\rm gp})\rightarrow \pi_0(N^{\rm gp})$   is surjective, then \cite[Lemma 3.8]{rognes2009topological} shows that the fiber product $N\times_{N^{\rm gp}}M^{\rm gp}$ is an exactification of $f$, and in which, Rognes refers to the construction $(M\rightarrow N)\mapsto N\times_{N^{\rm gp}}M^{\rm gp}$ as the \textit{repletion} of a virtually surjective map, c.f. \cite[Definition 3.6]{rognes2009topological}.
\end{rmk}
\begin{rmk}[The tangent bundle of $\mon$]\label{tangentbunlemonoids}
    Corollary\autoref{exactificationofmonoids} and Lemma\autoref{exactificationpresentable} permit us to study the tangent bundle $\tang_{\mone}$ of $\Einf$-monoids. By definition, the tangent bundle $\tang_{\mone}$ of $\Einf$-monoids is the relative stabilization of the evaluation functor ${\rm ev}_1:\Fun(\Delta^1,\mone)\rightarrow\mone$. Let us calculate the fiber  $\tang_\mone\times_{\mone}\{M\},$ which is equivalent to the stabilization of $\mone_{M//M}$ as there is a canonical equivalence $\Spt(\mathcal{ C})\simeq\Spt(\mathcal{C}_{*/})$ for $\CC$, as explained in \cite[Remark 1.4.2.18]{lurie2017higher}. Let $\mathcal E$ be an arbitrary presentable stable $\infty$-category, and let $F:\mone_{M//M}\rightarrow\mathcal E$ be a colimit-preserving functor.  For instance, we have $F(M)\simeq 0$. One can see that there is a unique factorization of $F$,
    $$\begin{tikzcd}
\mone_{M//M} \arrow[r, "F"] \arrow[d, "\pi"'] & \mathcal E \\
\mone \arrow[ru, "F'"']                       &           
\end{tikzcd}$$
such that the extension $F'$ is colimit-preserving, 
where $\pi$ is the forgetful functor. $F'$ is given  by sending an object $P\in\mone$ to $F(P\oplus M)\in\mathcal E.$  Now we claim that the functor $F'$ uniquely factors through $\mone^{\rm gp}$. Indeed, let $P_\bullet$ be the simplicial diagram 
$$\begin{tikzcd}
... \arrow[r, shift left=3] \arrow[r, shift left] \arrow[r, shift right] \arrow[r, shift right=3] & P^2 \arrow[r, shift left=2] \arrow[r] \arrow[r, shift right=2] \arrow[l, shift left=2] \arrow[l, shift right=2] \arrow[l] & P \arrow[r, shift right] \arrow[r, shift left] \arrow[l, shift left] \arrow[l, shift right] & 0 \arrow[l]
\end{tikzcd}$$
We have the following natural equivalences $$F'(BP)\simeq F'(\colim_{\Delta^{\rm op}}P_\bullet)\simeq\colim_{\Delta^{\rm op}}F'(P_\bullet)\simeq\Sigma F'(P).$$
This implies that we have  equivalences $$F'(P)\simeq\Omega F'(BP)\simeq\Omega F'(BP^{\rm gp})\simeq F'(P^{\rm gp}),$$
the second equivalence holds because $BP\simeq BP^{\rm gp}$.
Thus we have a commutative diagram of $\infty$-categories 
$$\begin{tikzcd}
\mone_{M//M} \arrow[d, "\pi"'] \arrow[r] & \mone^{\rm ex}_{M//M} \arrow[d, "\simeq"] \arrow[rd] &            \\
\mone \arrow[r]                          & \mone^{\rm gp} \arrow[r]                             & \mathcal E,
\end{tikzcd}$$ 
and therefore we get equivalences
$$\Fun^L(\mone,\mathcal E)\simeq\Fun^L(\mone^{\rm gp},\mathcal E)\simeq\Fun^L(\Spt,\mathcal{E}).$$
The second equivalence holds because $\mon^{\rm gp}\simeq\Spt_{\geq 0}$ by the recognition principle, and so its stabilization is $\Spt$.
i.e. we have $\tang_{\mone}\times_{\mone}\{M\}\simeq\Spt,$ for any $M\in\mon$, and the cotangent complex  $\LL_M$ of a monoid is nothing but its group completion $M^{\rm gp}$. 
\end{rmk}
\begin{defn}
    Let $f:(A,M)\rightarrow (B,N)$ be a morphism of connective $\Einf$-prelog rings. We say $f$ is exact if the induced map $M\rightarrow N$ is exact.
\end{defn}
Let $\plog^{\rm ex}_{/(B,N)}$ be the  subcategory of $\plog_{/(B,N)}$  spanned by exact morphisms $(A,M)\rightarrow (B,N)$. Hence we have a pullback square of $\infty$-categories 
$$\begin{tikzcd}
{\plog^{\rm ex}_{/(B,N)}} \arrow[d] \arrow[r] & {\plog_{/(B,N)}} \arrow[d] \\
\mone^{\rm ex}_{/N} \arrow[r]                 & \mone_{/N}                
\end{tikzcd}$$
Then $\plog^{\rm ex}_{/(B,N)}$ turns out to be  presentable by \cite[Theorem 5.5.3.18]{lurie2009higher}.
\begin{prop}\cite[Proposition 9.1]{lundemo2021relationship}\label{prelogexactification}
The fully faithful  inclusion $\plog^{\rm ex}_{/(B,N)}\subset\plog_{/(B,N)}$ admits a left adjoint ${\rm ex}:(A,M)\mapsto (A,M)^{\rm ex},$  which is called the exactification functor. More precisely, we have $(A,M)^{\rm ex}\simeq (A\otimes_{\mathbb S[M]}\mathbb S[M^{\rm ex}],M^{\rm ex}).$ 
\end{prop}
\begin{proof}
    Since the limits in $\plog_{/(B,N)}$ and $\plog^{\rm ex}_{/(B,N)}$ are computed point-wise. We obtain that the forgetful functor $\plog^{\rm ex}_{/(B,N)}\subset\plog_{/(B,N)}$ of presentable $\infty$-category preserves arbitrary (small) limits. So there is a left adjoint $${\rm ex}:\plog_{/(B,N)}\rightarrow\plog^{\rm ex}_{/(B,N)} $$
    of the inclusion $\plog^{\rm ex}_{/(B,N)}\subset\plog_{/(B,N)}.$
    Now let us compute the exactification of a morphism $f:(A,M)\rightarrow(B,N).$ Let $\DD(f)$ be the subcategory of $\plog_{(A,M)//(B,N)}$ spanned by $(A,M)\rightarrow(A',M')\rightarrow (B,N)$ such that $M'\rightarrow N$ is exact. Note that the morphism $(A,M)\rightarrow (A',M')$ factors through the morphism $(A,M)\rightarrow (A\otimes_{\mathbb S[M]}\mathbb S[M'],M').$ Hence, the subcategory of $\DD(f)$ spanned by those objects which are equivalent to $(A,M)\rightarrow (A\otimes_{\mathbb S[M]}\mathbb S[M'],M')\rightarrow (B,N)$ such that $M'\rightarrow N$ is exact, is cofinal in $\DD(f)$. We denote by $\DD'(f)$ the resulting $\infty$-category. It's easy to see that the functor $(A\otimes_{\mathbb S[M]}\mathbb S[M'],M')\mapsto M'$ gives rise to an equivalence  between $\DD'(f)$ and the  $\infty$-category $\mone^{\rm ex}_{M//N}$.  As a consequence, we have
    \begin{align*}
         (A,M)^{\rm ex}\simeq\varprojlim_{\DD(f)}(A',M')&\simeq\varprojlim_{M'\in\mone^{\rm ex}_{M//N}}(A\otimes_{\mathbb S[M]}\mathbb S[M'],M')\\
         &\simeq(A\otimes_{\mathbb S[M]}\mathbb S[M^{\rm ex}],M^{\rm ex}).
    \end{align*}
   
    \end{proof}
\begin{defn}\label{logexactification}
Let $(B,N)$ be a connective $\Einf$-log ring. The $\infty$-category of exact connective $\Einf$-log rings over $(B,N)$ is the pullback of presentable $\infty$-categories.
$$\begin{tikzcd}
{\llog^{\rm ex}_{/(B,N)}} \arrow[d] \arrow[r] & {\llog_{/(B,N)}} \arrow[d] \\
{\plog^{\rm ex}_{/(B,N)}} \arrow[r]           & {\plog_{/(B,N)}}          
\end{tikzcd}$$
\end{defn}
\begin{cor}\label{exactificationoflogrings}
    The fully faithful inclusion $\llog^{\rm ex}_{/(B,N)}\subset\llog_{/(B,N)}$ admits a left adjoint, which we still denote by $${\rm ex}:\llog_{/(B,N)}\rr\llog^{\rm ex}_{/(B,N)}$$
    and call it the exactification functor.
\end{cor}
\begin{proof}
    It follows from Proposition\autoref{prelogexactification} and \cite[Theorem 5.5.3.18]{lurie2009higher}.
\end{proof}

\begin{prop}\cite[Proposition 6.10]{lundemo2023deformation}\label{exacteq}
Let $(B,N)$ be a connective $\Einf$-log ring. Then the forgetful functor $\llog^{\rm ex}_{(B,N)//(B,N)}\rightarrow \alg^{\rm cn}_{B//B}$ is an equivalence.
\end{prop}
\begin{proof}
    Consider the commutative diagram of $\infty$-categories
   $$ \begin{tikzcd}
{\llog^{\rm ex}_{(B,N)//(B,N)}} \arrow[r] \arrow[d] & {\llog_{(B,N)//(B,N)}} \arrow[d] \\
{\plog^{\rm ex}_{(B,N)//(B,N)}} \arrow[r] \arrow[d] & {\plog_{(B,N)//(B,N)}} \arrow[d]          \\
\mone^{\rm ex}_{N//N} \arrow[r]                     & \mone_{N//N}                             
\end{tikzcd}$$
Every square is Cartesian. Corollary\autoref{exactificationofmonoids} shows that there is an equivalence $$\mone^{\rm gp}\stackrel{\simeq}\rr \mone^{\rm ex}_{N//N},G\mapsto G\oplus N.$$
So the $\infty$-category $\plog^{\rm ex}_{(B,N)//(B,N)}$ is equivalent to the full subcategory of $\plog_{(B,N)//(B,N)}$ spanned by morphisms having the form $(B,N)\rightarrow(A,N\oplus G)\rightarrow(B,N).$ The monoid morphism $G\rightarrow A$ factors through $$G\rr GL_1(A)\stackrel{\simeq}\rr GL_1(B)\oplus GL_1(A)/GL_1(B)\rr N\oplus GL_1(A)/GL_1(B)$$
It follows that  we have a commutative diagram 
$$\begin{tikzcd}
{(B,N)} \arrow[d, "="'] \arrow[r] & {(A,N\oplus G)} \arrow[r] \arrow[d]      & {(B,N)} \arrow[d, "="] \\
{(B,N)} \arrow[r]                 & {(A,N\oplus GL_1(A)/GL_1(B))} \arrow[r] & {(B,N)}               
\end{tikzcd}$$
It's easy to see that the morphism $(A,N)\rightarrow(A,N\oplus GL_1(A)/GL_1(B))$ exhibits $(A,N\oplus GL_1(A)/GL_1(B))$ as a logification of $(A,N\oplus G)$, and thus we have $(A,N\oplus GL_1(A)/GL_1(B))\in\llog^{\rm ex}_{(B,N)//(B,N)}.$ We must  show that  $(A,N\oplus GL_1(A)/GL_1(B))$ is the image of $(A,N\oplus G)$ under the localization $\plog^{\rm ex}_{(B,N)//(B,N)}\rightarrow \llog^{\rm ex}_{(B,N)//(B,N)}$ of the fully faithful inclusion $\llog^{\rm ex}_{(B,N)//(B,N)}\subset\plog^{\rm ex}_{(B,N)//(B,N)}$. Indeed, for any $(R,P)\in\llog^{\rm ex}_{(B,N)//(B,N)}$, and any morphism $(A,N\oplus G)\rightarrow (R,P)$, since $(R,P)$ is an $\Einf$-log ring, the map uniquely factors through $(A,N\oplus G)\rightarrow (A,N\oplus GL_1(A)/GL_1(B)).$
\end{proof}
\subsubsection{Log derivations}
 Let $U:\llog\rightarrow \algcn$ be the forgetful functor, which is a right adjoint,  and let $\Fun^{\rm ext}(\Delta^1,\textbf{Log}):=\Fun^{\algcn}(\Delta^1,\llog)$ be the subcategory of $\Fun(\Delta^1,\textbf{Log})$ spanned by exact morphisms. 
\begin{prop}\label{3245672}
\begin{enumerate}
    \item 
    The evaluation functor ${\rm ev}_1:\Fun^{\rm ext}(\Delta^1,\llog)\rightarrow\llog$ is a presentable fibration.
    \item The forgetful functor $U_*:\Fun^{\rm ext}(\Delta^1,\llog)\rightarrow\Fun(\Delta^1,\algcn)$ is a right adjoint.
    \end{enumerate}
\end{prop}
\begin{proof}
All fibers of the functor ${\rm ev}_1$ are presentable. Corollary\autoref{exactificationoflogrings} implies that it is also a coCartesian fibration. Indeed, let $f:(A,M)\rightarrow (B,N)$ be an arrow in $\llog$,
and let $g:(R,P)\rightarrow (A,M)$  be a connective $\Einf$-log ring over $(A,M)$. Then the exactification $(R,P)^{\rm ex}_{(B,N)}$ of $(R,P)$ with respect to $(B,N)$ is a pushout of $(R,P)$ along $f.$ We denote by $f_!:\llog^{\rm ex}_{/(A,M)}\rightarrow \llog^{\rm ex}_{/(B,N)}$ a pushout functor along $f$.\\
Using Grothendieck's construction, the coCartesian fibration ${\rm ev}_1$ gives rise to a functor $$\chi:\llog\rr\cat, (A,M)\mapsto\llog^{\rm ex}_{/(A,M)}.$$
We have to prove that the functor $\chi$ factors through $\llog\rightarrow \mathbf{Pr}_{L},$ i.e., the functor $f_!$ preserves small colimits.\\
There is a commutative diagram of presentable $\infty$-categories
$$\begin{tikzcd}
{\llog_{/(A,M)}} \arrow[d, "{\rm ex}"'] \arrow[r, "f_!"] & {\llog_{/(B,N)}} \arrow[d, "{\rm ex}"] \\
{\llog^{\rm ex}_{/(A,M)}} \arrow[r, "f_!"]               & {\llog^{\rm ex}_{/(B,N)}}             
\end{tikzcd}$$
Let $K$ be a simplicial set and $F:K\rightarrow \llog^{\rm ex}_{/(A,M)}$ be a functor. Then we have a lifting $\tilde{F}:K\rightarrow \llog_{/(A,M)}.$ Since the exactification functor preserves colimits, we have $\colim{F}\simeq(\colim{\tilde F})^{\rm ex}.$ Thus we have equivalences
\begin{align*}
f_!(\colim{ F})&\simeq f_!((\colim{\tilde F})^{\rm ex})\simeq (f_!\colim\tilde{F})^{\rm ex}\\
&\simeq(\colim f_!\tilde{F})^{\rm ex}\simeq\colim{f_!\tilde{F}}^{\rm ex}\simeq\colim f_!F.
\end{align*}
  As for the second statement, we only need to show that $U_*$ preserves small limits. It follows that we need to show that the inclusion $\Fun^{\rm ext}(\Delta^1,\llog)\subset\Fun(\Delta^1,\llog)$ preserves small limits. It's clear, because the inclusion is a relative  right adjoint of the exactification functor ${\rm ex}:\Fun(\Delta^1,\llog)\rightarrow\Fun^{\rm ext}(\Delta^1,\llog)$ over $\llog$ by Proposition\autoref{prelogexactification} and Definition\autoref{logexactification}.
 \end{proof}
 We let $\tang^{\rm rep}_{\llog}$ be the stable envelope of the presentable fibration ${\rm ev}_1:\Fun^{\rm ext}(\Delta^1,\llog)\rightarrow\llog$. Using Proposition\autoref{3245672}, we can form a commutative diagram of $\infty$-categories
 $$\begin{tikzcd}
\tang^{\rm rep}_{\llog} \arrow[d, "\Omega^{\infty}"'] \arrow[r]           & \tang_{\algcn} \arrow[d, "\Omega^{\infty}"]     \\
{\Fun^{\rm ext}(\Delta^1,\llog)} \arrow[d, "{\rm ev}_1"'] \arrow[r, "U_*"] & {\Fun(\Delta^1,\algcn)} \arrow[d, "{\rm ev}_1"] \\
\llog \arrow[r, "U"]    & \algcn                     
\end{tikzcd}$$
which satisfies condition $1,2$ and $3$ in the Definition\autoref{gabtan}, the following proposition guarantees that this forms a replete tangent bundle formalism.
 \begin{prop}\cite[Corollary 6.22]{lundemo2023deformation}\label{tangentreplog}\label{tlog=logcalg}
    There is a natural equivalence  $\tlog\simeq \llog\times_{\algcn}\mathrm T_{\algcn}$
    of $\infty$-categories over $\llog$.
\end{prop}
\begin{proof}
    We just need to show that the two stable presentable fibrations $\tlog\rightarrow \llog$ and $\llog\times_{\algcn}\mathrm T_{\algcn}$ represent the same prestack $$\chi_{\tlog}\simeq\chi_{\llog\times_{\algcn}\mathrm T_{\algcn}}:\llog\rr \mathbf{Pr}_L^{{\rm stab}}$$
    It's deduced immediately from Proposition\autoref{exacteq}.
\end{proof}
\begin{rmk}
 Thanks to Proposition\autoref{tlog=logcalg},
 we could identify the fiber $\tlog\times_{\llog}\{(A,M)\}$ with $\tlog\times_{\llog}\{(A,GL_1(A))\}\simeq\Mod_A$ via the functor $$i_{!}:\Mod_A\stackrel{\simeq}\rr\tlog\times_{\llog}\{(A,M)\}$$
 induced by the morphism $i:(A,GL_1(A))\rightarrow(A,M).$
\end{rmk}

Therefore, we can apply the replete tangent formalism to logarithmic geometry.
\begin{defn}
We denote the replete tangent correspondence associated with $\tlog$ as $\MTL$. The $\infty$-category of log derivations is the $\infty$-category of replete derivations in $\MTL$. The Gabber's cotangent complex $\LL^G$ is the Gabber's cotangent complex associated with $\tlog$. We also denote by 
$\Dlog_{(A,M)}$ the fiber product $\Dlog\times_{\llog}\{(A,M)\}$, and refer to it as the $\infty$-category of log derivations of $(A,M)$.

\end{defn}

  \begin{rmk}
    Using Proposition\autoref{tangentreplog}, Remark\autoref{tangentmodule} and Remark\autoref{tangentconnective}, we can concretely describe the infinity loop space functor $$\Omega^{\infty}:\tang^{\rm 
rep}_{\textbf{Log},(A,M)}\rr\textbf{Log}_{/(A,M)},$$
    which sends $((A,M),I)$ to $(A\oplus \tau_{\geq 0}I,M\oplus \tau_{\geq 0}I)\rightarrow (A,M)$.
    The proof of Proposition\autoref{exacteq} implies that $(A\oplus \tau_{\geq 0}I,M\oplus \tau_{\geq 0}I)$ is exact over $(A,M)$.
\end{rmk}
We deduce the following properties of Gabber's cotangent complex from its definition. 
\begin{thm}\label{cotan-der}
For Gabber's cotangent complexes, we have the following:
\begin{enumerate}
    \item(Universal characterization) Fix a morphism $(B,N)\rightarrow(A,M)$ of $\Einf$-log rings.
    The Gabber's cotangent complex defined above is characterized by the following property: for any $I\in\Mod^{\rm cn}_A$, we have 
$$\map_{A}(\LL^G_{(A,M)/(B,N)},I)\simeq\map_{{(B,N)//(A,M)}}((A,M),(A\oplus I,M\oplus I)).$$
\item(Conormal fiber sequence) Let $(R,L)\rightarrow(A,M)\rightarrow(B,N)$ be a sequence of maps of $\Einf$-log rings, the following sequence 
$$\LL^G_{(A,M)/(R,L)}\otimes_AB\rr\LL^G_{(B,N)/(R,L)}\rr\LL^G_{(B,N)/(A,M)}$$
is exact in $\Mod_B$.
\end{enumerate}
\end{thm}
\begin{proof}
    This is deduced immediately from the categorical definition of Gabber's cotangent complex.
\end{proof}
\begin{cor}
Fix a morphism  $(B,N)\rightarrow(A,M)$ of $\Einf$-log rings.
    There is an equivalence of $\infty$-categories:
$${\rm ev}_1:\Der^{\llog}_{(A,M)}\times_{\Der^{\llog}_{(B,N)}}\Der^{\llog,{\rm triv}}_{(B,N)}\rr (\Mod^{\rm cn}_{A})_{\LL^G_{(A,M)/(B,N)}/}$$
In which  $\Der^{\llog,{\rm triv}}_{(B,N)}$  is the full subcategory of $\Der^{\llog}_{(B,N)}$ spanned by trivial derivations, i.e., derivations having the form  
$0:\LL^G_{(A,M)/(B,N)}\rightarrow I$. In other words, the left hand term is the $\infty$-category of $(A,M)$-derivations which is trivial restricting to $(B,N).$
\end{cor}
\begin{proof}
Denote by $\Der^{\llog}_{(A,M)/(B,N)}$  the left hand term.
Theorem\autoref{cotan-der} guarantees that the evaluation $${\rm ev_1}:\Der^{\llog}_{(A,M)/(B,N)}\rr \Mod_A^{\rm cn}$$
factors through the comma category $(\Mod^{\rm cn}_{A})_{\LL^G_{(A,M)/(B,N)}/}$,
and clearly the induced functor 
$${\rm ev}_1:\Der^{\llog}_{(A,M)/(B,N)}\rr (\Mod^{\rm cn}_{A})_{\LL^G_{(A,M)/(B,N)}/}$$
is essentially surjective.
We must show that the functor ${\rm ev}_1$ is fully faithful. Let $d:(A,M)\rightarrow I$  and $d':(A,M)\rightarrow J$ be two derivations such that the  restrictions to $(B,N)$ are trivial. The mapping space $\map_{\Der^{\llog}_{(A,M)/(B,N)}}(d,d')$ is the anima spanned by commutative diagrams  having the form
$$\begin{tikzcd}
   & {(B,N)} \arrow[d] \arrow[ldd, "0"', bend right] \arrow[rdd, "0", bend left] &   \\
             & {(A,M)} \arrow[ld, "d"'] \arrow[rd, "d'"]        &   \\
I \arrow[rr] &                                                                             & J
\end{tikzcd}$$
in $\MTL$, such that the arrow $I\rightarrow J$ is an $A$-module morphism. Hence it follows that the mapping space is equivalent to the pullback 
of the following animae $$\begin{tikzcd}    & {\map_{\MTL}((A,M),I))\times_{\map_{\MTL}(B,N),I)}\{0\}} \arrow[d] \\
{\map_A(I,J)} \arrow[r] & {\map_{\MTL}((A,M),J))\times_{\map_{\MTL}(B,N),J)}\{0\}}          
\end{tikzcd}$$
Theorem\autoref{cotan-der} shows that there are natural equivalences of animae
$$\map_{\MTL}((A,M),I))\times_{\map_{\MTL}(B,N),I)}\{0\}\simeq\map_{A}(\LL_{(A,M)/(B,N)},I),$$ and
$$\map_{\MTL}((A,M),J))\times_{\map_{\MTL}(B,N),I)}\{0\}\simeq\map_{A}(\LL_{(A,M)/(B,N)},J).$$
Then it turns  out that we have equivalences
\begin{align*}
\map_{\Der^{\llog}_{(A,M)/(B,N)}}(d,d')&\simeq\map_A(\LL^G_{(A,M)/(B,N)},I)\times_{\map_A(\LL^GL_{(A,M)/(B,N)},J)}\map_{A}(I,J)\\
&\simeq\map_{\LL^G_{(A,M)/(B,N)}/}(I,J)
\end{align*}
\end{proof}
We have some elementary calculations of Gabber's cotangent complexes.
\begin{lem}\cite[Section 4]{sagave2016derived}, \cite[Section 3]{binda2023hochschild}\label{Lstrict}
\begin{enumerate}
\item 
Let $f:A\rightarrow  B$ be a morphism of $\Einf$-rings. Regard $f$ as a morphism of $\Einf$-log rings with trivial log structures, then there is an equivalence $\LL_{B/A}\simeq\LL^G_{(B,GL_1(B))/(A,GL_1(A))}$.
\item 
Let $f:(A,M)\rightarrow (B,N)$ be a strict morphism, then there is a natural  equivalence $\LL_{B/A}\simeq\LL^G_{(B,N)/(A,M)}.$
\item The Gabber's cotangent complex $\LL^G_{(A[M],M)/(A[N],N)}$ is equivalent to $A[M]\otimes_{\mathbb S}(M^{\rm gp}/N^{\rm gp}).$
\item $\LL^G_{(A,M)/(R,L)}$ is canonically equivalent to the colimit of the following diagram
$$A\otimes_{\mathbb S
}(M^{\rm gp}/L^{\rm gp})\longleftarrow A\otimes_{\mathbb S[M]}\LL_{\mathbb S[M]/\mathbb S[L]}\rr\LL_{A/R}.$$
    \end{enumerate}
\end{lem}
\begin{proof}
  The assertion (1)  are deduced from  Proposition \autoref{tangentreplog}. We now prove the $(2)$. There is a pushout square of $\Einf$-log rings:
  $$\begin{tikzcd}
{(A,GL_1(A))} \arrow[d] \arrow[r] & {(B,GL_1(B))} \arrow[d] \\
{(A,M)} \arrow[r]                 & {(B,N)}                
\end{tikzcd}$$
Since the morphism $(A,M)\rightarrow(B,N)$ is strict. On the other hand, Gabber's cotangent complex $\LL^G:\llog\rightarrow \tlog$ preserves colimits, therefore we have a pushout square in $\MTL$:
$$\begin{tikzcd}
{\LL^G_{(A,GL_1(A))}} \arrow[d] \arrow[r] & {\LL^G_{(B,GL_1(B))}} \arrow[d] \\
{\LL^G_{(A,M)}} \arrow[r]                 & {\LL^G_{(B,N)}}                
\end{tikzcd}$$
Taking cofiber, we get $$\LL_{B/A}\simeq\LL^G_{(B,GL_1(B)/(A,GL_1(A))}\stackrel{\simeq}\rr\LL^G_{(B,N)/(A,M)}.$$
To prove (3), let us consider  the  tangent bundle $\tang_{\mone}$ of 
the $\infty$-category $\mone$.
Let $M$ be an $\Einf$-monoid. Using Remark\autoref{tangentbunlemonoids}, we can form a commutative diagram of $\infty$-categories
$$\begin{tikzcd}
\Spt \arrow[d, shift left] \arrow[rr]                                                              &  & {\Mod_{\mathbb S[M]}} \arrow[d, shift left]                                                                          \\
\mone_{/M} \arrow[rr, "{P\mapsto(\mathbb S[P],P)}", shift left] \arrow[u, "\LL_{M/-}", shift left] &  & {\llog_{/\mathbb S[M]}} \arrow[ll, "{(R,L)\mapsto L}", shift left] \arrow[u, "{\LL^G_{\mathbb S[M]/-}}", shift left]
\end{tikzcd}$$
The relative cotangent complex $\LL_{M/N}$ is equivalent to $M^{\rm gp}/N^{\rm gp}$, for an arbitrary $\Einf$-monoid $N$.  Let $J$ be a connective $\mathbb S[M]$-module, then we have 
\begin{align*}
\map_{\mathbb S[M]}(\LL^G_{(\mathbb S[M],M)/(\mathbb S[N],N)},J)&\simeq\map_{\Spt}(\LL_{M/N},J)\\
&\simeq\map_\Spt(M^{\rm gp}/N^{\rm gp},J)\\
&\simeq\map_{\mathbb S[M]}(\mathbb S[M]\otimes_{\mathbb S}(M^{\rm gp}/N^{\rm gp}),J).
\end{align*}
   As for (4), note that there is a pushout square 
    $$\begin{tikzcd}
{(\mathbb S[L]\rightarrow \mathbb S[M])} \arrow[d] \arrow[r] & {((\mathbb S[L],L)\rightarrow (\mathbb S[M],M))} \arrow[d] \\
(R\rightarrow A) \arrow[r]                                   & {((R,L)\rightarrow(A,M))}                                 
\end{tikzcd}$$
in $\Fun(\Delta^1,\llog)$, and  Gabber's cotangent functor $\LL^G:\Fun(\Delta^1,\llog)\rightarrow\tlog$ is a left adjoint.
\end{proof}
\begin{cor}[Fundamental sequence] \cite[Corollary 3.9]{binda2023hochschild}
    Let $f:(R,L)\rightarrow(A,M)$ be a map of $\Einf$-log rings, then there is an exact sequence 
    $$A\otimes_{\mathbb S}(M^{\rm gp}/N^{\rm gp})\rr\LL^G_{(A,M)/(R,L)}\rr\LL_{A/{R\otimes_{\mathbb S[L]}\mathbb S[M]}}.$$
\end{cor}
\begin{proof}
    By Lemma\autoref{Lstrict} (4), we have a pushout square of $A$-modules
    $$\begin{tikzcd}
{ A\otimes_{\mathbb S[M]}\LL_{\mathbb S[M]/\mathbb S[L]}} \arrow[d] \arrow[r] & \LL_{A/R} \arrow[d] \\
A\otimes_{\mathbb S }(M^{\rm gp}/L^{\rm gp}) \arrow[r]                        & {\LL^G_{(A,M)/(R,L)}}
\end{tikzcd}$$
where the cofiber of the upper horizontal arrow is equivalent to 
$${\rm cofib}(\LL_{R/\mathbb S[L]}\otimes_RA\rr\LL_{A/\mathbb S[M]})\simeq\LL_{A/R\otimes_{\mathbb S[L]}\mathbb S[M]}.$$
\end{proof}
At the end of this section, we give the following  remark, which provides an intuition for our abstract construction of log cotangent complexes and log derivations, arising from the classical observations in log differential forms.
\begin{rmk}[Log Kähler differentials from replete diagonals]
   Recall that the classical module of Kähler differential forms $\Omega^1_{A/R}$ of an ordinary ring map $f:R\rightarrow A$ is isomorphic to the indecomposables of its diagonal $\mu:A\otimes_RA\rightarrow A$, namely, it is the conormal   $I/I^2$, where $I$ is the kernel of $\mu$. In the derived setting, a similar result also holds, which identifies the cotangent complex of  $f$ with its derived indecomposables, c.f. \cite{MR1732625}.\\
   However, in log geometry, the indecomposables are not well defined for a log ring map $g:(R,L)\rightarrow (A,M)$, and the indecomposables of the underlying ring map of the diagonal 
$$\mu:(A,M)\otimes_{(R,L)}(A,M)\rr (A,M)$$
only recover the differentials of the underlying ring map $R\rightarrow A$. An idea to resolve this issue is to pass the map $\mu$ to a strict map in a functorial way, and we consider the indecomposables of the underlying ring map. Consider the exactification of the map $\mu$, and then  we have a strict map whose underlying ring map is  
$$\mu^{\rm ex}:(A\otimes_RA)\otimes_{R[M\oplus_L M]}R[(M\oplus_L M)^{\rm ex}]\rr A.$$
We refer to $\mu^{\rm ex}$ as the replete diagonals of $g$.
As a consequence of \cite[Corollary 4.2.8(ii)]{Kato2004conductor}, the module of indecomposables of $\mu^{\rm ex}$ is canonically equivalent to the modules of log differentials $\Omega^1_{(A,M)/(R,L)}$.
\end{rmk}
\subsection{Deformation theory}
In this subsection, we apply the tangent bundles and tangent correspondences formalisms for studying the deformation of $\Einf$-log rings. We will define the notions of deformations of $\Einf$-log rings, which should be viewed as the logarithmic analog of deformations of $\Einf$-log rings. We also study the deformation theory of animated log rings. The main results of this subsection are Theorem\autoref{deform} and Theorem\autoref{deformani}.
\subsubsection{Square-zero extensions}
\begin{defn}
    Let $(A,M)$ be an $\Einf$-log ring, and let $I\in\Mod_A$ be an $A$-module. A \textit{square-zero extension}  of $A$ with kernel $I$ is a strict  morphism  $i:(A',M')\rr (A,M),$ such that $A'$ is a square-zero \cite[Definition 7.4.1.6]{lurie2017higher} extension of $A$ with kernel $I$. The $\infty$-category of square-zero extensions $\mathbf{Def}^{\llog}$  is the subcategory of $\Fun(\Delta^1,\llog)$ spanned by square-zero extensions.
\end{defn}
\begin{rmk}\label{torsor=sqzero}
If $(A',M')\rightarrow (A,M)$ is a square-zero extension, the associated $\Einf$-monoid map $M'\rightarrow M$ exhibits $M$ as a quotient $M'/I$, 
where the action of $I$ on $M'$ is given by the canonical map $I\simeq (1+I)\subset GL_1(A')\rightarrow M'$.
We can view $M'$ as an $I$-torsor lying over $M$. The $\Einf$-monoid
$M'$ is determined by an $\Einf$-monoid  map $M\rightarrow BI\simeq I[1]$, whose fiber is equivalent to $M'$. In particular, we have a map of  fiber sequences
$$\begin{tikzcd}
M' \arrow[d] \arrow[r] & M \arrow[d] \arrow[r] & BI \arrow[d, "\simeq"] \\
\Omega^\infty A' \arrow[r]           & \Omega^\infty A \arrow[r]           & BI                    
\end{tikzcd}$$
\end{rmk}
Consider a log derivation $d:(A,M)\rightarrow I[1], I\in\Mod_A$. By the definition of log derivations, this gives rise to a   pullback square in $\MTL$
$$\begin{tikzcd}
{(A',M')} \arrow[r] \arrow[d] & {(A,M)} \arrow[d, "d"] \\
0 \arrow[r]                   & {I[1]}                
\end{tikzcd}$$
where the pullback $(A',M')\in\llog$ and $(A',M')\rightarrow(A,M)$ is strict, since there is an equivalence  $$(A',M')\simeq(A,M)\times_{(A\oplus \tau_{\geq 0}(I[1]),M\oplus \tau_{\geq 0}(I[1]))}(A,M)$$ given by two maps $$({\rm id},0),({\rm id},d):(A,M)\rightarrow(A\oplus \tau_{\geq 0}(I[1]),M\oplus \tau_{\geq 0}(I[1])).$$ It turns out that it is a square-zero extension of $(A,M)$ with kernel $I$. After taking projection to the upper horizontal arrows, we get a
 functor
$$\gamma:\Der^{\llog}\rr\mathbf{Def}^\llog$$ 
\begin{prop}\cite[Proposition 7.5]{lundemo2023deformation}\label{sqzero-logder}
    The functor $\gamma$ is essentially surjective. In other words, let $(A',M')\rightarrow (A,M)$ be a square-zero deformation of $(A,M)$, then there exists a log derivation $d:(A,M)\rightarrow I[1]$, such that there is a pullback square:
    $$\begin{tikzcd}
{(A',M')} \arrow[d] \arrow[r] & {(A,M)} \arrow[d, "d"] \\
0 \arrow[r]                   & {I[1]}                
\end{tikzcd}$$
\end{prop}
\begin{proof}
 We work in the $(\infty,2)$-category $\mathbf{Corr}$ of correspondences. Consider a commutative diagram in $\mathbf{Corr}$
 \[\begin{tikzcd}
	\algcn & {\Fun(\Delta^1,\algcn)} & {\tang_{\algcn}} \\
	\llog & {\Fun(\Delta^1,\llog)} & {\Fun^{\rm ext}(\Delta^1,\llog)} & \tlog
	\arrow["{\mathfrak M^0}", squiggly, from=1-1, to=1-2]
	\arrow["{\mathfrak M^{T}_{\algcn}}", curve={height=-18pt}, squiggly, from=1-1, to=1-3]
	\arrow[squiggly, from=1-1, to=2-1]
	\arrow[squiggly, from=1-2, to=1-3]
	\arrow[squiggly, from=1-2, to=2-2]
	\arrow[squiggly, from=1-3, to=2-4]
	\arrow["{\mathfrak M^0}", squiggly, from=2-1, to=2-2]
	\arrow["\MTL"', curve={height=18pt}, squiggly, from=2-1, to=2-4]
	\arrow[squiggly, from=2-2, to=2-3]
	\arrow[squiggly, from=2-3, to=2-4]
\end{tikzcd}\]
This gives rise to a commutative diagram
    $$\begin{tikzcd}
\MT_{\algcn} \arrow[rd] \arrow[rr, "F"] &          & \MTL \arrow[ld] \\
 & \Delta^1 &       
\end{tikzcd}$$
Restricted to the fibers over $0\in\Delta^1$, one has $F|_0:\algcn\rightarrow\llog$ is the functor that sends an $\Einf$-ring $A$ to $(A,GL_1(A))$, and restricted to the fibers over $1\in\Delta^1$, one has $F|_1:\tang_{\algcn}\simeq\tang_{\algcn}\times_{\algcn}\algcn\rightarrow\tlog\simeq\tang_{\algcn}\times_{\algcn}\llog$ is just the obvious functor.
  Note that the functor $F$ is fiber-wise left and right adjointable \cite[Section 7]{rognes2009topological}, and obviously preserves Cartesian and  coCartesian edges. Then by \cite[Proposition 7.3.2.6]{lurie2017higher}, $F$ admits a relative right adjoint $G^R:\MTL\rightarrow\MT_{\algcn}$, which should be thought of as the forgetful functor, and a relative left adjoint functor $G^L:\MTL\rightarrow\MT_{\algcn}$, which should be thought of as the trivial locus functor. Informally, the functor $G^L$ sends an $\Einf$-log ring $(A,M)\in\MTL\times_{\Delta^1}\{0\}\simeq\llog$ to $A[M^{-1}]=A\otimes_{\mathbb S[M]}\mathbb S[M^{\rm gp}]$, and sends an object $((A,M),I)\in\MTL\times_{\Delta^1}\{1\}\simeq\tlog$ to $(A[M^{-1}],I\otimes_AA[M^{-1}])$.\\
  Thus the functor $F$ preserves both arbitrary small limits and colimits. Now consider a square-zero extension $(A',M')\rightarrow (A,M)$ of $\Einf$-log rings. The underlying square-zero extension $A'\rightarrow A$ gives rise to a pullback square in $\MTL$:
  $$\begin{tikzcd}
{(A',GL_1(A'))} \arrow[d] \arrow[r] & {(A,GL_1(A))} \arrow[d] \\
0 \arrow[r]                         & {I[1]}                 
\end{tikzcd}$$
    On the other hand, we have a pushout square of $\Einf$-log rings
    $$\begin{tikzcd}
{(A',GL_1(A'))} \arrow[d] \arrow[r] & {(A,GL_1(A))} \arrow[d] \\
{(A',M')} \arrow[r]                 & {(A,M)}                
\end{tikzcd}$$
because of the strictness of the morphisms $(A',GL_1(A'))\rightarrow(A,GL_1(A))$ and $(A',M')\rightarrow(A,M)$. Thus we can form a commutative diagram in $\MTL$
$$\begin{tikzcd}
{(A',GL_1(A'))} \arrow[d] \arrow[r] & {(A,GL_1(A))} \arrow[d] \\
{(A',M')} \arrow[r] \arrow[d]       & {(A,M)} \arrow[d]       \\
0 \arrow[r]                         & {I[1]}                 
\end{tikzcd}$$
We must show that the lower square is a pullback square. Unwinding the definition of replete  correspondence $\MTL$, we have to prove the following commutative diagram
$$\begin{tikzcd}
{(A',M')} \arrow[r] \arrow[d] & {(A,M)} \arrow[d]             \\
{(A,M)} \arrow[r]             & {(A\oplus \tau_{\geq 0}(I[1]),M\oplus \tau_{\geq 0}(I[1]))}
\end{tikzcd}$$
Since the limit in the $\infty$-category $\llog$ is calculated pointwise, we have to check the diagram
$$
\begin{tikzcd}
M' \arrow[d] \arrow[r] & M \arrow[d]    \\
M \arrow[r]            & {M\oplus \tau_{\geq 0}(I[1])}
\end{tikzcd}$$
is a pullback square in $\llog$. Indeed, we have $M\simeq M'/\tau_{\geq 0}I$, and the upper horizontal arrow and the lower horizontal arrow have the same fibers.
\end{proof}
\begin{rmk}
    As mentioned in \cite[Warning 7.4.1.10]{lurie2017higher}, the functor $\gamma$ is not an equivalence. A derivation $d:(A,M)\rightarrow I[1]$ defines a square-zero extension $(A',M')\rightarrow(A,M)$ via the functor $\gamma$, and every square-zero extension $(A',M')\rightarrow(A,M)$ arises in this way. However, the derivation $d$ is not uniquely determined by $(A',M')$. For example, it is not sensitive to the coconnective part $\tau_{\leq 0}I$ of $I$.
\end{rmk}
\begin{defn}
Let $(R,L)$ be an $\Einf$-log ring, and let $(R',L')$ be a square-zero extension of $(R,L)$ with kernel $I\in\Mod_R$. Let $(A,M)$ be an $(R,L)$-algebra.
A deformation of $(A,M)$ to $(R',L')$ is a 
square-zero extension  $(A',M')$ of $(A,M)$ with kernel $J\in\Mod_A$, such that there is an equivalence  $J\simeq I_A=I\otimes_RA$ of $A$-modules, and the following diagram 
\[\begin{tikzcd}
	{(R',L')} & {(R,L)} \\
	{(A',M')} & {(A,M)}
	\arrow[from=1-1, to=1-2]
	\arrow[from=1-1, to=2-1]
	\arrow[from=2-1, to=2-2]
	\arrow[from=1-2, to=2-2]
\end{tikzcd}\]
is a pushout square in $\llog$. 
\end{defn}
\begin{defn}
    Let $\mathbf{Def}^{\llog,+}$ be the subcategory of $\mathbf{Def}^{\llog}$ spanned by the following data:
    \begin{enumerate}
        \item Objects are same with those in $\mathbf{Def}^\llog$;
        \item  Morphisms  are those morphisms  $((R',L')\rightarrow (R,L))\rightarrow ((A',M')\rightarrow(A,M))$ in $\mathbf{Def}^\llog$ such that $(A',M')$ is a  deformation of $(A,M)$ to $(R',L').$
    \end{enumerate}
        \end{defn}
    \begin{defn}
            Let $\Der^{\llog,+}$ be the subcategory  of $\Der^{\llog}$ spanned by the following data:
\begin{enumerate}
    \item Objects are derivations $d
    :(A,M)\rightarrow I[1]$ such that $I$ is connective;
    \item Morphisms are those morphisms $f:(d:(A,M)\rightarrow I)\rightarrow(d':(B,N)\rightarrow J)$ in $\Der^\llog$ such that the induced map $I\otimes_AB\rightarrow J$ is an equivalence.
\end{enumerate}
\end{defn}
\begin{rmk}
    The existence of the $\infty$-categories $\mathbf{Def}^{\llog,+}$ and $\Der^{\llog,+}$ is provided by \cite[Proposition 4.1.2.10]{kerodon}.
\end{rmk}

The following is the main theorem of this section:
\begin{thm}[Deformation theory]\label{deform}
    The functor $\gamma^+:\Der^{\llog,+}\rightarrow\mathbf{Def}^{\llog,+}$ induced from $\gamma$ is a left fibration.
\end{thm}
The proof of Theorem\autoref{deform} will be given in the next subsection. Let us consider various consequences of  it. Unwinding Theorem\autoref{deform}, let $(d:(R,L)\rightarrow I[1])\in\Der^{\llog,+}$ be a derivation, we get an equivalence $\Der^{\llog,+}_{d/}\stackrel{\simeq}\rightarrow\mathbf{Def}^{\llog,+}_{\gamma^+(d)/}$. The left hand term is the $\infty$-category of morphisms of derivations 
$$\begin{tikzcd}
{(R,L)} \arrow[d] \arrow[r, "d"] & {I[1]} \arrow[d] \\
{(A,M)} \arrow[r, "d'"]          & {J[1]}          
\end{tikzcd}$$ 
such that $J[1]\simeq I[1]\otimes_RA,$ and the right hand term is the $\infty$-category of deformations $(A',M')\rightarrow(A,M)$ to $(R',L')$. In particular, we have 
\begin{cor}\label{deform=sq0extension}
    The functors $\Der^{\llog,+}_{d/}\stackrel{\gamma^+}\rightarrow\mathbf{Def}^{\llog,+}_{\gamma^+(d)/}\stackrel{F}\rightarrow\llog_{(R',L')/}$
    are equivalences. 
\end{cor}
\begin{proof}
    We only need to prove that the second functor $F$ is an equivalence of $\infty$-categories. Let $(A',M')$ be an $(R',L')$-algebra. The fiber of $F$ at $(A',M')$ is a contractible anima by \cite[Proposition 4.3.2.15]{lurie2009higher}. On the other hand, the functor $F$ is essentially surjective. It follows that $F$ is a trivial Kan fibration.
\end{proof}
\begin{cor}\label{obstruction}
 Let $(R',L')\rightarrow(R,L)$ be a square-zero extension with kernel $I\in\Mod_R^{\rm cn}$, induced by the derivation $d:(R,L)\rightarrow I[1]$.
\begin{enumerate}
    \item 
   Let $(R,L)\rightarrow(A,M)$ be a morphism of $\Einf$-log rings. The obstruction to the existence of deformations of $(A,M)$ to $(R',L')$ lies in $\mathrm{Ext}^2(\LL_{(A,M)/(R,L)},I_A).$
\item If the map $\LL_{(R,L)}\otimes_RA\rightarrow I_A[1]$  induced from $d$ vanishes (for example, $d$ is a trivial derivation), then there is an equivalence of animae
$$\map_{A}(\LL^G_{(A,M)/(R,L)},I_A[1])\simeq\mathbf{Def}^{\simeq}((A,M)/(R,L),d)$$
where $\mathbf{Def}^{\simeq}((A,M)/(R,L),d)$ is the anima of deformations of $(A,M)$ to $(R',L').$
    \end{enumerate}
\end{cor}
\begin{proof}
The existence of deformations of $(A,M)$ to $(R',L')$ is equivalent to the existence of extensions of $d\otimes_RA$ to the map $\LL^G_{(A,M)}\rightarrow I_A[1]$ by the above Corollary\autoref{deform=sq0extension}.
As we have the conormal exact sequence of Gabber's cotangent complexes
$$\LL^G_{(A,M)/(R,L)}[-1]\rr\LL^G_{(R,L)}\otimes_RA\rr\LL^G_{(A,M)},$$
the extension exists if and only if the composition $\LL^G_{(A,M)/(R,L)}[-1]\rightarrow \LL^G_{(R,L)}\otimes_RA\rightarrow I_A[1]$ vanishes, i.e., it's vanishing in the mapping class group $\pi_0\map(\LL^G_{(A,M)/(R,L)}[-1],I_A[1])\simeq{\rm Ext}^2(\LL^G_{(A,M)/(R,L)}[-1],I_A[1]).$ 
\end{proof}
The assertion (2) is directly obtained from Theorem \autoref{deform}.

\begin{rmk}
For classical log rings, the equivalences mentioned in Corollary \autoref{obstruction} can be constructed directly; see \cite{kato1996log} and \cite{olsson2005logarithmic} for details. We provide a brief sketch. Let $(R,L)$ be a discrete log ring and let $I$ be a discrete $R$-module. Let $\partial\in \pi_0{\rm\map}(\LL^{G}_{(R,L)},I[1])={\rm Ext}^1(\LL^G_{(R,L)},I)$ be a log derivation, the corresponding square-zero extension $(R^{\partial},L^{\partial})$ is given by the following pullback square 
\[\begin{tikzcd}
	{(R^{\partial},L^{\partial})} & {(R,L)} \\
	{(R,L)} & {(R\oplus I[1], L\oplus I[1])}
	\arrow["\partial", from=2-1, to=2-2]
	\arrow["{d_{\rm triv}}", from=1-2, to=2-2]
	\arrow[from=1-1, to=2-1]
	\arrow[from=1-1, to=1-2]
\end{tikzcd}\]
   Conversely, assume  $(R',L')$ is a square-zero extension of $(R,L)$. The underlying  ring  $R'$ is a square-zero extension of $R$, it is determined by a derivation $$d: R\rr R\oplus I[1].$$ The derivation $d_L:L\rr L\oplus I[1]$ on the log part is defined using the following diagram
\[\begin{tikzcd}
	{1+I} & {1+I} \\
	{R'^*} & {L'} \\
	{R^*} & L
	\arrow[from=2-1, to=3-1]
	\arrow[from=1-1, to=2-1]
	\arrow["\simeq", from=1-1, to=1-2]
	\arrow[from=1-2, to=2-2]
	\arrow[from=2-2, to=3-2]
	\arrow[from=2-1, to=2-2]
	\arrow[from=3-1, to=3-2]
\end{tikzcd}\]
    by letting $d_L$ be the natural map $L\rr\Sigma(1+I)\simeq I[1].$ 
\end{rmk}
\begin{rmk}
    As mentioned in the proof of Proposition\autoref{sqzero-logder}, there is a fully faithful embedding $\mathfrak{M}^T_{\algcn}\subset\MTL$ over $\Delta^1$, and it follows that there is a fully faithful embeddings $\Der(\algcn)\subset\Der^\llog$ and $\mathbf{Def}(\algcn)\subset\mathbf{Def}$.  Restricted to $\Einf$-log rings which have trivial log structures, the deformation theory of $\Einf$-log ring Theorem\autoref{deform} coincides with the deformation theory of $\Einf$-log rings established in \cite[Theorem 7.4.2.7]{lurie2017higher}
\end{rmk}
\subsection{Proof of Theorem\autoref{deform}}
In this subsection, we will prove Theorem\autoref{deform}. Our approach depends on certain constructions and results of deformation theory of $\Einf$-rings formulated by Lurie in \cite{lurie2017higher}.\\
 Let us consider the functors $\Omega^{\infty}:\tlog\rightarrow\llog$, and $\Omega^\infty:\tang_{\algcn}\rightarrow\algcn$. Applying Grothendieck construction, we get morphisms of fibered $\infty$-categories $\omega_{\mathrm{Log}}:\MTL\rightarrow \llog\times\Delta^1$ over $\Delta^1$, and $\omega:\MT_{\algcn}\rightarrow\algcn\times\Delta^1$ over $\Delta^1.$\\
Consider the functor $G^R:\MTL\rightarrow\MT_{\algcn}$ defined in the proof of Proposition\autoref{sqzero-logder}, which preserves Cartesian edges. Form a commutative diagram 
$$\begin{tikzcd}
\MTL \arrow[d, "\omega_{\mathrm{Log}}"'] \arrow[r, "G^R"] & \MT_{\algcn} \arrow[d, "\omega"]        \\
\llog\times\Delta^1 \arrow[r] \arrow[d, "\pi_1"']         & \algcn\times\Delta^1 \arrow[d, "\pi_1"] \\
\llog \arrow[r]                                           & \algcn                                 
\end{tikzcd}$$
Where the middle and  lower horizontal arrows are  forgetful functors. Therefore we get a relative functor: 
$$\begin{tikzcd}
\MTL \arrow[rd] \arrow[rr, "\Psi"] &          & \MT_{\algcn}\times_{\algcn}\llog \arrow[ld] \\  & \Delta^1 &     
\end{tikzcd}$$
One can see that the functor $\Psi$ preserves Cartesian  edges.
Moreover, we will show that it is an equivalence.
\begin{lem}\label{tangentlogalg}
    The functor $\Psi$ is an equivalence of $\infty$-categories.
\end{lem}
\begin{proof}
    We only need to check the equivalence at the level of  fibers. We have that $$\Psi_0:\MTL\times_{\Delta^1}\{0\}\simeq\llog\rr\algcn\times_{\algcn}\llog\simeq\llog$$  is actually equivalent to the identity functor. Consider the fibers over $1\in\Delta^1$. We have that the functor
$$\MTL\times_{\Delta^1}\{1\}\simeq\tlog,$$
    and $$(\MT_{\algcn}\times_{\algcn}\llog)\times_{\Delta^1}\{1\}\simeq\tang_{\algcn}\times_{\algcn}\llog.
    $$
    The functor $\Psi_1$ turns out to be an equivalence from Proposition\autoref{tlog=logcalg}.
\end{proof}
Now we consider  the $\infty$-category $\Der^\llog.$ The fully faithful  embedding $$\Der^\llog\hookrightarrow\Fun(\Delta^1,\MTL)$$
exhibits $\Der^\llog$ as a full subcategory of $$\Fun(\Delta^1,\MT_{\algcn})\times_{\Fun(\Delta^1,\algcn)}\Fun(\Delta^1,\llog).$$
Then it's easy  to see that we have a factorization
$$\begin{tikzcd}
\Der^{\llog} \arrow[d, "\Phi"]                                                  \\
{\Der\times_{\Fun(\Delta^1,\algcn)}\Fun(\Delta^1,\llog)} \arrow[d]        \\
{\Fun(\Delta^1,\MT_{\algcn})\times_{\Fun(\Delta^1,\algcn)}\Fun(\Delta^1,\llog)}
\end{tikzcd}$$
The functor $\Phi$ is fully faithful.
    Unwinding the definition of log derivations, the essential image of $\Phi$ is identified with the subcategory of 
$${\Der\times_{\Fun(\Delta^1,\algcn)}\Fun(\Delta^1,\llog)}$$
spanned by verticals $\partial$ satisfying the  following condition:
\begin{itemize}
    \item 
    The image of $\partial$ under the projection to the second factor$$\pi_2:\Der\times_{\Fun(\Delta^1,\algcn)}\Fun(\Delta^1,\llog)\rr \Fun(\Delta^1,\llog)$$ has the form of $(A,M)\rightarrow(A\oplus \tau_{\geq 0}(I[1]),M\oplus \tau_{\geq 0}(I[1]))$ for some $(A,M)\in\llog$ and $I[1]\in\Mod_A$. 
\end{itemize}
We denote by $\Der^{+}$ the $\infty$-category constructed in \cite[Notation 7.4.2.4]{lurie2017higher}. Informally, it's the subcategory of $\Der(\algcn)$ which has the objects $d:A\rightarrow I[1]$ with $I\in\Mod_A^{\rm cn}$, and  morphisms $f:(d:A\rightarrow I)\rightarrow(d':B\rightarrow J)$ that satisfy the condition that  the induced map $B\otimes_AI\rightarrow J$ is an equivalence of $A$-modules. Then we have a pullback square 
$$\begin{tikzcd}
{\Der^{\llog,+}} \arrow[d] \arrow[r] & \Der^+ \arrow[d] \\
\Der^{\llog} \arrow[r]               & \Der            
\end{tikzcd}$$
We also denote by $\mathbf{Def}^{+}$ the subcategory of $\Fun^+(\Delta^1,\algcn)$ constructed in \cite[Notation 7.4.2.6]{lurie2017higher}, spanned by square-zero extensions. Informally, objects in $\mathbf{Def}^+$ are square-zero extensions and morphisms in $\mathbf{Def}^+$ are deformations. There is a forgetful functor $\eta:\mathbf{Def}^{\llog,+}\rr\mathbf{Def}^+$
, and a pullback square
$$\begin{tikzcd}
{\mathbf{Def}^{\llog,+}} \arrow[d] \arrow[r] & \mathbf{Def}^+ \arrow[d] \\
\mathbf{Def}^{\llog} \arrow[r]               & \mathbf{Def}            
\end{tikzcd}$$
The above two pullback squares give rise to  fully faithful embeddings 
$$\Der^{\llog,+}\subset\Der^+\times_{\Fun(\Delta^1,\algcn)}\Fun(\Delta^1,\llog)$$
$$\mathbf{Def}^{\llog,+}\subset\mathbf{Def}^+\times_{\Fun(\Delta^1,\algcn)}\Fun(\Delta^1,\llog).$$
Then we can form a commutative diagram: 
$$\begin{tikzcd}
{\Der^{\llog,+}} \arrow[d, "\gamma^+"'] \arrow[r] & \Der^+ \arrow[d] \\
{\mathbf{Def}^{\llog,+}} \arrow[r, "\eta"]        & \mathbf{Def}^+  
\end{tikzcd}$$
\begin{lem}\label{12345436942}
The commutative diagram given as above is a pullback square.
\end{lem}
\begin{proof}
  We denote by $$\alpha:\Der^{\llog,+}\rr\Der^+\times_{\mathbf{Def}^+}\mathbf{Def}^{\llog,+}$$ the functor induced by the above commutative diagram. The functor $\alpha$ is essentially surjective. Indeed, an object $X$ in the right hand term is represented by a triple  $(d,\phi,\sigma),$ where $d: A\rightarrow I[1]$ is a derivation such that $I$ is connective, and $\phi:(A',M')\rightarrow(A,M)$ is a square-zero extension which fits into a pullback square in $\MTL$ as follows:
  $$\begin{tikzcd}
{(A',M')} \arrow[r, "\phi"] \arrow[d] & {(A,M)} \arrow[d, "d'"] \\
0 \arrow[r]                           & {J[1]}                 
\end{tikzcd}$$
and $\sigma$ is an equivalence of square-zero extensions $\sigma:(A''\rightarrow A)\stackrel{\simeq}\rightarrow (A'\rightarrow A)$, where $(A'\rightarrow A)$ is the  underlying square-zero extension of $\Einf$-rings of $\phi$ and $A'\rightarrow A$ is the square-zero extension of $\Einf$-rings induced by $d$.   Note that this implies that there is an $A$-module  equivalence $I[1]\simeq J[1]$. Therefore, the derivation $d'$ belongs to $\Der^{\llog,+},$ and the image of $d'$ under the functor $\alpha$ is equivalent to $X$.\\
Then we have to prove that the functor $\alpha$ is fully faithful. Using the embeddings
$$\Der^{\llog,+}\subset\Der^+\times_{\Fun(\Delta^1,\algcn)}\Fun(\Delta^1,\llog)$$ and
$$\mathbf{Def}^{\llog,+}\subset\mathbf{Def}^+\times_{\Fun(\Delta^1,\algcn)}\Fun(\Delta^1,\llog)$$ and then  we form a commutative diagram
$$\begin{tikzcd}
{\Der^{\llog,+}} \arrow[r, hook] \arrow[d, "\alpha"'] & {\Der^+\times_{\Fun(\Delta^1,\algcn)}\Fun(\Delta^1,\llog)} \arrow[d]                                                \\
{\mathbf{Def}^{\llog,+}} \arrow[r, hook]              & {\Der^+\times_{\mathbf{Def}^+}\mathbf{Def}^+\times_{\Fun(\Delta^1,\algcn)}\Fun(\Delta^1,\llog)} \arrow[d, "\simeq"] \\  & {\Der^+\times_{\Fun(\Delta^1,\algcn)}\Fun(\Delta^1,\llog)}     
\end{tikzcd}$$
Clearly, the functor $\alpha$ is fully faithful.
\end{proof}
Finally,  using Lemma\autoref{12345436942} and the deformation theory of $\Einf$-rings in \cite[Theorem 7.4.2.7]{lurie2017higher}, we get the desired result.
\subsection{Deformation theory of animated log rings}\label{deformationofanimated}
 The theory of animated log rings provided another building block of derived logarithmic geometry. The difference between $\Einf$-log rings and animated log rings is that the  animated log rings are required to have extra structures via the theorem of Barr-Beck-Lurie. As we have argued in \autoref{einfvsani}, the $\infty$-category of animated log rings can be viewed as the $\infty$-category of modules over $\llog$ with respect to a certain monad $T$. We can associate to every animated log ring $(A,M)$ an underlying $\Einf$-log ring structure $(A^\circ,M^\circ)$ via the forgetful functor $\llog^{\Delta}\stackrel{\Theta_{\rm Log}}\rightarrow\llog
_{\Z}\rightarrow\llog$. This functor admits a left adjoint $\Theta^L_{\rm Log}$, which is monadic, because the forgetful functor $\llog
_{\Z}\rightarrow\llog$ is monadic (because it is conservative, preserves arbitrary small limits, and geometric realizations). 
\begin{rmk}
    Even restricted to the characteristic zero case, the forgetful functor $\llog^\Delta_{\mathbb Q}\rightarrow\llog_{\mathbb Q}$ is not an equivalence.
\end{rmk}
\subsubsection{Algebraic log cotangent complexes and algebraic log derivations}
\begin{defn}
    The \textit{replete tangent bundle of animated log rings} is the fiber product $\ani\tang^{\rm rep}_{\llogd}:=\tlog\times_{\llog}\llogd$.
\end{defn}
\begin{rmk}
    The replete tangent bundle of animated log rings $\ani\tang^{\rm rep}_{\llogd}$ is not equivalent to the replete tangent bundle given by the forgetful functor $U:\llogd\rightarrow\alg^\Delta.$
\end{rmk}
Consider the following commutative diagram 
$$\begin{tikzcd}
       & \ani\tang^{\rm rep}_{\llogd} \arrow[d, "\widetilde{\Omega}^{\infty}"'] \arrow[r, "{\rm pr}_1"] & \tlog \arrow[d, "\Omega^\infty"]          \\
{\Fun(\Delta^1,\llog^{\Delta})} \arrow[r] \arrow[rd] & {\Fun^{\rm ext}(\Delta^1,\llog)\times_{\llog}\llog^{\Delta}} \arrow[d] \arrow[r]                & {\Fun^{\rm ext}(\Delta^1,\llog)} \arrow[d] \\
    & \llog^{\Delta} \arrow[r]   & \llog    
\end{tikzcd}$$
Here the functor $\Fun(\Delta^1,\llog^{\Delta})\rightarrow\Fun^{\rm ext}(\Delta^1,\llog)\times_{\llog}\llog^{\Delta}$ is given by the product $\lambda\times{\rm ev}_1$, where $\lambda$ is defined as the composition
$$\lambda:\Fun(\Delta^1,\llog^{\Delta})\stackrel{(-)^\circ}\rr\Fun(\Delta^1,\llog)\stackrel{\rm ex}\rr\Fun^{\rm }(\Delta^1,\llog)$$ 
\begin{lem}
The functor $\widetilde{\Omega}^{\infty}:\ani\tang^{\rm rep}_{\llogd}\rightarrow \Fun^{\rm ext}(\Delta^1,\llog)\times_{\llog}\llog^{\Delta}$ uniquely factors through a limit-preserving functor 
$\Omega^{\infty,\Delta}: \ani\tang^{\rm rep}_{\llogd}\rightarrow \Fun(\Delta^1,\llog^\Delta)$, such that the following diagram is commutative.
$$\begin{tikzcd}
\ani\tang^{\rm rep}_{\llogd} \arrow[d, "{\Omega^{\infty,\Delta}}"'] \arrow[r, "{\rm pr}_1"] & \tlog \arrow[d, "\Omega_{\rm Log}^\infty"] \\
{\Fun(\Delta^1,\llog^\Delta)} \arrow[r, "(-)^\circ"]                                        & {\Fun(\Delta^1,\llog)}                    
\end{tikzcd}$$
\end{lem}
\begin{proof}
    We first prove the  existence of the functor $\Omega^{\infty,\Delta}.$  
    Note that all functors except $\Omega^{\infty,\Delta}$ in the commutative diagram are left adjointable. We have to prove that there is a colimit-preserving functor $\mathcal L:\Fun(\Delta^1,\llog^\Delta)\rightarrow\ani\tang^{\rm rep}_{\llogd},$ such that the following diagram commutes
    $$\begin{tikzcd}
{\Fun(\Delta^1,\llog)} \arrow[d, "\LL^G"] \arrow[r, "(-)^\Delta"] & {\Fun(\Delta^1,\llogd)} \arrow[d, "\mathcal L"] \\
\tang^{\rm rep}_{\llog} \arrow[r]                                  & \ani\tang^{\rm rep}_{\llogd}                    
\end{tikzcd}$$
Here the functor $(-)^\Delta$ is the left adjoint of the functor $(-)^\circ.$ We can choose the functor $\mathcal L$ as the left Kan extension of $\LL^G$ 
along the functor $(-)^\Delta.$  The right adjoint $\hat{\Omega}^{\infty,\Delta}$ of $\mathcal L$ satisfies the desired property, which is terminal in the $\infty$-category of all functors $\omega:\ani\tang^{\rm rep}_{\llogd}\rightarrow\Fun(\Delta^1,\llogd)$ such that the diagram appearing in this proposition commutes.\\
We prove the uniqueness of the functor $\Omega^{\infty,\Delta}$. Let $\omega:\ani\tang^{\rm rep}_{\llogd}\rightarrow\Fun(\Delta^1,\llogd)$ be any functor which makes the diagram in this proposition commute. There is a unique natural transformation $\xi:\omega\rightarrow\hat{\Omega}^{\infty,\Delta}$ according to the construction of the functor $\hat{\Omega}^{\infty,\Delta}$.
The forgetful functor $(-)^\circ$ is monadic. Let $T$ be the monad associated with $(-)^\circ.$ Let us identify the $\infty$-category $\Fun(\Delta^1,\llogd)$ with ${\mathbf{L}}\Mod_T(\Fun(\Delta^1,\llog))$ the $\infty$-category of $T$-modules in $\Fun(\Delta^1,\llog)$. It follows that  the equivalence of $\xi$ is detected by its underlying  natural transformation $(-)^\circ\circ \xi$. It's clear that the natural transformation  $(-)^\circ\circ \xi$ is an equivalence.
\end{proof}
\begin{defn}
    The algebraic Gabber's cotangent complex functor $$\LL^{G,\Delta}:\Fun(\Delta^1,\llog^{\Delta})\rightarrow\ani\tang^{\rm rep}_{\llog^\Delta}$$ is the left adjoint of the functor $\Omega^{\infty,\Delta}$.
\end{defn}
\begin{rmk}[Gabber's cotangent complex as derived functors]
    The algebraic Gabber's cotangent complex of a map of animated log rings $(A,M)\rightarrow(B,N)$ is characterized by the following universal properties 
    $$\map_{(A,M)//(B,N)}((B,N),(B\oplus J,N\oplus J))\simeq\map_B(\LL^{G,\Delta}_{(B,N)/(A,M)},J),J\in\Mod_B^{\rm cn}.$$
     The Gabber's cotangent complex of animated log rings is  originally defined as the non-abelian derived functor of the functor of modules of log differential forms, c.f. \cite[Definition 8.5]{olsson2005logarithmic}, \cite[Definition 4.7]{sagave2016derived} and \cite[Definition 3.5]{binda2023hochschild}. By definition, it's characterized by the same universal property with $\LL^{G,\Delta}$. This implies that the two definitions of Gabber's cotangent complex coincide.
\end{rmk}
\begin{defn}
    The \textit{replete tangent correspondence of animated log rings} $\ani\mathfrak M^{\rm rep}_{\llog^{\Delta}}$ is the fiber product 
    $\MTL\times_{\llog}\llog^\Delta.$
\end{defn}
\begin{defn}
    The $\infty$-category of \textit{algebraic log  derivations} $\Der^{\llog^\Delta}$  is the subcategory of $\Fun(\Delta^1, \ani\mathfrak M^{\rm rep}_{\llog^{\Delta}})$ spanned by those functors  
  $d:\Delta^1\rightarrow\ani\mathfrak M^{\rm rep}_{\llog^{\Delta}}$ such that the composition 
  $$\Delta^1\rr \ani\mathfrak M^{\rm rep}_{\llog^{\Delta}}\rr \Delta^1\times \llog^\Delta$$
  is the functor $\Delta^1\rightarrow \Delta^1\times\{(A,M)\}\subset \Delta^1\times \llog^\Delta$ for some animated log ring $(A,M)$.
\end{defn}
Unwinding the definition, we can form a commutative diagram of $\infty$-categories 
$$\begin{tikzcd}
\Der^{\llog^\Delta} \arrow[d, hook] \arrow[r, "\simeq"]                         & {\Der^{\llog}\times_{\Fun(\Delta^1,\llog)}\Fun(\Delta^1,\llog^\Delta)} \arrow[d, hook] \\
{\Fun(\Delta^1,\ani\mathfrak M^{\rm rep}_{\llog^{\Delta}})} \arrow[r, "\simeq"] & {\Fun(\Delta^1,\MTL)\times_{\Fun(\Delta^1,\llog)}\Fun(\Delta^1,\llog^\Delta)}         
\end{tikzcd}$$
\subsubsection{Algebraic deformation theory}
\begin{defn}
    Let $(A,M)$ be an animated log ring, and let $I\in\Mod_A$ be an $A$-module. A \textit{square-zero extension}  of $A$ with kernel $I$ is a strict  morphism of animated log rings  $i:(A',M')\rightarrow (A,M)$, such that $A'$ is a square-zero  extension of $A$ with kernel $I$.
\end{defn}

\begin{rmk}
    Obviously there is a pullback square of $\infty$-categories 
    $$\begin{tikzcd}
\mathbf{Def}^{\llog^\Delta} \arrow[r] \arrow[d] & \mathbf{Def}^{\llog} \arrow[d] \\
{\Fun(\Delta^1,\llog^\Delta)} \arrow[r]         & {\Fun(\Delta^1,\llog)}        
\end{tikzcd}$$
\end{rmk}
\begin{defn}
Let $(R,L)$ be an animated log ring, and let $(R',L')$ be a square-zero extension of $(R,L)$ with kernel $I\in\Mod_R$. Let $(A,M)$ be an $(R,L)$-algebra,
a deformation of $(A,M)$ to $(R',L')$ is a 
square-zero extension  $(A',M')$ of $(A,M)$ with kernel $J\in\Mod_A$, such that there is an equivalence  $J\simeq I_A=I\otimes_RA$ of $A$-modules, and the following diagram 
\[\begin{tikzcd}
	{(R',L')} & {(R,L)} \\
	{(A',M')} & {(A,M)}
	\arrow[from=1-1, to=1-2]
	\arrow[from=1-1, to=2-1]
	\arrow[from=2-1, to=2-2]
	\arrow[from=1-2, to=2-2]
\end{tikzcd}\]
is a pushout square in $\llog^\Delta$. 
\end{defn}

\begin{defn}
    Let $\mathbf{Def}^{\llog^\Delta,+}$ be the subcategory of $\mathbf{Def}^{\llog^\Delta}$ spanned by the following data:
    \begin{enumerate}
        \item Objects are same with those in $\mathbf{Def}^{\llog^\Delta}$;
        \item An morphism   $((R',L')\rightarrow (R,L))\rightarrow ((A',M')\rightarrow(A,M))$ in $\mathbf{Def}^{\llog^\Delta}$ belongs to $\mathbf{Def}^{\llog^\Delta,+}$ if and only if $(A',M')$ is a  deformation of $(A,M)$ to $(R',L').$
    \end{enumerate}
        \end{defn}

    \begin{defn}
            Let $\Der^{\llog^\Delta,+}$ be the subcategory  of $\Der^{\llog\Delta}$ as follows:
\begin{enumerate}
    \item Objects are those derivations $d:(A,M)\rightarrow I[1]$ such that $I$ is connective.
    \item Morphisms are those mosprhisms having the form $f:(d:(A,M)\rightarrow I)\rightarrow(d':(B,N)\rightarrow J)$ such that the induced map $I\otimes_AB\rightarrow J$ is an equivalence.
\end{enumerate}
\end{defn}
\begin{rmk}
    There are pullback squares
    $$\begin{tikzcd}
{\Der^{\llog^\Delta,+}} \arrow[d] \arrow[r] & {\Der^{\llog,+}} \arrow[d] &  & {\mathbf{Def}^{\llog^\Delta,+}} \arrow[r] \arrow[d] & {\mathbf{Def}^{\llog,+}} \arrow[d] \\
{\Fun(\Delta^1,\llog^\Delta)} \arrow[r]     & {\Fun(\Delta^1,\llog)}     &  & {\Fun(\Delta^1,\llog^\Delta)} \arrow[r]             & {\Fun(\Delta^1,\llog)}            
\end{tikzcd}$$
\end{rmk}

    \begin{rmk}\label{anifromeinf}
    We have a functor $\gamma^+:\Der^{\llog^\Delta,+}\rightarrow\mathbf{Def}^{\llog^\Delta,+}$ same as that defined in Theorem\autoref{deform}
\end{rmk}
We have the analogous of Theorem\autoref{deform} in the animated context.
\begin{thm}[Algebraic deformation theory]\label{deformani}
    The functor $\gamma^+:\Der^{\llog^\Delta,+}\rightarrow\mathbf{Def}^{\llog^\Delta,+}$ induced from $\gamma$ is a left fibration.
\end{thm}
\begin{proof}
  Using Remark\autoref{anifromeinf},  we have a pullback square of $\infty$-categories 
 $$ \begin{tikzcd}
{\Der^{\llog^\Delta,+}} \arrow[d, "\gamma^+"'] \arrow[r] & {\Der^{\llog,+}} \arrow[d, "\gamma^+"] \\
{\mathbf{Def}^{\llog^\Delta,+}} \arrow[r]                & {\mathbf{Def}^{\llog,+}}              
\end{tikzcd}$$
Then applying Theorem\autoref{deform}, we get the desired result.
\end{proof}
\section{Log stacks}\label{logstacks}
The development  of logarithmic geometry is due to the study of the geometry and  cohomology of singular varieties and schemes. The foundation of this theory was given by Kato \cite{Kato1988}, following the ideas of Fontaine-Illusie. A \textit{log scheme} is a triple  $(X,\mathcal M,\alpha)$, where $X$ is a scheme and $\alpha:\mathcal M\rightarrow\mathcal O_X$ is a map of sheaves of monoids, with an  extra condition $\mathcal O_X^*\times_{\mathcal O_X}\mathcal M\simeq\mathcal O^*_X$. Earlier than that, in Deligne's papers \cite{Deligne1}, \cite{Deligne2} and \cite{Deligne3} on the theory of mixed Hodge structures, he studied an original version of logarithmic geometry, in which  the geometric objects he considered are complex projective manifolds equipped  with  normal crossing divisors $(X,D)$, which should be thought  as varieties with boundary. We note that this approach is  naturally related to the singular varieties, as for any resolution of singularities (if it exists) $\pi:\widetilde{X}\rightarrow X$ of a variety $X$ over a field $k$, it gives rise to a pair $(\widetilde{X},E)$, where $E$ is the exceptional divisor.  One feature of logarithmic geometry is that it enlarges the notion of smoothness to log-smoothness, which allows us to treat a large class of singular schemes, for example, varieties with resolvable singularities and semi-stable schemes over DVRs, similarly to smooth ones.\\
On the other hand, the derived algebraic geometry studies higher structures of geometric objects in algebraic geometry. Many algebraic and topological invariants associated to schemes came from derived algebraic  geometry, such as cotangent complexes, topological Andr\'e-Quillen homology, topological Hochschild homology and so on. So the logarithmic geometry and derived algebraic geometry meet in the study of the logarithmic version of certain higher invariants.\\
In this section, we continue working in the $\Einf$-context. We will study the notion of \textit{charted log stacks}, which has been studied in \cite{sagave2016derived} in the context of simplicial log rings, and \textit{log structures} associated with spectral Deligne-Mumford stacks. Then we will study the moduli stack of all log stacks and will prove the descent property.
\subsection{Charted log stacks}
In classical logarithmic geometry,  a chart on a log scheme $(X,\mathcal M)$ is a map of monoids $\tau:M\rightarrow \Gamma(X,\mathcal M),$ which yields an isomorphism of sheaves of monoids $\tau':\underline{M}^a\stackrel{\simeq}\rr\mathcal M,$ 
where $\underline{M}$ is the constant sheaf on $X$ associated with $M$.
A charted log scheme is a quadruple $(X,\mathcal M,M,\tau),$ where $(X,\mathcal M)$
is a log scheme, and $\tau:M\rightarrow\Gamma(X,\mathcal M)$ is a chart of $(X,\mathcal M)$.
We will generalize this notion to spectral stacks in this subsection.
\subsubsection{Topologies on $\Einf$-log rings}
Let $f:(A,M)\rightarrow(B,N)$ be a morphism of $\Einf$-log rings, recall that we have constructed Gabber's cotangent complex $\LL^G_{(B,N)/(A,M)}$. Similar to derived algebraic geometry, we have the following definition:
\begin{defn}\label{aaabbbvvvddd}
We say that a morphism of $\Einf$-log rings $f:(A,M)\rightarrow(B,N)$ is
\begin{enumerate}
    \item 
    \textit{Formally log-smooth} if the cotangent complex $\LL^G_{(B,N)/(A,M)}$ is projective in $\Mod_B$, i.e., it is a direct summand of a free $B$-module. 
\item \textit{Formally log-\'etale} if the cotangent complex $\LL^G_{(B,N)/(A,M)}$ is contractible. 
\item \textit{Log-smooth} if it is formally smooth and the underlying ring map $f:A\rightarrow B$ is almost  finitely presented, c.f. \cite[Definition 7.2.4.26]{lurie2017higher}.
\item \textit{Log-\'etale} if it is formally \'etale and the underlying ring map $f:A\rightarrow B$ is almost  finitely presented.
    \end{enumerate}
\end{defn}
\begin{defn}
    The collection of strict \'etale (resp. log-\'etale) morphisms forms a Grothendieck topology on the $\infty$-category $\llog^{\rm op}$. We denote by $\llog^{\rm op}_{\set}$ (resp. $\llog^{\rm op}_{{\rm l}\et}$) the corresponding site.\\
    Let $(A,M)$ be an $\Einf$-log ring,
    denote by $(A,M)_{\set}$ (resp. $(A,M)_{{\rm l}\et}$) the subcategory of $\llog_{(A,M)/}$ spanned by strict \'etale (resp. log-\'etale) objects over $(A,M)$.
    \end{defn}
The following proposition gives an analogue in the context of higher algebra of a result due to Vidal about the nil-invariance of the classical log étale site, c.f. \cite{vidal2001}.
\begin{prop}\cite[Theorem 1.3]{lundemo2023deformation}\label{strictetalerig}
    Let $(A,M)$ be an $\Einf$-log ring, and let $(A',M')$ be a square-zero extension of $(A,M)$. Then the base change gives rise to  equivalences
    \begin{align*}-\otimes_{A'}A:(A',M')_{\set}\stackrel{\simeq}\rr(A,M)_{\set},\\(A',M')_{{\rm l}\et}\stackrel{\simeq}\rr(A,M)_{{\rm l}\et}.\end{align*}
\end{prop}
\begin{proof}
   The two functors are essentially surjective by Corollary\autoref{deform=sq0extension}.  Now we need to prove that it's fully faithful. The two cases are similar, so we only give a proof for log-\'etale topology. Let $(B',N')$ and $(C',P')$ be objects in $(A',M')_{{\rm l}\et}$, and let $(B,N),(C,P)$ be the corresponding base changes to $(A,M)$ respectively. We have to prove the following equivalence 
   $$\map_{(A',M')}((B',N'),(C',P'))\stackrel{\simeq}\rr\map_{(A,M)}((B,N),(C,P)).$$
   Note that we have a factorization 
$$\begin{tikzcd}
{\map_{(A',M')}((B',N'),(C',P'))} \arrow[r] \arrow[d] & {\map_{(A,M)}((B,N),(C,P))} \\
{\map_{(A',M')}((B',N'),(C,P))} \arrow[ru, "\simeq"'] &                            
\end{tikzcd}$$
   We need to show the map
   $$\map_{(A',M')}((B',N'),(C',P'))\rr\map_{(A',M')}((B',N'),(C,P))$$ is an equivalence. Since the mapping space functor preserves limits with respect to the second variable, we just need to show the equivalence for letting $(C',P')$ be $(C\oplus I,P\oplus I)$ for some $1$-connective $C$-module $I.$ In this case, the fiber of this map at an arbitrary point $g\in\map_{(A',M')}((B',N'),(C,P))$ is equivalent to the mapping space 
   $$\map_{C}(C\otimes_B\LL^G_{(B,N)/(A,M)},I),$$
   which is contractible because $\LL^G_{(B,N)/(A,M)}\simeq 0$ by assumption.
\end{proof}
\subsubsection{Precharted log stacks}
    \begin{defn}
  A \textit{precharted log stack} is a sheaf over the site $\llog_{\set}.$ Denote by $\llog\stk^{\rm prechart}:=\shv(\llog^{\rm op}_{\set},\ani)$  the $\infty$-category of precharted log stacks.
  \end{defn}
  \begin{defn}
    Let $\spec(A,M):\llog_{\set}\rightarrow\ani$ be the functor $\map((A,M),-).$ We call it the \textit{charted log-affine scheme} associated with $(A,M)$.
\end{defn}

\begin{prop}\cite[Proposition 7.13]{sagave2016derived}\label{strictdescent}
    The functor $\spec(A,M)$
    satisfies hyperdescent with respect to strict \'etale topology.
\end{prop}
\begin{lem}\label{1884310}
    Let $f:A\rightarrow B$ be a faithfully flat map of connective $\Einf$-rings. Then one has $GL_1(A)\simeq GL_1(B)\times_{\Omega^\infty B}\Omega^{\infty}A$.
\end{lem}
\begin{proof}
    Consider  the fiber sequence
\[\Omega^{\infty+1} B\rr \Omega^{\infty}A\times_{\Omega^{\infty}B}GL_1(B)\rr \Omega^{\infty}A\times GL_1(B).\] 
It induces a long exact sequence of homotopies
$$\rr\pi_{*+1}(B)\rr\pi_*(\Omega^{\infty}A\times_{\Omega^{\infty}B}GL_1(B))\rr\pi_*(A)\times\pi_*(GL_1(B))\rr\pi_*(B)\rr$$
For $*\neq 0$, we have $\pi_*(B)\simeq\pi_*(GL_1(B))$. Hence, the exact sequence splits when $*\geq 1$, so we have an isomorphism of abelian groups $\pi_*(\Omega^{\infty}A\times_{\Omega^{\infty}B}GL_1(B))\simeq\pi_*(A),*\neq 0,$ and 
$\pi_0(A\times_BGL_1(B))\simeq \pi_0(A)\times_{\pi_0(B)}\pi_0(GL_1(B)).$
Since $A\rightarrow B$ is faithfully flat, $\pi_0(A)\rightarrow\pi_0(B)$ is faithfully flat too.
It is therefore enough to show the lemma for discrete rings. This is clear. Indeed, 
let $f:A\rightarrow B$ be a faithfully flat map of discrete  rings. Clearly, we have $GL_1(A)\subseteq f^{-1}(GL_1(B)).$  On the other hand, let $x$ be an element in $f^{-1}(GL_1(B))$, we will construct an inverse of $x$. The $A$-module map $\phi_x: A\rightarrow A, a\mapsto ax$ is an isomorphism because $\phi_x\otimes_A B$ is an isomorphism. Therefore, there exists $y\in A$ such that $\phi_x(y)=xy=1.$
\end{proof}
\begin{lem}\label{18843}
    Let $f:(A,M)\rightarrow (B,N)$ be a strict morphism such that the underlying map $A\rightarrow B$ of connective $\Einf$-rings is faithfully flat, then 
$N\simeq M\coprod_{GL_1(A)}GL_1(B)$ 
and $f$ is exact.
\end{lem}
\begin{proof}
Here $N$ is equivalent to the logification of the $\Einf$-prelog ring given by $M\rightarrow \Omega^{\infty}A\rightarrow\Omega^\infty B$. We note that 
\begin{align*}
M\times_{\Omega^\infty B}GL_1(B)&\simeq M\times_{\Omega^\infty A}\Omega^{\infty}A\times_{\Omega^\infty B}GL_1(B)\\&\simeq M\times_{\Omega^\infty A}GL_1(A)\\&\simeq GL_1(A)
\end{align*}
by Lemma\autoref{1884310}. We also have
\begin{align*}
M\times_NGL_1(B)&\simeq M\times_N(GL_1(B)\times_{\Omega^\infty B}N)\\
&\simeq M\times_{\Omega^\infty B}GL_1(B)\\
&\simeq M\times_{\Omega^\infty A}\Omega^\infty A\times_{\Omega^\infty B}GL_1(B)\\
&\simeq M\times_{\Omega^\infty A}GL_1(A)\\&\simeq GL_1(A)
\end{align*}
This implies that the commutative diagram 
$$\begin{tikzcd}
GL_1(A) \arrow[d] \arrow[r] & GL_1(B) \arrow[d] \\
M \arrow[r]                 & N                
\end{tikzcd}$$
is both Cartesian and coCartesian. Now, we will show that the map $M\rightarrow N$ is exact. Assume that there is a factorization $M\rightarrow M'\rightarrow N$ such that $M'$ is an exactification of $M$ over $N$, we have to show that $M\simeq M'$. We have a commutative diagram of $\Einf$-monoids with exact rows
$$\begin{tikzcd}
M' \arrow[rd] \arrow[dd, "{\rm gp}"', bend right] &                                   &                                     \\
M \arrow[d, "{\rm gp}"] \arrow[r] \arrow[u]       & N \arrow[d, "{\rm gp}"] \arrow[r] & GL_1(B)/GL_1(A) \arrow[d, "\simeq"] \\
M^{\rm gp} \arrow[r]                              & N^{\rm gp} \arrow[r]              & GL_1(B)/GL_1(A)                    
\end{tikzcd}$$
By the assumption $M'^{\rm gp}\simeq M^{\rm gp}$, we have that the composition $M'\rightarrow N\rightarrow GL_1(B)/GL_1(A)$ is vanishing, and therefore it uniquely factors through $M$. It shows that $M\simeq M'$ because $M$ also has the universal property of exactification.
\end{proof}
\begin{proof}[Proof of Proposition\autoref{strictdescent}]
    Let $f^\bullet:(B,N)\rightarrow (C^\bullet,P^\bullet)$ be a hypercover of $(B,N)$ with respect to  strict \'etale topology. We have to show that the induced map $\lim f^\bullet:(B,N)\rightarrow \lim_{\Delta}(C^\bullet,P^\bullet)=:(C^{-1},P^{-1})$ is an equivalence. By Lemma\autoref{18843}, it is also a hypercover consisting of exact maps. This means that the map $(B,N)\rightarrow(C^{-1},P^{-1}) $ could be calculated in the $\infty$-category $\Fun^{\rm ext}(\Delta^1,\llog)$. In particular, we have $N\rightarrow P^{-1}$ is exact and $N^{\rm gp}\simeq (P^{-1})^{\rm gp}$. By the universal property of exactness, we have $N\simeq
 P^{-1}$.
\end{proof}
\begin{rmk}
    According to the proof of Proposition\autoref{strictdescent}, we see that the functor $\spec(A,M)$ also satisfies the \textbf{strict flat hyperdescent}.
\end{rmk}
Besides strict \'etale topology, the log-\'etale covers also define a topology on $\llog,$ which is strictly  stronger than the strict \'etale topology.   But there is an issue that this topology is not subcanonical. We have the following example:
\begin{exmp}
Let $(A,M)$ be an $\Einf$-log ring. We can construct two different log structures on the $\Einf$-ring $A\times A$. Let $M^a$ be the logification of the diagonal  map $$M\rr \Omega^{\infty}A\stackrel{\Delta}\rr\Omega^\infty(A\times A),$$
    and let 
    $$M\times M\rr\Omega^\infty(A\times A) $$
    be the product of the  log structure $M\rightarrow \Omega^\infty A.$
    There is a natural morphism $(A\times A,M^a)\rightarrow(A\times A,M\times M),$ which is a log-\'etale cover. However, the induced {\v C}ech nerve of this map is the constant  diagram $\{(A\times A,M\times M)\}^{\infty}_{n=0}.$ This implies that the log-\'etale topology  is not subcanonical.
\end{exmp}
\subsubsection{Moduli functor of precharted log structures}
Applying Lemma\autoref{Lstrict},  one has that there is an adjoint pair  $$\begin{tikzcd}
\algcn_{\et} \arrow[r, "i", shift left] & \llog_{\set} \arrow[l, "U", shift left]
\end{tikzcd}$$
of sites. Here $U$ is the forgetful functor to the underlying $\Einf$-rings, and $i$ is the natural embedding from the $\infty$-category of connective $\Einf$-rings to the $\infty$-category of $\Einf$-log rings.
The induced functor 
$$U^*:\stk\rr\llog\stk^{\rm prechart}$$
preserves both small limits and colimits, and it follows that it admits a right adjoint $U_*$ and a left adjoint $U_!$. In particular, one can see that the functor $U_!$ is equivalent to the left Kan extension of the composition $\llog^{\rm op}\stackrel{U}\rightarrow\alg^{\rm cn
,op}\stackrel{\spec}\rightarrow\stk$ along the inclusion  $\llog^{\rm op}\subset\llog\stk^{\rm prechart}$. We have the following commutative diagram
$$\begin{tikzcd}
\llog^{\rm op} \arrow[r] \arrow[rr, "{(A,M)\mapsto\spec A}", bend left] \arrow[rd] & {\alg^{\rm cn,op}} \arrow[r]   & \stk\\
    & \llog\stk^{\rm prechart} \arrow[ru, "U_!"'] &           
\end{tikzcd}$$
We also denote by $\underline X:=U_!(X)$ a precharted log stack $X\in\llog\stk^{\rm prechart}$, and refer to it as the \textit{underlying  stack} of $X$. 
Let $\underline f:\underline X\rightarrow \underline Y$ be a morphism of stacks, and let $Y\in\llog\stk^{\rm prechart}\times_{\stk}\{\underline Y\}$, i.e., $Y$ be a precharted log stack whose underlying stack is equivalent to $\underline Y$. One can directly check that the  projection map $Y\times_{U^*\underline Y}U^*\underline X\rightarrow Y$ is a Cartesian edge  with respect to $\underline f.$ We will denote  $X$  by $\underline f^*Y$ in this case. We deduce that the functor $U_!$ is a Cartesian fibration.\\
\begin{defn}
    We say that a morphism $f:X\rightarrow Y$ is strict, if it is a Cartesian edge with respect to the underlying  map of stacks $\underline f:\underline X\rightarrow\underline Y.$
\end{defn}
\begin{rmk}
    By definition, the strictness is stable under base change and composition. 
\end{rmk}
\begin{lem}\label{pullbackaff=aff}
    Let $\alpha:\spec A\rightarrow\spec B$ be a map of affine spectral schemes, and let $(B,N)$ be an $\Einf$-log ring, then we have $\alpha^*\spec(B,N)\simeq\spec(A,N)^a$. So the strictness of maps of $\Einf$-log rings and that of maps of charted log-affine schemes coincide.
\end{lem}
\begin{proof}
The logification $(A,N)^a$ is identified with the pushout $(B,N)\otimes_{(B,GL_1(B))}(A,GL_1(A))$. Here, the tensor product is formed in the $\infty$-category $\llog$.   We  have the following equivalences of sheaves over the strict \'etale site\begin{align*}
\alpha^*\spec(B,N)&\simeq\spec(B,N)\times_{U^*\spec B}U^*\spec A\\ 
 &\simeq \spec((B,N)\otimes_{(B,GL_1(B))}(A,GL_1(A))\\
 &\simeq \spec(A,N)^a.
\end{align*}
\end{proof}
Let $X$ be a precharted log stack, and let $\underline{X}$ be its underlying stack $U_!(X)$, We have a Cartesian fibration $$U^X_!:\llog\stk^{\rm prechart}_{/X}\rr\stk_{/\underline{X}}.$$ 
Let $\underline Y\rightarrow\underline  X$ be a map of stacks. We denote by $\widehat{\R\llog}_{X}^{\rm prechart}(\underline Y)$  the fiber of $U^X_!$ at $\underline Y,$ and denote by $\R\llog^{\rm prechart}_{X}$ its core.  This defines  prestacks 
$$\widehat{\R\llog}^{\rm prechart}_{X}:\stk^{\rm op}_{/\underline{X}}\rr \cat, \underline Y\mapsto\widehat{\R\llog}_{X}^{\rm prechart}(\underline Y)$$
and
$$\R\llog^{\rm prechart}_{X}:\stk^{\rm op}_{/\underline{X}}\rr \ani, \underline Y\mapsto\R\llog_{X}^{\rm prechart}(\underline Y).$$
The functor $\R\llog_{X}^{\rm prechart}$ should be thought of as the moduli functor of precharted log stacks which are lying  over $X.$
\subsubsection{Deformation theory of stacks} 
In this subsection, we recall the basic theory of deformations of stacks, developed in \cite{lurie2004derived} and \cite{lurie2018spectral}.
Recall that for a prestack $X\in\Fun(\algcn,\ani)$, there is a notion of $\infty$-category of quasi-coherent complexes  defined in \cite[Section II.6.2]{lurie2018spectral}. We have seen  that the tangent bundle formalism  $\tang_{\algcn}\rightarrow\algcn$ is a presentable fibration. In particular, it's a coCartesian fibration,  and hence it determines a functor $$\Mod:\algcn\rr\cat,$$
which sends a connective $\Einf$-ring $A$ to the $\infty$-category of modules over $A$. Passing to right Kan extension, we get the  functor of the $\infty$-category of quasi-coherent complexes  
$$\qcoh:\Fun(\algcn,\ani)^{\rm op}\rr\cat.$$
Informally, one can refer to the $\infty$-category of quasi-coherent complexes  $\qcoh(X)$ as the limit $$\qcoh(X):={\varprojlim}_{(R,f)}\Mod_R,$$ 
where $(R,f)$ runs over all connective $\Einf$-rings lying over $X$, i.e., all maps $f:\spec R\rightarrow X$. What's more, there is  a spectral geometric variant \cite[Proposition 6.2.3.1]{lurie2018spectral} of Grothendieck's descent theorem for classical quasi-coherent sheaves  in algebraic geometry,
which asserts that the $\infty$-category of quasi-coherent complexes  is stable under flat sheafification and satisfies flat descent. 
\begin{defn}\cite[Definition 17.2.4.2]{lurie2018spectral}
Let $f:X\rightarrow S$ be a map of prestacks. We say that $f$ admits a cotangent complex if there exists an almost connective quasi-coherent complex $\LL_{X/S}\in\qcoh(X)$, satisfying the following property:\\
 For any connective $\Einf$-ring $R$, and any connective $R$-module $I,$ and any $x\in X(R)$, the fiber of the following natural map
    $$X(R\oplus I)\rr X(R)\times_{S(R)}S(R\oplus I)$$
    at the point $(x,\widetilde{f(x)}),$ where $\widetilde{f(x)}$ is the following map
    $$\spec(R\oplus I)\rightarrow\spec R\rr X\rr S,$$ which is denoted by $\mathfrak{D}_{X/S}(x,I)$, is functorially equivalent to the mapping space $\map_R(x^*\LL_{X/S},I).$
\end{defn}
\begin{rmk}
Informally, the anima $\mathfrak{D}_{X/S}(x,I)$ is equivalent to the anima consisting of the following liftings
$$\begin{tikzcd}
\spec R \arrow[d] \arrow[r]                   & X \arrow[d] \\
\spec(R\oplus I) \arrow[r] \arrow[ru, dashed] & S          
\end{tikzcd}$$
in which the left vertical  arrow is induced by the trivial square-zero extension, and the lower horizontal map is the map
$$\spec(R\oplus I)\rr\spec R\rr X\rr S.$$ From this point of view, one can see that the cotangent complex controls  the relative square-zero thickenings of relative prestacks. 
\end{rmk}
\begin{defn}\cite[Definition 17.3.1.5, 17.3.2.1 and 17.3.4.1]{lurie2018spectral}
Let $X$ be a prestack, we say that $X$ is:
\begin{enumerate}
    \item 
     \textit{Infinitesimally cohesive}, if for any pullback square  of connective $\Einf$-rings 
    \[\begin{tikzcd}
	{A'} & A \\
	{B'} & B
	\arrow[from=1-1, to=1-2]
	\arrow[from=1-1, to=2-1]
	\arrow[from=1-2, to=2-2]
	\arrow[from=2-1, to=2-2]
\end{tikzcd}\]
such that the maps $A\rightarrow B$ and $B'\rightarrow B$ are surjective on $\pi_0$ with  nilpotent kernels. the following diagram 
\[\begin{tikzcd}
	{X(A')} & {X(A)} \\
	{X(B')} & {X(B)}
	\arrow[from=1-1, to=1-2]
	\arrow[from=1-1, to=2-1]
	\arrow[from=1-2, to=2-2]
	\arrow[from=2-1, to=2-2]
\end{tikzcd}\]
is a pullback square.
\item  \textit{Nilcomplete}, if for any connective $\Einf$-ring $A$, we have an equivalence $X(A)\simeq\varprojlim_n X(\tau_{\leq n}A).$
\item 
\textit{Integrable}, if for any 
 complete connective $\Einf$-ring $A$, with respect to a finitely generated ideal $\mathfrak a\subseteq\pi_0(A)$, we have an equivalence $X(A)\simeq\varprojlim_n X(A/\mathfrak a^n)=\map(\spf A,X).$
\end{enumerate}
Let $f:X\rightarrow Y$ be a map of prestacks. We say that  $f$ is infinitesimally cohesive (resp. nilcomplete or integrable) if for any connective $\Einf$-ring $R$, and any map $\spec R\rightarrow Y,$ the pullback $X_R:=X\times_Y\spec R$ is infinitesimally cohesive (resp. nilcomplete or integral) as prestacks over $\algcn_R$.
\end{defn}

\begin{defn}\cite[Definition 3.4.3]{lurie2004derived}\cite[Definition 4.6]{Antieau_2014}
Let $f:X\rightarrow Y$ be a map of prestacks. We say that $f$ is
\begin{enumerate}
    \item 
    \textit{Formally smooth} if it is  {nilcomplete},  {infinitesimally cohesive}, and admits a cotangent complex $\LL_{X/Y}$, which is dual to a connective  perfect quasi-coherent complex. 
     \item 
     \textit{Smooth}, if it is  formally smooth and  locally of finite presentation.
     \item \textit{Formally \'etale}, if it is formally smooth and $\LL_{X/Y}\simeq 0.$
     \item \textit{\'Etale}, if it is smooth and formally \'etale.
     \end{enumerate}
\end{defn}

\subsubsection{Deformation theory of charted log stacks}
In this subsection, we will define the notion of charted log stacks and study the deformation property of the moduli functor $\R\llog_X^{\rm chart}$ of charted log stacks. We will prove that the functor $\R\llog_X^{\rm chart}$  admits a cotangent complex, and when restricted to charted log-affine schemes, it coincides with Gabber's cotangent complex.\\
In classical logarithmic geometry, a charted log scheme is a quadruple $(X,\mathcal{M},P,u)$, consisting of a log scheme $(X,\mathcal{M})$ (we assume that $X$ is quasi-compact and quasi-separated)   and a chart, i.e., a map of sheaves of monoids $u:\underline P\rightarrow\mathcal{M}$, where $\underline P$ is a constant sheaf induced from a monoid $P$.  We refer to \cite[Section 1.1]{beilinson2011crystalline} for more details on classical charted log schemes. We will generalize this to the $\Einf$-context.
\begin{defn}
    A precharted log stack $X$ is called a \textit{charted log stack} if, for any $\Einf$-ring $R$ and any map $f:\spec R\rightarrow\underline X$ from $\spec R$ to its underlying stacks, the pullback $f^*X$ is charted log-affine.
\end{defn}
In particular, if $X$ is a charted log stack whose underlying stack $\underline X\simeq\spec R$ is affine, then $X$ is charted log-affine, i.e., there is a unique $\Einf$-log ring $(R,L)$, such that $X\simeq\spec(R,L)$.
\begin{defn}
    Let $X$ be a charted log stack. We define the $\infty$-category $\llog\stk^{\rm chart}$ of charted log stacks  as the subcategory of $\llog\stk^{\rm prechart}$ spanned by charted log stacks, and define
    the
    moduli  $\R\llog_{X}^{\rm chart}$ of charted log stacks lying over $X$ as the subfunctor of $\R\llog^{\rm prechart}_X$ spanned by charted objects.
\end{defn}
\begin{rmk}
    One can see that the natural projection $\llog\stk^{\rm chart}\rightarrow\stk$ is equivalent to the Grothendieck construction associated with the right Kan extension of the functor $\algcn\rightarrow\cat$ which sends an $\Einf$-ring $A$ to the $\infty$-category of $\Einf$-log rings with underlying $\Einf$-rings equivalent to $A$.
\end{rmk}
\begin{lem}\label{chartaffcover}
    Let $X$ be a charted log stack, such that the underlying stack $\underline X$ is representable by a quasi-compact and quasi-separated   spectral Deligne-Mumford stack $(\mathcal X,\mathcal O_{\mathcal X})$. Then there is an $\Einf$-log ring $(R^0,L^0),$ and a strict \'etale  cover of charted log stacks $$\pi:\spec(R^0,L^0)\rr X.$$ As a consequence,  we can write $X$ as the geometric realization  of  a strict \'etale  of $X$ forms of charted log-affine schemes
$$\colim_{\Delta}\spec(R^\bullet,L^{\bullet})\stackrel{\simeq}\rr X.$$
\end{lem}
\begin{proof}
  Choose an \'etale cover $\pi:\spec A^0\rightarrow\underline{X}$. Then the pullback map $\pi^*X\rightarrow X$ forms a strict \'etale cover of $X$ from a charted log-affine scheme $\pi^*X$.
\end{proof}
\begin{thm}\label{gabbercotasfunctor}
Let $S$ be a charted log stack, whose underlying stack is representable by a spectral Deligne-Mumford stack $\MCS$, then
\begin{enumerate}
    \item The restricted functor $\R\llog_S^{\rm chart}:\Sptdm_{/\MCS}^{\rm op}\rightarrow\ani$  admits a cotangent complex $\LL_{\R\llog_S^{\mathrm{chart}}/\underline S}$, which is $(-1)$-connective.
\item If $\spec A\rightarrow \R\llog_{\spec(R,L)}^{\rm chart}$ classifies an  $\Einf$-log ring map $(R,L)\rightarrow(A,M)$ , then $\LL_{\spec A/\R\llog_{\spec(R,L)}^{\rm chart}}\simeq\LL^G_{(A,M)/(R,L)}.$
\end{enumerate}
\end{thm}
\begin{rmk}
    Here we identify  the functor $\R\llog_S^{\rm chart}$ as  a relative prestack lying  over $\underline S$ via the canonical  equivalence $\Fun(\Sptdm_{/\MCS
},\ani)\simeq\Fun(\Sptdm,\ani)_{/\underline S}$.
\end{rmk}
\begin{proof}[Proof of Theorem\autoref{gabbercotasfunctor}]\label{gabbercotasfunctor_proof}
    We first reduce to the case that $\underline S$ is equivalent to $\spec(R,L)$ for some $\Einf$-log ring $(R,L).$ Indeed, the Lemma\autoref{chartaffcover} shows that we have a strict \'etale cover $\spec(R^\bullet,L^\bullet)\rightarrow S$. We also have $\R\llog^{\rm chart}_{S}\times_S\spec R^\bullet\simeq \R\llog^{\rm chart}_{\spec(R^\bullet,L^\bullet)}$.
Thus we can assume that  the charted log stack $S$ is equivalent to a charted log-affine scheme $\spec(R,L)$.
Let $f:R\rightarrow A$ be an $R$-algebra, and let $I$ be a connective $A$-module. Let $x\in\R\llog^{\rm chart}_{\spec(R,L)}(A)$
 be a point which represents a charted log structure over $\spec A.$ The point $x$ represents a charted log-affine scheme $\spec(A,M)$. We will determine  the specific form of   the fiber $\mathfrak D_{\R\llog^{\rm chart}_{\spec(R,L)}/R}(x,I)$ of the following map 
 $$\R\llog^{\rm chart}_{\spec(R,L)}(A\oplus I)\rr\R\llog^{\rm chart}_{\spec(R,L)}(A)\times_{\spec R(A)}\spec R(A\oplus I)$$
 at $(x,\widetilde{f})$ lying over $f:R\rightarrow A$. 
Let $X$ be an object in $\mathfrak D_{\R\llog^{\rm chart
}_{\spec(R,L)}/R}(x,I).$ It gives rise to a strict morphism  of charted log stacks $$\spec(A,M)\rr X.$$
Using \cite[Lemma 17.1.3.7]{lurie2018spectral}, $X$ is charted log-affine.
Hence the anima $\mathfrak D_{\R\llog^{\rm chart}_{\spec(R,L)}/R}(x,I)$ is identified with the anima of deformations of $(A,M)$ to $A\oplus I$ over $(R,L)$. Applying Corollary\autoref{obstruction}, this space is equivalent to 
\begin{align*}
&{\rm fib}(\map(\LL^G_{(A,M)/(R,L)},I[1])\rr\map(\LL_{A/G},I[1]))\\
&\simeq\map(\LL^G_{(A,M)/(A,L^a)},I[1]).
\end{align*}
Let us write the $\Einf$-log rings $(A,M)$ and $(A,L^a)$ as colimits
$$\begin{tikzcd}
{\mathbb S[M]} \arrow[d] \arrow[r] & A \arrow[d] & {\mathbb S[L]} \arrow[d] \arrow[r] & A \arrow[d] \\
{(\mathbb S[M],M)} \arrow[r]       & {(A,M)}     & {(\mathbb S[L],L)} \arrow[r]       & {(A,L^a)}  
\end{tikzcd}$$

Using the left adjointness of Gabber's cotangent complexes, we obtain that  $\LL^G_{(A,M)/(A,L^a)}$ is naturally  equivalent to the  colimit of the following diagram
$$\begin{tikzcd}
{\LL^G_{(\mathbb S[L],L)}} \arrow[d] & {\LL_{\mathbb S[L]}} \arrow[d] \arrow[r] \arrow[l] & \LL_A \arrow[d] \\
{\LL^G_{(\mathbb S[M],M)}}           & {\LL_{\mathbb S[M]}} \arrow[r] \arrow[l]           & \LL_A          
\end{tikzcd}$$
where the colimit is formed in the $\infty$-category $\MTL$, which is equivalent to $$A\otimes_{\mathbb S[M]}{\rm cofib}(\LL_{\mathbb S[M]/\mathbb S[L]}\rr \LL^G_{(\mathbb S[M],M)/(\mathbb S[L],L)}).$$
Denote by $\mathscr D_{M/L}[1]$ the cofiber of the morphism $\LL_{\mathbb S[M]/\mathbb S[L]}\rr \LL^G_{(\mathbb S[M],M)/(\mathbb S[L],L)}$. we have an equivalence 
$$\mathfrak D_{\R\llog^{\rm chart}_{\spec(R,L)}/R}(x,I)\simeq\map_A(A\otimes_{\mathbb S[M]}\mathscr D_{M/L},I).$$
This implies that the object $A\otimes_{\mathbb S[M]}\mathscr D_{M/L}$ corepresents the functor 
$$I\mapsto {\rm fib}_x(\R\llog^{\rm chart}_{\spec(R,L)}(A\oplus I)\rr\R\llog^{\rm chart}_{\spec(R,L)}(A)).$$
Note that the collection of all maps $x:\spec A\rightarrow \R\llog^{\rm chart}_{\spec(R,L)}$  forms  a  cover  $$\coprod_{(A,x)} \spec A\rr \R\llog^{\rm chart}_{\spec(R,L)}$$ of $\R\llog^{\rm chart}_{\spec(R,L)}$. In particular, we have an equivalence of presentable stable $\infty$-categories 
$$\qcoh(\R\llog^{\rm chart}_{\spec(R,L)})\simeq\varprojlim_{(A,x)}\Mod_A.$$
Thus, we  need to show that those objects $\{A\otimes_{\mathbb S[M]}\mathscr D_{M/L}\}$ form a projective system, i.e., for any ring map $A\rightarrow B$ lying over $\R\llog^{\rm chart}_{\spec(R,L)}$, the following canonical map $$A\otimes_{\mathbb S[M]}\mathscr D_{M/L}\otimes_AB\stackrel{\simeq}\rr B\otimes_{\mathbb S[M]}\mathscr D_{M^a/L}$$
is an equivalence.
This holds because the left hand term is equivalent to $B\otimes_{\mathbb S[M]}\mathscr D_{M/L}$, which could be described as the $(-1)$-shift of the  colimit of the following  diagram
$$\begin{tikzcd}
{\LL^G_{(\mathbb S[L],L)}} \arrow[d] & {\LL_{\mathbb S[L]}} \arrow[d] \arrow[r] \arrow[l] & \LL_B \arrow[d] \\
{\LL^G_{(\mathbb S[M],M)}}           & {\LL_{\mathbb S[M]}} \arrow[r] \arrow[l]           & \LL_B          
\end{tikzcd}$$
which is equivalent to the $(-1)$-shift of the cofiber of the  map
$\LL^G_{(B,L^a)}\rr\LL^G_{(B,M^a)}$, 
which is equivalent to $B\otimes_{\mathbb S[M]}\mathscr D_{M^a/L}$ by the construction of $\mathscr D_{M^a/L}.$ The $(-1)$-connectivity of the cotangent complex of ${\R\llog_{\spec(R,L)}}$ is immediately deduced from the conormal exact sequence of stacks.
\end{proof}

\subsection{Spectral log stacks}\label{spectral log stacks}
We have defined the charted log stacks in the previous section. However, there is an issue that the charted log structures do not satisfy any descent property, which can be viewed as  a similar phenomenon in classical logarithmic geometry: one can only define  charts on a log scheme locally, but there is no way to glue them together to obtain a global chart in general. We have the following example, which shows that the functor $\R\llog^{\rm chart}_{S}$ does not satisfy \'etale descent.
\begin{exmp}\label{784512547845123659}
    Let $(A,M)$ be an $\Einf$-log ring. We denote by $(A\times A,M^a)$ the log ring associated to the diagonal map $M\rightarrow A\times A.$ Then we have a map of $\Einf$-log rings
    $(A\times A,M^a)\rightarrow (A\times A, M\times M).$ This defines a map of charted log-affine schemes $\pi:\spec(A\times A,M\times M)\rr \spec(A\times A,M^a),$ which is not an equivalence by the definition of charted log stacks. 
    On the other hand, let us consider  the natural embeddings $i_1,i_2:\spec A\rightarrow \spec (A\times A)$. The pullback of $\pi$ along  $i_\alpha,\alpha=1,2$ gives rise to a homotopy equivalence  $i_\alpha^*\pi\simeq id:\spec(A,M)\rightarrow\spec(A,M).$ But $\pi$ is not an equivalence. This shows that the functor $\R\llog_X^{\rm chart}$ doesn't preserve finite products.
\end{exmp}
The reason that the descent of the moduli of charted log stacks fails is that the "chart" is not allowed to glue in the following sense: consider a classical charted log scheme $(X,\mathscr M,P)$, which is equivalent to giving a map $X\rightarrow\spec\Z[P]$, such that $\mathscr M$ is isomorphic to the pullback of the canonical log structure on  $\spec \Z[P]$. However, if we have a class of strict \'etale maps of  log schemes 
$\{(U_i,\mathscr N_i)\rightarrow(X,\mathscr M)\}^n_{i=1}$ which we equip each $(U_i,\mathscr N_i)$ with a chart $P_i$, then we do not always obtain a charted map $(U_i,\mathscr N_i,P_i)\rightarrow(X,\mathscr M,P)$, because the following commutative diagram 
$$\begin{tikzcd}
U_i \arrow[r] \arrow[d]  & X \arrow[d]  \\
{\spec \Z[P_i]} \arrow[r] & {\spec\Z[P]}
\end{tikzcd}$$
will not exist.
\begin{defn}\label{logaffineschemedef}
    Let $(A,M)$ be an $\Einf$-prelog ring. The \textit{spectral log-affine scheme}  $\spet(A,M)$ associated with $(A,M)$ is a pair $(\spet A,\mathscr M),$ in which:
    \begin{enumerate}
        \item $\spet A$ is the \'etale spectrum of $A$ defined in \cite[Definition 1.4.2.5]{lurie2018spectral};
        \item $\mathscr M: \algcn_A\rightarrow\mon$ is the \'etale sheafification of the  functor
        $$\mathscr M_0: \alg^{\rm cn}_A\ni A'\mapsto M^a,$$
        where $M^a$ is the $\Einf
        $-monoid of the logification of the $\Einf$-prelog ring given by $\mathbb S[M]\rightarrow A\rightarrow A'$. 
    \end{enumerate}
\end{defn}
\begin{rmk}\label{afflog=logificationofconastantsheaf}
   By  the definition of the sheaf $\mathscr M$, it is the sheafification of the functor  which sends $A'\in A_\et$ to $\colim(M\leftarrow M\times_{\Omega^\infty A}GL_1(A)\rightarrow GL_1(A))$. Equivalently, $\mathscr M$ is equivalent to the colimit of the following diagram $$\underline M\longleftarrow \underline M\times_{\Omega^\infty\OO_A}GL_1(\OO_A)\rr GL_1(\OO_A),$$
   where the colimit is formed in $\shv_{\et}(\alg^{\rm cn,op},\mon)$. In other words, the sheaf $\mathscr M$ is equivalent to the sheafification of the point-wise logification of the constant prelog structure $\underline M\rightarrow\OO_A$.
\end{rmk}
\begin{cau}
    We have to point out that the functor $\algcn_A\rightarrow\mon,A'\mapsto M^a$ is not a sheaf, as it does not preserve finite products. Although it has descent along an fpqc hypercover $A\rightarrow A^\bullet$ by Proposition\autoref{strictdescent}.
\end{cau}
\begin{defn}
    A \textit{spectral prelog Deligne-Mumford stack} is a triple $(\mathcal X,\mathscr M,\alpha),$ where $\mathcal X$ is a quasi-compact and  quasi-separated  spectral Deligne-Mumford stack\footnote{See \cite[Definition 1.4.4.2]{lurie2018spectral}}, $\mathscr M$ is an \'etale sheaf of $\Einf$-monoids over the underlying $\infty$-topos of $\mathcal X$, and $\alpha:\mathscr M\rightarrow\Omega^\infty\mathcal O_\mathcal X$ is a morphism of $\Einf$-monoids.\\
     We denote $\plog\Sptdm$ as the $\infty$-category of spectral log Deligne-Mumford stacks.
\end{defn}
\begin{defn}
    A \textit{spectral log Deligne-Mumford stack} $(\mathcal X,\mathscr M,\alpha)$ is a spectral prelog Deligne-Mumford stack, such that the induced morphism $\mathscr M\times_{\Omega^\infty\OO_\X}GL_1(\OO_\X)\rightarrow GL_1(\OO_\X)$ is an equivalence.\\
    We denote $\llog\Sptdm$ as the $\infty$-category of spectral log Deligne-Mumford stacks.
\end{defn}
\begin{cau}
    The functor $\llog^{\rm op}\rightarrow\llog\Sptdm,(A,M)\mapsto\spet(A,M)$ is not fully faithful.
\end{cau}
Let us give a concrete definition of the $\infty$-categories $\plog\Sptdm$ and $\llog\Sptdm$ as follows. We denote by $\mone(\X):=\shv(\X,\mone)$ the $\infty$-category of  sheaves of $\Einf$-monoids over the underlying $\infty$-topos of $\X.$ The assignment $\X\mapsto\shv(\X,\mone)$ defines  a functor $\mon(-):\Sptdm^{\rm op}\rightarrow\cat$. We can define a functor $\widehat{\R\plog}:\Sptdm^{\rm op}\rightarrow\cat$ as the limit of the following diagram of functors
$$\begin{tikzcd}
* \arrow[rd, "\Omega^\infty\mathcal O_{(-)}"'] &                 & {\Fun(\Delta^1,\mone)(-)} \arrow[ld, "{\rm ev}_1"] \\  & \mone(-) &     
\end{tikzcd}$$
Using the Grothendieck construction, we obtain a Cartesian fibration $$\pi:\plog\Sptdm\rightarrow\Sptdm.$$
We refer to $\plog\Sptdm$ as the $\infty$-category of spectral prelog Deligne-Mumford stacks, and 
we refer to $\llog\Sptdm$ as the subcategory of $\plog\Sptdm$ spanned by log objects, which is called the $\infty$-category of spectral log Deligne-Mumford stacks.
\begin{rmk}\label{44556644}
    The $\infty$-category $\llog\Sptdm$ has finite limits, and it's easy to see that $$\spet(A,M)\times_{\spet(R,L)}\spet(B,N)\simeq\spet((A,M)\otimes_{(R,L)}(B,N))$$ by  definition.
\end{rmk}
\begin{rmk}[Derived log stacks]
Similar to the $\Einf$-context, one can also construct a Cartesian fibration $\llog\ddm\rightarrow\ddm$ over the $\infty$-category of derived Deligne-Mumford stacks in the same way, where $\llog\ddm$ is the $\infty$-category of  derived log  Deligne-Mumford stacks. A  derived log Deligne-Mumford stack is a triple $(\X,\mathscr M,\alpha)$, where $\X$ is a derived Deligne-Mumford stack,  $\mathscr M$ is a sheaf of animated monoids on the underlying $\infty$-topos of $\X$, and $\alpha:\mathscr M\rightarrow \Omega^\infty\OO_\X$ is a map of sheaves of  animated monoids, such that the induced map $\mathscr M\times_{\Omega^\infty\OO_X}GL_1(\OO_\X)\rightarrow GL_1(\OO_\X)$ is an equivalence.\\
We should point out that all constructions, definitions, and results in this paper for spectral log Deligne-Mumfords also hold for derived log Deligne-Mumford stacks. Especially, the deformations and  representability properties hold for  the moduli stack of derived log structures. For this reason, we will mainly work on the $\Einf$-context, and only collect the main results for derived log Deligne-Mumford stacks in \autoref{dersetting}.
\end{rmk}
\subsubsection{Functoriality}
We recall some basic functoriality of (pre)-log structures developed in \cite[Section 3.4]{sagave2016derived} \cite[Section 2]{binda2023hochschild} and \cite[Section II.1.2]{ogus2018lectures}. We have the following operations on prelog structures:
\begin{enumerate}
\item(Pullback) Since the functor $\pi:\plog\Sptdm\rightarrow\Sptdm$ is a Cartesian fibration, we have pullback functors for prelog structures. We can describe it as follows: let $(\Y, \mathscr M, \alpha)$ be a spectral prelog Deligne-Mumford stack, and let $f: \X \rightarrow \Y$ be a morphism of spectral  Deligne-Mumford stacks. 
The spectral  prelog Deligne-Mumford stack with a structure map 
$f^{-1}\mathscr M \xrightarrow{\alpha}  \Omega^\infty \mathcal O_\X$
is a Cartesian edge lying over $f$.
\item(Pushout) The functor $\pi:\plog\Sptdm\rightarrow\Sptdm$ is also a coCartesian fibration. We construct coCartesian edges as follows: let $(\X, \mathscr N, \beta)$ be a spectral prelog Deligne-Mumford stack, and let $f \colon \X \to \Y$ be a morphism of Deligne-Mumford stacks. 
The following fiber product of sheaves of  $\Einf$-monoids 
\[\begin{tikzcd}
f_*\mathscr N \arrow[r]\arrow[d] & \Omega^\infty\mathcal O_\Y \arrow[d, "f"]\\
f^{\shv}_*\mathscr N \arrow[r, "\beta"] &f^{\shv}_*\Omega^\infty\mathcal O_\X.
\end{tikzcd}
\]
gives a coCartesian edge lying over $f$,
where $f_*^{\shv}$ is the sheaf-theoretic direct image functor.
\end{enumerate}

\begin{rmk}\label{asdfg}
    It's easy to see that the pushout functor $f_*$ preserves log objects.
\end{rmk}
Thus we have the following fact:
\begin{prop}\cite[Proposition 3.26]{sagave2016derived}
  The restricted functor $\pi':\llog\Sptdm\rightarrow\Sptdm$ of $\pi$ is a presentable fibration.
\end{prop}
\begin{proof}
    Remark\autoref{asdfg} implies that $\pi'$ is a coCartesian fibration. Observe that the pushout functor preserves  limits of prelog structures, and log structures are stable under limits. Therefore, we just need to show that $\pi'$ has presentable fibers. The same arguments as Proposition\autoref{presentabilityoflog} give the desired result.
\end{proof}

\begin{defn}\label{qcohdefinition}
Let $(\X,\mathscr M,\alpha)$ be a spectral log Deligne-Mumford stack, we say that $(\X,\mathscr M,\alpha)$ is:
\begin{enumerate}
    \item \textit{Quasi-coherent}, if there is an \'etale cover  $U\rightarrow \X,$ such that $U\simeq\coprod_j\spet A_j$ is a finite disjoint union of spectral affine schemes, and the restriction $(U,\mathscr M|_U)\simeq\coprod_j\spet(A_j,M_j)$ is a finite disjoint union of spectral log-affine schemes;
    \item \textit{Coherent}, if it is quasi-coherent, and there is a choice of $\{(A_j,M_j)\}$, which appears in assertion (1), such that for any $j$, $(A_j,M_j)$ is obtained from an $\Einf$-prelog ring $(A_j,M_j')$ satisfying the property that $\mathbb S[M_j']$ is almost finitely presented over $\mathbb S$.
    \item \textit{Fine}, if it's coherent, and there is a choice of $\{(A_j,M_j)\}$, which appears in assertion (2), such that for any $j$, the commutative ring $\pi_0\mathbb S[M_j]=\Z[\pi_0M_j]$ is a finitely generated  integral $\Z$-algebra.
\end{enumerate}
\end{defn}
\begin{defn}
    Let $f:(\X,\mathscr M)\rightarrow(\Y,\mathscr{N})$ be a morphism of spectral log Deligne-Mumford stacks. We say that $f$ is strict if it is a Cartesian edge with respect to $\underline{f}:\X\rightarrow\Y$. Or equivalently, the log structure $\mathscr M$ is equivalent to the pullback of $\mathscr N$ along $\underline f$.
\end{defn}
\begin{rmk}
    The log structure on $\spet (A,M)$ is equivalent to the pullback of the log structure on $\spet(\mathbb S[M],M)$ along $\spet A\rightarrow\spet\mathbb S[M]$. 
\end{rmk}
\begin{rmk}[Local charts for strict maps]\label{676756}
Let $f:\spet(A,M)\rightarrow\spet(R,L)$ be a strict map. We denote by $\mathscr L$ and $\mathscr M\simeq f^*\mathscr L$ the log structures of $\spet(R,L)$ and $\spet(A,M)$ respectively. Then  there is an $\Einf$-log ring $(A,M')$ such that $\spet(A,M)\simeq\spet(A,M')$, and the map $f$ is induced from an $\Einf$-log ring map $(R,L)\rightarrow (A,M')$. Indeed, we can form the following commutative diagram 
$$\begin{tikzcd}
{\spet(A,M)} \arrow[rd] \arrow[rr, "f"] &                         & {\spet(R,L)} \arrow[ld] \\& {\spet(\mathbb S[L],L)} &  
\end{tikzcd}   
$$
The composition $\spet(A,M)\rightarrow \spet(\mathbb S[L],L)$ is strict, as the strictness is stable under compositions. It turns out that $\mathscr M$ is equivalent to the pullback of the log structure of $\spet(\mathbb S[L],L)$.  It  is easy to see that $\spet(A,M)\simeq\spet(A,L)^a$
$f$ is induced from $(R,L)\rightarrow(A,L)^a$. We have to point out that there is no reason wonder that there is a map $(R,L)\rightarrow (A,M)$ induces $f$ in general. 
\end{rmk}
\begin{cau}[Non-existence of local charts for general maps]
    Let us work in classical logarithmic geometry. Recall that a map of quasi-coherent log schemes $f:(X,\mathscr M)\rightarrow (S,\mathscr L)$ \textit{locally admits a chart}, if after an \'etale localization on both source and target, there is a commutative diagram 
    $$\begin{tikzcd}
{(X,\mathscr M)} \arrow[r] \arrow[d] & {(Y,\mathscr N)} \arrow[d] \\
{\spet(\Z[M],M)} \arrow[r]           & {\spet(\Z[L],L)}          
\end{tikzcd}$$
such that vertical maps are strict. The existence of local charts will not hold in general. But under some finiteness assumption on the target, one can show that such maps locally admit charts, see \cite[Section 1.1]{beilinson2011crystalline}.
\end{cau}
At the end of this subsection, we consider the topology used to define log structures. Recall that a log structure on a spectral Deligne-Mumford stack $\X$ is a sheaf of $\Einf$-monoids  defined on the underlying $\infty$-topos of $\X$, namely, the big \'etale $\infty$-topos $\X_{\rm \Acute{E}t}$. In the vast majority of cases in geometry and arithmetic, when logarithmic geometry appears, this definition is  perfectly applicable. For example, if we consider log structures  coming from  strict normal crossing divisors on algebraic varieties over fields or semi-stable proper schemes defined over DVRs. One can show that this kind of log structure also has fpqc (hyper)descent in the following sense: to give such a kind of log structure $\mathscr M$ on a scheme $X$, it is equivalent to give a map 
$X\rightarrow \mathbb A^1/\mathbb G_m\simeq{\rm Div}_\Z$.  More generally, it seems that log structures on (spectral) schemes given by \textit{Deligne-Faltings structures} are also suitable for work on the fpqc topology instead of the \'etale topology, see \cite[Section 2]{Talpo_2018} for details. We have the following variants of the definition of log structures:
\begin{enumerate}
    \item In the definition of $\spet(A,M)$, we define  the log structure  $\mathscr M$ as the (hyper)sheafification  with respect to the fpqc topology instead of the \'etale topology.
    \item Let $\Sptdm_{\rm fpqc}$ be the subcategory of $\Sptdm$ spanned by objects $\X$ having fpqc (hyper)descent\footnote{This means the associated Yoneda functor $h_\X:\algcn\rightarrow\ani$ has fpqc (hyper)descent.}. The $\infty$-category $\plog\Sptdm_{\rm fpqc}$ is the Grothendieck construction associated with  the  functor 
    $\Sptdm^{\rm op}_{\rm fpqc}\rightarrow\cat$ given as 
$$\lim(*\stackrel{\Omega^\infty\OO_{(-)}}\rr\mon^{(\wedge)}(-)\longleftarrow\Fun(\Delta^1,\mon)^{(\wedge)}(-)).$$
The $\infty$-category $\llog\Sptdm_{\rm fpqc}$ is the subcategory of $\plog\Sptdm_{\rm fpqc}$ spanned by log objects.
\item In the definitions of quasi-coherent, coherent, and fine log structures, we also use fpqc covers instead of \'etale covers.
\end{enumerate}
\begin{rmk}
   By the same procedure, one can also construct the $\infty$-category of spectral log Delinge-Mumford stacks using other topologies:  Zariski, Nisnevich, fppf, pro-\'etale...... Each topology has its own advantages. In particular, as we will work on infinite root  stacks in \autoref{section7}, which is defined  as an inverse limit of a sequence of certain functors. It seems that using fpqc topology or pro-\'etale topology is more convenient than others, as a sequential  limit  of fpqc (resp. pro-\'etale) covers  is also an fpqc (resp. pro-\'etale) cover.
\end{rmk}
\begin{cov}
    For the sake of simplicity, we omit set-theoretic issues. Under mild assumptions on set theory, we always assume that any presheaf has an fpqc (resp. pro-\'etale) sheafification, and that the $\infty$-category $\stk_{\rm fpqc}$ (resp. $\stk_{\rm pro\et}$) of fpqc  (resp. pro-\'etale) stacks is presentable.
\end{cov}
\subsection{The moduli stack of spectral log stacks}\label{subsection43}
In this subsection, we provide a construction of the moduli stack of log structures. Then we study certain properties of this functor, including the descent property and deformation properties in the sense of \cite[Part VIII]{lurie2018spectral}.
Let $(\MCS,\MSL)$ be a spectral log Deligne-Mumford stack. The Cartesian fibration $\pi':\llog\Sptdm\rightarrow\Sptdm$ induces a  Cartesian fibration $$\pi'_{(\MCS,\MSL)}:\llog\Sptdm_{/(\MCS,\MSL)}\rr\Sptdm_{/\MCS}.$$ This gives rise to a moduli functor of spectral log structures:
\begin{defn}\label{rlogdefn}
    Let $(\MCS,\MSL)$ be a spectral log Deligne-Mumford stack. We define the following functor 
    $$\R\llog_{(\MCS,\MSL)}:\Sptdm^{\rm op}_{/\MCS}\rr\ani,$$
    which sends a quasi-compact and quasi-separated spectral Deligne-Mumford stack $\X$ to the anima consisting  of maps of spectral log Deligne-Mumford stacks $(\X,\mathscr M)\rightarrow(\MCS,\MSL).$  More precisely, $\R\llog_{(\MCS,\MSL)}$ is the functor associated with $\pi'_{(\MCS,\MSL)}$ via the Grothendieck construction.\\
    We denote by $\R\llog_{(\MCS,\MSL)}^{\rm qcoh}$, $\R\llog_{(\MCS,\MSL)}^{\rm coh}$ and $\R\llog_{(\MCS,\MSL)}^{\rm fin}$ the subfunctors of $\R\llog_{(\MCS,\MSL)}$ spanned  by quasi-coherent, coherent, or fine log structures, respectively.
\end{defn}
\subsubsection{Descent property of log structures}
In this subsection, we will prove the descendability of the moduli of  log structures.
\begin{cov}
    In this section, we should fix a topology for defining our spectral log Deligne-Mumford stacks.  By the proof of Theorem\autoref{des}, any choice of one of the following topologies:
\begin{center}
    Zariski, \'etale, fppf, fpqc, pro-\'etale......
\end{center}
and any choice of one of the following descendability:
\begin{center}
    descent, hyperdescent
\end{center}
for using to define spectral log Deligne-Mumford stacks will give the parallel result of the descent property. More precisely, if we use the topology  $\tau$, and (hyper-)descendability to define our log Deligne-Mumford stacks, the resulting moduli functor has (hyper-)descent with respect to $\tau$.  For simplicity, in this section,  we will employ \textbf{\'etale topology} and \textbf{sheaves} to define our log stacks.
\end{cov}

\begin{thm}\label{des}
 Let $\mu\in\{\emptyset,{\rm qcoh, coh, fin }\}$, and let $(\MCS,\MSL)$ be a $\mu$-spectral Deligne-Mumford stack. Then
    the following functors $$\R\llog_{(\MCS,\MSL)},\R\llog^{\rm qcoh
}_{(\MCS,\MSL)},{\ }\R\llog^{\rm coh
}_{(\MCS,\MSL)},{\ }\R\llog^{\rm fin
}_{(\MCS,\MSL)}$$ are \'etale sheaves.
\end{thm}
As a corollary, we obtain the following result that will contribute to the local study of log stacks.
\begin{cor}\label{afflogstk}
   Let $\mu\in\{\rm qcoh, coh, fin \}$. Let $(R,L)$ be an $\Einf$-log ring, and let 
   $$\R\llog_{\spet(R,L)}^{\rm aff,\mu}:\algcn_R\rr\ani$$ be the subfunctor of $\R\llog^\mu_{\spet(R,L)}$ consisting of all log-affine structures. The inclusion  $\R\llog_{\spet(R,L)}^{\rm aff,\mu}\subset \R\llog^{\mu}_{\spet(R,L)}$ exhibits the latter  as the sheafification of $\R\llog_{\spet(R,L)}^{\rm aff,\mu}$.
\end{cor}

\begin{proof}
    Let $\mathcal F:\algcn_R\rightarrow\ani$  be a sheaf. We aim to prove that the following natural  map
    $$i^*:\map_{\shv}(\R\llog_{\spet(R,L)}^{\mu},\mathcal F)\rr\map_{\rm pre\shv}(\R\llog_{\spet(R,L)}^{\rm aff,\mu},\mathcal F)$$
    induced by the inclusion $i:\R\llog_{\spet(R,L)}^{\rm aff,\mu}\subset \R\llog^{\mu}_{\spet(R,L)}$ is an equivalence.
    We prove that the map $i^*$ is essentially surjective and $(-1)$ -truncated,  i.e., it has contractible or empty homotopy fibers.\\
   Obviously, the map $i^*$ is  $(-1)$-truncated, because the functor $i$ is point-wise $(-1)$-truncated. We now prove that $i^*$ is essentially surjective. 
    Let $f:\R\llog_{\spet(R,L)}^{\rm aff,\mu}\rightarrow\mathcal F$ be a morphism. We have to construct an extension  $\tilde{f}:\R\llog_{\spet(R,L)}^{\mu}\rightarrow\mathcal F$, making the following diagram  
\[\begin{tikzcd}
	{\R\llog_{\spet(R,L)}^{\rm aff,\mu}} & {\mathcal F} \\
	{\R\llog_{\spet(R,L)}^{\mu}}
	\arrow["f", from=1-1, to=1-2]
	\arrow[from=1-1, to=2-1]
	\arrow["{\tilde{f}}"', dashed, from=2-1, to=1-2]
\end{tikzcd}\]
commutative. Let $A\in\algcn_R$ be a connective $R$-algebra, and  let $\mathscr M$ be a $\mu$-log structure on $X=\spet A$. Then  we know that there is a faithfully flat \'etale $A$-algebra $A'$, such that the log structure $\mathscr M':=\mathscr M|_{A'}$ 
on $\spet A'$ obtained from pulling back the log structure $\mathscr M$ along the covering map $\spet A'\rightarrow\spet A$  is log-affine. By Theorem\autoref{des}, one can write $(\spet A,\mathscr M)$ as a colimit $$(\spet A,\mathscr M)\stackrel{\simeq}\longleftarrow{\colim}_{\Delta^{\rm op}}(\spet A'^{\otimes_R(\bullet+1)},\mathscr M|_{A'^{\otimes_R(\bullet+1)}}).$$
Then we define the functor $\tilde{f}$ as  the  left Kan extension of $f$ along the inclusion $\R\llog_{\spet(R,L)}^{\rm aff,\mu}\subset \R\llog^{\mu}_{\spet(R,L)}.$ 
Therefore, we conclude that the map $i^*$ is essentially surjective. 
\end{proof}
\begin{rmk}\label{covtoricscheme}
    The Corollary\autoref{afflogstk} implies that $\R\llog_{\spet(R,L)}^{\mu}$ admits a surjective map from the union of all spectral affine schemes having the form of $\spec R[M]$, where $M$ runs over all $\Einf$-monoids with the desired property.
\end{rmk}
Now, we are ready to prove the descent theorem for moduli of log structures. In fact, it is enough to show the statement of Theorem \autoref{des} for the stack of all  log structures, in light of the following Lemma.
\begin{lem}\label{redqcoh}
If $\R\llog_{(\MCS,\MSL)}$ is an \'etale sheaf, then so are $\R\llog^{\rm qcoh}_{(\MCS,\MSL)}, \R\llog^{\rm coh}_{(\MCS,\MSL)}$ and $\R\llog^{\rm fin}_{(\MCS,\MSL)}$.
\end{lem}
\begin{proof}
Assume that  $\R\llog_{(\MCS,\MSL)}$ satisfies \'etale descent. We will show that the   functors $\R\llog^{\rm qcoh}_{(\MCS,\MSL)},\R\llog^{\rm coh}_{(\MCS,\MSL)}$ and $\R\llog^{\rm fin}_{(\MCS,\MSL)}$ also satisfy \'etale descent.
Let $f:\X\rightarrow\Y$ be an \'etale cover of  spectral Deligne-Mumford stacks. Let $\mu\in\{\rm qcoh, coh,fin\}$. Consider the following commutative diagrams
  $$\begin{tikzcd}
{\R\llog_{(\MCS,\MSL)}^{\mu}(\Y)} \arrow[d] \arrow[r, "f^*"] & {\R\llog_{(\MCS,\MSL)}^\mu(\X)} \arrow[d] &  &  &    \\
{\R\llog_{(\MCS,\MSL)}(\Y)} \arrow[r, "f^*"]                 & {\R\llog_{(\MCS,\MSL)}(\X)}               &  &  & {}
\end{tikzcd}$$
The fiber product
$\R\llog_{(\MCS,\MSL)}^{\mu}(\X)\times_{\R\llog_{(\MCS,\MSL)}(\X)}\R\llog_{(\MCS,\MSL)}(\Y)$
is the anima consisting of triples $(\mathscr M,\mathscr N,\rho)$, where $\mathscr M\in\R\llog_{(\MCS,\MSL)}^{\mu}(\X)$ is a $\mu$-log structure over $\X$, $\mathscr N\in\R\llog_{(\MCS,\MSL)}(\Y)$ is a  log structure over $Y$, and $\rho: f^*\mathscr N\stackrel{\simeq}\rightarrow\mathscr M$ is an equivalence of log structures. This implies that $\mathscr N\in\R\llog_{(\MCS,\MSL)}(\Y)$ is actually  a $\mu$-log structure. since a log structure has the  property $\mu$, it is, by definition, a local question. We have that  the canonical  map 
$$\R\llog_{(\MCS,\MSL)}^\mu(\Y)\rr \R\llog_{(\MCS,\MSL)}^{\mu}(\X)\times_{\R\llog_{(\MCS,\MSL)}(\X)}\R\llog_{(\MCS,\MSL)}(\Y)$$
has a cosection $\pi$, given by the projection onto the second variable. It's obvious that $\pi$ has contractible fibers.  We deduce that the projection $\pi$ is an equivalence.  Let $\Y_\bullet$ be an arbitrary {\v C}ech cover  of $\Y$; we obtain an equivalence of cosimplicial animae $$\R\llog_{(\MCS,\MSL)}^\mu(\Y)\simeq\R\llog_{(\MCS,\MSL)}^\mu(\Y_\bullet)\times_{\R\llog_{(\MCS,\MSL)}(\Y_\bullet)}\R\llog_{(\MCS,\MSL)}(\Y),$$
Passing to the limit, we acquire the following equivalences
\begin{align*}
\R\llog_{(\MCS,\MSL)}^\mu(\Y)&\stackrel{\simeq}\rr{\lim}_{\Delta}\R\llog_{(\MCS,\MSL)}^\mu(\Y_\bullet)\times_{{\lim}_{\Delta}\R\llog_{(\MCS,\MSL)}(\Y_\bullet)}\R\llog_{(\MCS,\MSL)}(\Y)\\
&\stackrel{\simeq}\rr {\lim}_{\Delta}\R\llog_{(\MCS,\MSL)}^\mu(\Y_\bullet)\times_{\R\llog_{(\MCS,\MSL)}(\Y)}\R\llog_{(\MCS,\MSL)}(\Y)\\
&\stackrel{\simeq}\rr{\lim}_{\Delta} \R\llog_{(\MCS,\MSL)}^{\mu}(\Y_\bullet).
 \end{align*}
\end{proof}
For this reason, we will focus on the moduli functor of arbitrary  log structures.
Let $(\MCS,\MSL)$ be a  spectral log Deligne-Mumford stack, we want to define two new functors
$$\widehat{\R\llog_{(\MCS,\MSL)}},\widehat{\R\plog_{(\MCS,\MSL)}}:\Sptdm^{\rm op}_{/\MCS}\rr\cat,$$
which sends a relative spectral Deligne-Mumford stack $\X/\MCS$ to the $\infty$-category of maps to spectral log (or prelog) Deligne-Mumford stacks $(\X,\mathscr M)\rightarrow(\MCS,\MSL).$ We thus have enough to show the descendability of $\widehat{\R\llog_{(\MCS,\MSL)}}$ and $\widehat{\R\plog_{(\MCS,\MSL)}}$. \cite[Corollary 2.3.8]{Ayoub_2022} shows that the functor $$\mone(-):\Sptdm_{/\mathcal S}^{\rm op}\rr \cat$$ 
satisfies \'etale descent.
We let the functor $\widehat{\R\plog_{(\MCS,\MSL)}}$ be the limit of the diagram of functors
$$\begin{tikzcd}
         & \Fun(\Delta^2,\mone)(-) \arrow[ld, "{\rm ev}_2"'] \arrow[rd, "{\rm ev}_0"] &          \\
\mone(-) & * \arrow[l, "\Omega^\infty\mathcal O(-)"'] \arrow[r, "f\mapsto f^*\mathscr L"]      & \mone(-)
\end{tikzcd}$$
and let $\widehat{\R\llog{(\MCS,\MSL)}}$ be the subfunctor of $\widehat{\R\plog_{(\MCS,\MSL)}}$ spanned by log objects. We immediately obtain the following fact:
\begin{lem}
    The functor $\widehat{\R\plog_{(\MCS,\MSL)}}$ satisfies \'etale descent.
\end{lem}
Now we can prove Theorem\autoref{des}. 
\begin{proof}[Proof of Theorem\autoref{des}]
 Let $f:\X\rightarrow \Y$ be an \'etale cover of a spectral Deligne-Mumford stack.    We will show that the following diagram 
 $$\begin{tikzcd}
{\widehat{\R\llog_{(\MCS,\MSL)}}(\Y)} \arrow[r, "f^*"] \arrow[d] & {\widehat{\R\llog_{(\MCS,\MSL)}}(\X)} \arrow[d] \\
{\widehat{\R\plog_{(\MCS,\MSL)}}(\Y)} \arrow[r, "f^*"]           & {\widehat{\R\plog_{(\MCS,\MSL)}}(\X)}          
\end{tikzcd}$$
is a pullback square. This is clear because a prelog structure $\mathscr M$ over $\Y$ is a log structure if and only if its restriction to $\X$ is a log structure. Then applying the same argument in Lemma\autoref{redqcoh}, we obtain that the functor 
$\widehat{\R\llog_{(\MCS,\MSL)}}$ satisfies \'etale descent, then so does $\R\llog_{(\MCS,\MSL)}$, as $\R\llog_{(\MCS,\MSL)}$ is the core of $\widehat{\R\llog_{(\MCS,\MSL)}}$. Applying Lemma\autoref{redqcoh}, we know that $\R\llog^{\mu}_{(\MCS,\MSL)},\mu\in\{\rm qcoh,coh,fin\}$  also satisfies \'etale hyperdescent.
\end{proof}
\begin{rmk}[Log structures in fpqc topology]\label{fpqcrlog}
 We can define another variation of the functor $\R\llog_{(\MCS,\MSL)}$ using fpqc topology. In this case, we can regard as a functor 
 $\R\llog^{\mu}_{(\MCS,\MSL)}:\Sptdm_{\rm fpqc}\rightarrow\ani$. The same method of the proof Theorem\autoref{des} shows this is an fpqc sheaf.
 \end{rmk}
\subsubsection{Comparison with the charted log structures and logarithmic cotangent complexes}
Let $S$ be a charted log stack whose underlying stack $\underline S$ is represented by a spectral Deligne-Mumford stack $\MCS.$ Denote by $\llog\Sptdm^{\rm chart}$ the subcategory of $\llog\stk^{\rm chart}$ that consists of charted log stacks whose underlying stacks are representable by  spectral Deligne-Mumford stacks. We define a functor $$|-|:\llog\Sptdm^{\rm chart}\rr\llog\Sptdm$$ given by the left Kan extension of the functor $\spet(-):\llog^{\rm op}\rightarrow \llog\Sptdm.$ We have the following commutative diagram $$
\begin{tikzcd}
\llog^{\rm op} \arrow[d, "\spec(-)"'] \arrow[r, "\spet(-)"] & \llog\Sptdm \\
\llog\Sptdm^{\rm chart} \arrow[ru, "|-|"', dashed]          &            
\end{tikzcd}$$
Thus, we get a morphism of presheaves 
$$|-|:\R\llog^{\rm chart}_{S}\rr\R\llog^{\rm qcoh}_{|S|}.$$
\begin{prop}\label{descent=forgetchart}
The morphism $|-|$ exhibits $\R\llog^{\rm qcoh}_{|S|}$  as a sheafification of $\R\llog^{\rm chart}_{S}$.
\end{prop}
\begin{proof}
Without loss of generality, we can assume that $S=\spec (R,L)$ is charted log affine.
Let $\CC$ be the subcategory of $\aff_{/\R\llog^{\rm chart}_{\spet(R,L)}}$ forms of maps $\spec R[M]\rightarrow\R\llog^{\rm chart}_{\spet(R,L)}$ classifying log structures $M\rightarrow\Omega^\infty R[M]$.
Using Remark\autoref{covtoricscheme},  the disjoint union of objects in $\CC$ forms a cover of $\R\llog^{\rm qcoh}_{\spet(R,L)}$.  This proves that $|-|$ is an effective epimorphism after sheafification. Note that $\CC$ in $\aff_{/\R\llog^{\rm chart}_{\spet(R,L)}}$ and its essential image $\CC'$ in $\aff_{/\R\llog^{\rm qcoh}_{\spet(R,L)}}$ are both cofinal. It follows  that the sheafification of $\R\llog^{\rm qcoh}_{\spet(R,L)}$ is equivalent to $\colim_\CC\spec R[M]\simeq\R\llog^{\rm qcoh}_{\spet(R,L)}.$
\end{proof}
\begin{rmk}
    If we work in the fpqc topology, then the same method shows that the fpqc sheafifiaction of $\R\llog_{S}^{\rm chart}$ also gives the moduli stack of quasi-coherent log structures on $|S|$ defined by fpqc topology.  
\end{rmk}
\begin{lem}\label{defchartlog}
    Let $(\X,\mathscr M)$ be a quasi-coherent spectral log Deligne-Mumford stack. Assume that there is a strict morphism $f:\spet(A,M)\rightarrow (\X,\mathscr M)$, such that the underlying map of stacks $\underline{f}:\spet A\rightarrow\X$ is a square-zero extension. Then  there is a unique (up to homotopy) square-zero extension $i:(A',M')\rightarrow (A,M)$ of $\Einf$-log rings, such that $\spet (A',M')\simeq(\X,\mathscr{M})$, and $f$ is induced from $i$.
\end{lem}
\begin{proof}
     \cite[Lemma 17.1.3.7]{lurie2018spectral} shows that the underlying stack $\X$ is affine, and we denote it by $\spet A'$, where $A'$ is a square-zero extension of $A$ with kernel $I\in\Mod_A$. Let us identify the underlying $\infty$-topoi of $\spet A$ and $\spet A'$, and restrict $\mathscr M$ to the small \'etale site of $\spet A$, we therefore regard $\mathscr M$ as a sheaf of $\Einf$-monoids on $A_\et$. Denote by $\mathscr M_0$ the pullback of $\mathscr M$ to $\spet A'$, which is equivalent to the log structure of $\spet(A,M)$ by assumption. As sheaves on $\spet A$, we have a  fiber sequence
     $$\mathscr M\rr\mathscr M_0\rr B\mathcal I\simeq\mathcal{I}[1]$$
     , where $\mathcal{I}$ is the quasi-coherent sheaf associated with $I$.\\
     As $\mathscr M_0$ is given by the sheafification of the functor $A_\et\rightarrow\mon, A'\mapsto M^a$, we can form a commutative diagram of presheaves on $A_\et$ as follows:
$$\begin{tikzcd}
{M'}^a \arrow[r] \arrow[d] & M^a \arrow[r] \arrow[d] & B\mathcal I \arrow[d, "\simeq"] \\
\mathscr M \arrow[r]       & \mathscr M_0 \arrow[r]  & B\mathcal I                    
\end{tikzcd}$$
where  $(-)^a$ is the point-wise logification,  $M^a\rightarrow B\mathcal{I}$ is induced from the composition 
$$\phi:M\rr \Gamma(\spet A,\mathscr M_0)\rr\Gamma(\spet A,B\mathcal{I})\simeq BI,$$
and $M'$ is its fiber. After taking the  sheafification, the middle arrow becomes an equivalence, and so does the left vertical map. It follows that $(\X,\mathscr M)$ is equivalent to $\spet(A',M')$, where $(A',M')$ is clearly a square-zero extension of $(A,M)$.\\
We now consider the uniqueness. If there is another square-zero extension $j:(A',M'')\rightarrow(A,M)$, such that $(\X,\mathscr M)$ is equivalent to $\spet(A',M'')$, and $f$ is induced from $j$. Then, by virtue of the  same argument as above, we can show that the induced map ${M''}\rightarrow M$ can be given as the fiber of the map $\phi$, and it turns out that $M'$ is canonically equivalent to $ M''$. We conclude that $(A',M')$ is unique up to homotopy.
\end{proof}
\begin{rmk}
The proof of Lemma\autoref{defchartlog} also holds when we work in the fpqc topology,
 as we observe that the sequence $\mathscr M\rightarrow\mathscr M_0\rightarrow B\mathcal I$  is also a   fiber sequence (as fpqc sheaves restricted to the small \'etale site of $\spet 
 A$). 
\end{rmk}
\begin{rmk}\label{nil-defchartlog}
    One can check that the proof of Lemma\autoref{defchartlog} also holds for the following case: $A$ is $n$-truncated, and $\spet A\rightarrow \X\simeq\spet A'$ is the embedding induced from the $n$-truncation map.
\end{rmk}
Using Lemma\autoref{defchartlog}, one can easily deduce the following result:
\begin{cor}\label{fiberstability}
    Let $S$ be a charted spectral log Deligne-Mumford stack. Let $A'\rightarrow A$ be a square-zero extension with kernel $I\in\Mod_A^{\rm cn}$. Then the following commutative diagram
    $$\begin{tikzcd}
\R\llog_S^{\rm chart}(A') \arrow[r] \arrow[d] & \R\llog_S^{\rm qcoh}(A') \arrow[d] \\
\R\llog_S^{\rm chart}(A) \arrow[r]            & \R\llog_S^{\rm qcoh}(A)           
\end{tikzcd}$$
is a pullback square.
\end{cor}
Combining  Theorem\autoref{gabbercotasfunctor} with Corollary\autoref{fiberstability}, we can conclude the existence of the cotangent complex on the moduli of log structures.
\begin{thm}\label{almostperfect}
 Let $\mu\in\{\rm qcoh,coh,fin\},$ and let $(\MCS,\MSL)$ be a $\mu$-spectral log Deligne-Mumford stack.  The functor $\R\llog^{\mu}_{(\MCS,\MSL)}$ admits a $(-1)$-connective cotangent complex $\LL_{\R\llog^{\mu}_{(\MCS,\MSL)}/\MCS}$. If $\mu=\rm coh$ or $\rm fin$, then the cotangent complex $\LL_{\R\llog^{\mu}_{(\MCS,\MSL)}/\MCS}$ is almost perfect.
\end{thm}
\begin{proof}
    As the existence of the cotangent complex of a relative stack $F/G$ is local on the target, we replace  $(\MCS,\MSL)$ with a strict \'etale log-affine cover $\spet(R,L)$. On the other hand, it's easy to see that if a log structure $\mathscr M$ over $\X$ has property $\mu$, then so does $\mathscr{M}'$ of a square-zero extension by Lemma\autoref{defchartlog}.
 From this point of view, we will let $\mu={\rm qcoh}$. Using the following Lemma\autoref{0987890}, together with Lemma\autoref{defchartlog}, we obtain the desired result.
\end{proof}
\begin{lem}\label{0987890}
    Let $X:\algcn\rightarrow\ani$ be a presheaf, assume that $X$ admits a cotangent complex $\LL_{X}$. If  its \'etale (or fpqc) (hyper)sheafification $a:X\rightarrow X'$
    satisfies the following condition: for any $s\in X'(A)$ who belongs to the essential image of $a$, the fiber of the natural map 
    $\eta:X'(A\oplus I)\rightarrow X'(A)$ at $s$ is equivalent to the fiber of $X(A\oplus I)\rightarrow X(A)$ at $s$ via the functor $a$.\\
    Then $X'$ also admits a cotangent complex.
    In particular, one has $a^*\LL_{X'/R}\simeq\LL_{X}$. 
\end{lem}
\begin{proof}
First, the pullback along $a$ gives rise to an equivalence $a^*:\qcoh(X')\rightarrow\qcoh(X)$ by \cite[Proposition 6.2.3.1]{lurie2018spectral}.    Let $\mathcal{L}$ be the quasi-coherent complex over $X'$ given by $(a^*)^{-1}\LL_X$, here $(a^*)^{-1}$ is a quasi-inverse of $a^*$.
    We have to show that $\mathcal{L}$ is indeed a cotangent complex of $X'$. Let $t\in X'(A)$ be an arbitrary section. Since $a$ is a (hyper)sheafification morphism, there is an \'etale (resp. flat) hypercover $A^\bullet$ of $A$, such that the restriction $t^\bullet\in X'(A^\bullet)$ belongs to the essential image of $a$. It follows that the fiber of $\eta$ at $t$ is equivalent to the limit of the following mapping spaces $$\lim_\Delta\map_{A^\bullet}((t^\bullet)^*\mathcal{L},I\otimes_AA^\bullet)\simeq\lim_\Delta\map_{A}(t^*\mathcal{L},I\otimes_AA^\bullet)\simeq\map_A(t^*\mathcal{L},I),$$
    where the equivalences above are deduced from the quasi-coherence of $\mathcal{L}$ and the flat hyperdescent of quasi-coherent complexes. This means that $\mathcal{L}$ is a cotangent complex of $X'$.
\end{proof}
\begin{defn}\label{DEF OF COTANGENT CPX}
    Let $f:(\X,\mathfrak M)\rightarrow(\MCS,\MSL)$ be a map of quasi-coherent spectral log Deligne-Mumford stacks, the log cotangent complex of $f$ is the relative cotangent complex $\LL^{\rm Log}_{(\X,\mathfrak M)/(\MCS,\MSL)}:=\LL_{\X/\R\llog^{\rm qcoh}_{(\MCS,\MSL)}}.$
\end{defn}
\begin{rmk}
    By the proof of Theorem\autoref{gabbercotasfunctor}. If the map $f:\spet(A,M)\rightarrow\spet(R,L)$ is induced from an $\Einf$-log ring map $(R,L)\rightarrow(A,M)$. Then we have a canonical equivalence $\LL^{\rm Log}_{\spet(A,M)/\spet(R,L)}\simeq\LL^G_{(A,M)/(R,L)}.$
\end{rmk}
Using the conormal sequence for Gabber's cotangent complex and reducing to the local arguments, we prove the following results:
\begin{thm}\label{global tri seq}
    Let $(\X,\mathfrak M)\stackrel{f}\rightarrow(\Y,\mathscr N)\stackrel{g}\rightarrow(\MCS,\MSL)$ be a sequence of quasi-coherent spectral log Deligne-Mumford stacks, then there is a fiber sequence of quasi-coherent sheaves over $\X$:
    $$f^*\LL^{\rm Log}_{(\Y,\mathscr N)/(\MCS,\MSL)}\rr\LL^{\rm Log}_{(\X,\mathfrak M)/(\MCS,\MSL)}\rr\LL^{\rm Log}_{(\X,\mathfrak M)/(\Y,\mathscr N)}.$$
\end{thm}
\begin{rmk}
  The proof of  Theorem\autoref{almostperfect} shows that the moduli stack of quasi-coherent log structures defined using fpqc also has a cotangent complex, which coincides with the cotangent complex on $\R\llog_{(\MCS,\MSL)}^{\rm qcoh}$.
\end{rmk}
\subsubsection{Infinitesimal properties and  finiteness}\label{zxzxzxz}
\begin{prop}[Infinitesimally Cohesive]\label{infcohesive}
  Let $\mu\in\{\rm qcoh,coh,fin\}$, and let $(\MCS,\MSL)$ be a $\mu$-spectral log Deligne-Mumford stack. Then, the relative functor $\R\llog^{\mu}_{(\MCS,\MSL)}\rightarrow\MCS$ is relatively infinitesimally cohesive. 
\end{prop}
\begin{proof}
Without loss of generality, we assume that $(\MCS,\MSL)\simeq\spet(R,L)$ is log affine. Thanks to \cite[Proposition 17.3.6.1]{lurie2018spectral}, we only need to prove that for any $R$-algebra $A$, any connective $A$-module I, and  any $R$-linear derivation $d:A\rightarrow A\oplus I[1]$, the following commutative diagram
$$\begin{tikzcd}
{\R\llog_{\spet(R,L)}^{\mu}(A')} \arrow[r] \arrow[d] & {\R\llog_{\spet(R,L)}^{\mu}(A)} \arrow[d]  \\
{\R\llog_{\spet(R,L)}^{\mu}(A)} \arrow[r]             & {\R\llog_{\spet(R,L)}^{\mu}(A\oplus I[1])}
\end{tikzcd}$$
is a pullback square, where $A'$ is the square-zero extension associated with $d$. Note that we can assume that $\mu=$qcoh.
We then use Corollary\autoref{fiberstability}. Therefore, we only need to show the following commutative diagram 
$$\begin{tikzcd}
{\R\llog_{\spec(R,L)}^{\rm chart}(A')} \arrow[r] \arrow[d] & {\R\llog_{\spec(R,L)}^{\rm chart}(A)} \arrow[d]  \\
{\R\llog_{\spec(R,L)}^{\rm chart}(A)} \arrow[r]             & {\R\llog_{\spec(R,L)}^{\rm chart}(A\oplus I[1])}
\end{tikzcd}$$
is a pullback square. In other words, we should show that the following natural map 
$$v:{\R\llog_{\spec(R,L)}^{\rm chart}(A')}\rr {\R\llog_{\spec(R,L)}^{\rm chart}(A)}\times_{{\R\llog_{\spec(R,L)}^{\rm chart}(A\oplus I[1])}} {\R\llog_{\spec(R,L)}^{\rm chart}(A')}$$
is an equivalence. The right hand term is the anima consisting of triples $((A,M),(A,M'),\tau)$, where $\tau:d^*(A,M)\stackrel{\simeq}\rightarrow d^*_{\rm triv}(A,M')$, and where we let $d_{\rm triv}$ be the trivial derivation $A\rightarrow A\oplus I[1]$.
We can construct an inverse $u$ of $v$, which sends a triple $((A,M),(A,M'),\tau)$ to the fiber product $(A,M)\times_{d^*(A,M)} (A,M')$. One can check that $u\circ v\simeq {\rm id}$ and $v\circ u\simeq{\rm id}$.
\end{proof}
\begin{rmk}[Fpqc quasi-coherent log structures]
    The moduli stack of quasi-coherent  log structures defined using  fpqc topology also has the infinitesimally cohesive property.
\end{rmk}
\begin{prop}[Nil-completeness and integrability]\label{nilcompint}
    Let $\mu\in\{\rm qcoh,coh,fin\}$, and let $(\MCS,\MSL)$ be a $\mu$-spectral log Deligne-Mumford stack, which is equivalent to $|S|$ for some charted spectral log stack $S$. Fix a sequence of $\Einf$-rings $$...\rr A_n\rr A_{n-1}\rr...\rr A_0$$
    lying over $\MCS$, with limit $A$. 
Assume that one of the following conditions holds:
    \begin{enumerate}
        \item $A_n\simeq\tau_{\leq n} A$, and the transition maps $A_n\rightarrow A_{n-1}$ are truncation maps;
        \item  $A$ is a discrete complete Noetherian local ring, and the transition maps $A_n\simeq A/\mathfrak m^n\rightarrow A_{n-1}\simeq A/\mathfrak m^{n-1}$ are the   maps given in the obvious way, where $\mathfrak m\subset A$ is the maximal ideal of $A$.
    \end{enumerate}
    
    Then the following  canonical map of animae
    $$\varphi:\R\llog^\mu_{(\MCS,\MSL)}(\spet A)\rr{\varprojlim}_n \R\llog^\mu_{(\MCS,\MSL)}(\spet A_n)$$
    is an equivalence.
\end{prop}
\begin{lem} \label{chartinfdef}
Fix a charted log stack $S$ with an underlying spectral Deligne-Mumford stack. We have the following facts:
    \begin{enumerate}
        \item The functor $\R\llog^{\rm chart}_{S}:\alg^{\rm cn}_{/\underline S}\rightarrow\ani$ is nil-complete and infinitesimally cohesive.
        \item Let $A$ be a discrete local $\Einf$-ring that is complete with respect to a finitely generated ideal $\mathfrak a\subset A$, then one has $\R\llog^{\rm chart}_S(A)\simeq\varprojlim\R\llog^{\rm chart}_S(A/\mathfrak a^n)$.
    \end{enumerate}
\end{lem}
\begin{proof}
For assertion $(1)$, we first consider the natural map $\phi:\R\llog^{\rm chart}_{S}(A)\rightarrow\varprojlim \R\llog^{\rm chart}_{S}(\tau_{\leq n}A)$. This is an equivalence, as we can construct an inverse $\phi'$ of $\phi$, which sends a projective system $\{(\tau_{\leq n A},M_n)\}^\infty_{n=0}$ to $\varprojlim (\tau_{\leq n}A,M_n)$. Now, we let $\{(\tau_{\leq n A},M_n)\}^\infty_{n=0}$ be the image of $(A,M)$ under $\phi$. As pointed out in Remark\autoref{torsor=sqzero}, $M_{n+1}\rightarrow M_n$ is a $\pi_{n+1}A[n+1]$-torsor, and we have a fiber sequence 
$$M_{n+1}\rr M_n\rr\pi_{n+1}A[n+2].$$
By induction, we can regard $M\rightarrow \lim_n M_n$ 
as a map of  torsors over $\varprojlim_m (\tau_{\leq m}\tau_{\geq n+1}A)[1]\simeq(\tau_{\geq n+1} A)[1]$, which is obviously an equivalence. This shows that $\phi'\circ\phi\simeq{\rm id}$. The same argument also gives $\phi\circ\phi'\simeq{\rm id}$. This shows that $\R\llog^{\rm chart}_{S}$ is nil-complete. As $\R\llog^{\rm chart}_S$ is nil-complete and admits a cotangent complex. The differential criterion of infinitesimal cohesivity \cite[Proposition 17.3.6.1]{lurie2018spectral} implies that we only need to check the following diagram
$$\begin{tikzcd}
\R\llog_S(A') \arrow[d] \arrow[r] & \R\llog_S(A) \arrow[d]    \\
\R\llog_S(A) \arrow[r]            & {\R\llog_S(A\oplus I[1])}
\end{tikzcd}$$
is a pullback square, where $A'\rightarrow A$ is the square-zero extension of $A$ induced by a derivation $A\rightarrow A\oplus I[1]$. This can be deduced from a similar argument as stated above.\\
For the assertion $(2)$. We also want to show that the natural map 
$\nu:\R\llog^{\rm chart}_S(A)\rightarrow\varprojlim\R\llog^{\rm chart}_S(A/\mathfrak a^n)$ admits an inverse $\mu:\{(A/\mathfrak a^n,M_n)\}^\infty_{n=0}\mapsto \varprojlim (A/\mathfrak a^n,M_n)$.
As $A$ is a local ring,  then $A/\mathfrak a^n$ is also a local ring for all $n\geq 0$. An element $x$ in $A$ is invertible if and only if it is not contained in its maximal ideal. This is equivalent  to saying that its image $x_n\in A/\mathfrak a^n$. So the following commutative diagram 
$$\begin{tikzcd}
GL_1(A/\mathfrak a^{n+k}) \arrow[d] \arrow[r] & GL_1(A/\mathfrak a^{n}) \arrow[d] \\
A/\mathfrak a^{n+k} \arrow[r]                 & A/\mathfrak a^n                  
\end{tikzcd}$$
is a pullback square for any $0\leq k\leq\infty$. It follows that the map $M_{n+k}\rightarrow M_n$ exhibits $M_{n+k}$ as a ${\rm fib}(GL_1(A/\mathfrak a^{n+k})\rightarrow GL_1(A/\mathfrak a^{n}))$-torsor. Passing to the limit, we see that $M:=\varprojlim_k M_{n+k}\rightarrow M_n$ is a ${\rm fib (GL_1(A)\rightarrow GL_1(A/\mathfrak a^n))}$-torsor, and the natural map $M\rightarrow\varprojlim M_{n}$ is an equivalence of torsors over $M_n$. This holds because the filtered limits commute with the classifying space functor when restricted to discrete groups. The similar argument as the proof of the nil-completeness part of assertion (1) demonstrates that $\mu$ is an inverse of $\nu$.
\end{proof}
\begin{proof}[Proof of Proposition\autoref{nilcompint}]
The natural map $\R\llog^{\rm chart}_S\rightarrow \R\llog^{\rm qcoh}_{|S|}$ is a sheafification map. We will abbreviate the functor $\R\llog^{\rm chart}_S$ to $\F$ for short.
We have a canonical equivalence
$$\R\llog^{\rm qcoh}_{|S|}(A)\simeq
\varinjlim_{C^0_A\in\mathcal{U}_A}\lim_{B\in C^0_A}\F(B),$$
where $B$ runs over  the  covering sieve $C^0_A$, and the colimit runs over the filtered $\infty$-category $\mathcal{U}_A$ of  all covering sieves of $A$. The map $\varphi
$ can be reformulated into the following  form
$$\gamma: \varinjlim_{C^0_A\in \mathcal U_A}\lim_{B\in C^0_A}\F(B)\rr 
\varprojlim_n\varinjlim_{C^0_{A_n}\in\mathcal{U}_{A_n}}\lim_{B_n\in C^0_{A_n}}\F(B_n)
.$$
For case $(1)$, we have $A_\et\simeq
(A_n)_\et$. We can first replace $C^0_{A_n}$ by its subcategory $C_A^{\et}$ spanned by \'etale objects, and then we identify $C^\et_A$ with $C^\et_{A_n}$, and identify $\mathcal{ U}_A$ with $\mathcal{U}_{A_n}$. We can rewrite $\gamma$ as follows:
$$\gamma^{(1)}: \varinjlim_{C^0_A\in \mathcal U_A}\lim_{B\in C^\et_A}\F(B)\rr 
\varprojlim_n\varinjlim_{C^0_{A}\in\mathcal{U}_{A}}\lim_{B\in C^\et_{A}}\F(\tau_{\leq n}B)
.$$
 For case $(2)$, we observe that any covering  $C^0_A$ has a refinement which is generated by a single map $\{A\rightarrow B\}$ such that $B$ is connected by Hensel's lemma. In particular, in this case, $B$ is a finite local algebra lying over $A$. Note that the  {\v C}ech nerve of $A\rightarrow B$ is cofinal in this covering. 
Then we can simplify the map $\gamma$ as the following map
$$\gamma^{(2)}:\varinjlim_{B}\lim_{\Delta}\F(B^\bullet)\rr\varprojlim_n\varinjlim_B\lim_{\Delta}\F(B^\bullet\otimes_AA/\mathfrak m^n),$$
where $B$ runs over the category of finite \'etale local $A$-algebras.\\
Now, we prove the equivalenceness of $\gamma^{(1)}$ and $\gamma^{(2)}$. Denote by $B_n$ either $\tau_{\leq n}B$ or $B\otimes_AA/\mathfrak m^n$ 
We  first compute the fiber of the map $\F(B)\rightarrow\F(B_n)$. We denote by $L$  a cotangent complex of $\F$.
Let $(B_n,M_n)\in\F(B_n)$ be a point that lies in the essential image of $\F(B)\rightarrow\F(B_n)$.
The fiber of $\F(B_{n+1})\rightarrow\F(B_n)$ at a point $(B_n,M_n)\in\F(B_n)$ is equivalent to $\map_{B_n}(L|_{B_n},J_n)$. Here, we let 
$$J_n=
\begin{cases}
    \pi_{n+1}B[n+1], &\quad B_n=\tau_{\leq n}B;\\
    \mathfrak m^nB/\mathfrak m^{n+1}B,  &\quad B_n=B\otimes_AA/\mathfrak m^n.
\end{cases}$$
Applying induction, if we denote by $F_{m,n}(B)$ be the fiber of the map $\F(B_m)\rightarrow\F(B_n)$, then there is a tower 
$$...\rr F_{m,n}(B)\rr F_{m-1,n}(B)\rr...\rr F_{n+1,n}(B)\simeq\map_{B_n}(L|_{B_n},J_n),$$
whose layers are given by ${\rm fib}(F_{m,n}(B)\rightarrow F_{m-1,n}(B))\simeq F_{m+1,m}(B)\simeq\map_{B_m}(L|_{B_m},J_m)$.
Using the fact $\F(B)\simeq\varprojlim\F(B_n)$ proved in Lemma\autoref{chartinfdef}, the fiber of the map $\F(B)\rightarrow \F(B_n)$ at $(B_n,M_n)$ is equivalent to the inverse limit $F_{\infty,n}(B):=\varprojlim_{k}F_{n+k,n}(B).$ 
Note that the  layer $B\mapsto \map_{B_n}(L|_{B_n},J_n)$ of the tower $\{F_{m,n}(-)\}_m$ defines an \'etale sheaf on $A$. Applying the induction, we know that $F_{m,n}, n\leq m\leq\infty$ are sheaves on the \'etale site of $A$. It follows that the fiber of $\gamma^{(i)},i=1,2$ at a point $\{(A_n,M_n)\}_{n=0}^\infty$ is equivalent to $\varprojlim_nF_{\infty,n}(A)\simeq\varprojlim_n\varprojlim_m F_{m,n}(A)\simeq\varprojlim_nF_{n,n}(A)=0$. This shows that $\varphi$ is $(-1)$-truncated.\\
We still need to show that $\varphi$ is essentially surjective.  Choose an \'etale cover $A_0\rightarrow A_0'$ of $A_0$ such that the pullback $\mathscr M'_0=\mathscr M_0|_{\spet A_0'}$ is log-affine.  
Then by the rigidity theorem of Lurie \cite[Theorem 7.5.4.2]{lurie2017higher}, there is essentially a unique commutative diagram $\Einf$-rings 
$$\begin{tikzcd}
... \arrow[r] & A_n \arrow[r] \arrow[d] & A_{n-1} \arrow[r] \arrow[d] & ... \arrow[r] \arrow[d] & A_0 \arrow[d] \\
... \arrow[r] & A'_n \arrow[r]          & A'_{n-1} \arrow[r]          & ... \arrow[r]           & A_0'         
\end{tikzcd}$$
such that all vertical maps are \'etale covers. The limit map $A\rightarrow\varprojlim_nA'_n$ is also an \'etale cover by \cite[Theorem 7.5.4.2]{lurie2017higher} again. Let $\mathscr M_n'$ be the pullback of $\mathscr M_n$ to $\spet A'_n$ for $n\geq 0$. By Lemma\autoref{defchartlog}, the square-zero extension of a charted log-affine object $\spec(A,M)$ also has a chart, which is functorially dependent on the $\Einf$-log ring $M$. More precisely,  one can construct a sequence of $\Einf$-log rings $(A'_n,M'_n)$, such that for any $n\geq 0,$ the pullback of $M'_{n+1}$ along the map  $A'_{n+1}\rightarrow A'_n$ is equivalent to $ M'_n$, and the map $\spet(A'_n,M'_n)\rightarrow\spet(A'_{n+1}, M_{n+1}')$is induced from $\Einf$-ring map $(A'_{n+1},M'_{n+1})\rightarrow(A'_n,M'_n)$.  Passing to the limit, we get a log-affine stack $\spet(A',M')=|\spec(A',M')|,$ where $(A',M')=\varprojlim(A'_n,M'_n),$ which is a pre-image of $\{\spet(A_n',M_n')\}_{n=0}^\infty\in\varprojlim\R\llog^{\rm qcoh}_{|S|}(A_n')$.  Finally, applying the Descendability Theorem\autoref{des} of  $\R\llog^{\rm qcoh}_{|S|},$ one gets a log structure $\mathscr M$ over $\spet A,$ which is as desired.
\end{proof}
\begin{rmk}[Pro-\'etale quasi-coherent log structures]\label{proetalequasicoherent}
The proof of Proposition\autoref{nilcompint} depends on the fact that the \'etale topology are stable under infinitesimal thickening, so the same argument is  also suitable for the moduli stack of quasi-coherent log structures  defined using pro-\'etale topology.  
\end{rmk}

Let $f:X\rightarrow Y$ be a morphism of functors defined over the $\infty$-category $\algcn.$ We say that $f$ is \textit{locally almost of finite presentation} if it satisfies the following condition:\\
 Let $m\geq 0$, and let $\{A_j\}_{j\in J}$ be a filtered diagram of $m$-truncated connective $\Einf$-rings
with colimit $A$. Then the canonical map
$$\theta:\varinjlim X(A_j)\rr X(A)\times_{Y(A)}\varinjlim Y(A_j)$$
is an equivalence.
\begin{prop}\label{almostfinite}
    Let $\mu\in\{\rm coh,fin\}$, and let $\MSL$ be a $\mu$-log structure over $\MCS$.The morphism $\R\llog^{\mu}_{(\MCS,\MSL)}\rightarrow\MCS$ is locally almost of finite presentation.
\end{prop}
\begin{proof}
    Without loss of generality, we may replace  $\MCS$ by some \'etale cover of it and, therefore, assume that $(\MCS,\MSL)\simeq\spet(R,L)$ is log-affine.  We have to show that the functor 
    $$\R\llog^{\mu}_{\spet(R,L)}:\algcn_R\rr\ani$$
    is locally almost of finite presentation, i.e., for any  filtered diagram of $m$-truncated $R$-algebras $\{A_j\}_{j\in J}$ with colimit $A$, the following natural map $$\theta:\varinjlim\R\llog^{\mu}_{\spet(R,L)}(A_j)\rr\R\llog^{\mu}_{\spet(R,L)}(A)$$
    is an equivalence. Applying \cite[Corollary 17.4.2.2]{lurie2018spectral}, Theorem\autoref{almostperfect} and Proposition\autoref{infcohesive}, we will assume that $A_j$ is discrete for any $j\in J.$
    We first prove that the map $\theta$ is effectively epimorphic.  Let $\mathscr M$ be a $\mu$-log structure over $\spet A.$ Choose an \'etale cover $A\rightarrow A'$, such that the restriction  $\mathscr M|_{\spet A'}$ is log-affine. Using \cite[Lemma 29.36.14]{stacks-project}, we might assume that $A'$ is a finite product of standard \'etale  $A$-algebras, i.e., there exist $a_1,a_2,...,a_n\in A$ and $f
    _1,f_2,...,f_n\in A[x],$ such that $A'=\prod_iA[a_i^{-1}][x]/(f_i).$ Then we might also assume that $a_i,i=1,2,...,n,$ and all coefficients of $f_i,i=1,2,...,n,$ belong to $A_j,j\in J.$ Let $A'_j=\prod_iA_j[a_j^{-1}][x]/(f_j)$, then we may assume that  the natural map $A_j\rightarrow A_j'$ is an \'etale cover for any $j\in J.$ We have an isomorphism $\varinjlim A_j'\simeq A'$  of discrete rings. Form a commutative diagram of animae
    $$\begin{tikzcd}
{\varinjlim\R\llog^{\mu}_{\spet(R,L)}(A_j)} \arrow[d] \arrow[r] & {\R\llog^{\mu}_{\spet(R,L)}(A)} \arrow[d] \\
{\varinjlim\R\llog^{\mu}_{\spet(R,L)}(A'_j)} \arrow[r]          & {\R\llog^{\mu}_{\spet(R,L)}(A')}         
\end{tikzcd}$$
      Since $(\spet A',\mathscr M)$ is log-affine, there is a $\mu$-$\Einf$-log ring $(A',M'),$ such that $(\spet A',\mathscr M)\simeq\spet(A',M').$ In particular, there is an $\Einf$-monoid $P$, such that $\mathbb S[P]$ is almost of finite presentation, and there is an equivalence $M'=P^a.$ The finiteness of $P$ implies that the $\Einf$-ring map $\mathbb S[P]\rightarrow A'$ factors through some $A'_j$. Denote by $M'_j$  the monoid of logification of the $\Einf$-prelog ring $(A'_j,P),$ and let $\mathscr M'_j$ be the corresponding log structure over $\spet A_j'.$ Applying the descendability Theorem\autoref{des} of  $\R\llog^{\mu}_{\spet(R,L)},$ one obtains a unique log structure $\mathscr M_j$ over $\spet A_j$, and the pullback of $\mathscr M_j$ along the map $\spet A\rightarrow\spet A_j$ is equivalent to $\mathscr M.$\\
     We prove that the map $\theta$ has contractible fibers. Let $F$ be the fiber of $\theta$ at $\mathscr M$, and let $\mathscr M_i\in\R\llog^{\mu}_{\spet(R,L)}(A_i),\mathscr M_j\in\R\llog^{\mu}_{\spet(R,L)}(A_j)$, such that the  pullback of them to $\spet A$ are equivalent to $\mathscr M$ . By virtue of the argument above, there is an index $k\geq  i,j$ and a log structure $\mathscr M_k$ over $\spet A_k$, such that $\mathscr M_k$ is equivalent to the pullback of $\mathscr M_i$ and $\mathscr M_j.$ It turns out that we can identify $\mathscr M_i$ and $\mathscr M_j$ via the canonical equivalence $f_{ik}^*\mathscr M_i\simeq f_{jk}^*\mathscr M_j$, in which we denote $f_{ik}$ and $f_{jk}$ as the transition maps $A_i\rightarrow A_k$ and $A_j\rightarrow A_k$. This shows that $F$ is contractible. 
\end{proof}
\subsubsection{Truncatedness}
\begin{defn}
    Let $(\X,\mathscr M)$ be a spectral log Deligne-Mumford stack. We say that $(\X,\mathscr M)$ is $n$-\textit{admissible}, if there is an \'etale cover $\coprod^N_{j=1}\spet R_j\rightarrow\X$ of spectral Deligne-Mumford stacks, such that  the pullback 
    $\mathscr M|_{\spet R_j}$ are given by 
    $n$-admissible $\Einf$-log rings. We say that it is admissible if it is $n$-admissible for some $n<+\infty$.
\end{defn}
\begin{rmk}
   By Remark\autoref{admissiblelocal}, the $n$-admissibility of spectral log Deligne-Mumford stacks is local on the underlying Deligne-Mumford stacks with respect to the fpqc topology.
\end{rmk}
\begin{rmk}
    $\spet(A,M)$ is $n$-admissible if and only if $(A,M)$ is $n$-admissible.
\end{rmk}
\begin{prop}\label{1-trunc-2}
    Let $(\mathcal X,\mathscr M),(\Y,\mathscr N)$ be  spectral log Deligne-Mumford stacks. Assume that
    the following conditions are satisfied:
    \begin{enumerate}
        \item $\X$ and $\Y$ are affine;
        \item $\X$ is $m$-truncated;
        \item $(\X,\mathscr M)$ is $n$-admissible 
    \end{enumerate}
Then the mapping space $\map((\X,\mathscr M),(\Y,\mathscr N))$ is $(1+\max\{n,m\})$-truncated.
\end{prop}
\begin{proof}
Let $(\Y,\mathscr N)$ be an arbitrary  spectral log Deligne-Mumford stack.  Then by the definition of the $\infty$-category $\llog\Sptdm$ introduced in \autoref{spectral log stacks}, there is a canonical map 
$$\theta:\map((\X,\mathscr M),(\Y,\mathscr N))\rr\map(\X,\Y).$$
The anima $\map(\X,\Y)$ is $m$-truncated. We must consider the fibers of the map $\theta$. Let $f\in\pi_0\map(\X,\Y)$. The fiber of $\theta$ at $f$ is equivalent to the anima $\map_{/\OO_\X}(f^*\mathscr N,\mathscr M)\simeq\lim_{R\in\X_{\et}}\map_{/R}(f^*\mathscr N(R),\mathscr M(R)).$ Note that each term $\map_{/R}(f^*\mathscr N(R),\mathscr M(R))$ is $(1+\max\{n,m\})$-truncated, then so is $\map_{/\OO_\X}(f^*\mathscr N,\mathscr M).$ It follows that $\map((\X,\mathscr M),(\Y,\mathscr N))$ is $(1+\max\{n,m\})$-truncated.
\end{proof}

\begin{defn}\label{rlogdefn2}
Let $\mu\in\{\rm qcoh,coh,fin\}$, and let $(\MCS,\MSL)$ be a  $\mu$-spectral log stack.
Note that $n$-admissible objects are stable under pullbacks, so there are  subfunctors $\R\llog^{n-\rm Adm,\mu}_{(\MCS,\MSL)}$ and $\R\llog^{\rm Adm,\mu}_{(\MCS,\MSL)}$ of $\R\llog^{\mu}_{(\MCS,\MSL)}$ consisting of $n$-admissible and admissible objects respectively. 
\end{defn}
\begin{rmk}\label{admdeformation}
    It's easy to see that $\R\llog^{n-\rm Adm,\mu}_{(\MCS,\MSL)}$ and $\R\llog^{\rm Adm,\mu}_{(\MCS,\MSL)}$ satisfy \'etale hyperdescent and all infinitesimal properties investigated in \autoref{zxzxzxz}. 
\end{rmk}
The stack $\R\llog^{\rm Adm,\mu}_{(\MCS,\MSL)}$ admits a filtration of substacks 
$$\R\llog^{0-{\rm Adm,\mu}}_{(\MCS,\MSL)}\subset\R\llog^{1-{\rm Adm},\mu}_{(\MCS,\MSL)}\subset...\subset \R\llog^{n-{\rm Adm},\mu}_{(\MCS,\MSL)}\subset...$$
where $\R\llog^{n-{\rm Adm},\mu}_{(\MCS,\MSL)}$ is the substack of $\R\llog^{{\rm Adm},\mu}_{(\MCS,\MSL)}$ consisting of $n$-admissible objects. We have $\R\llog^{{\rm Adm},\mu}_{(\MCS,\MSL)}\simeq\bigcup_{n}\R\llog^{n-{\rm Adm},\mu}_{(\MCS,\MSL)}$.
\begin{lem}\label{admcot}
    The inclusions $\R\llog^{n-{\rm Adm},\mu}_{(\MCS,\MSL)}\subset \R\llog^{m-{\rm Adm},\mu}_{(\MCS,\MSL)},n\leq m$ and $\R\llog^{{\rm Adm},\mu}_{(\MCS,\MSL)}\subset\R\llog^{\mu}_{(\MCS,\MSL)}$ are formally \'etale. As a consequence, the stack $\R\llog^{n-{\rm Adm},\mu}_{(\MCS,\MSL)}$ admits a cotangent complex for any $n\leq+\infty$.
\end{lem}
\begin{proof}
    Let $A$ be a connective $\Einf$-ring and let $I\in\Mod_A^{\rm cn}$. Let us form a commutative diagram  of stacks
    $$\begin{tikzcd}
\spec(A) \arrow[d] \arrow[r, "x"]    & {\R\llog^{n-{\rm Adm},\mu}_{(\MCS,\MSL)}} \arrow[d] \\
{\spec(A\oplus I)} \arrow[r, "x'"] & {\R\llog^{m-{\rm Adm},\mu}_{(\MCS,\MSL)}}          
\end{tikzcd}$$
Without loss of generality, we assume that both $x$ and $x'$ classify affine log structures $\spet(A,M)$ and $\spet(A\oplus I,M')$ respectively, such that $(A\oplus I,M')$ is a square-zero extension of $(A,M)$.
Using Remark\autoref{admissiblelocal}, the $n$-admissibility is preserved under  square-zero extensions, and therefore the space of  liftings $\spec(A\oplus I)\rightarrow \R\llog^{n-{\rm Adm},\mu}$ is contractible. 
\end{proof}

\section{Spectral algebraic stacks}\label{ALREP}
In order to study the moduli functor of  spectral log structures, we have to consider algebraicity in some sense. In other words, we want a conception of \textit{Artin stacks} in spectral algebraic geometry. At first, we need a concept of smoothness for defining the notion of atlas $U$ over  stacks $X$. However, in spectral algebraic geometry, following \cite[Section III.11.2.2]{lurie2018spectral}, there are at least two kinds of smoothness  for maps of $\Einf$-rings (and  spectral Deligne-Mumford stacks) $f:A\rightarrow B$:
\begin{enumerate}
    \item \textit{Fiber smooth}, the map of $\Einf$-rings $f:A\rightarrow B$ is flat, and the induced map $\pi_0(f):\pi_0(A)\rightarrow \pi_0(B)$ is smooth as ordinary rings;
    \item \textit{Differentially smooth} (in the following we refer to “smooth” for simplicity),
    the map of $\Einf$-rings $f:A\rightarrow B$ is almost of finite presentation, and 
    the cotangent complex $\LL_{B/A}$ is a  projective object in the $\infty$-category $\Mod_B$ of $B$-modules.
\end{enumerate}
It seems that if we adopt the  fiber-smoothness to define our algebraic stack, the associated deformation theory will be pathological, because the fiber-smoothness does not determine the infinitesimal lifting property in spectral algebraic geometry. Hence, in this section, we define {Artin stacks} using differential smoothness. Informally, an Artin stack is an \'etale  sheaf over $\algcn$ with representable diagonal and admits a differentially smooth surjective map from a spectral algebraic space.
\subsection{Artin stacks}
\begin{defn}
    Let $f:X\rightarrow Y$ be a morphism in $\stk$. We say that $f$ is $0$-representable, or a \textit{spectral relative algebraic space}, if for any connective $\Einf$-ring $R$, and any morphism $\spec R\rightarrow Y$, the fiber product $X\times_Y\spec R$ is a  spectral algebraic space over $R$. Here a spectral algebraic space means a functor which is represented by some spectral Deligne-Mumford $0$-stack\footnote{Here a spectral Deligne-Mumford $0$-stack $\X$ is a spectral Deligne-Mumford stack such that $\X(R)$ is discrete for any discrete $R$.}.\\
    Let $n$ be a natural number. The morphism $f$ is called $n$-representable, if the diagonal $\Delta_{X/Y}:X\rightarrow X\times_YX$ is $(n-1)$-representable. 
\end{defn}
\begin{defn}
    Let $\mathbb P$ be one of the following properties of morphisms of spectral relative algebraic spaces: 
$$ \text{Smoothness, Fiber-smoothness, Flatness, \'Etaleness}$$
We say that a morphism $f:X\rightarrow Y$ of sheaves has the property $n$-$\mathbb P$, if it's $n$-presentable, and for any $\Einf$-ring $R$, and any map $\spec R\rightarrow Y$, there is an  $(n-1)$-$\mathbb P$ surjective map $U\rightarrow X\times_Y\spec R$, such that the composition  $U\rightarrow X\times_Y\spec R$ is a spectral relative algebraic space,  and has the property $\mathbb P.$ We will say  that the morphism  $f$ has the property $\mathbb P$, if it has the property $n$-$\mathbb P$ for some $n\in\Z_{\geq 0}.$
\end{defn}

\begin{defn}
Assume that we have a commutative diagram 
$$\begin{tikzcd}
U \arrow[rr] \arrow[rd] &   & X \arrow[ld] \\
                        & Y &             
\end{tikzcd}$$
such that $U/Y$ is a spectral relative algebraic space, and the map $U\rightarrow X$ is surjective and has the property $\mathbb P$. We will say that $U$ is a $\mathbb P$-\textit{atlas} of $X$. 
\end{defn}
\begin{defn}\label{artstkdef}
    Let $f:X\rightarrow Y$ be a morphism in $\stk$. We say that $f$ is 
    \begin{enumerate}
        \item 
    A \textit{relative} $n$-\textit{Artin stack}, if it is $n$-representable, and admits a smooth atlas locally, i.e., for any map $\spec A\rightarrow Y$, the fiber product $X\times_Y\spec A$ admits a smooth atlas;
    \item A \textit{relative Artin stack} is a relative $n$-Artin stack for some $n\geq 0$;
    \item 
     A \textit{relative locally Artin stack} is such that for any map $\spec R\rightarrow Y$, the fiber product $X\times_Y\spec R$ is a union of open substacks that are Artin. 
    \end{enumerate}
\end{defn}
\begin{rmk}
It's easy to see that:
\begin{enumerate}
    \item 
     If $f:X\rightarrow Y$ is a relative $n$-Artin stack, then it's a relative $m$-Artin stack, for any $m\geq n$;
     \item  The pullback of a relative $n$-Artin stack is also a relative $n$-Artin stack;
    \item The finite limit of relative $n$-Artin stacks is also an $n$-Artin stack.
    \end{enumerate}
\end{rmk}
We present some basic properties of spectral Artin stacks, which are  proved for derived Artin stacks by Lurie in \cite{lurie2004derived}. However, the proofs also translate word by word  for spectral Artin stacks by  definition. We only summarize the main results below and refer to the proofs in Lurie's thesis.
\begin{prop}\cite[Proposition 5.1.5]{lurie2004derived}\label{cotartinstk}
    A relative $n$-Artin stack $f:X\rightarrow Y$ admits a $(-n)$-connective cotangent complex. 
\end{prop}
\begin{prop}\cite[Proposition 5.3.7]{lurie2004derived} and \cite[Theorem 6.4.1]{lurie2004derived}
     A relative $n$-Artin stack $f:X\rightarrow Y$ is  infinitesimally cohesive, nil-complete, and integrable. 
\end{prop}
\begin{lem}\cite[Proposition 5.1.4]{lurie2004derived}\label{n-1ton}
Let $f
:X\rightarrow Y$ be a map of sheaves.    The following conditions are equivalent. 
    \begin{enumerate}
    \item  The map $f$ is a relative $n$-Artin stack.
     \item For every spectral algebraic space $U/Y$, and every map $\pi:U\rightarrow X$,  then the map $\pi$ is a relative $(n-1)$-Artin stack.
     \item There is a spectral algebraic space $V/Y,$ and a smooth cover $V\rightarrow X$ which is a relative $(n-1)$-Artin stack.
     \item There is an $(n-1)$-Artin stack $W/Y$, and a smooth cover $W\rightarrow X,$ which is a relative $(n-1)$-Artin stack.
    \end{enumerate} 
\end{lem}
%\begin{proof}
 %  Without loss of generality, we might as well assume that $Y$ is equivalent to a spectral affine scheme $\spec A$ for some connective $\Einf$-ring $A.$ Thus we can regard $X$ as a sheaf over the $\infty$-category $\algcn_A$, with respect to the \'etale topology. The implication $(1)\Rightarrow(2)\Rightarrow(3)\Rightarrow (4)$ is trivial. We will then prove the implication of assertions $(4)\Rightarrow(1).$ Let us suppose that the cover $W\rightarrow X$ is $(n-1)$-Artin. Let $U,V$ be  spectral algebraic spaces, and  let $a:U\rightarrow X$, and $b:V\rightarrow X$ be two maps. Assume that  there is a lifting $U\rightarrow W.$ Hence we have a commutative diagram  
%    \[\begin{tikzcd}
%%%	U & {U\times_XV} \\
%%%%	W & {W\times_XV} \\
%%%%	X & V
%%%%%	\arrow[from=1-1, to=2-1]
%%%	\arrow["a"', curve={height=12pt}, from=1-1, to=3-1]
%%%	\arrow[from=1-2, to=1-1]
%%%	\arrow[from=1-2, to=2-2]
%%	\arrow[from=2-1, to=3-1]
%%	\arrow[from=2-2, to=2-1]
%%%	\arrow[from=2-2, to=3-2]
%%	\arrow["b"', from=3-2, to=3-1]
%\end{tikzcd}\]
%Every square is Cartesian. By assumption, the product $W\times_XV$ is $(n-1)$-Artin, and thus $U\times_XV$ is $(n-1)$-Artin. This guarantees that $X$ is  $n$-representable. To construct an $(n-1)$-atlas of $X$, we only need to construct an $(n-2)$-atlas of $W$. Then processing by induction, we complete the proof.
%\end{proof}

\subsection{Artin-Lurie's representability theorem}
We generalize  the proof of Artin-Lurie representability for derived Artin stacks \cite[Theorem 7.1.6]{lurie2004derived} and spectral Deligne-Mumford stacks \cite[Theorem 18.3.0.1]{lurie2018spectral} to spectral Artin stacks.\\
Let us first recall the existence of approximate charts by Lurie.
\begin{prop}\cite[Proposition 18.3.1.1]{lurie2018spectral}\label{appchart}
    Let $R$ be an $\Einf$-G-ring, and  let $X\in\stk$. Suppose  there is a morphism $q:X\rightarrow\spec R, $ that satisfies the following conditions:
    \begin{enumerate}
        \item The functor $X$ is infinitesimally cohesive, nilcomplete, and integrable.
        \item The morphism  $q$ is locally almost of finite presentation.
        \item The map $q$ admits a cotangent complex $\LL_{X/R}$.
    \end{enumerate}
Suppose  we are given a field $k$ and a map $f : \spec k\rightarrow X$ that exhibits $k$ as a finitely
generated field extension of some residue field of $R.$ Then the map $f$ factors as a composition
$\spec k \rightarrow  \spec B \rightarrow X$, where $B$ is almost of finite presentation over R and the vector space
$\pi_1(\LL_{B/X}\otimes k)$ vanishes.
\end{prop}
    Let $X\in \stk$ be a sheaf that satisfies conditions $(1)$-$(3)$ above. Let $f:\spec B\rightarrow X$ be a map such that the homotopy $\pi_1(\LL_{B/X}\otimes k)=0$, where $x:\spec k\rightarrow \spec B$ is a $k$ point. Let $B_0$ be the discrete $\Einf$-ring $\pi_0B$. Using the exact sequence $$\LL_{B/X}\otimes_BB_0\rr\LL_{B_0/X}\rr\LL_{B_0/B}$$
    and \cite[Corollary 7.4.3.2]{lurie2017higher}, the cotangent complex  $\pi_1\LL_{B_0/B}=0$, we deduce that $\pi_0(\LL_{B_0/X}\otimes_{B_0}k)=0.$
Thus we might as well assume that $B=B_0$ is $0$-truncated.\\
We can reformulate Lurie's theorem on the refinement of approximate charts.
\begin{prop}\label{refinement}
 Let $\spec k\rightarrow \spec B\stackrel{f}\rightarrow X\rightarrow \spec R$ be as above. Then, after a suitable localization of $B$, there is a finitely presented $\Einf$-algebra $B'$ over $R$, such that there is a factorization $$\spec B\rr\spec B'\rr X,$$ and the map $B'\rightarrow X$ is formally smooth.
\end{prop}
\begin{proof}
Assume that $\pi_j
(\LL_{B/X}\otimes_Bk)=0,j=1,2,...,r$.  Recall that every  complex $C$ over a field $k
 $ is quasi-isomorphic  to a split chain complex $C\simeq \oplus_{j}(\pi_jC)[j]$, the complex $\LL_{B/X}\otimes_Bk$ has a decomposition 
 $$N\rr\LL_{B/X}\otimes_Bk\rr M$$
   where $M$ is $(r+1)$-connective and $N$ is coconnective, and bounded below. Assume that $\pi_{-m}N\neq 0$ and $\tau_{\leq -m-1}N=0.$
   Since $B$ and $X$ are locally almost of finite presentation over $R$, the cotangent complex $\LL_{B/X}$ is almost perfect. Then $\pi_{-m}N\simeq k^{n_{-m}}$ for some $n_{-m}\in \Z_{\geq 0}.$ This defines a map $k^{n_{-m}}[-m]\rightarrow N\rightarrow \LL_{B/X}\otimes_Bk$. After a suitable shrinking of $B$, we can lift it to a map $N_0^{-m}:=B^{n_{-m}}[-m]\rightarrow \LL_{B/X}$. The cofiber $M_0^{-m}$ is ${(-m)}$-connective, and almost  finitely  presented as $B$-modules. Note that we  have that $\pi_{-m+1}(M^{-m}_0\otimes_B\eta)\simeq(\pi_{-m+1}M^{-m}_0)\otimes^{\rm cl}_B\eta$\footnote{We let $\otimes^{\rm cl}$ denote the classical tensor product of modules}, for any $\eta\in{\rm MaxSpec} B$, which is a finite dimensional $\eta$-vector space. Then  shrinking $B$ again, we can construct a map $B^{n_{-m+1}}[-m+1]\rightarrow M_0^{-m},$ lifting the isomorphism $k^{n_{-m+1}}\simeq\pi_{-m+1}(M_0^{-m}\otimes_Bk).$ Let $N_0^{-m+1}$ be the cofiber of the composition 
   $$B^{n_{-m+1}}[-m]\rr M_0^{-m}[-1]\rr N_0^{-m}.$$
Let $M_0^{-m+1}$ be the cofiber of the map $N_0^{-m+1}\rightarrow \LL_{B/X}$. Then by induction, we construct a filtration on $\LL_{B/X},$
$$N_0^{-m}\rr N_0^{-m+1}\rr...\rr N_0^0=N_0\rr\LL_{B/X}$$
such that the cofiber $M_0$ of the map $N_0\rightarrow \LL_{B/X}$ is $(r+1)$-connective.
Because $N_0^{-j}$ is obtained by taking the cofiber  $B^{n_{-j-1}}[-j]\rightarrow N_0^{-j-1}$, we deduce that $N_0$ is the dual of a connective perfect $R$-module.\\
The map $\LL_{B/X}\rightarrow M_0$ exhibits $M_0$ as a $B$-derivation $\LL_{B/R}\rightarrow\LL_{B/X}\rightarrow M_0$. This defines a square-zero extension $B^1\rightarrow B$, lying over $X$, with kernel $M_0[-1]$. We have a factorization
   $$\spec B\rr\spec B^1\rr X.$$
   Using \cite[Theorem 7.4.3.1 and Corollary 7.4.3.2]{lurie2017higher}, the cotangent complex $\LL_{B/B^1}$ is $(r+1)$-connective, and  there is a $(2r+2)$-connective map
   $M_0\otimes_{B^1}B\rightarrow\LL_{B/B^1}.$ Then  the exact sequence 
   $$\LL_{B^1/X}\otimes_{B^1}B\rr\LL_{B/X}\rr\LL_{B/B^1
   }$$
   implies that $\pi_j(\LL_{B^1/X}\otimes_{B^1}k)=0,1\leq j\leq r+1.$ 
  Then replace $B$ by $B^1$, and passing to induction, there is a large enough $n$, and a sequence of maps 
  $$\spec B\rr\spec B^1\rr...\rr \spec B^n\rr X,$$
  such that $\pi_j(\LL_{B^n/X}\otimes_{B^n}k)=0,1\leq j\leq n+r,$ and there is an exact sequence 
  $$N_i\rr\LL_{B_i/X}\rr M_i,$$ such that 
  $N_i$ is dual to a connective perfect $B_i$-module and $M_i$ is $(i+r)$-connective, and there is an equivalence  $N_{i}\otimes_{B^i}B^{i-1}\simeq N_{i-1}.$ Because the construction of $M_i$ and $N_i
  $ requires shrinking $B$ finitely many times for each $i$, we can only construct such a sequence of sufficient large length $n$. We give another construction, which is suitable for the case $n$ arbitrary large. Let $B^{n+1}$ be the square-zero extension of $B^n$ with kernel $M_m[-1]$. By construction, the $B^n$-module  $N_n$ admits a filtration 
  $$N_n^{-m}\rr N_n^{-m+1}\rr...\rr N_n^{0}=N_n$$
where  $N_n^{-m+j}$ is obtained by taking the cofiber of the map $$(B^n)^{\oplus n_{-m+j-1}}[-m]\rr N_{n}^{-m+j-1}$$ where the number $ n_{-m+j-1}$ is  independent of $n$. We want to lift this filtration as a filtered $B^{n+1}$-module. For this, we proceed by induction. We first lift $N_{n}^{-m}\simeq (B^{n})^{n_{-m}}[-m]\rr\LL_{B^n/X}$ as the map $B^{n+1}$-module  $ (B^{n+1})^{n_{-m}}[-m]\rr\LL_{B^{n+1}/X}$. This is possible because the cofiber of the map $\LL_{B^{n+1}/X}\otimes_{B^{n+1}}B^n\rightarrow \LL_{B^{n}/X}$ is $(n+r+1)$-connective. Assume that we already have constructed an lifting 
$$N_{n+1}^{-m}\rr N_{n+1}^{-m+1}\rr...\rr N_{n+1}^{j}\rr\LL_{B^{n+1}/X}$$
We have  the following diagram 
$$\begin{tikzcd}
{(B^{n+1})^{\oplus n_{j}}[j]} \arrow[d] & N_{n+1}^j \arrow[d] &           \\
{(B^n)^{\oplus n_{j}}[j]} \arrow[r]     & N_n^j \arrow[r]     & N_n^{j+1}
\end{tikzcd}
$$
To get a lifting  $(B^{n+1})^{\oplus n_{j}}[j]\rightarrow N_{n+1}^{j+1}$, we have to prove that the map $(B^n)^{\oplus n_{j}}\rightarrow N_n^j$ could be lifted to a map of $B^{n+1}$-modules. It's equivalent to prove that the corresponding homotopy classes in $\pi_jN_{n}^j$ could be lifted to  $\pi_jN_{n+1}^j$. This holds whenever the induced map $\pi_jN_{n+1}^j\rightarrow \pi_jN_{n}^j$ is surjective. This could be deduced from the following exact sequence 
$$N_{n+1}^j\otimes_{B^{n+1}}J\rr N_{n+1}^j\rr N_{n}^j$$
Where $J$ is the fiber $B^{n+1}\rightarrow B^n$, and is $(n+r)$-connective. Therefore, we have that $\pi_{j-1}(N_{n+1}^j\otimes_{B^{n+1}}J)=0$, as long as $n$ is large enough.\\
We constructed a projective system of almost of finitely presented  $\Einf$-algebras over $R$
$$\rr B^n\rr B^{n-1}\rr...\rr B^1\rr B$$
, which are lying over $X$, such that $\tau_{\leq n} B^{n+1}\simeq \tau_{\leq n}B^n.$ Let $B'$ be the inverse limit of the above sequence. Clearly it's almost  finitely presented over $R$. The induced map $\spec B'\rightarrow X$ is formally smooth, because it's easy to see that $\LL_{B'/X}$ is equivalent to the
inverse limit of the system $\{N_{n}\},$ which is dual to a connective perfect $B'$-module by construction.
\end{proof}
Let us consider the following adjoint pair
$$\begin{tikzcd}
\algcn \arrow[r, "\Theta^L", shift left] & \alg^{\Delta} \arrow[l, "\Theta", shift left]
\end{tikzcd}$$
in which the functor $\Theta$ is the forgetful functor that carries an animated ring $A$ to its underlying $\Einf$-ring $\Theta(A)$. Passing to Yoneda embeddings, we can define a functor 
$\Theta_!:\stk^\Delta:=\shv_{\et}(\alg^\Delta,\ani)\rightarrow\stk$ given by the left Kan extension of the composition of functors $$\alg^{\Delta,\rm op}\rr\alg^{\rm cn,op}\stackrel{\spec}\rr\shv_{\et}$$ along the embedding $$\spec:\alg^{\Delta,\rm op}\rr\Fun(\alg^{\Delta},\ani).$$ The functor $\Theta_!$ is left adjoint to the restriction functor $\Theta^*.$  In addition, the pushout $\Theta_*$ is right adjoint to $\Theta^*$.
\begin{rmk}
    The functor $\Theta^*$ preserves affine objects. Indeed, if $A$ is a connective $\Einf$-ring, we have $\Theta^*\spec A=\map_{\algcn}(A,\Theta(-))\simeq\map_{\alg^\Delta}(\Theta^L(A),-)$.
\end{rmk}
\begin{rmk}
    Let $R$ be an animated ring, and let $R^\circ$ be the underlying $\Einf$-ring of $R$. Consider the $\Einf$-ring map $R\rightarrow\pi_0(R)$, this map has a unique factorization $$R\rr \Theta^L(R)^\circ\rr \pi_0(R).$$ Passing to the 
    $0$-th homotopy, we know that $\pi_0(R)$ is a retraction of $\pi_0(\Theta^L(R^\circ)).$
\end{rmk}
\begin{lem}\label{spectoder}
    Consider the adjoint pair
$$\begin{tikzcd}
{\stk^\Delta} \arrow[r, "\Theta_!", shift left] & {\stk} \arrow[l, "\Theta^*", shift left]
\end{tikzcd}$$
induced by the forgetful functor $\alg^\Delta\rightarrow\algcn$, then we have
\begin{enumerate}
\item  $\Theta^*$ preserves locally almost  finitely presented, nilcomplete and infinitesimally cohesive maps;
  \item $\Theta^*$ preserves surjective map;
   \item $\Theta^*$ preserves formally \'etale maps and \'etale maps;
    \item $\Theta^*$ preserves $n$-representable maps;
   
    \item  If $f:X\rightarrow Y$ is $n$-differentially smooth  then $\Theta^*
    f$ is a smooth map of derived stacks;
  
    \item If $f:X\rightarrow Y$ is relative $n$-Artin, then so is $\Theta^*f$.
\end{enumerate}
\end{lem}
\begin{proof}
$(1).$ This is clear because the forgetful functor $\Theta:\alg^\Delta\rightarrow\algcn$  preserves  arbitrary limits and filtered colimits.\\
$(2)$. Let $f:X\rightarrow Y$ be a surjective map. Let $R$ be an animated ring, and let $u:\spec R\rightarrow \Theta^*Y$. By adjointness, this is equivalent to giving a map $\spec R^\circ\rightarrow Y,$ where $R^\circ=\Theta(R)$ is the underlying $\Einf$-ring of $R$. After an \'etale localization  $R^{\circ'}$ of $R^\circ$, there is a lifting $\spec R^{\circ'}\rightarrow X.$ \cite[Theorem 3.4.13]{lurie2004derived} and \cite[Theorem 7.5.4.2]{lurie2017higher}   implies that $R^{\circ'}$ has a unique animated $R$-algebra structure $R'$, such that  $R'$ is \'etale over $R$. This gives rise to a lifting $\spec R'\rightarrow \Theta^*X.$\\
$(3)$. Because of $(1),$ we just only need to consider the formally \'etale maps. Let $X\rightarrow Y$ be a formally \'etale map. Let us consider the following lifting property of derived stacks
$$\begin{tikzcd}
\spec R \arrow[d] \arrow[r]                    & \Theta^*X \arrow[d] \\
\spec(R\oplus I ) \arrow[r] \arrow[ru, dashed] & \Theta^*Y          
\end{tikzcd}$$
This is equivalent to studying the following lifting property of spectral stacks
$$\begin{tikzcd}
\spec R^{\circ} \arrow[d] \arrow[r]                 & X \arrow[d] \\
\spec(R\oplus I)^\circ \arrow[r] \arrow[ru, dashed] & Y          
\end{tikzcd}$$
By the assumption that $X\rightarrow Y$ is formally \'etale, the lifting $\spec(R\oplus I)^\circ\rightarrow X$ is unique. Therefore, the lifting $\spec (R\oplus I)\rightarrow  \Theta^*X$ is unique.\\
$(4)$. If $f:X\rightarrow Y$ is $0$-representable, 
as $\Theta^*$ preserves small limits, and affine objects, we can assume that $Y\simeq \spec R$ is an affine scheme, then $X$ is an algebraic space over $R.$ There is an \'etale cover of spectral stacks $\coprod_j\spec A_j\rightarrow X.$ Then applying $(1)$ and $(2)$,  the pullback $\coprod_j\spec \Theta^L(A_j)\rightarrow \Theta^*X$ is an \'etale cover of $\Theta^*X.$
Assume this holds for $n$-representable maps, then let $g$ be an $(n+1)$-representable map, it's equal to say that the diagonal $\Delta_g$ of $g$ is $n$-representable. Since the functor $\Theta^*$ preserves limits, we know that the functor $\Theta^*g$ is $(n+1)$-representable.\\
$(5)$. Without loss of generality, assume that $Y\simeq \spec A$ is affine. If $f$ is $0$-smooth, then it's a relative algebraic space. Since $X$ is locally equivalent to some $\spec B$, where $A\rightarrow B$ is a smooth map, by \cite[Proposition 11.2.2.1]{lurie2018spectral}, we can localize $B$ suitably, such that there is an \'etale map $B\rightarrow A\{x_1,x_2,...,x_m\}$. Applying the functor $i^*$, we get an \'etale map 
$\Theta^L(B)\rightarrow\Theta^L(A)[x_1,x_2,...,x_m],$
which is equivalent to the map $\Theta^L(A)\rightarrow\Theta^L(B)$ being smooth as an animated ring map. Assume that this holds for $n$-differentially smooth maps. Now let $X$ be $(n+1)$-smooth over $Y$, there is an atlas
$U\rightarrow X$, which is $n$-differentially smooth. By assumption, $\Theta^*U\rightarrow \Theta^*X$ and $\Theta^*U\rightarrow \Theta^*\spec A$ are smooth. Thus $\Theta^*X\rightarrow \Theta^*\spec A$ is smooth.\\
$(6)$ is deduced from $(1)-(5)$.
\end{proof}
\begin{prop}\label{smoothcover}
    If $q:X\rightarrow \spec R$ satisfies those conditions in Proposition\autoref{appchart}, then there is an effective epimorphism $U\rightarrow X,$ such that $U$ is equivalent to a disjoint union of spectral schemes, and formally smooth over $X$.
\end{prop}
\begin{proof}
Let $U$ be the union of all possible $\spec B$, which 
are constructed in Proposition\autoref{refinement}.
    Let $\spec A\rightarrow X$ be an arbitrary map, we will prove that, after an \'etale localization $A'$, there is a lifting $\spec A'\rightarrow U.$
    Note that we just need to prove that there is a  lifting  $\spec \pi_0A'\rightarrow U$ of $\spec \pi_0A'\rightarrow X$. Indeed, suppose that there is a lifting $\spec\pi_0A'\rightarrow U$.  Consider the following commutative diagram
    $$\begin{tikzcd}
   &        &    &       &     & U \arrow[d] \\
\spec \pi_0A' \arrow[r] \arrow[rrrrru] & \spec \tau_{\leq 1}A' \arrow[r] & ... \arrow[r] & \spec \tau_{\leq n}A' \arrow[r] & ... \arrow[r] & X    
\end{tikzcd}
$$
Since $U\rightarrow X$ is formally smooth, applying 
 induction, one can construct  a lifting $\spec\tau_{\leq n }A'\rightarrow U$ for any $n\in\N$, such that the above diagram commutes. Because $X$ is nil-complete, these maps $\{\spec \tau_{\leq n}A'\rightarrow U\}$ determine a unique map $\spec A'\rightarrow U$. For this reason, we assume that $A$ is discrete.\\
Consider the following commutative  diagram
    $$\begin{tikzcd}
\Theta_!\Theta^*U \arrow[r] \arrow[d] & U \arrow[d] \\
\Theta_!\Theta^*X \arrow[r]           & X          
\end{tikzcd}$$
Note that any map $\spec A\rightarrow X$ from a discrete ring $A$ factors through $\Theta_!\Theta^*X$, this induces a map of derived stacks $\spec A\rightarrow \Theta^*X$. The proof of \cite[Lemma 7.2.2]{lurie2004derived} shows that $\Theta^*U\rightarrow \Theta^*X$  is an effective epimorphism in the $\infty$-topos $\stk^\Delta$. We then replace  $A$ with some \'etale algebra  $A'$, such that there is  a lifting $\spec A'\rightarrow \Theta^*U.$ Then we have a lifting $\spec A'\rightarrow U$ using the following commutative diagram:
$$\begin{tikzcd}  & \Theta_!\Theta^*U \arrow[r] \arrow[d] & U \arrow[d] \\
\spec A' \arrow[r] \arrow[ru] & \Theta_!\Theta^*X \arrow[r]           & X          
\end{tikzcd}$$
\end{proof}
Let $\alg_{\leq j}$ be the subcategory of $\algcn$ spanned by $j$-truncated $\Einf$-rings; the 
fully faithful inclusion 
$i:\alg_{\leq j}\rightarrow \algcn$ admits a left adjoint $\tau_{\leq j}$. Therefore, the inclusion functor $i$ preserves limits, and it turns out that the left Kan-extension $i_!:\Fun(\alg_{\leq j},\ani)\rightarrow\Fun(\algcn,\ani)$ of $i$ preserves colimits. In addition, the left adjoint $i^*$ of $i_!$, which is, in fact, the restriction functor, preserves these objects that satisfy \'etale (hyper)-descent. Hence, we obtain an adjoint pair 
\[\begin{tikzcd}
	{\shv_{\et}(\alg_{\leq j},\ani)} & {\stk}
	\arrow["{i_!}", shift left=2, from=1-1, to=1-2]
	\arrow["{i^*}", shift left=2, from=1-2, to=1-1]
\end{tikzcd}\]
We denote by  $\tau_{\leq j}$  the composition $i_!\circ i^*$, and refer to it as the \textit{$j$-th truncation of a sheaf $X$}.\\
Note that the functor $i_!$ is a fully faithful embedding of $\infty$-categories, and  we denote by ${\shv_{\et}(\algcn,\ani)_{\leq j}}$ its essential image, which we refer to as the \textit{$\infty$-category of $j$-truncated sheaves.} We have the following fact:
\begin{lem}\cite[Theorem 18.1.0.2.]{lurie2018spectral}\label{defoDM}
Let  $q:X\rightarrow\spec R$ be a map of sheaves that is infinitesimally cohesive and nilcomplete.  Assume that it has a cotangent complex $\LL_{X/R}$. If the truncation  $\tau_{\leq 0}X$ is equivalent to a $0$-truncated Deligne-Mumford stack over $\pi_0R
$, then $X$ is equivalent to a spectral Deligne-Mumford stack over $R$. In particular, if $\tau_{\leq 0}X$ is equivalent to $\spec A$ for some discrete algebra $A$ over $\pi_0R$, then $X$ is equivalent to $\spec A'$ for some $\Einf$-algebra $A'$ over $R$, such that $\pi_0(A')=A.$  
\end{lem}
\begin{thm}[Artin-Lurie representability]\label{alrep}
    Let $X:\algcn\rightarrow \ani$ be a sheaf, and let $R$ be an $\Einf$-G-ring. Suppose that there is a morphism $q:X\rightarrow \spec R$. Let $n$ be a non-negative integer. Then $X$ is represented by a spectral $n$-Artin stack locally almost of finite presentation if and only if $X$ satisfies the following conditions:
\begin{enumerate}
\item The functor $X$ is nilcomplete, infinitesimally cohesive, and integrable.
\item The functor $X$ admits an almost connective cotangent complex $\LL_{X/R}.$
\item The map $q$ is locally almost of finite presentation, that is, let $A_{i}$ be a filtered system of $m$-truncated $\Einf$-rings over $R$, let $X_R$ be the fiber of the map $q$, then $X_R(\varinjlim A_i)\simeq \varinjlim X_R(A_i).$

  \item  For every discrete  $R$-algebra  $A$, the space $X_R(A)$ is n-truncated.
\end{enumerate}
\end{thm}
\begin{proof}  
    We proceed by induction.  For $n=-2$, i.e., the natural transformation $X\rightarrow \spec R$ is an equivalence after restricting to the subcategory $\alg_{\leq 0}$ of discrete $\Einf$-rings in $\algcn$. By lemma \autoref{defoDM}, we have that the functor $X$ is represented by an affine scheme $\spec A$ for some $A$. Now assume that the theorem  holds for $(n-1)$.  Let $X$ be $n$-truncated restricted to discrete rings. According to lemma \autoref{n-1ton}, we just need to prove there exists a smooth  $(n-1)$-cover $U\rightarrow X$ from an $(n-1)$-Artin stack. Thanks to proposition \autoref{appchart} and Proposition \autoref{smoothcover}, there is a spectral algebraic space $U$, which is locally almost  finitely presented over $R$, and a formally smooth effective epimorphism  $p:U\rightarrow X$. Then the fiber of the natural  map 
    $U\times_XU\rightarrow U\times_RU$ is $(n-1)$-truncated after restricting to $\alg_{\leq 0}$, because it's the base change of the diagonal $\Delta_{X/R}$ along the map $U\times_RU\rightarrow X\times_RX$. By induction assumption, $U\times_X U$ is  an $(n-1)$-Artin stack. We have to  show that the induced map $U\times_XU\rightarrow X$ is a relative $(n-1)$-Artin stack. Let $V=\spec A$ be an affine scheme and consider a map $x:V\rightarrow X$. Shrink $V$ such that $x$ has a lifting $x':V\rightarrow U$. Denote by $\delta$ the induced diagonal map $\delta:V\rightarrow U\times_XU$. We have a commutative diagram 
    \[\begin{tikzcd}
	{U\times_XU\times_XV} & V \\
	{U{\times_XU\times_XU\times_XU}} & {U\times_XU} \\
	{U\times_XU} & X
	\arrow[from=1-1, to=1-2]
	\arrow[from=1-1, to=2-1]
	\arrow["\delta", from=1-2, to=2-2]
	\arrow[from=2-1, to=2-2]
	\arrow[from=2-1, to=3-1]
	\arrow[from=2-2, to=3-2]
	\arrow[from=3-1, to=3-2]
\end{tikzcd}\]
Every square is Cartesian. The sheaf
$U{\times_XU\times_XU\times_XU}$ is an $(n-1)$-Artin stack. But $U\times_XU$ is $(n-1)$-Artin, one has that the fiber product $U\times_XU\times_XV\simeq U{\times_XU\times_XU\times_XU}\times_{(U\times_XU)} V$ is an $(n-1)$-Artin stack. This shows that the map $U\times_XU\rightarrow X$ is  $(n-1)$-Artin.
\end{proof}

\section{Algebraicity of Moduli of log stacks}\label{section6}
As we have established the theory of deformations of $\Einf$-log rings,  constructed the moduli stack of log structures over spectral log stacks, and studied the relevant deformation properties, now we are able to prove the representability theorem of the moduli of log structures. The main result is Theorem\autoref{algebraicity}. We also study the moduli stack of animated log structures over derived stacks, and we have Theorem\autoref{algebraicityani}.
\subsection{Algebraicity}
We prove the algebraicity of the moduli stack of log structures.
\begin{thm}\label{algebraicity}
    Let $\MCS$ be a quasi-compact and quasi-separated spectral Deligne-Mumford stack. Assume that there is an \'etale cover $\spet R\rightarrow\MCS$ from a spectral affine scheme such that $R$ is an $\Einf$-Grothendieck ring. Let $\mu\in\{\rm coh,fin\}$, and let $\MSL$  be a $\mu$-log structure over $\MCS$. Fix a natural number $n\geq 0.$
    Then the stack $\R\llog^{n-{\rm Adm},\mu}_{(\MCS,\MSL)}$ is a $(1+n)$-Artin and locally almost finitely presented over $\MCS$.
\end{thm}
\begin{proof}
    Without loss of generality, we might as well assume that there is an $\Einf$-$\mu$ log ring $(R,L)$, in which $R$ is a Grothendieck $\Einf$-log ring, and there is an equivalence $(\MCS,\MSL)\simeq\spet(R,L)$. We apply Artin-Lurie's representability Theorem\autoref{alrep}.
    We have to check the descent and deformation properties of the map $\R\llog^{n-{\rm Adm},\mu}_{\spet(R,L)}\rightarrow \spec R$.
    They are guaranteed by Theorem\autoref{des}, Theorem\autoref{almostperfect}, Proposition\autoref{infcohesive}, Proposition\autoref{nilcompint}, Proposition\autoref{almostfinite}  Proposition\autoref{1-trunc-2}, Remark\autoref{admdeformation} and Lemma\autoref{admcot}.
 \end{proof}
 We also have the following result.
\begin{cor}
    Under the same assumption as Theorem\autoref{algebraicity}. The Stack $\R\llog^{\rm Adm,\mu}_{(\MCS,\MSL)}$ is locally Artin and locally almost finitely presented over $\MCS$. 
\end{cor}
\subsection{Classical truncations and comparisons}
Let $(\MCS^{\heartsuit},\mathscr L^{\heartsuit})$ be a classical log Deligne-Mumford stack. Denote by $\MCS^\heartsuit_{\et}\simeq \shv_{\et}(\MCS^{\heartsuit},{\rm Set})$ the underlying $1$-topos of $\MCS^\heartsuit$.
The log structure $\MSL^{\heartsuit}\rightarrow\mathcal{O}_{\MCS^\heartsuit}$ should be thought of as a map in the $1$-category $\MCS^\heartsuit_{\et}\otimes_{\rm Set}\mone^{\heartsuit}.$
Let $\MCS$ be the associated spectral Deligne-Mumford stack of $\MCS^\heartsuit$, and let $\MSL\rightarrow\mathcal O_{\MCS}$ be the image of $\mathscr L^{\heartsuit}\rightarrow\mathcal O_{\mathcal S^\heartsuit}$ in $\shv_{\et}(\MCS,\ani)$ under the canonical functor $$\MCS^{\heartsuit}_{\et}\otimes_{\rm Set}\mone^\heartsuit\rr\MCS_{\et}\otimes_{\ani}\mone$$ 
It's not difficult to see that the map $\MSL\rightarrow\mathcal O_{\MCS}$ is a log structure over $\MCS$. This defines a fully faithful embedding 
$$\llog\Sptdm^{\rm cl}\rr\llog\Sptdm.$$
Now assume that $\MCS^{\heartsuit}$ admits an \'etale cover of Grothendieck rings, and $\mathscr L^{\heartsuit}$ is a fine log structure. Define the moduli functor of classical log structures lying over $(\MCS^{\heartsuit},\mathscr L^{\heartsuit})$ as the following:
$$\llog_{(\MCS^{\heartsuit},\mathscr L^{\heartsuit})}:{\mathbf{DM}^{\rm op
}_{/\mathcal S^\heartsuit}}\rr\mathbf{Grpd}$$
which sends a relative classical Deligne-Mumford stack $\X/\MCS^\heartsuit$ to the $1$-groupoid consisting of maps of 
 log Deligne-Mumford stacks $(\X,\mathscr M)\rightarrow(\MCS^{\heartsuit},\mathscr L^{\heartsuit})$. In \cite{olsson2003logarithmic}, Olsson proved that this functor is a classical Artin stack and is locally  finitely presented over $\MCS^{\heartsuit}$ in the classical sense. We denote by $\R\llog_{(\MCS,\MSL)}^\heartsuit$ the classical truncation of the stack $\R\llog^{\rm fin}_{(\MCS,\MSL)}$. By the classical Artin's representability theorem, the functor $\R\llog_{(\MCS,\MSL)}^\heartsuit$ is a classical Artin stack and is locally finitely presented over $\MCS^\heartsuit$.
 We have a canonical map of classical Artin stacks
 $$\theta:\llog_{(\MCS^{\heartsuit},\mathscr L^{\heartsuit})}\rr\R\llog_{(\MCS,\MSL)}^\heartsuit.$$
 \begin{thm}\label{compder-clas}
     The map $\theta$ is an open immersion, whose image is isomorphic to the classical truncation of $\R\llog^{0-{\rm Adm},\rm fin}_{(\MCS,\MSL)}$.
 \end{thm}
 \begin{proof}
     The map $\theta$ is $(-1)$-truncated, and essentially surjective. We must show that it's \'etale. Since we already know that both of the two stacks are locally  finitely presented over $\MCS^\heartsuit$, we just need to show that $\theta$ is formally \'etale. Let $R$ be a classical ring, and let $I\in\Mod_R^{\heartsuit}$. We form a commutative diagram
     $$\begin{tikzcd}
\spec R \arrow[d] \arrow[r, "x"]  & {\llog_{(\MCS^{\heartsuit},\mathscr L^{\heartsuit})}} \arrow[d, "\theta"] \\
\spec (R\oplus I) \arrow[r, "x'"] & {\R\llog^{\heartsuit}_{(\MCS,\mathscr L)}}                               
\end{tikzcd}$$
Without loss of generality, we assume that both $x$ and $x'$ classify affine log structures $\spet(R,M)$ and $\spet(R\oplus I,M')$ respectively, such that $(R\oplus I,M')$ is a square-zero extension of $(R,M)$. Then $M'$ is discrete because we have an exact sequence of $\Einf$-monoids
$$(1+I)\rr M'\rr M$$
This implies that we have a unique lifting $\spec (R\oplus I)\rightarrow\llog_{(\MCS^{\heartsuit},\mathscr L^{\heartsuit})}$ of $x'$, making the commutative diagram as above commute. It turns out that the map $\theta$ is formally \'etale.
 \end{proof}
 \subsection{Moduli of derived log structures}\label{dersetting}
 In this subsection, we consider the moduli stack of derived log stacks. 
\subsubsection{Derived algebraic geometry} Let $\X$ be a derived Deligne-Mumford stack in the sense of \cite{lurie2004derived}, we denote by $\X^\circ$ the underlying spectral Deligne-Mumford stack defined in \cite[2.18]{chough2020brauer}. This gives rise to a functor $(-)^\circ:\ddm\rightarrow\Sptdm$, which carries a \textit{derived Deligne-Mumford stack} to its underlying spectral Deligne-Mumford stack. Here we denote by $\ddm$
 the $\infty$-category of derived Deligne-Mumford stacks.  We have a commutative diagram of $\infty$-categories 
 $$\begin{tikzcd}
\ddm \arrow[d] \arrow[r, "(-)^\circ"]                  & \Sptdm \arrow[d]          \\
{\stk^\Delta} \arrow[r, "\Theta_!"] & {\stk}
\end{tikzcd}$$. 
Both vertical arrows are fully faithful embeddings \cite[Proposition 4.6.2]{lurie2004derived} \cite[Proposition 1.6.4.2]{lurie2018spectral}. Applying Lemma\autoref{spectoder}, the stack $\Theta^*h_{\X}$ of a functor $h_\X$ which is represented by a spectral Deligne-Mumford stack $\X$ is also representable by some derived Deligne-Mumford stacks. In other words, the essential image of  functor $$\Sptdm\rr\stk\stackrel{\Theta^*}\rr \stk^\Delta$$ is  landed in $\ddm$, thus it has a unique factorization $(-)^\Delta:\Sptdm\rightarrow\ddm$, which is  a right adjoint of $(-)^\circ$.
\subsubsection{Derived log stacks}
  \begin{defn}
  The strict \'etale map of animated log rings defines a sub-canonical topology on $\llog^\Delta$. We denote by $\llog^{\Delta}_{\set}$ the corresponding site.
  A \textit{derived charted log stack} is a sheaf over the site $\llog^\Delta_{\set}.$ Denote by $\shv^\Delta_{\set}:=\shv(\llog^{\Delta,\rm op}_{\set},\ani)$  the $\infty$-category of derived charted log stacks.
  \end{defn}
  \begin{defn}
    Let $\spec(A,M):\llog^\Delta_{\set}\rightarrow\ani$ be the functor $\map((A,M),-).$ We call it the derived charted log-affine scheme associated with $(A,M)$.
\end{defn}
\begin{defn}
    Let $(A,M)$ be an animated log ring. The \textit{derived log-affine scheme}  $\spet(A,M)$ associated with $(A,M)$ is a pair $(\spet A,\mathscr M),$ in which:
    \begin{enumerate}
        \item $\spet A$ is the \'etale spectrum of $A$ defined in \cite[Section 4.3]{lurie2004derived};
        \item $\mathscr M$ is the \'etale sheafification  of the following presheaf
        $$\alg^\Delta_A\ni A'\mapsto M^a,$$ where $M^a$ is the animated monoid of the logification of the pre-animated log ring  $(A',M)$. 
    \end{enumerate}
\end{defn}
\begin{defn}
\begin{enumerate}
    \item 
    A \textit{derived prelog Deligne-Mumford stack} is a triple $(\mathcal X,\mathscr M,\alpha),$ where $\mathcal X$ is a quasi-compact and  quasi-separated  derived  Deligne-Mumford stack, $\mathscr M$ is an \'etale sheaf of animated monoids over the underlying $\infty$-topos of $\mathcal X,$ and $\alpha:\mathscr M\rightarrow\mathcal O_\mathcal X$ is a morphism of animated monoids.\\
     We denote $\plog\ddm$ as the $\infty$-category of derived log Deligne-Mumford stacks.
\item 
    A \textit{derived log Deligne-Mumford stack} $(\mathcal X,\mathscr M,\alpha)$ is a derived prelog Deligne-Mumford stack, such that the induced morphism $\mathscr M\times_{\OO_\X}GL_1(\OO_\X)\rightarrow GL_1(\OO_\X)$ is an equivalence.\\
    We denote $\llog\ddm$ as the $\infty$-category of derived log Deligne-Mumford stacks.
    \end{enumerate}
\end{defn}
\begin{defn}
Let $(\X,\mathscr M,\alpha)$ be a derived log Deligne-Mumford stack, we say that $(\X,\mathscr M,\alpha)$ is:
\begin{enumerate}
    \item \textit{Quasi-coherent}, if there is a cover $U\rightarrow \X,$ such that $U\simeq\coprod_j\spet A_j$ is a finite disjoint union of derived affine schemes, and the restriction $(U,\mathscr M|_U)\simeq\coprod_j\spet(A_j,M_j)$ is a finite disjoint union of derived log-affine schemes;
    \item \textit{Coherent}, if it's quasi-coherent, and there is a choice of $\{(A_j,M_j)\}$, which appears in the assertion (1), such that for any $j$, $(A_j,M_j)$ is obtained from an animated prelog ring $(A_j,M_j')$ satisfying the property that $\mathbb Z[M_j']$ is almost finitely presented over $\mathbb Z$.
    \item \textit{Fine}, if it's coherent and there is a choice of $\{(A_j,M_j)\}$, which appears in assertion (2), such that for any $j$, the commutative ring $\pi_0\mathbb Z[M_j]=\Z[\pi_0M_j]$ is a finitely generated  integral $\Z$-algebra.
\end{enumerate}   
\end{defn}
\begin{defn}
   Let $(\X,\mathscr M)$ be a derived log Deligne-Mumford stack. We say that $(\X,\mathscr M)$ is $n$-\textit{admissible}, if there is an \'etale cover $\coprod^N_{j=1}\spet R_j\rightarrow\X$ of derived Deligne-Mumford stacks, such that  the pullback 
    $\mathscr M|_{\spet R_j}$ are given by 
    $n$-admissible $\Einf$-log rings. We say that it is admissible if it is $n$-admissible for some $n<+\infty$.
\end{defn}
\begin{defn}\label{rlogdefn4}
    Let $(\MCS,\MSL)$ be a derived log Deligne-Mumford stack. We define the following functor 
    $$\R\llog^\Delta_{(\MCS,\MSL)}:\llog\mathbf{DM}^{\rm op}_{/\MCS}\rr\ani,$$
    which sends a quasi-compact and quasi-separated derived Deligne-Mumford stack $\X$ to the anima consisting  of maps of derived log Deligne-Mumford stacks $(\X,\mathscr M)\rightarrow(\MCS,\MSL).$ \\
    Let $\mu\in\{\rm qcoh,coh,fin\}$, and let $(\MCS,\MSL)$ be a $\mu$-derived log Deligne-Mumford stack.
    We denote by $$\R\llog_{(\MCS,\MSL)}^{\mu,\Delta}, \R\llog_{(\MCS,\MSL)}^{n-{\rm Adm},\mu,\Delta}, {\rm{and } {\ } } \R\llog_{(\MCS,\MSL)}^{{\rm Adm},\mu,\Delta}$$ the subfunctors of $\R\llog^\Delta_{(\MCS,\MSL)}$ spanned  by $\mu$-log structures, $n$-admissible $\mu$-log structures, and admissible $\mu$-log structures, respectively.
\end{defn}
The same arguments in \autoref{logstacks} guaranty that the following facts hold.
\begin{thm}\label{algebraicityani}
    \begin{enumerate}
        \item Let $\mu\in\{\rm qcoh,coh,fin\}$. The functors $\R\llog_{(\MCS,\MSL)}^{\mu,\Delta}$, $\R\llog_{(\MCS,\MSL)}^{n-{\rm Adm},\mu,\Delta}$, and $\R\llog_{(\MCS,\MSL)}^{{\rm Adm},\mu,\Delta}$ satisfy descent with respect to the \'etale topology.
        \item Let $\mu\in\{\rm qcoh,coh,fin\}$. Let $(\MCS,\MSL)$ be a $\mu$-derived log Deligne-Mumford stack. Then $\R\llog_{(\MCS,\MSL)}^{\mu,\Delta}$ admits a $(-1)$-connective cotangent complex.
        \item Let $\mu\in\{\rm coh,fin\}$. Let $(\MCS,\MSL)$ be a $\mu$-derived log Deligne-Mumford stack. Assume that $\MCS$ is covered by derived Grothendieck rings. Then the functor  $\R\llog_{(\MCS,\MSL)}^{n-{\rm Adm},\mu,\Delta}$ is derived $(1+n)$-Artin, and $\R\llog_{(\MCS,\MSL)}^{{\rm Adm},\mu,\Delta}$ is locally Artin.
        \item Let $f:\spet(A,M)\rightarrow\spet(B,N)$ be a map induced from the  derived log ring map $(B,N)\rightarrow(A,M)$, then there is a canonical equivalence $\LL_{\spet A/\R\llog^{\mu,\Delta}_{\spet(B,N)}}\simeq\LL^{G,\Delta}_{(A,M)/(B,N)}.$
    \end{enumerate}
\end{thm}
\subsection{Moduli of general log stacks}
The functor $\mathcal{F}:\Sptdm^{\rm op}\rightarrow \cat$ associated with the 
Cartesian projection $\pi:\llog\Sptdm\rightarrow\Sptdm$
 extends to a functor 
 $\mathcal{ F}':\stk^{\rm op}\rightarrow \cat$ via right Kan extension along the inclusion $\Sptdm\subset\stk$. Then, passing to the Grothendieck construction, we  obtain a Cartesian fibration $\pi:\llog\stk\rightarrow\stk$. We refer to the $\infty$-category $\llog\stk$ as the $\infty$-category of log stacks. We say a spectral log stack $X$ is quasi-coherent if, for any map $f:\spec A\rightarrow X$, the pullback $f^*X$ is  quasi-coherent.
\begin{rmk}
A quasi-coherent spectral log  stack is given by a map $X\rightarrow\R\llog^{\rm qcoh}$, where $X$ is an stack.
\end{rmk}
\begin{rmk}\label{genlogaff}
    The subcategory $\llog\stk^{\rm qcoh}$ of $\llog\stk$ consisting of quasi-coherent log stacks is generated by log-affine objects under colimits of strict morphisms.
\end{rmk}
\begin{exmp} We have the following examples of log stacks:
\begin{enumerate}
    \item 
 The identity map on $\R\llog^{\rm qcoh}$ determines a log structure $\mathscr M^{\rm univ}$ on $\R\llog^{\rm qcoh}$. We call this log structure the \textit{universal log structure} on $\R\llog^{\rm qcoh}$.
\item Fix an $\Einf$-ring $R$. The  map $\spec R[M]\rightarrow \R\llog^{\rm qcoh}_R$ classifying $\spet(R[M],M)$, which factors through the quotient stack $$\spec R[M]/\spec R[M^{\rm gp}]\rr \R\llog^{\rm qcoh}_R.$$
Indeed, the composition $\spec R[M^{\rm gp}]\rightarrow\spec R[M]\rightarrow\R\llog_R^{\rm qcoh}$ classifies a trivial log structure, as the  logification of the $\Einf$-prelog ring $(R[M^{\rm gp}],M)$ agrees with $(R[M^{\rm gp}],GL_1(R[M^{\rm gp}])$, and we have a commutative diagram 
$$\begin{tikzcd}
... \arrow[r, shift left=2] \arrow[r, shift right=2] \arrow[r] & {\spec R[M]\times \spec R[M^{\rm gp}]} \arrow[r, shift left] \arrow[r, shift right] \arrow[l, shift left] \arrow[l, shift right] & {\spec R[M]} \arrow[l] \arrow[r] & \R\llog^{\rm qcoh}_R
\end{tikzcd}$$
Passing to the colimit, we get the desired map $\spec R[M]/\spec R[M^{\rm gp}]\rightarrow \R\llog^{\rm qcoh}_R$.
This also defines a log structure on $\spec R[M]/\spec R[M^{\rm gp}]$. 
\item  Let $\spet (A,M)$ be a log-affine spectral log scheme, where $A$ is 
$p$-complete. Then the natural inclusion $\spf A\rightarrow\spec A$ defines a log structure on $\spf A$. This  coincides with the classical definition of log formal schemes in \cite[Appendix A]{koshikawa2022logarithmic}. 
 \end{enumerate}
\end{exmp}
Let $S$ be a quasi-coherent spectral  log stack. one can define the functor $\R\llog^{\rm qcoh}_S$ as the following colimit $$\R\llog^{\rm qcoh}_S:=\colim_{(\mathcal{U},f)}\R\llog^{\rm qcoh}_{f^*X},$$ where the index $(\mathcal{ U},f)$ runs over the $\infty$-category of spectral  Deligne-Mumford stacks lying over $\underline S$, and the colimit is formed in the $\infty$-category $\stk$. Then we obtain a natural map $\R\llog^{\rm qcoh}_S\rightarrow\underline S$ that classifies quasi-coherent log stacks lying over $S$. 
\begin{rmk}[Representability for general bases]
    Let $S\in\stk$ be an arbitrary stack. We define $\R\llog_S$ as the colimit $\colim_{\mathcal{ U}}\R\llog^\mu_{\mathcal U}$.  We  have natural equivalences $\R\llog^\mu_\mathcal{U}\simeq\R\llog^\mu\times\mathcal{U}$ for all $\mathcal{ U}\in\Sptdm$ by  definition,  and therefore we have $\R\llog^\mu_S\simeq \R\llog^\mu\times S$. In particular, the stack $\R\llog^{\rm Adm}_{S}$ is  locally Artin  over $S$, and relatively of finitely almost presented.
\end{rmk}
\section{($p$-typical) Infinite root  stacks}\label{section7}
The aim of this section is to develop a derived version of the theory of infinite root  stacks, as in \cite{Borne_2012} and \cite{Talpo_2018}. Our construction generalizes the original definition of infinite root  stacks to all quasi-coherent stacks and coincides with the definition of the classical ones in  certain cases, as provided by Remark\autoref{258789} and  Proposition\autoref{localroot}.\\
In this section, we will first give the definition of infinite root  stacks, and then we will study the associated geometric  and functorial properties. Our main result is Theorem\autoref{123256456}, which states that the $p$-completed cotangent complexes of certain  quasi-coherent spectral log Deligne-Mumford stacks could be recovered from those of associated infinite root  stacks. This generalizes the result of Binda-Lundemo-Merici-Park in \cite[Theorem 4.12]{binda2024logarithmicprismaticcohomologymotivic} and \cite[Lemma 3.18]{binda2024logarithmictcinfiniteroot} to the quasi-coherent case.\\
All constructions and results in this section are suitable for both $\Einf$- and derived contexts.
\begin{cov}
Throughout this section, we will work with the \textbf{fpqc topology}. All stacks are assumed to satisfy fpqc descent, and all operations are performed in the $\infty$-category $\stk_{\rm fpqc}$ of fpqc stacks. In particular, we denote by $\R\llog_{S}^{\rm qcoh}$ the moduli stack of quasi-coherent log structures defined using the  fpqc topology, which
is an fpqc sheaf by Remark\autoref{fpqcrlog}. 
\end{cov}
\begin{cau}
   The reason why we use the fpqc topology to define the infinite root stacks is that our infinite root stack $\sqrt[\infty]{(X,\mathscr M)}$ of a log stack $(X,\mathscr M)$ is defined as a inverse limit of a tower of  \'etale surjective maps to the underlying stack $X$. As the fpqc topoi are \textit{replete topoi}, the surjective maps are stable under sequential inverse limits, its more convenient to work in the fpqc topology than that in the \'etale topology for studying the infinite root stacks. However, we have to point out that, in the fpqc topology, the infinitesimal  properties of the moduli stack of quasi-coherent structures are  probably  not as good as those in the \'etale topology, as \textbf{the nilcompleteness and integrability are missing}.
\end{cau}
\subsection{Definitions and basic geometric properties}\label{rootstkdefprop}
For simplicity, throughout this section, we fix a prime number $p$, and a $p$-complete $\Einf$-ring $R$. We work with the stack $\R\llog_R^{\rm qcoh}$, and its $p$-adic completion $\logfr=\R\llog^{\rm qcoh}_R\times_{\spec R}\spf R$. 
\begin{defn}\label{qazwsxedc}
Denote by $\R\llog_R^{\wedge}$ the fiber product $\R\llog_R^{\rm qcoh}\times_{\spec R}\spf R$. The \textit{universal ({$p$-typical}) infinite root  stack} $\rtr$ is defined as the inverse limit of the following sequence
$$...\rr\R\llog_R^{\wedge}\stackrel{p}\rr\R\llog_R^{\wedge}\stackrel{p}\rr...\stackrel{p}\rr \R\llog_R^{\wedge}.$$
The  \textit{($p$-typical) infinite root  stack} $\sqrt[\infty]{(X,\mathscr M)}$ of a quasi-coherent spectral log stack $(X,\mathscr M)$ is the fiber product $X\times_{\logfr}\rtr$. Similarly, one can define the \textit{$n$-th root stack} $\sqrt[n]{(X,\mathscr M)}$ as the following fiber product
$$\begin{tikzcd}
{\sqrt[n]{(X,\mathscr M)}} \arrow[d] \arrow[r] & \logfr \arrow[d, "p^n"] \\
X \arrow[r]                                    & \logfr                 
\end{tikzcd}$$
We have $\sqrt[\infty]{(X,\mathscr M)}\simeq\varprojlim_n\sqrt[n]{(X,\mathscr M)}$. 
\end{defn}
\begin{exmp}
    If $X$ is a spectral log stack with trivial log structure, then $\sqrt[\infty]{X}\simeq X$.
    \end{exmp}
    \begin{rmk}\label{ptwist}
        Let $(X,\mathscr M)$ be a quasi-coherent spectral log stack. We denote by $(X,p^n\mathscr M)$ the spectral log stack which is classified by the map 
        $X\stackrel{\mathscr M}\rightarrow\logfr\stackrel{p^n}\rightarrow\logfr$. This gives rise to a concrete description of the infinite root  stack $\sqrt[\infty]{(X,\mathscr M)}$. The  mapping space $\map(U,\sqrt[\infty]{(X,\mathscr M)})$ from a stack $U$
        is equivalent to the anima consisting of maps in $\llog\stk$ as follows
        $$\begin{tikzcd}
{(U,\mathscr N)} \arrow[d] \arrow[r] & {(U,p\mathscr N)} \arrow[d] \arrow[r] & {(U,p^2\mathscr N)} \arrow[d] \arrow[r] & ... \arrow[r] & {(U,p^n\mathscr N_2)} \arrow[d] \arrow[r] & ... \\
{(X,\mathscr M)} \arrow[r]             & {(X,\mathscr M)} \arrow[r]              & {(X,\mathscr M)} \arrow[r]                & ... \arrow[r] & {(X,\mathscr M)} \arrow[r]                & ...
\end{tikzcd}$$
such that the left vertical map $(U,\mathscr N)\rightarrow(X,\mathscr M)$ is strict.
    \end{rmk}
\subsubsection{Basic properties}
    By the definition of $\rtr$, it admits a surjective map from the  union of all $${\spf R[M[1/p]]^\wedge_p}$$
    where $M$ runs over all $\Einf$-monoids. Indeed, if we have a map $\spec A\rightarrow \rtr$ that is lying over $\spf R[M]^\wedge_p\rightarrow\logfr$, this is equivalent to providing a map $\spec A\rightarrow\spf R[M]^\wedge_p\times_{\logfr}\rtr\simeq\sqrt[\infty]{\spf R[M]^\wedge_p}$. Unwinding the definition of  infinite root  stacks, there is a sequence of fpqc maps of $\Einf$-rings $$A\rr A_1\rr A_2\rr...\rr A_n\rr...,$$together with a map of the projective system of stacks
    $$\begin{tikzcd}
... \arrow[r] & \spec A_{n+1} \arrow[d] \arrow[r] & \spec A_n \arrow[r] \arrow[d] & \spec A_{n-1} \arrow[d] \arrow[r] & ... \\
... \arrow[r] & \logfr \arrow[r, "p"]             & \logfr \arrow[r, "p"]         & \logfr \arrow[r]                  & ...
\end{tikzcd}$$
such that each vertical map classifies a log-affine log structure $\spec A_n\rightarrow \spf R[1/p^nM]^\wedge_p$, and is compatible with each other. Passing to the limit, we have a map $\spec A_\infty:=\spec(\varinjlim_n A_n)\rightarrow\spf R[M[1/p]]^\wedge_p$. As faithfully flat algebras are stable under sequential colimits, $\spec A_\infty\rightarrow\spec A$ is a flat cover, and we obtain the desired result.
\begin{prop}
    If $R$ is discrete with bounded $p^\infty$-torsions, then $\rtr$ is classical\footnote{We say a stack $X$ is classical if it belongs to the essential image of the fully faithful embedding $$i_!:\shv_{\rm fpqc}(\alg^\heartsuit,\ani)\rightarrow\stk=\shv_{\rm fpqc}(\algcn,\ani).$$}, and the natural map $$\iota:\sqrt[\infty]{\R\llog_R^{\wedge,\heartsuit}}\rr\rtr$$
    is an equivalence.
\end{prop}
\begin{proof}
By  the argument above, we can form a cover of stacks 
$$\pi:\coprod_{M}\spf R[M[1/p]]^\wedge_p\rr \rtr.$$
The \cite[Corollary 2.10]{antieau2024sphericalwittvectorsintegral} guaranties that the $\Einf$-ring $R[M[1/p]]^\wedge_p\simeq R[\pi_0M[1/p]]^\wedge_p$ is discrete. This holds because $R$ has bounded $p^\infty$-torsion and then follows from \cite[Lemma 4.7]{bhatt2019topological}. Consider the subcategory $\mathbf{FMon}_R$ of $\stk_{/\rtr}$ spanned by such canonical maps $\chi_M:\spf R[\pi_0M[1/p]]_p^\wedge\rightarrow\rtr$.  The map $\chi_M$ has a canonical factorization 
$$\spf R[\pi_0M[1/p]]_p^\wedge\rr\sqrt[\infty]{\R\llog_R^{\wedge,\heartsuit}}\stackrel{\iota}\rr\rtr.$$
Passing to the colimit, we have the  
$$\begin{tikzcd}
{\colim_{\mathbf{FMon}_R}\spf R[\pi_0M[1/p]]_p^\wedge} \arrow[r] \arrow[rr, "\simeq"', bend right] & {\sqrt[\infty]{\R\llog_R^{\wedge,\heartsuit}}} \arrow[r, "\iota"] & \rtr
\end{tikzcd}$$
This follows that $\rtr$ is classical, and the map $\iota$ is surjective. Since both sides of $\iota$ are classical stacks, the equivalence of $\iota$ is tested by discrete rings. Then it follows.
\end{proof}
\begin{rmk}[Comparing with classical infinite root  stacks]\label{258789}
If $(A,M)^a$ is a classical fine and saturated log ring, where $M$ is a fine and saturated monoid. Assuming that $A$ is a $p$-complete algebra over $R$, and $R[M]\rightarrow A$ is $p$-completely faithfully flat (this is equivalent to the induced map $\spf A\rightarrow \R\llog_R^{\wedge,\heartsuit}$ being $p$-completely flat, see \cite[Corollary 5.31]{olsson2003logarithmic}). Then our definition of infinite root  stack $\sqrt[\infty]{\spf(A,M)}$ coincides with the classical definition, using the local description of classical infinite root  stacks established in \cite[Section 3]{Talpo_2018} or Proposition\autoref{localroot}. See also \cite[Section 13]{diao2024logarithmicarminfcohomology} and \cite[Section 3.10]{binda2024logarithmictcinfiniteroot}.
\end{rmk}
\subsubsection{Infinite root  stacks of toric algebras}
Let $M$ be an $\Einf$-monoid. The toric log structure on $\spf R[M]^\wedge_p$ is the log structure induced from the unit map $\alpha:M\rightarrow \Omega^\infty R[M]$. We denote by $\mathbb A^{\rm Log}_M$ the resulting spectral log formal scheme and refer to $\mathbb A_M$ as the underlying spectral formal scheme of it. We also consider the perfected spectral toric formal scheme
$$\mathbb A_M^\infty:=\varprojlim_n\spf R[1/p^nM]^\wedge_p\simeq\spf R[M[1/p]]^\wedge_p\simeq\spf R[\pi_0M[1/p]]^\wedge_p$$
of $\mathbb A_M$. The map $\alpha$ also defines a log structure on $\mathbb A_M^{\infty}$, and we let $\mathbb A_M^{\infty,{\rm Log}}$ be the resulting spectral log formal scheme.
\begin{rmk}
    We observe that if $\pi_0M\simeq\pi_0N$, then $\mathbb A_M^\infty\simeq\mathbb A^\infty_N$. In particular, we have 
    $\mathbb A^{1,\infty}_{\spf R}\simeq \mathbb A^{1,\flat,\infty}_{\spf R}$. Here $\mathbb A_{\spf R}^{1,\flat}$ is the flat affine line $\mathbb A^{1,\flat}\times\spf R$ over $\spf 
    R$.
\end{rmk}
The functor $\map_{\mon}(M,GL_1(-)):\algcn_R\rightarrow\mon^{\rm gp}, A\mapsto\map(M,GL_1(A))$ is an fpqc hypersheaf of $\Einf$-groups over $R$. Let $n$ be a natural number, and let $\mu^M_{p^n}$ be the fiber of the $p^n$-th power self-map of $\map_{\mon}(M,GL_1(-))$. Thus, it's an $\Einf$-group stack over $R$. Let $\mu^M_{p^\infty}:=\varprojlim_n\mu^M_{p^n}$.
\begin{rmk}
 We observe that   $\mu^{\coprod_{n\geq 0}B\Sigma_n}_{p^n}\simeq\mu_{p^n}=\spf R\{x\}^\wedge_p/(x^{p^n}-1)$ is a spectral  group formal scheme. As the functor $\mon^{\rm op}\rightarrow\stk_{/\spf R}, M\mapsto \mu^M_{p^n}$ is limit-preserving, and $\mon$ is generated by $\coprod_{n\geq 0}B\Sigma_n$ under colimits, $\mu^M_{p^n}$ is an affine $\Einf$-group formal scheme.\\
 One can show that $\mu^M_{p^n}\simeq\spf R[(1/p^nM)^{\rm gp}/M^{\rm gp}]^\wedge_p$ and $\mu^M_{p^\infty}\simeq\spf R[M[1/p]^{\rm gp}/M^{\rm gp}]^\wedge_p$.
\end{rmk}

\begin{prop}\label{localroot}
Let $V=\spf R[M]^\wedge_p$ be a toric spectral formal scheme equipped with the toric log structure, and let $V^\infty$ be $\spf R[M[1/p]]^\wedge_p$. Then the map $\pi:V^\infty\rightarrow V$ admits a functorial lifting $\pi^\infty:V^\infty\rightarrow\sqrt[\infty]{V}$. If we equip $V^\infty$ with the natural $\mu_{p^\infty}$-action and equip $\sqrt[\infty]{V}$ with the trivial $\mu^M_{p^\infty}$-action,
the map $\pi^\infty$ is $\mu^M_{p^{\infty}}$-equivariant, and the induced map $V^\infty/{\mu^M_{p^\infty}}\rightarrow \sqrt[\infty]{V}$ is an equivalence.
\end{prop}
\begin{proof}
Let $A$ be an algebra lying over $\spf R$. Unwinding the definition of infinite root  stacks, after a suitable flat localization of $A$, a section $s=(f_0,\{f_n\})\in\sqrt[\infty]{V}(A)\simeq V(A)\times_{\logfr(A)}\rtr(A)$  is   given by   a collection of $R$-linear $p$-complete  maps $f_n:R[1/p^nM]^\wedge_p\rightarrow A$, $n=1,2,...,\infty$, which are compatible with transition maps
    $$R[M]^\wedge_p\hookrightarrow R[1/pM]]^\wedge_p\hookrightarrow...\hookrightarrow R[1/p^nM]^\wedge_p\hookrightarrow...$$ Taking the limit, we get a map $f_\infty:=\lim_mf_n:R[M[1/p]]^\wedge_p\rightarrow A$. In this case, we refer to $f_\infty$ as an infinite root  of $f_0$.
We summarize it as the following commutative diagram 
$$\begin{tikzcd}
\spec A \arrow[r, "s", dashed] \arrow[rd, "f_0"'] \arrow[rr, "\{f_n\}", bend left] & {\sqrt[\infty]{V}} \arrow[r] \arrow[d] & \rtr \arrow[d] \\ & V \arrow[r]                            & \logfr        
\end{tikzcd}$$
The limit map $f_\infty$ determines a factorization  $$s:\spec A\stackrel{f_\infty}\rr V^\infty\stackrel{\pi^\infty}\rr\sqrt[\infty]{V}$$ 
of $s$, lying over $V$. Note that the definition of $\pi^\infty$ is independent of the  choice of the $\Einf$-ring $A$ and the map $s$. The map $\pi^\infty$ is a surjective map of  fpqc stacks. It's easy to see that  $\pi^\infty$ is $\mu^M_{p^\infty}$-equivariant. To show that it's a quotient map, we consider its fiber over an $\Einf$-ring $B$ lying over $\spf R$. Let $t\in\sqrt[\infty]{V}(B)$, which corresponds to a pair $(g_0:R[M]^\wedge_p\rightarrow B,\{g_n:R[1/p^nM]^\wedge_p\rightarrow B\})$. Denote by $g_\infty:=\lim_n g_n$ the associated infinite root  of $g_0$. We observe that the fiber of $\pi^\infty$ lying over $t$ is  equivalent to the following animae
\begin{align*}
&\varprojlim_n{\rm fib}(\map_R(R[1/p^nM]^\wedge_p,B)\rr \map_R(R[M]^\wedge_p,B))\\
\simeq &{\rm fib}(\varprojlim_n \map_R(R[1/p^nM]^\wedge_p,B)\rr \map_R(R[M]^\wedge_p,B))\\
\simeq&{\rm fib}(\map_R(R[M[1/p]]^\wedge_p,B)\rr \map_R(R[M]^\wedge_p,B))\\
\simeq&{\rm fib}(\map_{\mon}(M[1/p],\Omega^\infty B)\rr\map_{\mon}(M,\Omega^\infty B))
. \end{align*}
In other words, it's equivalent to $\mu^M_{p^\infty}(B)$.
\end{proof}

\subsection{Functoriality of infinite root  stacks}
We will prove that the construction of infinite root  stacks is functorial in this subsection.
\begin{thm}\label{FUNCROOTSTK}
   There is  a limit-preserving functor 
    $$\gamma:\llog\stk_{/\spf R}\rr\stk_{/\spf R},$$
    such that for any  quasi-coherent spectral log stack $X$, the stack $\gamma(X)$ coincides with the infinite root  stack $\sqrt[\infty]{X}$ of $X.$
   
\end{thm}
We will give a proof of Theorem\autoref{FUNCROOTSTK} later.\\
Let $A$ be an  $\Einf$-algebra lying over $\spf R$. For any $n\geq 0$, the $p$-th power map $p^n\Omega^\infty A\rightarrow \Omega^\infty A$ gives rise to another $p$-complete $\Einf$-prelog ring $(A,p^n\Omega^\infty A)$. Then we  can form the following sequence 
$$\begin{tikzcd}
... \arrow[r] & {(A,p^n\Omega^\infty A)} \arrow[r] & {(A,p^{n-1}\Omega^\infty A)} \arrow[r] & ... \arrow[r] & {(A,p\Omega^\infty A)} \arrow[r] & {(A,\Omega^\infty A)}
\end{tikzcd}$$
Passing to the corresponding spectral log schemes, we define:
$$(\spet A)^\sharp:=\colim_{n\geq 0}\spet(A,p^n\Omega^\infty A).$$
This defines a functor 
$$(-)_0^\sharp:\alg^{{\rm cn, op}}_{/\spf R}\rr\llog\stk_{/\spf R}.$$
\begin{rmk}
We define $p^\infty \Omega^\infty A$ as the limit $\lim_np^n \Omega^\infty A$. We have a natural map $(\spet A)^\sharp\rightarrow\spet(A,p^\infty \Omega^\infty A)$. This map is not an isomorphism in general. For example, let $A=\Z_p$. The log structure of the log stack  $\spet(\Z_p,p^\infty \N )$ is trivial, while $(\spet\Z_p)^\sharp$ is not quasi-coherent.
\end{rmk}
Taking the left Kan extension of  $(-)^\sharp$ along the Yoneda inclusion $\spet(-):\alg^{{\rm cn, op}}_{/\spf R}\rightarrow\stk_{/\spf R}$, we have the following  functor 
$$(-)^\sharp:\stk_{/\spf R}\rr\llog\stk_{/\spf R}.$$
It's easy to see that we have the following result.
\begin{lem}
    The functor $(-)^\sharp$ preserves colimits.
\end{lem}
\begin{proof}
    Taking the left Kan extension of $(-)_0^\sharp:\alg_{/\spf R}^{\rm cn,op}\rightarrow\llog\stk_{/\spf R}$ along the Yoneda inclusion $\alg_{/\spf R}^{\rm cn}\rightarrow\Fun(\alg_{/\spf R}^{\rm cn},\ani)$, we get a colimit-preserving functor $(-)_1^\sharp:\Fun(\alg_{/\spf R}^{\rm cn},\ani)\rightarrow\llog\stk_{/\spf R}$, and $(-)^\sharp$ is equivalent to the restriction of $(-)_1^\sharp$ on its subcategory of sheaf objects. Applying \cite[ Proposition 1.3.1.7]{lurie2018spectral}, we need to show that the functor $(-)_0^\sharp$ has  fpqc codescent, which is obvious.
\end{proof}
\begin{proof}[Proof of Theorem\autoref{FUNCROOTSTK}]
As $(-)^\sharp$ is colimit-preserving, it's right adjointable. We denote by $\gamma:\llog\stk_{/\spf R}\rightarrow\stk_{/\spf R}$ its right adjunction. Let $(X,\mathscr N)$ be a quasi-coherentspectral log stack lying over $\spf R$. We have the following equivalences
\begin{align*}
    \map_{\stk_{/\spf R}}(\spet A,\gamma(X,\mathscr  N))&\simeq\map_{\llog\stk_{/\spf R}}((\spet A)^\sharp,(X,\mathscr  N))\\
    &\simeq\varprojlim_n\map_{\llog\stk_{/\spf R}}(\spet (A,p^n\Omega^\infty A),(X,\mathscr N)).
    \end{align*}
The mapping space
$\lim_n\map_{\llog\stk_{/\spf R}}(\spet (A,p^n\Omega^\infty A),(X,\mathscr N))$ is equivalent to the anima consisting of the following diagrams in $\llog\stk_{/\spf R}$
$$\begin{tikzcd}
                &                                                                      & {(X,\mathscr N)}        &                                                                   &                                                                    \\
... \arrow[rru] & {\spet(A,p^n\Omega^\infty\mathcal O_A)} \arrow[ru, "f_n"'] \arrow[l] & ... \arrow[u] \arrow[l] & {\spet(A,p\Omega^\infty\mathcal O_A)} \arrow[lu, "f_1"] \arrow[l] & {\spet(A,\Omega^\infty\mathcal O_A)} \arrow[llu, "f_0"'] \arrow[l]
\end{tikzcd}$$
Here we use the notation given in Remark\autoref{ptwist}.
 This coincides with the anima $\sqrt[\infty]{(X,\mathscr N)}(A)$. In particular, if $\mathscr N\simeq GL_1(\mathcal O_X)$, then $\gamma(X,\mathscr N)\simeq X.$
We then have to show that there is a morphism $ \sqrt[\infty]{(X,\mathscr N)} \rightarrow\gamma(X,\mathscr N)$, which gives the above equivalence. Dually, we will construct a map $\sqrt[\infty]{(X,\mathscr N)}^\sharp \rightarrow X$. Unwinding the definition of infinite root  stacks, there is an equivalence 
$\sqrt[\infty]{(X,\mathscr N)}^\sharp\simeq \colim_{\spf C}{\spf C}^\sharp$, where $\spf C$ runs over all formal log-affine objects lying over $\sqrt[\infty]{(X,\mathscr N)}$. So there is an equivalence 
$$\map({\sqrt[\infty]{(X,\mathscr N)}},\gamma((X,\mathscr N)))\simeq\lim\map({{\spf C}},\gamma(X,\mathscr N))\simeq \lim\map({({\spf C})}^\sharp,(X,\mathscr N)).$$
Denote by $\eta:\sqrt[\infty]{(X,\mathscr N)}\rightarrow X$.
Then by the definition of the functor $(-)^\sharp$, a commutative diagram of log stacks as follows
$$\begin{tikzcd}
... & {({\sqrt[\infty]{(X,\mathscr N)}},p^2\Omega^\infty\mathcal O_{\sqrt[\infty]{(X,\mathscr N)}})^a} \arrow[l] \arrow[rd, "\eta_2"'] & {({\sqrt[\infty]{(X,\mathscr N)}},p\Omega^\infty\mathcal O_{\sqrt[\infty]{(X,\mathscr N)}})^a} \arrow[d, "\eta_1"] \arrow[l] & {({\sqrt[\infty]{(X,\mathscr N)}},\Omega^\infty\mathcal O_{\sqrt[\infty]{(X,\mathscr N)}})^a} \arrow[ld, "\eta"] \arrow[l] \\
    &                                                                                                                                  & {(X,\mathscr N)}                                                                                                             &                                                                                                                           
\end{tikzcd}$$
 defines a map $\sqrt[\infty]{(X,\mathscr N)}^\sharp\rightarrow X$.  Then the induced  map $\eta:{\sqrt[\infty]{(X,\mathscr N)}}\rightarrow {\sqrt[\infty]{(X,\mathscr N)}}$ gives such a map by the definition of infinite root  stacks.
\end{proof}

%\begin{prop}
%Let $V=\spf R[M]^\wedge_p$. Then the map $\eta:\sqrt[\infty]{V}\rightarrow V$ induces an identity on maps of $\infty$-category of quasi-coherent sheaves 
%$$\eta_*\eta^*\simeq\mathrm{id}:\qcoh(V)\rr\qcoh(\sqrt[\infty]{V})\rr\qcoh(V).$$
%\end{prop}
%\begin{proof}
   % Let  $V^\infty:=\spf R[M[1/p]]^\wedge_p$. Proposition\autoref{localroot} shows that $\sqrt[\infty]{V}\simeq V^\infty/{\mu^M_{p^\infty}}$.
%Applying the projection formula, we only need to show that the unit map $u:\OO_{V}\rightarrow\eta_*\OO_{\sqrt[\infty]{V}}$ is an equivalence. Observe that the cohomology ${\rm R}\Gamma(\sqrt[\infty]{V},\OO_{\sqrt[\infty]{V}})$ is equivalent to the cohomology ${\rm R}\Gamma(B\mu^M_{p^\infty}, R[M[1/p]]^\wedge_p)$ because $\sqrt[\infty]{V}\simeq V^\infty/\mu^M_{p^\infty}$. Hence, we have to show that the following 
%$$u:R[M]^\wedge_p\rr{\rm R}\Gamma(B\mu^M_{p^\infty}, R[M[1/p]]^\wedge_p)$$
%is an equivalence. As both sides of $u$ are $p$-complete, the equivalence of $u$ is detected by modulo $p$. Using the base change formula for quasi-coherent sheaves, we can reduce to the case that $R$ is a $\mathbb F_p$-algebra. Without loss of generality, we can reduce to the case $R=\mathbb F_p$. In this case, $R[M]^\wedge_p\simeq\mathbb F_p[M]$ and 
%${\rm R}\Gamma(B\mu^M_{p^\infty}, R[M[1/p]]^\wedge_p)\simeq{\rm R}\Gamma(B\mu^M_{p^\infty}, \mathbb F_p[M[1/p]]).$ 
%\end{proof}
\subsection{Saturated descent}
Recall that a discrete monoid $M\in\mon^\heartsuit$ is called $p$-saturated if, for any $x\in M^{\rm gp}$, there exists a non-negative integer $n\in\Z_{\geq 0}$ such that $p^nx$ lies in the image of the group-completion map $M\rightarrow M^{\rm gp}$, then $x\in M$. In other words, the natural map $M\rightarrow M[1/p]\times_{M[1/p]^{\rm gp}}M^{\rm gp}$ is an isomorphism of monoids, where the fiber  product is formed in $\mon^\heartsuit$. In practice, this condition is required in almost all concrete applications in arithmetic and geometry. In this subsection, we give its topological generalization.
\begin{defn}\label{satdefn}
 let $M$ be an $\Einf$-monoid. Consider  the map $a_M:M[1/p]\rightarrow M[1/p]\times (M[1/p]^{\rm gp}/M^{\rm gp})$ induced from the diagonal map of $M$, and we regard it as the natural coaction of $M[1/p]^{\rm gp}/M^{\rm gp}$ on $M[1/p]$. Form the following cosimplicial diagram $C_M^\bullet$ induced from the coaction $a_M$:
 $$\begin{tikzcd}
M[1/p] \arrow[r, shift right] \arrow[r, shift left] & {M[1/p]\times(M[1/p]^{\rm gp}/M^{\rm gp})} \arrow[l] \arrow[r, shift left=2] \arrow[r, shift right=2] \arrow[r] & {M[1/p]\times(M[1/p]^{\rm gp}/M^{\rm gp})^2} \arrow[l, shift left] \arrow[l, shift right] \arrow[r, shift left=3] \arrow[r, shift left] \arrow[r, shift right] \arrow[r, shift right=3] & ... \arrow[l, shift left=2] \arrow[l, shift right=2] \arrow[l]
\end{tikzcd}$$
We say that $M$ is\textit{ topologically $p$-saturated} if the augmented cosimplicial diagram $M\rightarrow C^\bullet_M$ is equivalent to the
left Kan extension of the map $M\rightarrow M[1/p]$ along the inclusion $\Delta^{\leq 0}_+\rightarrow\Delta$. Equivalently, it is equivalent to say that:
\begin{enumerate}
    \item The augmented cosimplicial diagram is a limit diagram;
    \item The following diagram 
    $$\begin{tikzcd}
M \arrow[r] \arrow[d]                & {M[1/p]} \arrow[d, "a_M"]                  \\
{M[1/p]} \arrow[r, "{({\rm id},0)}"] & {M[1/p]\times(M[1/p]^{\rm gp}/M^{\rm gp})}
\end{tikzcd}$$
is a pushout square.
\end{enumerate}
\end{defn}
Our main result in this subsection is the following.
\begin{thm}[Saturated descent]\label{satdescent}
  Let $(A,M)$ be a $p$-complete $\Einf$-prelog ring, such that $M$ is topologically $p$-saturated, and let $K$ be a $p$-complete almost connective module over $A$. Then the map 
  $$K\rr \lim_{\bullet\in\Delta}(K\widehat{\otimes}_AA[C_M^\bullet]^\wedge_p)$$
  is an equivalence.
\end{thm}
\begin{rmk}[Saturated descent as descent along root stacks]
  Fix a $p$-complete $\Einf$-prelog ring $(A,M)$. Proposition\autoref{localroot} shows that the simplicial diagram $\spf A[C^\bullet_M]^\wedge_p$  gives rise to an atlas  of the infinite root stack $\sqrt[\infty]{(A,M)}$ in the fpqc topology. Fix a $p$-complete $A$-module $K$, the simplicial diagram 
  $K\widehat{\otimes}_AA[C^\bullet_M]^\wedge_p$ computes the cohomology ${\rm R}\Gamma(\sqrt[\infty]{(A,M)},\eta^* K)\simeq \eta_*\eta^* K$. Here, $\eta$ is the natural map $\sqrt[\infty]{(A,M)}\rightarrow \spf A$.
  So the equivalence in Theorem\autoref{satdescent} is equivalent to say that the pullback functor 
  $$\eta^*:\qcoh^{\rm acn}(\spf A)\rr \qcoh^{\rm acn}(\sqrt[\infty]{(A,M)})$$
  is fully faithful.
\end{rmk}
We will give the proof of Theorem\autoref{satdescent} in this section.
Before doing that, we have to do some preparation.
\begin{defn}
Let $f: R\rightarrow A$ be a map of connective $\Einf$-rings, and let $A^\bullet$ be its \v Cech nerve.   We say that:
\begin{enumerate}
    \item $f$ is \textit{convergent} if the natural map $R\rightarrow\lim_{\bullet\in\Delta}A^\bullet$ is an equivalence;
    \item $f$ is \textit{linearly convergent} if the natural functor 
    $\Mod_R^{\rm acn}\rightarrow \lim_{\bullet\in\Delta}\Mod^{\rm acn}_{A^\bullet}$ is an equivalence.  Here, we let $\Mod^{\rm acn}$ be the $\infty$-category of almost connective modules.
\end{enumerate}
\begin{exmp}\label{dsahuikmnbvfghjjkmnbg}
We have the following  basic examples of linearly convergent maps.
\begin{enumerate}
    \item 
     Descendable maps in the sense of \cite[section 3]{mathew2016galoisgroupstablehomotopy}, are linearly convergent.
    \item 
     Faithfully flat maps are linearly convergent.
     \item  The map $\Z_p\rightarrow \Z/p$ is \textit{$p$ completely linearly convergent}. Indeed, the base change functor $(-)\otimes\Z/p:\Mod^\wedge_{\Z_p}\rightarrow \Mod_{\Z/p}$ is conservative and preserves
     both colimits and limits because $\Z/p$ is perfect as a $\Z_p$-module. In other words, 
     it is both monadic and comonadic.
     Then  the linear convergence follows from the Barr-Beck-Lurie theorem.
     \end{enumerate}
\end{exmp}
\end{defn}
\begin{lem}\label{qwertyuiop}
The class of  linearly convergent maps defines a subcanonical and finitary Grothendieck topology on the $\infty$-category $\alg^{\rm cn,op}$. We call it the linearly convergent topology.
\iffalse
    We say a map of $\Einf$-ring $f:R\rightarrow A$ has descent if $R$ is equivalent to the totalization of the {\v C}ech nerve of $f$. Give a sequence of maps $R\rightarrow A\rightarrow B$. Assume that both $R\rightarrow A$ and $A\rightarrow B$ have descent, and for any map $A\rightarrow C$, if the map $C\rightarrow B\otimes_AC$ has descent, then $R\rightarrow B$ has descent.
    \fi
\end{lem}
\begin{proof}
We will apply \cite[Proposition A.3.2.1]{lurie2018spectral} to prove the assertion.\\
(1). The class of linearly convergent maps is stable under base change in $\algcn$.\\
(2). The class of linearly convergent maps is stable under composition. Let $R\rightarrow A\rightarrow B$ be a composition of linearly convergent maps. We want to show that the composition is also linearly convergent. Using \cite[Proposition 5.2.2.36]{lurie2017higher}, we can reduce to check for descendability of almost connective modules along the map $R\rightarrow B$. 
We first observe that the functor $\Mod^{\rm acn}_R\rightarrow\Mod_B^{\rm acn}$ is conservative. Then, we have to show that for an almost connective $R$-module $K$, it has descent along the map $R\rightarrow B$.
 We denote by
$$C^\bullet_{A/R},C^\bullet_{B/R},C^\bullet_{B/A}$$
    the {\v C}ech nerves of the functors $R\rightarrow A$, $\Mod^{\rm acn}_R\rightarrow B$ and $A\rightarrow B$ respectively. 
    Consider the map $A\rightarrow C^n_{A/R}$, and its pushout $C^{n}_{A/R}\rightarrow B\otimes_AC^n_{A/R}$
    along the map $A\rightarrow B$.
     We have the following pushout square
    $$\begin{tikzcd}
A \arrow[d] \arrow[r] & C^{n}_{A/R} \arrow[d] \\
B \arrow[r]           & B\otimes_AC^{n}_{A/R}
\end{tikzcd}$$
    Denote by $C^{\bullet,n}$ the {\v C}ech nerve of this map, and therefore we get a bi-cosimplicial object $\Delta\times\Delta\ni(n,m)\mapsto C^{n,m}$. Note that the diagonal of $C^{\bullet,\bullet}$ is equivalent to $C^\bullet_{B/R}$. Thus, we have the following equivalences 
    $$\lim_{\bullet\in\Delta}(M\otimes_RC^\bullet_{B/R})\simeq\lim_{(p,q)\in\Delta\times\Delta}(M\otimes_RC^{p,q})\simeq\lim_{p\in\Delta}\lim_{q\in\Delta}(M\otimes_RC^{q,p}).$$
    As the map $C^{n}_{A/R}\rightarrow B\otimes_AC^n_{A/R}$ is linearly convergent by (1), we know that 
$$\lim_{p\in\Delta}\lim_{q\in\Delta}(M\otimes_RC^{p,q})\simeq\lim_{p\in\Delta} M\otimes_RC^p_{A/R})\simeq M.$$
Then the result follows.\\
(3). The class of linearly convergent maps is stable under finite products. This is clear because the totalization and  $\Mod^{\rm acn}_{(-)}$ commute with finite products.\\
(4). The coproducts are universal in $\alg^{\rm cn,op}$. This is obvious.\\
Then there is a unique Grothendieck topology on $\alg^{\rm cn,op}$ given as follows: for any $R\in\algcn$, a sieve $\CC^0_{R}$ on $\spec R$ is a covering sieve if and only if it contains finitely many affine schemes $\spec A_i,i=1,2,...,n$, such that the map $R\rightarrow A_1\times A_2\times...\times A_n$ is linearly convergent.
\end{proof}
\begin{rmk}
    Similar to the fpqc topology, the linearly convergent topology is also admits arbitrary large covers, and the subcateogry $\shv_{\rm lcon}\subset\Fun(\algcn,\ani)$ spanned by sheaves in the linearly convergent topology might be not presentable. However, this issue disappears under some mild set-theoretic assumptions. For example, we fix two regular cardinal $\kappa<\kappa'$. We let $\algcn$ be the  $\infty$-category of $\kappa$-small connective $\Einf$-rings, and replace the target $\ani$ of  $\kappa$-small animae with $\widehat{\ani}$ the $\infty$-category of $\kappa'$-small animae. Just as before, we omit set theory issues.
\end{rmk}
Then we have the following naive result.
\begin{prop}\label{cxfnfdsdszxmncxdhsbxvdgxv}
The functor $\Fun(\algcn,\widehat{\ani})^{\rm op}\rightarrow \widehat\cat, X\mapsto\qcoh^{\rm acn}(X)$ factors through the localization $\Fun(\algcn,\widehat{\ani})\rightarrow \widehat\shv_{\rm lcon}$. The resulting functor 
$\widehat{\shv}_{\rm lcon}\rightarrow\widehat{\cat}$ is limit-preserving. Moreover, a map $R\rightarrow A$ of connective $\Einf$-rings is linearly convergent if and only if the corresponding affine scheme map $\spec A\rightarrow\spec R$ is surjective in $\widehat\shv_{\rm lcon}$
\end{prop}
\begin{proof}
    The same argument with \cite[Proposition 6.2.3.1]{lurie2018spectral} will give the first assertion. The second assertion is obvious.
\end{proof}
\begin{lem}\label{asdfghjkl}
Let $A$ be a connective $\Einf$-ring. Then
the $0$-th truncation map $A\rightarrow\pi_0 A$ is linearly convergent. In particular, we have that 
    for any $\Einf$-monoid $M$, the map $R[M]\rightarrow R[\pi_0M]$ is linearly convergent.
\end{lem}
\begin{proof}
    This is immediately deduced from Dundas and Rognes' theorem of cosimplicial descent  \cite[Theorem 1.2]{MR3873116}.  But we can give a direct argument as follows.\\
    We first show that the base change functor $\Mod_A^{\rm acn}\rightarrow\Mod_{\pi_0 A}^{\rm acn}$ is conservative. Fix an almost connective $A$-module $K$, and assume that $K\otimes_A\pi_0 A\simeq 0$.
    If $A$ is $n$-truncated for some $n$, the map $A\rightarrow\pi_0 A$ is a descendable map, and it is automatically conservative for the base change of modules by \cite[Proposition 3.19]{mathew2016galoisgroupstablehomotopy}.
    For a general $A$, we have that 
    $$K\otimes_A \pi_0 A\simeq M\otimes_A\tau_{\leq m}A\otimes_{\tau_{\leq m }A}\pi_0 A\simeq 0.$$
    It follows that the term $K\otimes_A\tau_{\leq m} A\simeq 0$. By the almost connectivity of $K$. We know that $K\simeq \varprojlim_m (K\otimes_A\tau_{\leq m} A)\simeq 0$.\\
    We then want to show that the natural map $\theta: K\rightarrow\lim_{\bullet\in\Delta}(K\otimes_A A^\bullet)$  is an equivalence.
    We have the following commutative diagram 
    $$\begin{tikzcd}
K \arrow[d, "\simeq"'] \arrow[r]                             & \lim_{\bullet\in\Delta}(K\otimes_A A^\bullet) \arrow[d]                          \\
\varprojlim_m(K\otimes_A \tau_{\leq m}A) \arrow[r, "\simeq"] & \varprojlim_m\lim_{\bullet\in \Delta}(K\otimes_A A^\bullet\otimes_A\tau_{\leq m}A)
\end{tikzcd}$$
   Here, the left vertical map is an equivalence because of the almost connectivity of $K$, and the lower horizontal map is also an equivalence because the diagram $K\otimes_A A^\bullet\otimes_A\tau_{\leq m}A$ is equivalent to the {\v C}ech nerve of the $\tau_{\leq m} A$-module $K\otimes_A\tau_{\leq m} A$ along the base change $\tau_{\leq m}A\rightarrow \tau_{\leq m}A\otimes_A \pi_0 A$, which is a descendable map. While, on the other hand, the right vertical map is also an equivalence because limits commute with limits. It follows that the upper horizontal map is an equivalence. The result follows.
    \iffalse 
$$\begin{tikzcd}
M \arrow[d] \arrow[r]    & \pi_0M \arrow[d] \\
\tau_{\leq n}M \arrow[r] & N               
\end{tikzcd}$$
It then follows that the map $R[\tau_{\leq n}M]\rightarrow R[N]$ is a descendable map.
        Denote by $A^\bullet$ the {\v C}ech nerve of the map $R[M]\rightarrow R[\pi_0M]$, and denote by $A^\bullet_n$ the {\v C}ech nerve of the map $R[\tau_{\leq n}M]\rightarrow R[N]$.
The cosimplicial diagram $A^\bullet_n$ is eventually constant; therefore, there is a natural number $r_n$ such that $$\lim_{\bullet\in \Delta} A_n^\bullet \simeq \lim_{\bullet\in \Delta^{\leq r_n}} A_n^\bullet\simeq \lim_{\bullet\in \Delta^{\leq r_n+1}} A_n^\bullet\simeq ......$$
We form the following commutative diagram
$$\begin{tikzcd}
{R[M]} \arrow[r] \arrow[d]              & \lim_{\bullet\in\Delta^{\leq r_n} }A^\bullet \arrow[d] \\
{R[\tau_{\leq n}M]} \arrow[r, "\simeq"] & \lim_{\bullet\in\Delta^{\leq r_n} }A^\bullet_n        
\end{tikzcd}$$
where the lower horizontal map is given by base change of the upper horizontal map along $R[M]\rightarrow R[\tau_{\leq n }M]$.
\fi   
\end{proof}
\begin{lem}
 Let $J\rightarrow\algcn_A,j\mapsto A_j$ be a filtered diagram of descendable algebras over $A$
 \iffalse, such that all transition maps are descendable \fi. Denote by $B$ the filtered colimit $\colim_{j\in J}A_j$.
 Assume that the base change $\Mod^{\rm acn}_A\rightarrow\Mod^{\rm acn}_B$ is conservative. 
 Then $A\rightarrow B$ is linearly convergent.
\end{lem}
\begin{proof}
 Denote by $A^\bullet_j$ the \v Cech nerve of the map $A\rightarrow A_j$ for $j\in J$, and denote by $B^\bullet$ the \v Cech nerve of the map $A\rightarrow B$. 
 Note that, as cosimplicial objects of $A$-algebras, the maps $A_j^\bullet\rightarrow B, j\in J$ induce an equivalence $\colim_{j\in J} A^\bullet_j\simeq B^\bullet$.
 We want to show the following fact: let $K$ be a module over $A$. Then, the natural map $K\rightarrow\lim_{\bullet \in\Delta}( B^\bullet\otimes_A K)$ is an equivalence. \\
 We first assume that $K$ is coconnective. We have the following maps:
 $$\lim_{\bullet \in\Delta}( B^\bullet\otimes_A K)\simeq \lim_{\bullet \in\Delta}( (\colim_{j\in J}A_j^\bullet)\otimes_A K)\longleftarrow\colim_{j\in J}\lim_{\bullet\in\Delta}(A^\bullet_j\otimes_AK)\simeq K.$$
 The second map is also an equivalence because the totalization of connective spectra commutes with the filtered colimit.\\
 We then assume that $K$ is arbitrary. We can write it as a limit $K\simeq\varprojlim_n\tau_{\leq m}K$ of almost coconnective objects. Then, we have the following equivalences
 $$\lim_{\bullet\in\Delta}(B^\bullet\otimes_AK)\simeq \lim_{\bullet\in\Delta}\lim_n(B^\bullet\otimes_A\tau_{\leq n}K)\simeq \lim_n\lim_{\bullet \in\Delta}(B^\bullet\otimes_A\tau_{\leq n}K)\simeq K.$$
 The result follows.
 \end{proof}
 \begin{cor}\label{descentalongfrobenius}
     Let $A$ be a $p$-complete connective  $\Einf$-ring. Assume that it admits a self map $F:A\rightarrow A$, which lifts the Frobenius ${\rm Fr}: \pi_0A/p\rightarrow\pi_0 A/p$. Then the induced map $F_\infty:A\rightarrow \colim^\wedge_F A$ is $p$-completely. linearly convergent. In particular, for any $\Einf$-monoid, the map $R[M]^\wedge_p\rightarrow R[M[1/p]]^\wedge_p$ is $p$-completely linearly convergent.
 \end{cor}
\begin{proof}
By Proposition\autoref{cxfnfdsdszxmncxdhsbxvdgxv}, we just need to show that $F_\infty$ is surjective in $\widehat{\shv}_{\rm lcon}$.
By Lemma\autoref{asdfghjkl}, we can assume that $A$ is discrete, i.e., a classical $p$-complete ring.
By Example\autoref{dsahuikmnbvfghjjkmnbg}(3), we can assume that $A$ is an $\mathbb Z/p$-algebra. 
In this case, $\colim^\wedge_F A$ is nothing but the inductive  perfection $A_{\rm perf}$ of $A$.
But we can also assume that $A$ is semi-perfect. Indeed, we let $A'$ be the ring obtained by formally joining $p^\infty$-roots of elements in $A$. Namely, we define
$$A':= A[X_a^{1/p^\infty};a\in A]/(X-a;a\in A).$$
We observe that $A\rightarrow A'$ is linearly convergent because it is a quasi-syntomic cover. The map $A_{\rm perf}\rightarrow A'_{\rm perf}$ is an isomorphism. Now, we assume that $A$ is semiperfect. In this case, the Frobenius self map $F:A\rightarrow A$ is descendable, as its fiber coincides with the kernel, which is nilpotent.\\
The only thing we need to show is that the base change along $A\rightarrow A_{\rm perf}$ is conservative for almost connective modules.
This is deduced from the following Lemma\autoref{xcvfmdsajskxmczdnsj}.
\end{proof}
\begin{lem}\label{xcvfmdsajskxmczdnsj}
    Let $R\rightarrow A$ be a universal homeomorphism of discrete rings. Then the base change functor $\Mod_R^{\rm acn}\rightarrow\Mod_A^{\rm acn}$ is conservative.
\end{lem}
\begin{proof}
    Let $K\in\Mod_R^{\rm acn}$. Assume that $K\otimes_RA\simeq 0$. Without loss of generality, we can assume that $K$ is connective. We observe that, for any prime ideal $\mathfrak p\in\spec R$, the tensor product $K\otimes_R k(\mathfrak p)\simeq 0$. Indeed, the base change $(K\otimes_Rk(\mathfrak p))\otimes_{k(\mathfrak p)}(A\otimes_Rk(\mathfrak p))\simeq (K\otimes_RA)\otimes_A(A\otimes_Rk(\mathfrak p))\simeq 0$. It turns out  that $K\otimes_Rk(\mathfrak p)\simeq 0$ because the composed map $k(\mathfrak p)\rightarrow A\otimes_Rk(\mathfrak p)\rightarrow \pi_0(A\otimes_Rk(\mathfrak p))_{\rm red}$ is faithfully flat.\\
    We then reduce to show that for an almost connective module $K$ over $R$, such that for any prime ideal $\mathfrak p\in\spec R$, the tensor product $K\otimes_Rk(\mathfrak p)\simeq 0$, then $K\simeq 0$. We want to show that $\tau_{\leq n}K\simeq 0$ for any $n\in\N$. Consider the exact sequence 
    $$\tau_{\geq n+1}K\rr K\rr \tau_{\leq n}K.$$
    After tensoring with $k(\mathfrak p)$, we have that $(\tau_{\leq n}K)\otimes_R k(\mathfrak p)\simeq (\tau_{\geq n+1}K)\otimes_Rk(\mathfrak p)[1]$, which is $(n+2)$-connective. If $n=0$, it's obvious that $\tau_{\leq 0}K=0$, because it vanishes at all prime ideals. It turns out that $K$ is $1$-connective. Then replace $K$ by $K[-1]$, we also have $0\simeq\tau_{\leq 0}(K[-1])\simeq(\tau_{\leq 1}K)[1]$. Passing to the induction, we know that $K$ is $\infty$-connective, i.e., it is vanishing.
\end{proof}
Now, we can prove  Theorem\autoref{satdescent}.
\begin{proof}[Proof of Theorem\autoref{satdescent}]
Unwinding the definition,  if $M$ is $p$-saturated, for an almost connective module $K$ over $A$, the term $\lim_{\bullet\in\Delta} K\otimes_A A[C^\bullet_M]$ is equivalent to $\lim_{\bullet\in \Delta} K\otimes_A A^\bullet$, where $A^\bullet$ is the \v Cech nerve of the map $A\rightarrow(A\otimes_{\mathbb S[M]}\mathbb S[M[1/p]])^\wedge$. Applying Lemma\autoref{descentalongfrobenius}, we get the following equivalences
$$\lim_{\bullet\in\Delta} K\otimes_A A[C^\bullet_M]\simeq \lim_{\bullet\in \Delta} K\otimes_A A^\bullet\simeq K.$$
We get the desired result.
\end{proof}

\begin{defn}
    An $\Einf$-log ring $(A,M)$ is called topologically $p$-saturated if it is a logification of  a topologically $p$-saturated $\Einf$-prelog ring.\\
    A quasi-coherent spectral log stack $X$ is called topologically $p$-saturated if, for any map $f:\spec A\rightarrow\underline X$, the pullback $f^*X$ is fpqc locally equivalent to $\spet(A,M)$ for some $(A,M)$, where $(A,M)$ is topologically $p$-saturated.
\end{defn}
\subsection{Cotangent complexes via infinite root  stacks}
In this section, we will prove that the $p$-completion of the log cotangent complex of a topologically $p$-saturated quasi-coherent
 spectral log stack over $\spec R$, can be recovered from its infinite root  stack. 
\begin{lem}\label{7878744}
  Fix a stack $S$.  Let $\CC\rightarrow\Fun(\Delta^1,\stk)_{/S}$ be a diagram of stacks lying over $S$. Assuming that for any $i\in\CC$, the relative stack $X_i/S$ admits a cotangent complex $\LL_{X_i/S}$. If all $\LL_{X_i/S}$ are uniformly bounded below, then the limit $X:=\lim_{i\in \CC} X_i\rightarrow S$ admits a cotangent complex, and the following natural map $$\colim_{i\in\CC} f_i^*\LL_{X_i/S}\rr\LL_{X/S}$$
  is an equivalence. Here $f_i:X\rightarrow X_i$ is the canonical map induced from the limit.
\end{lem}
\begin{proof}
  Without loss of generality, we might as well assume that $S=\spec A$, and thus we work with the $\infty$-category $\shv_{\et}(\algcn_A,\ani)\simeq\stk_{/\spec A}$.  Let $B$ be an $A$-algebra, and $I$ be a connective $B$-module. As each $X_i$ admits a cotangent complex $\LL_{X_i/A}$, we have that the fiber of the map $X(B\oplus I)\rightarrow X(B)$ over a point $x:\spec B\rightarrow X$ is equivalent to 
  \begin{align*}
  \lim\map_{B}(x^*f_i^*\LL_{X_i/A},I)&\simeq\map_B(\colim x^*f_i^*\LL_{X_i/A},I) \\
  &\simeq\map_B(x^*\colim f^*_i\LL_{X_i/A},I)
  \end{align*}
  Applying the lower bound assumption of $\{\LL_{X_i/S}\}_{i\in\CC}$, $\colim f^*_i\LL_{X_i/A}$ is almost connective. Therefore, it's a cotangent complex of $X/A$.
\end{proof}
We have the following corollary.
\begin{cor}
    The relative stack $\eta:\rtr\rightarrow\spf R$ admits a cotangent complex.
\end{cor}
\begin{lem}\label{7878755}
  Let $A$ be a $p$-complete $\Einf$-algebra over $R$, then the associated map  $f:\spf A\rightarrow\spf R$ admits a cotangent complex, which is equivalent to $(\LL_{A/R})^\wedge_p$, the $p$-completion of the cotangent complex $\LL_{A/R}$.  
\end{lem}
\begin{proof}
   We observe that $\spf A\simeq \spec A\times_{\spec R}\spf R $. It follows that 
   $\LL_{\spf A/\spf R}$ is equivalent to the pullback of $\LL_{A/R}$ along the map $\spf A\rightarrow\spec A$.
\end{proof}
\begin{lem}\label{71717}
    For any  $\Einf$-monoid $M$ having the form $M=N[1/p]$, the $\Einf$-algebra $R\rightarrow R[M]^\wedge_p$ is $p$-completely formally \'etale, i.e., $(\LL_{R[M]^\wedge_p/R})^\wedge_p$ is contractible. 
\end{lem}
\begin{proof}
 As $\LL_{R[M]^\wedge_p/R}$ is connective, the contractibility can be detected by taking the  base change to $R\rightarrow\pi_0 R$. We then assume that $R$ is discrete.
 Moreover, we can assume that $R=\Z_p$.
 By \cite[Corollary 2.10]{antieau2024sphericalwittvectorsintegral}, we can also assume that $M$ is discrete. As $R[M[1/p]]^\wedge_p$ is a discrete $p$-complete ring, the contractibility of modules is detected by passing to the base change to the mod $p$ fiber. Because $\mathbb F_p[M[1/p]]$ is perfect, its cotangent complex vanishes.
\end{proof}
\begin{lem}\label{descentofcotangent}
    Let $R$ be an $\Einf$-ring or an animated ring, regard the topological (resp. algebraic ) cotangent complex as a functor $\LL^{(\Delta)}_{(-)/R}:\alg^{(\Delta)}_R\rightarrow\Mod_R$. Then, one has:
    \begin{enumerate}
        \item In the animated setting, $\LL^{\Delta}_{(-/R)}$ has descent in the fpqc topology;
        \item In the $\Einf$-setting,
        $\LL_{(-/R)}$ has hyperdescent in the fpqc topology.
    \end{enumerate}
\end{lem}
\begin{proof}
    Assertion $(1)$ is proved in \cite[Corollary 2.7]{bhatt2012derham}. We now assume that $R$ is an $\Einf$-ring. Recall that there is a concrete description of cotangent complexes in terms of indecomposables: the cotangent complex $\LL_{A/R}$ of $R\rightarrow A$ is canonically equivalent to ${\rm cofib}(I_A\otimes_A I_A\rightarrow I_A)$, where $I_A$ is the fiber of the multiplication map $A\otimes_RA\rightarrow A$. Let $A\rightarrow A^\bullet$ be a faithfully flat hypercover of $A$, and therefore we form an exact sequence of cosimplicial diagrams in $\Mod_R$:
    $$I_{A^\bullet}\rr A^\bullet\otimes_R A^\bullet\rr A^\bullet.$$
    Taking the limit, we have that $$\lim_{\Delta}(A^\bullet\otimes_R A^\bullet)\simeq\lim_{i,j\in\Delta\times\Delta}(A^i\otimes_RA^j)\simeq\lim_{j\in\Delta}\lim_{i\in \Delta} (A^i\otimes_RA^j),$$
    because $\Delta$ is sifted. Fix $j$, the functor $i\mapsto A^i\otimes_RA^j$ is an fpqc hypercover of $A\otimes_RA^j$. It follows that the limit of the above diagram is equivalent to $A\otimes_R A$. This shows that the functor $A\mapsto I_A$ has hyperdescent in the fpqc topology. Then, the hyperdescentness of the functor $A\mapsto I_A\otimes_AI_A$ can be deduced from the following exact sequence
    $$I_A\otimes_AI_A\rr A\otimes_RA\otimes_AI_A\simeq A\otimes_RI_A\rr I_A.$$
 We conclude that $\LL_{(-)/R}$ has fpqc hyperdescent.
\end{proof}
\begin{lem}\label{cotaslim}
    Let $X/S$ be a relative stack in the fpqc topology, with a cotangent complex $\LL_{X/S}$. Assume that $X$ satisfies one of the following conditions:
    \begin{enumerate}
        \item The diagonal $\Delta_{X/S}:X\rightarrow X\times_SX$
        is representable by a relative algebraic space;
        \item  $X$ can be written in the form of $X\simeq\colim_{j\in J}\spec R_j$, for which any $\spec R_j$ is formally \'etale over $X$, and $J$ is sifted.
    \end{enumerate}
        Then the following natural
map
$$\LL_{X/S}\stackrel{\simeq}\rr\varprojlim_{(A/B,s)}s'_*\LL_{A/B}$$
is an equivalence,   where the index runs over the $\infty$-category $\aff_{X/S}$ consisting of all     commutative diagrams
    $$\begin{tikzcd}
\spec A \arrow[r, "s'"] \arrow[d] & X \arrow[d] \\
\spec B \arrow[r, "s"]            & S          
\end{tikzcd}$$
Here the limit is formed in $\qcoh(X)$. 
\end{lem}
\begin{proof}
According to the base change formula for the cotangent complex functor, without loss of generality, one can assume that $S\simeq\spec B$ is affine. Denote by $\aff_{/X}$ the $\infty$-category of spectral affine schemes lying over $X$.
For any map $s:\spec A\rightarrow X$, consider the conormal exact sequence 
$$s_*s^*\LL_{X/B}\rr s_*\LL_{A/B}\rr s_*\LL_{A/X}.$$
 Passing to the limit, we get the following new exact sequence of quasi-coherent sheaves over $X$,
$$\LL_{X/R}\rr\varprojlim_{(A,s)}s_*\LL_{A/R}\rr\varprojlim_{(A,s)}s_*\LL_{A/X}.$$
We want to show that the last term $\varprojlim_{(A,s)}s_*\LL_{A/X}$ is vanishing. \\
Assume condition $(1)$.
Denote by $\Mod(X)$ the $\infty$-category  of $\OO_X$-modules, c.f. \cite[Notation 17.2.4.1]{lurie2018spectral}. The fully faithful embedding $i_X:\qcoh(X)\rightarrow\Mod(X)$ preserves small colimits and thus admits a right adjoint $Q_X:\Mod(X)\rightarrow\qcoh(X)$.
    We denote by $\mathcal{ L}:=\varprojlim^{\Mod(X)}_{(A,s)}s_*\LL_{A/X}\in\Mod(X)$. Note that the  $\OO_X$-module $\mathcal{L}$ is not quasi-coherent in general, but there is an equivalence    $$Q_X(\mathcal{L})\simeq\varprojlim_{(A,s)}s_*\LL_{A/X}.$$
We want to show that $\mathcal{L}\simeq 0$.  Let $C\in\aff_{/X}$, we have the following equivalences
  \begin{align*}
  \mathrm{R}\Gamma(\spec C,\mathcal{L})&\simeq\varprojlim_{(A,s)} \mathrm{R}\Gamma(\spec C,s_*\LL_{A/X})\\
  &\simeq\varprojlim_{(A,s)}  \mathrm{R}\Gamma(\spec C\times_X\spec A,\LL_{\spec C\times_X\spec A/\spec C})\\
  &\simeq{\rm R}\Gamma(\colim_{(A,s)}(\spec C\times_X\spec A),\LL_{-/C}).
  \end{align*}
 The second equivalence holds because the diagonal is representable by a relative algebraic space.
 Note that the natural map $\colim_{(A,s)}(\spec C\times_X\spec A)\rightarrow\spec C$ is an equivalence. Then we have
 \begin{align*}
     {\rm R}\Gamma(\colim_{(A,s)}(\spec C\times_X\spec A),\LL_{-/C})&\simeq   {\rm R}\Gamma(\spec C,\LL_{-/C})\simeq 0.
 \end{align*}
 Assume condition $(2)$. For any map $\spec B\rightarrow X$, there exists a faithfully flat map $\spec B\rightarrow\spec B'$, and some $j\in J$, such that the induced map $\spec B'\rightarrow X$ factors through $\spec R_j$. We therefore have an exact sequence of $B'$-modules:
$$\LL_{B/X}\otimes_BB'\rr\LL_{B'/X}\simeq\LL_{B'/R_j}\rr\LL_{B'/B}.$$
Applying Lemma \autoref{descentofcotangent}, the faithfully flat descent of cotangent complexes for $\Einf$-rings, the term $\LL_{B/X}$ is equivalent to the limit of the diagram $\LL_{B^\bullet/R_j}$, where $B^\bullet$ is the {\v C}ech nerve of $(B\rightarrow B')$. Let $\CC$ be the $\infty$-category of triples $(x:\spec C\rightarrow X,x':\spec C\rightarrow\spec R_i,i)$, where $x'$ is a lifting of $x$. Observe that the forgetful functor $\CC\rightarrow\mathbf{Aff}_{/X}$ is cofinal, and we can compute the limit $\varprojlim_{(B,s)} s_*\LL_{B/X}$ using $\CC$.  On the other hand, the subcategory of $\CC$ spanned by objects having the form $(\spec R_j\rightarrow\spec R_j,\spec R_j\rightarrow\spec R_j,j)$, which is equivalent to $J$ is also cofinal. Thus we have that 
$$\varprojlim_{(B,s)} s_*\LL_{B/X}\simeq\lim_{j\in J}s_*\LL_{R_j/X}\simeq0.$$
The proof is complete.

\end{proof}
As a corollary, we have the following result.
\begin{prop}\label{719719}
    The relative stack $\eta:\rtr\rightarrow\spf R$ is formally \'etale.
\end{prop}
\begin{proof}
    Using Lemma\autoref{71717} and \autoref{cotaslim}, the canonical map $$\theta:\LL_{\rtr/\spf R}\rr\lim_{s\in\mathbf{FMon}_R}s_*\LL^\wedge_{R[M[1/p]]_p^\wedge/R}\simeq 0$$
    is an equivalence.
\end{proof}

Now we can prove our main theorem.
\begin{thm}\label{123256456}
 Let $R$ be a $p$-complete $\Einf$-ring.   Let $X$ be a topologically $p$-saturated quasi-coherent spectral log Deligne-Mumford stack over $\spec R$, and let $\underline{X}^\wedge_p\rightarrow\logfr$ be the associated $p$-completion. 
    Then the  map 
    $$\phi:\LL_{X/R}^{\rm Log}\rr\eta_*\LL_{\sqrt[\infty]{X_p^\wedge}/\spf R}$$
     induced from the canonical map $\eta:\sqrt[\infty]{X^\wedge_p}\rightarrow\underline{X}^\wedge_p$
    becomes an equivalence after $p$-completion. In particular, the following map
    $$\phi:{\rm R}\Gamma(\underline{X},\LL^{{\rm Log}}_{{X}/R})^\wedge_p\rr{\rm R}\Gamma(\sqrt[\infty]{X},\LL_{\sqrt[\infty]{X^\wedge_p}/\spf R})$$
    is an equivalence.
\end{thm}
\begin{proof}
Without loss of generality, we might as well  assume that $X\simeq\spet(A,M)$ is log-affine. We only need to show that the following natural map 
$$\phi:(\LL^{G}_{(A,M)/R})^\wedge_p\rr\eta_*\LL_{\sqrt[\infty]{\spf(A,M)}/\spf R}$$ is an equivalence. This follows from Theorem\autoref{satdescent} and Proposition\autoref{719719}.
\end{proof}
\bibliographystyle{alpha}
\bibliography{ref}

\end{document}